%% file: DY_Mod_Cat_arXiv_v1.tex
\documentclass[12pt,a4paper]{article}

\usepackage{amsfonts,amsmath,amsthm}
\usepackage{graphicx}
\usepackage{amssymb}
\usepackage{mathrsfs}
\usepackage[top=2cm, bottom=2cm, left=2cm, right=2cm]{geometry}
\usepackage[all]{xy}
\usepackage{xcolor}
\usepackage{hyperref}
\usepackage{authblk}
\usepackage{enumitem}

\usepackage[normalem]{ulem}

\allowdisplaybreaks 

\newtheorem{theorem}{Theorem}[section]
\newtheorem{proposition}[theorem]{Proposition}
\newtheorem{corollary}[theorem]{Corollary}
\newtheorem{definition}[theorem]{Definition}
\newtheorem{lemma}[theorem]{Lemma}

\newtheorem{thIntro}{Theorem}
\newtheorem{coroIntro}[thIntro]{Corollary}
\newtheorem{propIntro}[thIntro]{Proposition}

\theoremstyle{remark}
\newtheorem{remark}[theorem]{Remark}
\newtheorem{example}[theorem]{Example}

\newcommand{\comp}{%
  \overline{\circ}%
}

\DeclareMathOperator{\Ext}{Ext}

\DeclareMathOperator{\Hom}{Hom}
\DeclareMathOperator{\End}{End}
\DeclareMathOperator{\Aut}{Aut}

\title{\bf Deformations of mixed associators \\in module  categories}
\author[1]{M. Faitg}
\author[2]{A.M. Gainutdinov}
\author[3]{C. Schweigert}
\author[ ]{J.-O. Willprecht}
\affil[1]{\small \textit{Institut de Mathématiques de Toulouse, Universit\'e Paul Sabatier, 118 route de Narbonne, 
\newline F-31062 Toulouse, France.}}
\affil[2]{\small \textit{Institut Denis Poisson, CNRS, Université de Tours, Universit\'e d’Orl\'eans, Parc de Grandmont, 37200 Tours, France}}
\affil[3]{\small \textit{Fachbereich Mathematik, Universität Hamburg, Bundesstra{\ss}e 55, 20146 Hamburg, Germany}}
\date{}

\begin{document}

\maketitle

\vspace{-2em}

\noindent {\small {\em E-mail adresses:} matthieu.faitg@math.univ-toulouse.fr, azat.gainutdinov@cnrs.fr,

\hspace{5.95em}christoph.schweigert@uni-hamburg.de}

\bigskip

\begin{abstract}
We study deformations of the mixed associator which is part of the
data of a module category $\mathcal{M}$ over a $\Bbbk$-linear monoidal category $\mathcal{C}$ for a field $\Bbbk$. We set up a 
cohomology theory $\mathrm{H}^\bullet_{\mathrm{mix}}(\mathcal{M})$ whose coefficients are $\mathcal{C}$-module
endofunctors. The underlying complexes control deformations of mixed associators, $\mathcal C$-module functors and their
obstructions.
We show that they are isomorphic to the Davydov--Yetter  (DY) complex of the representation functor $\rho : \mathcal{C} \to \mathrm{End}(\mathcal{M})$.
Using our previous results on DY cohomology [\href{https://arxiv.org/abs/2411.19111}{arXiv:2411.19111}], we prove for any finite tensor category $\mathcal{C}$ and any finite $\mathcal{C}$-module category $\mathcal{M}$ that the cohomology groups $\mathrm{H}^\bullet_{\mathrm{mix}}(\mathcal{M})$ are isomorphic to the relative Ext groups $\mathrm{Ext}^\bullet_{\mathcal{Z}(\mathcal{C}),\mathcal{C}}(\boldsymbol{1},\mathcal{A}_{\mathcal{M}})$ for the usual adjunction of the forgetful functor from the Drinfeld center $\mathcal{Z}(\mathcal{C})$ to $\mathcal{C}$, where $\mathcal{A}_{\mathcal{M}}$ is the so-called adjoint algebra of the module category $\mathcal{M}$. This allows us to give a dimension formula for $\mathrm{H}^n_{\mathrm{mix}}(\mathcal{M})$ in terms of certain Hom spaces in $\mathcal{Z}(\mathcal{C})$, and also to show that every mixed associator deformation of the regular module category $\mathcal{M} = \mathcal{C}$ is trivial. We also show that the algebra $\mathcal{A}_{\mathcal{M}}$ is the ``full center'' of an algebra in $\mathcal{C}$ realizing~$\mathcal{M}$. We furthermore prove a generalized version of Ocneanu rigidity for monoidal functors with coefficients, and provide its application to general (non-exact and non-finite) $\mathcal{C}$-module categories over a fusion category $\mathcal{C}$ such that $\dim(\mathcal{C}) \neq 0$.
As an illustration, we spell out these results for module categories defined by 
finite-dimensional comodule algebras over finite-dimensional Hopf algebras. Examples based on comodule 
algebras over Sweedler's Hopf algebra are worked out in detail; in particular we obtain new continuous families of inequivalent non-exact module categories.
\end{abstract}

\tableofcontents

\section{Introduction}

A {\em module category} $\mathcal{M}$ over a monoidal category $\mathcal{C}$ is the categorified version of the notion of a module $M$ over an algebra $A$. The latter is defined by an ``action map'' $\cdot : A \times M \to M$ such that $a \cdot (a' \cdot v) = (aa') \cdot v$ and $1_A \cdot v = v$. A standard procedure categorifies maps by functors while equalities are categorified by natural isomorphisms which have to satisfy further compatibility conditions. A module category is thus defined by an ``action functor'' $\rhd : \mathcal{C} \times \mathcal{M} \to \mathcal{M}$ together with a natural isomorphism $m$ having components of the form
\begin{equation}\label{mixedAssoIntro}
m_{X,Y,M} : X \rhd (Y \rhd M) \to (X \otimes Y) \rhd M
\end{equation}
and called {\em mixed associator}, which is required to satisfy the obvious coherence diagram (see \S\ref{subsecModCatModFun}). We also require that $\mathrm{Id}_{\mathcal{M}} = \boldsymbol{1} \rhd -$, and $m$ behaves well with this condition. The goal of this paper is to introduce the deformation theory of such mixed associators and to provide a computational tool for it which is based on relative homological algebra. For simplicity we assume everywhere that $\mathcal{C}$ is strict monoidal.

\smallskip

There are several reasons to study such deformations. In
the same way the study of modules is necessary to understand the
structure of rings and algebras, module categories lead to a deeper
understanding of monoidal categories. The class of {\em exact}
module categories is particularly well understood \cite{EGNO}; it
leads, for example, to the notion of Morita equivalent 
finite tensor categories.  It was moreover shown~\cite{CMZ} that exact module categories over a finite tensor category $\mathcal{C}$ is the same as semi-simple algebras in $\mathcal{C}$, and in many interesting cases there are even classification results~\cite{EO,Mombelli,NSS}. In contrast,
non-exact finite module categories  are poorly understood due to absence of appropriate techniques, despite their importance in studying extensions of exact module categories. In this context, it is 
 remarkable that our methods provide new deformation techniques applicable to any finite module category, not necessarily exact, notably Thm.\,\ref{thmAdjDYMod} together with the continuous family of examples of non-exact module categories constructed thanks to Prop.\,\ref{propCohomAmn}.

Module categories also naturally appear in applications in mathematical physics. A subclass of exact module categories, 
known as pivotal, describes boundary conditions
for 3d topological field theories of Turaev-Viro
type \cite{fusV}. It has been known for a long time \cite{fuRs4} that 
objects in exact module categories over modular fusion categories
describe boundary conditions in 2d rational conformal field
theories. There is a strong indication that this structure generalizes to non-semisimple or logarithmic theories:
an object in a pivotal module category $\mathcal{M}$ over a modular
finite tensor category should still describe a boundary condition, while inner 
Hom objects describe boundary fields, and the mixed
associator of $\mathcal{M}$ is responsible for the boundary OPEs. We refer to \cite[\S 2.2]{fuSc27}
for a more detailed discussion. Our results are thus also first
steps towards an algebraic deformation theory of 2d
conformal field theories \textit{via} their boundary OPEs.

\smallskip

\indent Let $\Bbbk$ be a field. In all the sequel we assume that $\mathcal{C}$ and $\mathcal{M}$ are $\Bbbk$-linear and that the monoidal product $\otimes: \mathcal{C} \times \mathcal{C} \to \mathcal{C}$ and the action $\rhd : \mathcal{C} \times \mathcal{M} \to \mathcal{M}$ are $\Bbbk$-bilinear bifunctors. 

Given an integer $N \geq 1$ and a formal variable $h$, it is natural to define a {\em deformation of the mixed associator $m$  of order $N$} as an expression $\mathbf{m} = m + \sum_{i=1}^N f_i h^i$ where each $f_i$ is a natural transformation of the form \eqref{mixedAssoIntro}, such that $\mathbf{m}$ satisfies the defining properties of a mixed associator over $\Bbbk[h]/\langle h^{N+1}  \rangle$; see \S\ref{sectionObstructionsMix} for the precise definition. The important case $N=1$ corresponds to {\em infinitesimal deformations of $m$} and is discussed first in \S\ref{subsecInfDef}. We note that the categories $\mathcal{C}, \mathcal{M}$ and the action $\rhd$ are not deformed; we only extend their scalars to $\Bbbk[h]/\langle h^{N+1} \rangle$ in order to make sense of $\mathbf{m}$.

\smallskip

As usual, infinitesimal deformations inspire the definition of a cochain complex in which they are cocycles. This complex, denoted by $\mathrm{C}^\bullet_{\mathrm{mix}}(\mathcal{M},m)$ or simply $\mathrm{C}^\bullet_{\mathrm{mix}}(\mathcal{M})$, is introduced in \S\ref{subsecMixAssoCohom}. We show that its second cohomology group $\mathrm{H}^2_{\mathrm{mix}}(\mathcal{M})$ classifies infinitesimal deformations of $m$ up to $\mathcal{C}$-module isomorphisms which are ``tangent'' to the identity, see Prop.\,\ref{relDefAssoAndCohom} for a precise statement. We furthermore show that the third cohomology group $\mathrm{H}^3_{\mathrm{mix}}(\mathcal{M})$ controls the obstructions for lifting a deformation to higher orders, in Prop.\,\ref{propLiftObstruction}.

\smallskip

More generally, we define a cochain complex $\mathrm{C}^\bullet_{\mathrm{mix}}(\mathcal{M};\mathsf{F},\mathsf{G})$ for any $\mathsf{F},\mathsf{G} \in \mathrm{End}_{\mathcal{C}}(\mathcal{M})$, the category of $\mathcal{C}$-module endofunctors of $\mathcal{M}$. The objects $\mathsf{F},\mathsf{G}$ are called {\em coefficients} and the deformation complex $\mathrm{C}^\bullet_{\mathrm{mix}}(\mathcal{M})$ is recovered in the special case when $\mathsf{F}=\mathsf{G}=\mathrm{Id}_{\mathcal{M}}$. These complexes with coefficients also have a deformation-theoretic interpretation: $\mathrm{H}^1_{\mathrm{mix}}(\mathcal{M};\mathsf{F},\mathsf{F})$ classifies infinitesimal deformations of the $\mathcal{C}$-module structure of $\mathsf{F}$ up to $\mathcal{C}$-module natural transformations, see~\S\ref{subsecMixAssoCohom}, while $\mathrm{H}^2_{\mathrm{mix}}(\mathcal{M};\mathsf{F},\mathsf{F})$ controls the obstruction to lift such deformations to higher orders, see Prop.\,\ref{propExtensionCmodStruct}.

\medskip

\indent Our strategy to study the mixed associator cohomology is to relate it to Davydov--Yetter (DY) cohomology, and to further use the adjunction theorem for DY cohomology \cite[\S2]{FGS2}.

\smallskip

Recall that the DY complex $\mathrm{C}^\bullet_{\mathrm{DY}}(\Gamma)$ of a monoidal functor $\Gamma : \mathcal{C} \to \mathcal{D}$ controls the deformations of the monoidal structure $\Gamma(-) \otimes \Gamma(-) \overset{\sim}{\implies} \Gamma(- \otimes -)$ of $\Gamma$ \cite{CY,davydov,yetter1}. In \cite{GHS} a version with coefficients $\mathrm{C}^\bullet_{\mathrm{DY}}(\Gamma;\mathsf{V},\mathsf{W})$ was introduced. The coefficients $\mathsf{V}, \mathsf{W}$ are objects in the {\em centralizer} $\mathcal{Z}(\Gamma)$, which is a generalization of the Drinfeld center consisting of pairs $(V,h)$ where $V \in \mathcal{D}$ and $h : V \otimes \Gamma \Rightarrow \Gamma \otimes V$ is a half-braiding relative to images of $\Gamma$. See \S\ref{subsubDYCoeff} for a reminder on DY cohomology with coefficients.

Let $\rho : \mathcal{C} \to \mathrm{End}_\Bbbk(\mathcal{M})$ be the representation functor, defined by $\rho(X) = X \rhd -$. The category $\mathrm{End}_\Bbbk(\mathcal{M})$ is monoidal, with monoidal product being the composition of endofunctors. It is easily seen that the mixed associator $m$ of $\mathcal{M}$ gives rise to a monoidal structure $\widehat{m}$ for $\rho$, i.e.\ a natural isomorphism of the form $\widehat{m}_{X,Y} : \rho(X)\rho(Y) \Rightarrow \rho(X \otimes Y)$, see \S\ref{subsubMonStructRho}. Moreover, the centralizer $\mathcal{Z}(\rho)$ can be identified with $\mathrm{End}_{\mathcal{C}}(\mathcal{M})$, see \S\ref{subsubCentralRepFunct}. Within such an identification:

\begin{propIntro}\label{propIsoMixDYIntro} {\em (Prop.\,\ref{propMixCohomDY})}
For all $\mathsf{F},\mathsf{G} \in \mathrm{End}_{\mathcal{C}}(\mathcal{M}) = \mathcal{Z}(\rho)$, the cochain complexes $\mathrm{C}^\bullet_{\mathrm{mix}}(\mathcal{M};\mathsf{F},\mathsf{G})$ and $\mathrm{C}^\bullet_{\mathrm{DY}}(\rho;\mathsf{F},\mathsf{G})$ are isomorphic.
\end{propIntro}
\noindent All our results on mixed associator cohomology are based on this isomorphism.

\smallskip

\indent Building upon results of \cite{GHS} it has been shown in \cite{FGS} that, under certain finiteness assumptions, the Davydov--Yetter cohomology groups $\mathrm{H}^\bullet_{\mathrm{DY}}(\Gamma; \mathsf{V}, \mathsf{W})$ of a monoidal functor $\Gamma : \mathcal{C} \to \mathcal{D}$ are isomorphic to the relative Ext groups $\Ext^\bullet_{\mathcal{Z}(\Gamma),\mathcal{D}}(\mathsf{V},\mathsf{W})$.  In \cite[\S 3.2]{FGS2} we generalized this isomorphism without finiteness assumption on the target category $\mathcal{D}$. 
In particular, it is not necessary to require that $\mathcal{D}$ 
is rigid or has an exact tensor product. It suffices to require 
that the tensor product is right exact, which holds e.g.\ in the
category of bimodules $\mathrm{Rex}_\Bbbk(\mathcal{M})$. This generalization is crucial for applications in this paper to module categories. 

We recall that {\em relative Ext groups} can be defined for any {\em resolvent pair} between abelian categories, i.e.\ an adjunction satisfying certain assumptions, and these groups are computed through {\em relative projective resolutions}, a generalization of usual projective resolutions satisfying a ``splitting" condition \cite[Chap.\,IX]{macLane},  see also~\cite[\S 2.1]{FGS} for a short and self-contained review. In the case of $\Ext^\bullet_{\mathcal{Z}(\Gamma),\mathcal{D}}$ we consider the adjunction between $\mathcal{Z}(\Gamma)$ and $\mathcal{D}$ defined by the forgetful functor $\mathcal{Z}(\Gamma) \to \mathcal{D}$, $(V,h) \mapsto V$ and its left adjoint. This fact together with Prop.\,\ref{propIsoMixDYIntro} yields a description of mixed associator cohomology in terms of relative
cohomology:

\begin{propIntro}\label{propIsoRelExtIntro}{\em (Prop.\,\ref{PropExtForMixAssoCohom})}
Assume that $\mathcal{C}$ and $\mathcal{M}$ are $\Bbbk$-linear abelian categories, that $\mathcal{C}$ is finite and rigid, and that  $- \rhd M : \mathcal{C} \to \mathcal{M}$ is right-exact for all $M \in \mathcal{M}$\footnote{Equivalently: the representation functor $\rho$ is right-exact. Note that rigidity of $\mathcal{C}$ implies that $\rho$ automatically takes values in $\mathrm{Rex}_\Bbbk(\mathcal{M})$. See \S\ref{subsecExtModCat} for details.} Then for all $\mathsf{F},\mathsf{G} \in \mathrm{Rex}_{\mathcal{C}}(\mathcal{M})$ we have
\[ H^\bullet_{\mathrm{mix}}(\mathcal{M};\mathsf{F},\mathsf{G}) \cong \Ext^\bullet_{\mathrm{Rex}_{\mathcal{C}}(\mathcal{M}), \,\mathrm{Rex}_\Bbbk(\mathcal{M})}(\mathsf{F}, \mathsf{G}). \]
\end{propIntro}

In the terminology of \cite[Def.\,4.1.1]{EGNO},\footnote{We refer to definitions and results in~\cite{EGNO}, even if they are formulated for algebraically closed ground fields, in the case it is straightforward to check their validity for an arbitrary ground field $\Bbbk$.} the monoidal category $\mathcal{C}$ in Prop.\,\ref{propIsoRelExtIntro} is a finite multitensor category, in particular its monoidal unit $\boldsymbol{1}$ is in general semisimple;
we refer to \S\ref{subsecExtModCat} for a more detailed discussion. Further, $\mathrm{Rex}_{\mathcal{C}}(\mathcal{M})$ and $\mathrm{Rex}_\Bbbk(\mathcal{M})$ are the full subcategories of $\mathrm{End}_{\mathcal{C}}(\mathcal{M})$ and $\mathrm{End}_\Bbbk(\mathcal{M})$ consisting of right-exact endofunctors. The relative Ext groups are the ones associated to the forgetful functor $\mathrm{Rex}_{\mathcal{C}}(\mathcal{M}) \to \mathrm{Rex}_\Bbbk(\mathcal{M})$ and its left adjoint, which exists thanks to the assumptions. 

We also recall that if $\mathcal{M}$ is an exact $\mathcal{C}$-module category then $\mathrm{End}_{\mathcal{C}}(\mathcal{M}) = \mathrm{Rex}_{\mathcal{C}}(\mathcal{M})$ is a so-called Morita dual to $\mathcal{C}$, and it is known~\cite[Sec.\,7.12]{EGNO} to be a finite multitensor category. Within the above mentioned Morita context, it is thus worth noticing that our results are  also applicable to Morita duals of finite tensor categories.

\smallskip

It might be hard in general to do relative homological algebra with respect to the adjunction between the categories $\mathrm{Rex}_{\mathcal{C}}(\mathcal{M})$ and $\mathrm{Rex}_\Bbbk(\mathcal{M})$. Under the extra assumption that $\mathcal{M}$ is {\em finite} we can apply the {\em adjunction theorem} of \cite[\S2]{FGS2}, which transforms the relative Ext groups with trivial coefficients  for the resolvent pair $\mathrm{Rex}_{\mathcal{C}}(\mathcal{M}) \rightleftarrows \mathrm{Rex}_\Bbbk(\mathcal{M})$ into relative Ext groups with one non-trivial coefficient for the more familiar resolvent pair $\mathcal{Z}(\mathcal{C}) \rightleftarrows \mathcal{C}$ given by the forgetful
functor from the Drinfeld center. This non-trivial coefficient is an object of the Drinfeld center $\mathcal{Z}(\mathcal{C})$ and turns out to be the so-called {\em adjoint algebra of $\mathcal{M}$} \cite{shimizuMonCent, shimizuCoend}, which we denote by~$\mathcal{A}_{\mathcal{M}}$. It can be described as the end
\[ \mathcal{A}_{\mathcal{M}} = \int_{M \in \mathcal{M}} \underline{\Hom}(M,M) \in \mathcal{C} \]
equipped with a half-braiding $b : \mathcal{A}_{\mathcal{M}} \otimes - \overset{\sim}{\Longrightarrow} - \otimes \mathcal{A}_{\mathcal{M}}$ defined via the universal property of the end, see \S\ref{subsecAdjThm}. Here $\underline{\Hom}$ denotes the internal Hom bifunctor of $\mathcal{M}$, whose definition is recalled in \eqref{adjIntHom}. In this way, we obtain the main result of this paper:

\begin{thIntro}\label{thmAdjIntro}{\em (Thm.\,\ref{thmAdjDYMod})}
Assume that $\mathcal{C},\mathcal{M}$ are $\Bbbk$-linear finite abelian categories, that $\mathcal{C}$ is rigid and that $- \rhd M : \mathcal{C} \to \mathcal{M}$ is right-exact for all $M \in \mathcal{M}$.\footnote{These assumptions actually imply that $\rhd$ is exact in each variable, see Rmk.\,\ref{remarkRightExactDSPS}.} Then
\[ H^\bullet_{\mathrm{mix}}(\mathcal{M}) \cong \Ext^\bullet_{\mathcal{Z}(\mathcal{C}),\mathcal{C}}(\boldsymbol{1},\mathcal{A}_{\mathcal{M}}). \]
\end{thIntro}
\noindent From this theorem we deduce in Cor.\,\ref{cor:dim-formula-Id} a formula for the dimension of $H^n_{\mathrm{mix}}(\mathcal{M})$ in terms of the dimensions of certain Hom spaces in $\mathcal{Z}(\mathcal{C})$, based on a relatively projective covering of $\boldsymbol{1} \in \mathcal{Z}(\mathcal{C})$.

For any object $\mathsf{V} = (V,t) \in \mathcal{Z}(\mathcal{C})$,
the half-braiding $t$ endows the endofunctor $V\rhd-$ of
$\mathcal{M}$ with the structure of a $\mathcal{C}$-module 
endofunctor $\mathsf{V} \rhd -$. A generalization of Thm.\,\ref{thmAdjIntro} for the cohomologies $\mathrm{H}^\bullet_{\mathrm{mix}}(\mathcal{M}; \mathsf{V} \rhd -, \mathsf{W} \rhd -)$ is given in Prop.\,\ref{propGeneralizationAdjThm}, together with a generalized dimension formula. This gives a description of
the deformations of the $\mathcal{C}$-module structure of 
$\mathsf{V} \rhd -$ in terms of relative Exts, 
see Cor.\,\ref{coroInterpretDefGenralizedAdjThm}.

\smallskip

As its name indicates, the object $\mathcal{A}_{\mathcal{M}}$ has a natural structure of an algebra in $\mathcal{Z}(\mathcal{C})$, which is recalled in \S\ref{subsubAdjAlg}. Since the $\mathcal{C}$-module category $\mathcal{M}$ is finite we can assume that $\mathcal{M} = \mathrm{Mod}_{\mathcal{C}}(A)$ for some algebra object $A \in \mathcal{C}$, see \cite[Cor.\,7.10.5]{EGNO}. In Thm.\,\ref{thmAdjObjFullCent}, we prove that $\mathcal{A}_{\mathcal{M}}$ is isomorphic as an algebra in $\mathcal{Z}(\mathcal{C})$ to the {\em full center} of the algebra $A$, which is a categorical definition of the center of an algebra introduced in \cite[Def.\,4.2.1]{Sch} and \cite[\S 4]{davCenter}.
 Based on the description of bulk fields in 2d rational conformal field theories \cite{fuRs4}, the name ``full center'' was introduced in \cite[Def.\,4.9]{fjfrs2} for a commutative algebra in $\mathcal{C} \boxtimes\mathcal{C}^{\mathrm{rev}}$ for $\mathcal{C}$ modular fusion category, in which case $\mathcal{C} \boxtimes\mathcal{C}^{\mathrm{rev}}$ is braided equivalent to $\mathcal{Z}(\mathcal{C})$. This has led to a characterization of the full center as a commutative algebra in $\mathcal{Z}(\mathcal{C})$ in terms of a universal property \cite[\S 4]{davCenter}. It is thus an intriguing problem to explore the consequences of the appearance of the full center in an algebraic deformation theory of 2d conformal field theories.

\smallskip

\indent  Theorem \ref{thmAdjIntro} implies an interesting rigidity result. Note that a  natural example of a $\mathcal{C}$-module category is the {\em regular module}, where $\mathcal{C}$ acts on itself by means of the monoidal product $\otimes_{\mathcal{C}}$. The adjoint algebra $\mathcal{A}_{\mathcal{C}} \in \mathcal{Z}(\mathcal{C})$ is the end $\int_{C \in \mathcal{C}} C \otimes C^* \in \mathcal{C}$ equipped with the half-braiding defined in Example \ref{exampleAdjObjReg}. We prove that $\mathcal{A}_{\mathcal{C}}$ is relatively projective with respect to the resolvent pair $\mathcal{Z}(\mathcal{C}) \rightleftarrows \mathcal{C}$, and thus Thm.\,\ref{thmAdjIntro} gives: 

\begin{coroIntro}
{\em (Prop.\,\ref{propRigidityReg})} Assume that $\mathcal{C}$ is rigid and is finite as a $\Bbbk$-linear abelian category. Then $\mathrm{H}^{> 0}_{\mathrm{mix}}(\mathcal{C}) = 0$. In particular,  there are no non-trivial deformations of the mixed associator on the regular $\mathcal{C}$-module $\mathcal{C}$.
\end{coroIntro}

The second rigidity result concerns the case where $\mathcal{C}$ is a fusion category (as defined e.g. in \cite[Def.\,4.1.1]{EGNO}).  To establish it, we first generalize in Thm.\,\ref{thmGeneralizedOcneanuDY} the Ocneanu rigidity result for DY cohomology with coefficients~\cite{GHS} to the case of $\Bbbk$-linear monoidal functors with fairly general targets (these targets do not need to be semisimple or finite, but   only to have right exact tensor product), with the proof based on the Maschke theorem for Hopf monads~\cite{BV} and~\cite[Rmk.\,2.9]{BLV}. For any $\mathcal{C}$-module category $\mathcal{M}$, we apply this generalized Ocneanu rigidity to the representation functor $\rho\colon \mathcal{C} \to \mathrm{Rex}_\Bbbk(\mathcal{M})$ whose target is not necessarily a fusion or even a finite tensor category and whose monoidal product is only right exact:

\begin{thIntro}{\em (Thm.\,\ref{thmRigidityMix})}\label{thmOcneanuIntro}
Assume that $\mathcal{C}$ is a $\Bbbk$-linear fusion category such that $\dim(\mathcal{C})\neq 0$\footnote{The dimension assumption on $\mathcal{C}$ is automatically fulfilled when the ground field $\Bbbk$ has characteristic 0.}, that $\mathcal{M}$ is a $\Bbbk$-linear abelian $\mathcal{C}$-module category and that the action functor $\rhd : \mathcal{C} \times \mathcal{M} \to \mathcal{M}$ is $\Bbbk$-bilinear. Then for any coefficients $\mathsf{F}, \mathsf{G} \in \mathrm{Rex}_{\mathcal{C}}(\mathcal{M})$ we have
\[ \forall \, n > 0, \quad \mathrm{H}^n_{\mathrm{mix}}(\mathcal{M}; \mathsf{F}, \mathsf{G}) = 0. \]
\end{thIntro}

\noindent It is remarkable that in Thm.\,\ref{thmOcneanuIntro} semisimplicity and finiteness
has only to be imposed on the monoidal category $\mathcal{C}$ and not 
on the module category $\mathcal{M}$. We also note that there are no assumptions on $\rhd$ besides bilinearity. This theorem implies that there are no non-trivial infinitesimal deformations of the mixed associator of $\mathcal{M}$ and of the $\mathcal{C}$-module structure of any $\mathsf{F} \in \mathrm{Rex}_{\mathcal{C}}(\mathcal{M})$.

\smallskip

An important source of module categories is given by comodule algebras. Let $H$ be a finite-dimensional Hopf algebra over the field $\Bbbk$ and let $A$ be a finite-dimensional $H$-comodule algebra, i.e.\ an associative algebra equipped with a coaction $\Delta_A : A \to H \otimes A$ which is an algebra morphism and coassociative. Then $A\text{-}\mathrm{mod}$ is a module category over $H\text{-}\mathrm{mod}$: indeed, if $X \in H\text{-}\mathrm{mod}$ and $M \in A\text{-}\mathrm{mod}$ then $A$ acts on the space $X \otimes_\Bbbk M$ by means of $\Delta_A$ and this defines the $A$-module $X \rhd M$. We first note in \S\ref{subsecAlgComplex} that the cochain spaces for the deformation complex of $\mathcal{M} = A\text{-}\mathrm{mod}$ have an explicit description:
\[ \mathrm{C}^n_{\mathrm{mix}}(A\text{-}\mathrm{mod}) \cong \mathcal{Z}\bigl( \Delta_A^{(n)}(A) \bigr) \]
where $\Delta^{(n)}_A : A \to H^{\otimes n} \otimes A$ is the iterated coaction and $\mathcal{Z}$ denotes the centralizer in the algebra $H^{\otimes n} \otimes A$. The differential can also be rephrased in terms of the coproduct of $H$ and the coaction of $A$. Then in \S\ref{subAdjAlgAmod} we describe the adjoint algebra $\mathcal{A}_{A\text{-}\mathrm{mod}} \in \mathcal{Z}(H\text{-}\mathrm{mod}) \cong D(H)\text{-}\mathrm{mod}$; this is not a new result as an equivalent description was given in \cite[\S 4.2]{BM}.

Finally, computations with examples of comodule algebras over the Sweedler's 4-dimensional Hopf algebra $\mathsf{Sw}$ are presented in \S\ref{sectionExamples}: categories of vector and super-vector spaces (corresponding to the Hopf subalgebras $\mathbb{C} \cong \mathbb{C}\langle 1_{\mathsf{Sw}} \rangle$ and $\mathbb{C}\langle g \rangle$ in $\mathsf{Sw}$), a family of $\mathsf{Sw}$-comodule algebras $A_\xi$ indexed by a parameter $\xi \in \mathbb{C}$ such that $A_\xi\text{-}\mathrm{mod}\cong \mathrm{vect}_{\mathbb{C}}$ as a linear category (but not as a module category) and a family $A_{m,n}$ of $\mathsf{Sw}$-comodule algebras indexed by $m,n\in \mathbb{N}_{\geq 1}$ which unlike the other listed examples yield non-exact module categories for $n>1$. For each of these examples we describe the structure of the adjoint algebra as a module over the Drinfeld double $D(\mathsf{Sw})$, and deduce the dimension of the mixed associator cohomology spaces thanks to Thm.\,\ref{thmAdjIntro}, by using a relatively projective resolution of the trivial module $\mathbb{C} \in D(\mathsf{Sw})\text{-}\mathrm{mod}$. In each case we provide bases of 2-cocycles by direct computation, then we extend them to formal deformations which can be specialized to continuous families of mixed associators. In particular, we find a $m(n-1)$-dimensional continuous family of mixed associators on $A_{m,n}\text{-}\mathrm{mod}$. 

\begin{remark}
All the content of \S\ref{sectionMixAssoCohom} and of App.\,\ref{appHigherOrderDef} would remain valid for lax module categories (and lax monoidal functors on the DY side), where ``lax'' means that the mixed associator $m$ (and the monoidal structure $\Gamma^{(2)}$) is not required to be an isomorphism. However Prop.\,\ref{propIsoRelExtIntro}, and hence Thm.\,\ref{thmAdjIntro}, needs mixed associators which are isomorphisms; we thus assume this everywhere. Also the field $\Bbbk$ can be replaced by any commutative ring in \S\ref{sectionMixAssoCohom} and App.\,\ref{appHigherOrderDef}.
\end{remark}

\begin{remark}
The results in this paper are a generalization of  results presented in the fourth author's PhD thesis \cite{willprecht}. In particular, Thm.\,\ref{thmAdjIntro} was obtained there under the additional assumption that $\mathcal{M}$ is an {\em exact} module category. Here we use the adjunction theorem of \cite{FGS2} to drop this assumption. Moreover, we have here a much more general statement on Ocneanu rigidity for module categories.
\end{remark}

\begin{remark}
    In the context of the above discussed applications, it would be interesting to see whether a given pivotal structure on a (necessarily exact) $\mathcal{C}$-module category can always be transferred along the mixed associator deformations. This is indeed the case for Sweedler's category $\mathcal{C} = \mathsf{Sw}\text{-}\mathrm{mod}$ whose module categories and their deformations are studied in \S\ref{sectionExamples}. In fact, using \cite[Table\,1 in \S 5.1]{shimizuSerre} together with Remark~\ref{remarkVectLambda}, we see that all $\mathrm{vect}_{\mathbb{C}}^{\lambda}$ from \S\ref{subDeformVectSw} are pivotal, while all $\mathrm{svect}_{\mathbb{C}}^{\lambda}$ from  \S\ref{subDeformSvectSw} are not, where we also used a comodule algebra realization of $\mathrm{svect}_{\mathbb{C}}^{\lambda}$ from~\cite[Prop.\,5.22]{willprecht} in order to compare with \cite[Table\,1]{shimizuSerre}. Therefore, in this case, the mixed associator deformations  preserve pivotality of a $\mathcal{C}$-module category. 
\end{remark}

\medskip

\noindent \textbf{Acknowledgements.} 
 The first three authors   acknowledge the hospitality of Institut Pascal in Orsay, particularly thanking Vesna Cupic, whose contributions transformed the program {\sl CFT: Algebraic, Topological and Probabilistic Approaches in Conformal Field Theory} into an unforgettable event. AMG acknowledges support by C.N.R.S.\ and l'Agence nationale de la recherche
(ANR grant ``New algebraic structures in quantum integrability: towards 3D" NASQI3D ANR-24-CE40-7252). 
AMG is also grateful to
Hamburg University for its kind hospitality in 2023-2024 and  the Humboldt Foundation for the financial support during the visit when an important part of this work was accomplished.  CS acknowledges support by the Deutsche Forschungsgemeinschaft (DFG, German Research Foundation) under Germany's Excellence Strategy - EXC 2121 ``Quantum Universe'' - 390833306 and the Collaborative Research Center - SFB 1624 ``Higher structures, moduli spaces and integrability'' - 506632645. JOW was partially supported 
 by the DFG under 
the RTG 1670 ``Mathematics inspired by String theory and Quantum Field Theory''.

\bigskip

\noindent \textbf{Diagrammatic conventions.} Certain computations in strict monoidal categories are presented with the help of the well-known diagrammatic calculus, defined by the following rules:

\begin{center}
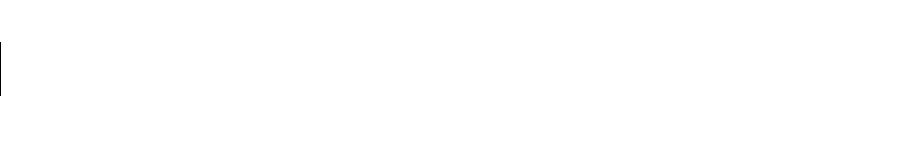
\end{center}
Note that here we read diagrams from bottom to top. When $\mathcal{C}$ is rigid \cite[\S 2.10]{EGNO}, we denote the left (resp. right) dual object of $X \in \mathcal{C}$ by $X^*$ (resp. $^*\!X$) and the duality morphisms are represented as follows:
\begin{center}
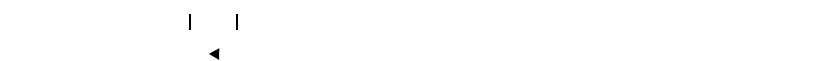
\end{center}

\bigskip

\noindent \textbf{Extension of scalars for categories and functors.} Everywhere in this paper, $\Bbbk$ denotes a field. Let $\mathcal{C}$ be a $\Bbbk$-linear category, meaning that its Hom sets are $\Bbbk$-vector spaces and the composition is $\Bbbk$-bilinear. For a fixed integer $N \geq 1$, the {\em extension} $\mathcal{C}_h = \mathcal{C} \otimes_\Bbbk \, \Bbbk[h]/\langle h^{N+1} \rangle$ is defined by
\begin{equation}\label{defExtScalCat}
\mathrm{Ob}(\mathcal{C}_h) = \mathrm{Ob}(\mathcal{C}), \qquad \Hom_{\mathcal{C}_h}(X,Y) = \Hom_{\mathcal{C}}(X,Y) \otimes_\Bbbk \Bbbk[h]/\langle h^{N+1} \rangle.
\end{equation}
Explicitly, an element in $\Hom_{\mathcal{C}_h}(X,Y)$ is a formal sum $\sum_{i=0}^N f_ih^i$ with $f_i \in \Hom_{\mathcal{C}}(X,Y)$. The composition in $\mathcal{C}_h$ is the $\Bbbk[h]/\langle h^{N+1} \rangle$-bilinear extension of the composition in $\mathcal{C}$ and identity morphisms are unchanged. The integer $N$ is implicit in the notation $\mathcal{C}_h$ and has to be specified.

If $\mathcal{D}$ is another $\Bbbk$-linear category and $F : \mathcal{C} \to \mathcal{D}$ is a $\Bbbk$-linear functor (meaning that it induces linear maps between Hom spaces), its {\em trivial extension} $F_h : \mathcal{C}_h \to \mathcal{D}_h$ is defined on objects and morphisms by
\begin{equation}\label{defExtLinFunct}
\textstyle  F_h(X) = F(X), \qquad F_h\!\left( \sum_{i=0}^N f_ih^i \right) = \sum_{i=0}^N F(f_i)h^i.
\end{equation}
As an immediate generalization of \eqref{defExtLinFunct}, if $B : \mathcal{C} \times \mathcal{C}' \to \mathcal{C}$ is $\Bbbk$-bilinear bifunctor, its trivial extension $B_h : \mathcal{C}_h \times \mathcal{C}'_h \to \mathcal{C}_h$ is defined by
\begin{equation}\label{defExtBilinFunct}
\textstyle  B_h(X,X') = B(X,X'), \qquad B_h\!\left( \sum_{i=0}^N f_ih^i, \sum_{j=0}^N g_j h^j \right) = \sum_{i=0}^N\sum_{l=0}^i B(f_l, g_{i-l})h^i.
\end{equation}
The same remark applies for any $\Bbbk$-multilinear functor. In particular if $\mathcal{C}$ is a monoidal category such that its product $\otimes : \mathcal{C} \times \mathcal{C} \to \mathcal{C}$ is $\Bbbk$-bilinear then $\mathcal{C}_h$ is also monoidal with product $\otimes_h$.

\smallskip

If $F,G : \mathcal{C} \to \mathcal{D}$ are $\Bbbk$-(multi)linear functors then we have from the definitions
\begin{equation}\label{defExtScalNat}
\textstyle \mathrm{Nat}(F_h,G_h) = \left\{ \left.\sum_{i=0}^N \alpha_ih^i \,\right| \, \forall \, i, \: \alpha_i \in \mathrm{Nat}(F,G) \right\}.
\end{equation}

\section{Mixed associator cohomology and DY cohomology}\label{sectionMixAssoCohom}
The goal of this section is to introduce a cohomology theory which controls the deformations of a given mixed associator in a module category, and to relate it with the Davydov--Yetter cohomology of the representation functor of the module category.

\subsection{Deformation theory of mixed associators}\label{subsecMixMon}
Let $(\mathcal{C},\otimes,\boldsymbol{1})$ be a monoidal category. For simplicity, in the whole paper we assume that $\mathcal{C}$ is strict, meaning that the associator for $\otimes$ is the identity morphism, and $X \otimes \boldsymbol{1} = \boldsymbol{1} \otimes X = X$ for all $X \in \mathcal{C}$.

Let $\mathcal{M}$ be another category equipped with a functor $\rhd : \mathcal{C} \times \mathcal{M} \to \mathcal{M}$ which satisfies $\boldsymbol{1} \rhd M = M$ for all $M \in \mathcal{M}$. In the sequel we write $\mathcal{M}$ for $(\mathcal{M},\rhd)$, meaning that $\mathcal{M}$ always comes implicitly with $\rhd$. As for monoidal product, we write $X \rhd M$ instead of $\rhd(X,M)$.

\subsubsection{Module categories and module functors}\label{subsecModCatModFun}
Given $\mathcal{C}$, $\mathcal{M}$ and $\rhd$ as above, a {\em mixed associator} for $\mathcal{M}$ is a family of natural isomorphisms
\[ m = \bigl[ m_{X,Y,M} : X \rhd (Y \rhd M) \overset{\sim}{\longrightarrow} (X \otimes Y) \rhd M \bigr]_{X,Y \in \mathcal{C}, M \in \mathcal{M}} \]
such that
\begin{equation}\label{mixedAssoAxiom}
\xymatrix@C=6em{
X \rhd \bigl( Y \rhd (Z \rhd M) \bigr)  \ar[r]^{m_{X,Y,Z \rhd M}} \ar[d]_{\mathrm{id}_X \,\rhd\, m_{Y,Z,M}} & (X \otimes Y) \rhd (Z \rhd M) \ar[d]^{m_{X \otimes Y,Z,M}}\\
X \rhd \bigl( (Y \otimes Z) \rhd M \bigr) \ar[r]_{m_{X,Y \otimes Z,M}} & (X \otimes Y \otimes Z) \rhd M
} \end{equation}
commutes for all $X,Y,Z \in \mathcal{C}$ and $M \in \mathcal{M}$ and
\begin{align}
\begin{split}\label{mixedAssoUnit}
&\bigl[ X \rhd (\boldsymbol{1} \rhd M) = X \rhd M \xrightarrow{m_{X,\boldsymbol{1},M}} X \rhd M = (X \otimes \boldsymbol{1}) \rhd M \bigr] = \mathrm{id}_{X \rhd M},\\
&\bigl[ \boldsymbol{1} \rhd (X \rhd M) = X \rhd M \xrightarrow{m_{\boldsymbol{1},X,M}} X \rhd M = (\boldsymbol{1} \otimes X) \rhd M \bigr] = \mathrm{id}_{X \rhd M}.
\end{split}
\end{align}
The datum $(\mathcal{M},\rhd,m)$, usually written more shortly $(\mathcal{M},m)$ or even $\mathcal{M}$ when $m$ is fixed, is then called a {\em $\mathcal{C}$-module category} or more shortly a $\mathcal{C}$-module. Let us denote by $\mathrm{Mix}(\mathcal{M})$ the set of such mixed associators for the fixed action~$\rhd$.

\smallskip

\indent A {\em $\mathcal{C}$-module functor} between two $\mathcal{C}$-module categories $(\mathcal{M},\rhd,m)$ and $(\mathcal{M}',\rhd',m')$ is a pair $(F,\gamma^F)$ where $F : \mathcal{M} \to \mathcal{M}'$ is a functor and $\gamma^F = \bigl( \gamma_{X,M}^F :  F(X \rhd M) \overset{\!\!\sim}{\longrightarrow} X \rhd' F(M) \bigr)_{X \in \mathcal{C}, M \in \mathcal{M}}$ is a natural isomorphism such that
\begin{equation}\label{defCModFunct}
\xymatrix@C=5.5em{
F\bigl( X \rhd (Y \rhd M) \bigr) \ar[d]^{F(m_{X,Y,M})} \ar[r]^{\gamma^F_{X, Y \rhd M}} & X \rhd' F(Y \rhd M) \ar[r]^-{\mathrm{id}_X \,\rhd'\, \gamma^F_{Y,M}} & X \rhd' \bigl(Y \rhd' F(M)\bigr) \ar[d]^{m'_{X,Y,F(M)}}\\
F\bigl( (X \otimes Y) \rhd M \bigr) \ar[rr]_{\gamma^F_{X \otimes Y,M}} && (X \otimes Y) \rhd' F(M)
} \end{equation}
commutes for all $X,Y \in \mathcal{C}$, $M \in \mathcal{M}$. Because of \eqref{mixedAssoUnit}, taking $X=Y=\boldsymbol{1}$ in \eqref{defCModFunct} implies that we automatically have
\begin{equation}\label{UnitalityCModFunct}
\bigl[ F(M) = F(\boldsymbol{1} \rhd M) \xrightarrow{\:\gamma^F_{\boldsymbol{1},M}\:} 1 \rhd' F(M) = F(M) \bigr] = \mathrm{id}_{F(M)}.
\end{equation}

A {\em $\mathcal{C}$-module natural transformation} between two $\mathcal{C}$-module functors $(F,\gamma^F)$, $(G,\gamma^G)$ as above is a natural transformation $\nu : F \Rightarrow G$ such that
\begin{equation}\label{CmodNatTransfo}
\xymatrix@C=4em@R=1.6em{
F(X \rhd M) \ar[r]^{\nu_{X \rhd M}} \ar[d]_{\gamma^F_{X,M}} & G(X \rhd M) \ar[d]^{\gamma^G_{X,M}}\\
X \rhd' F(M) \ar[r]_{\mathrm{id}_X \rhd' \nu_M} & X \rhd' G(M)
} \end{equation}
commutes for all $X \in \mathcal{C}$ and $M \in \mathcal{M}$.

These are classical notions, see e.g.\ \cite[\S 7.1, \S 7.2]{EGNO}. In general one requires a natural isomorphism $l_M : \boldsymbol{1} \rhd M \overset{\sim}{\to} M$ satisfying axioms; here we assume $l_M = \mathrm{id}_M$ for simplicity.

\begin{definition}\label{defEquivMix}
We say that two mixed associators $m_1,m_2 \in \mathrm{Mix}(\mathcal{M})$ are equivalent, denoted by $m_1 \simeq m_2$, if there exists a $\mathcal{C}$-module functor of the form $(\mathrm{Id}_{\mathcal{M}}, \gamma)$ between $(\mathcal{M},m_1)$ and $(\mathcal{M},m_2)$.
\end{definition}
\noindent Let $\Aut(\rhd)$ be the group of all natural automorphisms of the functor $\rhd$; these are natural isomorphisms $\eta$ with components $\eta_{X,M} : X \rhd M \overset{\sim}{\longrightarrow} X \rhd M$. Given $\eta \in \Aut(\rhd)$ and $m \in \mathrm{Mix}(\mathcal{M})$, define a natural isomorphism $\eta \cdot m$ whose components $(\eta \cdot m)_{X,Y,M}$ are
\begin{align}
\begin{split}\label{actionAutMix}
X \rhd (Y \rhd M) \xrightarrow{\mathrm{id}_X \,\rhd\, \eta_{Y,M}^{-1}} X \rhd (Y \rhd M) &\xrightarrow{\eta_{X,Y \rhd M}^{-1}} X \rhd (Y \rhd M) \\
&\xrightarrow{m_{X,Y,M}} (X \otimes Y) \rhd M \xrightarrow{\eta_{X \otimes Y,M}} (X \otimes Y) \rhd M.
\end{split}
\end{align}

\begin{lemma}
$\eta \cdot m$ is a mixed associator and this defines an action of $\Aut(\rhd)$ on $\mathrm{Mix}(\mathcal{M})$. The orbits of this action are exactly the equivalence classes of $\simeq$: 
\[ \mathrm{Mix}(\mathcal{M}) \big/ \Aut(\rhd) = \mathrm{Mix}(\mathcal{M})/\!\simeq. \]
\end{lemma}
\begin{proof}
The verification of the first claim is straightforward. For the second claim, it suffices to observe that $(\mathrm{Id}_{\mathcal{M}},\gamma)$ is a $\mathcal{C}$-module functor $(\mathcal{M},m_1) \to (\mathcal{M},m_2)$ if and only if $m_2 = \gamma \cdot m_1$.
\end{proof}

\subsubsection{Infinitesimal deformations of mixed associators}\label{subsecInfDef}
Let $\Bbbk$ be a field. We continue with a category $\mathcal{M}$ equipped with a functor $\rhd : \mathcal{C} \times \mathcal{M} \to \mathcal{M}$ as before, but now add $\Bbbk$-linearity conditions, namely:

\smallskip

\noindent \textbf{Assumption.} The categories $\mathcal{C}$ and $\mathcal{M}$ are $\Bbbk$-linear (meaning that the Hom spaces are $\Bbbk$-vector spaces and composition is $\Bbbk$-bilinear) and the functors $\otimes : \mathcal{C} \times \mathcal{C} \to \mathcal{C}$ and $\rhd : \mathcal{C} \times \mathcal{M} \to \mathcal{M}$ are $\Bbbk$-bilinear on morphisms.

\smallskip

Let $\mathcal{C}_{h}$ and $\mathcal{M}_h$ be the extension of scalars to the ring of dual numbers $\mathcal{C} \otimes_\Bbbk \Bbbk[h]/\langle h^2\rangle$, see \eqref{defExtScalCat}. By \eqref{defExtBilinFunct}, the monoidal product $\otimes$ of $\mathcal{C}$ and the action $\rhd$ on $\mathcal{M}$ can be trivially extended to $\Bbbk[h]/\langle h^2\rangle$-bilinear bifunctors $\otimes_h : \mathcal{C}_h \times \mathcal{C}_h \to \mathcal{C}_h$ and $\rhd_h : \mathcal{C}_h \times \mathcal{M}_h \to \mathcal{M}_h$. In the sequel $\mathcal{M}_h$ is always implicitly equipped with the functor $\rhd_h$.

\smallskip

\indent For $m \in \mathrm{Mix}(\mathcal{M})$, consider
\[ \mathbf{T}_m\mathrm{Mix}(\mathcal{M}) = \bigl\{ \alpha \in \mathrm{Nat}\bigl( - \rhd (- \rhd -), (- \otimes -) \rhd - \bigr) \: \big|\, m + h \alpha \in \mathrm{Mix}(\mathcal{M}_{h}) \bigr\}. \]
By  definition, the condition defining this set is equivalent to requiring that $(\mathcal{M}_{h}, \rhd_h, m +h\alpha)$ is a $\mathcal{C}_{h}$-module category. This is a $\Bbbk$-vector space. Indeed, by \eqref{mixedAssoAxiom} and \eqref{mixedAssoUnit}, a natural transformation $\alpha = \bigl( \alpha_{X,Y,M} : X \rhd (Y \rhd M) \to (X \otimes Y) \rhd M \bigr)_{X,Y \in \mathcal{C},M \in \mathcal{M}}$ is in $\mathbf{T}_m\mathrm{Mix}(\mathcal{M})$ if and only if the following linear conditions are satisfied for all $X,Y,M$
\begin{align}
\begin{split}\label{cocycleCondAsso}
&m_{X, Y\otimes Z,M} \circ (\mathrm{id}_X \rhd \alpha_{Y,Z,M}) - \alpha_{X \otimes Y, Z, M} \circ m_{X,Y,Z \rhd M}\\
+\:\,& \alpha_{X, Y\otimes Z,M} \circ (\mathrm{id}_X \rhd m_{Y,Z,M}) - m_{X \otimes Y, Z, M} \circ \alpha_{X,Y,Z \rhd M} = 0
\end{split}
\end{align}
and
\begin{equation}\label{normCondAsso}
\alpha_{X,\boldsymbol{1},M} = \alpha_{\boldsymbol{1},X,M} = 0.
\end{equation}

\begin{definition}\label{defInfDeformationsMix}
1. For any $\alpha \in \mathbf{T}_m\mathrm{Mix}(\mathcal{M})$, the natural transformation $m + h\alpha$ is called an infinitesimal deformation of the mixed associator $m$.
\\2. Two elements $\alpha,\beta \in \mathbf{T}_m\mathrm{Mix}(\mathcal{M})$ are called equivalent, denoted by $\alpha \cong_m \beta$, if there is a natural transformation $s : \rhd \Rightarrow \rhd$ such that $(\mathrm{Id}_{\mathcal{M}_{h}}, \mathrm{id}_{\rhd} + h s)$ is a $\mathcal{C}_{h}$-module functor from $(\mathcal{M}_{h}, \rhd_h, m + h \alpha)$ to $(\mathcal{M}_{h}, \rhd_h, m + h \beta)$.
\\3. We say that an infinitesimal deformation $m + h \alpha$ is trivial if $\alpha \cong_m 0$.
\end{definition}
\noindent Item 2 is an infinitesimal version of Def.\,\ref{defEquivMix}. We require module natural transformations of the form $\mathrm{id}_{\rhd} + h s$ in order to remain tangent to the point $m$. Explicitly, the condition on $s$ reads:
\begin{equation}\label{coboundaryCondAsso}
\alpha_{X,Y,M} - \beta_{X,Y,M} = m_{X,Y,M} \circ (\mathrm{id}_X \rhd s_{Y,M}) - s_{X \otimes Y,M} \circ m_{X,Y,M} + m_{X,Y,M} \circ s_{X,Y \rhd M}
\end{equation}
for all $X,Y \in \mathcal{C}$ and $M \in \mathcal{M}$. Note that by \eqref{UnitalityCModFunct}, $s$ satisfies $s_{\boldsymbol{1},M} = 0$ for all $M$. Equivalently, $m + h \alpha$ and $m + h \beta$  belong to the same orbit in $\mathrm{Mix}(\mathcal{M}_h)$ under the action \eqref{actionAutMix} of the subgroup of $\Aut(\rhd_h)$ consisting of elements of the form $\mathrm{id} + h s$.

\begin{remark}
We use a tangent space notation for infinitesimal deformations although in general $\mathrm{Mix}(\mathcal{M})$ is just a set. However in certain cases $\mathrm{Mix}(\mathcal{M})$ has a natural structure of affine algebraic variety, and then $\mathbf{T}_m\mathrm{Mix}(\mathcal{M})$ really is the Zariski tangent space at the point $m$; see Remark \ref{rmkQuasiComod}.
\end{remark}

\subsubsection{Mixed associator cohomology}\label{subsecMixAssoCohom}
Continue with the $\Bbbk$-linearity assumptions made in \S\ref{subsecInfDef} and fix a mixed associator $m \in \mathrm{Mix}(\mathcal{M})$. In the sequel, the notation $\mathcal{M}$ implicitly means $(\mathcal{M},\rhd,m)$. Our goal is to build a complex of vector spaces $\mathrm{C}_{\mathrm{mix}}^\bullet(\mathcal{M};\mathsf{F}, \mathsf{G})$, where $\mathsf{F}$ and $\mathsf{G}$ are called {\em coefficients} and will be described 
below, whose cohomology classifies infinitesimal deformations of $m$ in the sense of  \S\ref{subsecInfDef}.

\smallskip

\indent Let $\mathsf{F} = (F,\gamma^F), \mathsf{G} = (G,\gamma^G)$ be $\Bbbk$-linear $\mathcal{C}$-module endofunctors, as defined in \S\ref{subsecModCatModFun}. For all $n \geq 0$, consider the vector space $\mathrm{C}_{\mathrm{mix}}^n(\mathcal{M};\mathsf{F},\mathsf{G})$ consisting of {\em all} natural transformations $g$ with components of the form
\begin{equation}\label{defMixCochains}
g_{X_1,\ldots,X_n,M} : F\bigl(  X_1 \rhd X_2 \rhd \ldots \rhd X_n \rhd M \bigr) \to \bigl( X_1 \otimes \ldots \otimes X_n \bigr) \rhd G(M)
\end{equation}
for all $X_1,\ldots,X_n \in \mathcal{C}$ and $M \in \mathcal{M}$. In particular $\mathrm{C}_{\mathrm{mix}}^0(\mathcal{M};\mathsf{F},\mathsf{G}) = \mathrm{Nat}(F,G)$.

Define coface maps $\partial^{(n)}_i : \mathrm{C}_{\mathrm{mix}}^n(\mathcal{M};\mathsf{F},\mathsf{G}) \to \mathrm{C}_{\mathrm{mix}}^{n+1}(\mathcal{M};\mathsf{F},\mathsf{G})$ as follows. Given $g \in \mathrm{C}_{\mathrm{mix}}^n(\mathcal{M};\mathsf{F},\mathsf{G})$, the components $\partial^{(n)}_0(g)_{X_1,\ldots,X_{n+1},M}$ are
\begin{align*}
F\bigl(X_1 \rhd X_2 \rhd \ldots \rhd X_{n+1} \rhd M \bigr)&\xrightarrow{\gamma^F_{X_1, X_2 \rhd \ldots \rhd X_{n+1} \rhd M}} X_1 \rhd F\bigl( X_2 \rhd \ldots \rhd X_{n+1} \rhd M \bigr)\\
&\xrightarrow{\mathrm{id}_{X_1} \rhd g_{X_2,\ldots,X_{n+1},M}} X_1 \rhd (X_2 \otimes \ldots \otimes X_{n+1}) \rhd G(M)\\
&\xrightarrow{m_{X_1,X_2 \otimes \ldots \otimes X_{n+1},G(M)}} \bigl(X_1 \otimes \ldots \otimes X_{n+1} \bigr) \rhd G(M),
\end{align*}
then for $1 \leq i \leq n$ the components $\partial^{(n)}_i(g)_{X_1,\ldots,X_{n+1},M}$ are
\begin{align*}
&F\bigl( X_1 \rhd \ldots \rhd X_i \rhd X_{i+1} \rhd \ldots \rhd X_{n+1} \rhd M \bigr)\\
&\xrightarrow{F(\mathrm{id}_{X_1} \rhd \ldots \rhd \mathrm{id}_{X_{i-1}} \rhd m_{X_i, X_{i+1}, X_{i+2} \rhd \ldots \rhd X_{n+1} \rhd M})} F\bigl(X_1 \rhd \ldots \rhd (X_i \otimes X_{i+1}) \rhd \ldots \rhd X_{n+1} \rhd M \bigr)\\
&\xrightarrow{g_{X_1, \ldots, X_i \otimes X_{i+1}, \ldots, X_{n+1},M}} \bigl(X_1 \otimes \ldots X_i \otimes X_{i+1} \otimes \ldots \otimes X_{n+1} \bigr) \rhd G(M)
\end{align*}
and finally the components $\partial^{(n)}_{n+1}(g)_{X_1,\ldots,X_{n+1},M}$ are
\begin{align*}
F\bigl( X_1 \rhd \ldots \rhd X_n \rhd X_{n+1} \rhd M \bigr) &\xrightarrow{g_{X_1,\ldots,X_n,X_{n+1} \rhd M}} (X_1 \otimes \ldots \otimes X_n) \rhd G\bigl(X_{n+1} \rhd M\bigr)\\
&\xrightarrow{\mathrm{id}_{X_1 \otimes \ldots \otimes X_n} \,\rhd\, \gamma^G_{X_{n+1},M}} (X_1 \otimes \ldots \otimes X_n) \rhd X_{n+1} \rhd G(M)\\
&\xrightarrow{m_{X_1 \otimes \ldots \otimes X_n, X_{n+1},G(M)}} \bigl(X_1 \otimes \ldots \otimes X_n \otimes X_{n+1} \bigr) \rhd G(M).
\end{align*}

It is straightforward to check that these maps satisfy the cosimplicial identity $\partial^{(n+1)}_j \partial^{(n)}_i = \partial^{(n+1)}_i \partial^{(n)}_{j-1}$ for all $0 \leq i < j \leq n+2$, so that the collection of linear maps 
\begin{equation}\label{mixDiff}
\forall \, n \in \mathbb{N}, \quad d^{(n)} = \sum_{i=0}^{n+1} (-1)^i \partial^{(n)}_i : \mathrm{C}_{\mathrm{mix}}^n(\mathcal{M};\mathsf{F},\mathsf{G}) \to \mathrm{C}_{\mathrm{mix}}^{n+1}(\mathcal{M};\mathsf{F},\mathsf{G})
\end{equation}
is a differential for the complex $\mathrm{C}_{\mathrm{mix}}^\bullet(\mathcal{M};\mathsf{F},\mathsf{G})$. Superscripts indicating the degree on the differential are usually omitted. We denote by $\mathrm{H}^\bullet_{\mathrm{mix}}(\mathcal{M}; \mathsf{F},\mathsf{G})$ the resulting cohomology.

\begin{remark}
The generalization of the complex $\mathrm{C}_{\mathrm{mix}}^\bullet(\mathcal{M};\mathsf{F},\mathsf{G})$ to the case where the monoidal category $\mathcal{C}$ has a non-trivial associator is detailed in \cite[App.\,B.2]{willprecht}.
\end{remark}

\indent For $0 \leq i \leq n-1$ we can also define codegeneracy maps $s_i^{(n)} : \mathrm{C}_{\mathrm{mix}}^n(\mathcal{M};\mathsf{F},\mathsf{G}) \to \mathrm{C}_{\mathrm{mix}}^{n-1}(\mathcal{M};\mathsf{F},\mathsf{G})$ by letting
\begin{equation}\label{codegeneracies}
s_i^{(n)}(g)_{X_1,\ldots,X_{n-1},M} = g_{X_1,\ldots,X_i, \boldsymbol{1},X_{i+1},\ldots,X_{n-1},M}.
\end{equation}
for all $X_1,\ldots,X_{n-1} \in \mathcal{C}$ and $M \in \mathcal{M}$. In this way $\mathrm{C}_{\mathrm{mix}}^\bullet(\mathcal{M};\mathsf{F},\mathsf{G})$ becomes a cosimplicial vector space. It is straightforward to check the cosimplicial identities; one can also deduce them from the isomorphism with the DY cochain complex explained in \S\ref{subsubDYRepFunct}.

\smallskip

\indent There is an important particular case. Let $\mathsf{Id} = \bigl( \mathrm{Id}_{\mathcal{M}}, \mathrm{id}_{\rhd} \bigr) \in \End_{\mathcal{C}}(\mathcal{M})$, which we call the {\em trivial coefficient}. We define the {\em deformation complex of $\mathcal{M}$} and its cohomology as
\begin{equation}\label{defDeformCohom}
\mathrm{C}^{\bullet}_{\mathrm{mix}}(\mathcal{M}) = \mathrm{C}_{\mathrm{mix}}^{\bullet}(\mathcal{M}; \mathsf{Id}, \mathsf{Id}), \quad \mathrm{H}_{\mathrm{mix}}^{\bullet}(\mathcal{M}) = \mathrm{H}_{\mathrm{mix}}^{\bullet}(\mathcal{M}; \mathsf{Id}, \mathsf{Id}).
\end{equation}
We sometimes use the notations $\mathrm{C}_{\mathrm{mix}}^{\bullet}(\mathcal{M},m)$ and $\mathrm{H}_{\mathrm{mix}}^{\bullet}(\mathcal{M},m)$ to insist on the fact that the differential (and hence the cohomology) depends on the mixed associator $m$ of $\mathcal{M}$.

\smallskip

Define the {\em normalized deformation complex} $\mathrm{NC}^\bullet_{\mathrm{mix}}(\mathcal{M},m)$ as
\begin{equation}\label{normalizedMixCochain}
\mathrm{NC}^n_{\mathrm{mix}}(\mathcal{M},m) = \bigl\{ f \in \mathrm{C}^n_{\mathrm{mix}}(\mathcal{M},m) \,\big|\, f_{X_1,\ldots,X_n,M} = 0 \text{ if } X_i = \boldsymbol{1} \text{ for some } i \bigr\}
\end{equation}
and equipped with the same differential as $\mathrm{C}^\bullet_{\mathrm{mix}}(\mathcal{M},m)$. Using the codegeneracies \eqref{codegeneracies} we can write $\mathrm{NC}^n_{\mathrm{mix}}(\mathcal{M},m) = \bigcap_{i=0}^{n-1} \ker\bigl( s_i^{(n)}  \bigr)$, which is the definition of the normalized cochain complex of a cosimplicial vector space. We denote its cohomology by $\mathrm{NH}^\bullet_{\mathrm{mix}}(\mathcal{M},m)$.

It is immediate from definitions that the space $\mathbf{T}_m\mathrm{Mix}(\mathcal{M})$ of infinitesimal deformations of $m$, defined by the conditions \eqref{cocycleCondAsso} and \eqref{normCondAsso}, equals the subspace of normalized 2-cocycles.

\begin{proposition}\label{relDefAssoAndCohom}
Using the notations from Definition \ref{defInfDeformationsMix} we have
\[ \bigl[ \mathbf{T}_m\mathrm{Mix}(\mathcal{M})/\!\cong_m \bigr] = \mathrm{NH}^2_{\mathrm{mix}}(\mathcal{M},m) \cong \mathrm{H}^2_{\mathrm{mix}}(\mathcal{M},m). \]
\end{proposition}
\begin{proof}
We see from \eqref{coboundaryCondAsso} that $\alpha \cong_m \beta$ if and only if $\alpha - \beta = d^{(1)}(s)$ for some $s \in \mathrm{NC}^1_{\mathrm{mix}}(\mathcal{M},m)$. Hence $\mathbf{T}_m\mathrm{Mix}(\mathcal{M})/\!\cong_m$ equals $\mathrm{NH}^2_{\mathrm{mix}}(\mathcal{M},m)$. Moreover, for any cosimplicial vector space $V^\bullet$, it is known that the normalized subspace $\mathrm{N}V^\bullet$ has the same cohomology up to isomorphism, see e.g.\ \cite[\S 3.1 and App.\,B]{FGS2}.
\end{proof}

\indent Recall that the coefficients $\mathsf{F} = (F,\gamma^F), \mathsf{G} = (G,\gamma^G)$ are $\Bbbk$-linear $\mathcal{C}$-module endofunctors. The low degree cohomology groups {\em with coefficients} also have a relevant meaning:

\smallskip

\textbullet~$\mathrm{H}^0_{\mathrm{mix}}(\mathcal{M};\mathsf{F},\mathsf{G}) = \ker(d^{(0)})$ consists of natural transformations $\nu : F \Rightarrow G$ such that
\[ \forall \, X \in \mathcal{C}, \:\: \forall \, M \in \mathcal{M}, \quad (\mathrm{id}_X \rhd \nu_M) \circ \gamma^F_{X,M} - \gamma^G_{X,M} \circ \nu_{X \rhd M} = 0. \]
These are exactly the $\mathcal{C}$-module natural transformations recalled in \eqref{CmodNatTransfo}.

\smallskip

\textbullet~Extend $F$ to a $\Bbbk[h]/\langle h^2 \rangle$-linear endofunctor $F_h$ of $\mathcal{M}_h$, see \eqref{defExtLinFunct}. A cochain $s \in \mathrm{C}^1_{\mathrm{mix}}(\mathcal{M};\mathsf{F},\mathsf{F})$ is a cocycle if and only 
\begin{align*}
m_{X,Y,F(M)} \circ (\mathrm{id}_X \rhd s_{Y,M})\circ \gamma_{X, Y \rhd M}^F &- s_{X \otimes Y,M} \circ F(m_{X,Y,M})\\
&+ m_{X,Y,F(M)} \circ (\mathrm{id}_X \rhd \gamma^F_{Y,M}) \circ s_{X, Y \rhd M} = 0
\end{align*}
for all $X,Y \in \mathcal{C}$ and $M \in \mathcal{M}$, which by \eqref{defCModFunct} precisely means that $(F_h, \gamma^F + h s)$ is a $\mathcal{C}_h$-module endofunctor of $(\mathcal{M}_h,\rhd_h,m)$. We say that $\gamma^F + hs$ is an {\em infinitesimal deformation of $\gamma^F$}. Moreover, $s$ is coboundary to $s'$ if and only if there exists $a \in \mathrm{Nat}(F,F)$ such that $s'_{X,M} - s_{X,M} = (\mathrm{id}_X \rhd a_M) \circ \gamma^F_{X,M} - \gamma^F_{X,M} \circ a_{X \rhd M}$ for all $X \in \mathcal{C}$ and $M \in \mathcal{M}$,
which by \eqref{CmodNatTransfo} precisely means that $\mathrm{id}_F + h a$ is a $\mathcal{C}_h$-module natural isomorphism $(F_h, \gamma^F + hs) \Rightarrow (F_h,\gamma^F + hs')$. Hence  we proved the following:

\begin{proposition}\label{prop:H1-F}    
$\mathrm{H}^1_{\mathrm{mix}}(\mathcal{M};\mathsf{F},\mathsf{F})$ classifies infinitesimal deformations of $\gamma^F$ up to isomorphisms of $\mathcal{C}_h$-module functors of the form $\mathrm{id}_F + h a$, with $\mathsf{F} = (F,\gamma^F)$ and $a \in \mathrm{Nat}(F,F)$.
\end{proposition}

\subsection{Relation with DY cohomology}
Associated to the ``action functor'' $\rhd : \mathcal{C} \times \mathcal{M} \to \mathcal{M}$ is the ``representation functor'' $\rho : \mathcal{C} \to \End(\mathcal{M})$. Here we explain that the  notions introduced in \S\ref{subsecMixMon} can be restated in terms of monoidal structures of $\rho$ and their DY cohomology. We begin by recalling the definition of the DY cochain space for a general monoidal functor and its deformation-theoretic interpretation.

\subsubsection{DY cohomology with coefficients}\label{subsubDYCoeff}
We recall that $\Bbbk$ is a field. Let $\mathcal{C}, \mathcal{D}$ be $\Bbbk$-linear monoidal categories with $\Bbbk$-bilinear monoidal products, assumed strict for simplicity, and let $\Gamma : \mathcal{C} \to \mathcal{D}$ be a $\Bbbk$-linear monoidal functor. Denote by
\[ \Gamma^{(2)}: \Gamma(-) \otimes \Gamma(-) \overset{\sim}{\Longrightarrow} \Gamma(- \otimes -) \]
its {\em monoidal structure}, which by definition satisfies
\begin{equation}\label{defMonStruct}
\forall \, X,Y,Z \in \mathcal{C}, \quad \Gamma^{(2)}_{X \otimes Y,Z} \circ \bigl( \Gamma^{(2)}_{X,Y} \otimes \mathrm{id}_{\Gamma(Z)} \bigr) = \Gamma^{(2)}_{X,Y \otimes Z} \circ \bigl( \mathrm{id}_{\Gamma(X)} \otimes \Gamma^{(2)}_{Y,Z} \bigr).
\end{equation}
We also require that $\Gamma(\boldsymbol{1}) = \boldsymbol{1}$ and\footnote{In full generality one requires the existence of an isomorphism $\Gamma(\boldsymbol{1}) \overset{\sim}{\to} \boldsymbol{1}$. Here we replace this isomorphism by an equality for simplicity.}
\begin{equation}\label{unitalMonStruct}
\forall \, X \in \mathcal{C}, \quad \Gamma^{(2)}_{X,\boldsymbol{1}} = \Gamma^{(2)}_{\boldsymbol{1},X} = \mathrm{id}_{\Gamma(X)}.
\end{equation}
The centralizer of $\Gamma$ \cite{majidZF}, denoted by $\mathcal{Z}(\Gamma)$, is a category whose objects are pairs $(V,t^V)$ with $V \in \mathcal{D}$ and $t^V : V \otimes \Gamma(-) \overset{\sim}{\Longrightarrow} \Gamma(-) \otimes V$ is a {\em half-braiding relative to $\Gamma$}, \textit{i.e.}
\begin{equation}\label{axiomHalfBr}
t^V_{X \otimes Y} \circ \bigl( \mathrm{id}_V \otimes \Gamma^{(2)}_{X,Y} \bigr) = \bigl( \Gamma^{(2)}_{X,Y} \otimes \mathrm{id}_V \bigr) \circ \bigl( \mathrm{id}_{\Gamma(X)} \otimes t^V_Y \bigr) \circ \bigl( t^V_X \otimes \mathrm{id}_{\Gamma(Y)} \bigr)
\end{equation}
for all $X,Y \in \mathcal{C}$.

\smallskip

 Fix objects $\mathsf{V} = (V,t^V), \mathsf{W} = (W,t^W) \in \mathcal{Z}(\Gamma)$. The $m$-th DY cochain space $\mathrm{C}^m_{\mathrm{DY}}(\Gamma; \mathsf{V},\mathsf{W})$ consists of natural transformations $f$ with components
\[ f_{X_1,\ldots, X_m} : V \otimes \Gamma(X_1) \otimes \ldots \otimes \Gamma(X_m) \to \Gamma(X_1 \otimes \ldots \otimes X_m) \otimes W \]
for all $X_i \in \mathcal{C}$. For $0 \leq i \leq m+1$, the partial differentials $\partial^{(m)}_i : \mathrm{C}^m_{\mathrm{DY}}(\Gamma; \mathsf{V},\mathsf{W}) \to \mathrm{C}^{m+1}_{\mathrm{DY}}(\Gamma; \mathsf{V},\mathsf{W})$ are
\begin{align*}
&V \, \Gamma(X_1) \, \Gamma(X_2) \, \ldots \, \Gamma(X_{m+1}) \xrightarrow{t^V_{X_1} \,\otimes\, \mathrm{id}} \Gamma(X_1) \, V \, \Gamma(X_2) \, \ldots \, \Gamma(X_{m+1})\\
&\xrightarrow{\mathrm{id}_{\Gamma(X_1)} \,\otimes\, f_{X_2,\ldots,X_{m+1}}} \Gamma(X_1) \, \Gamma(X_2 \, \ldots \, X_{m+1}) \, W \xrightarrow{\Gamma^{(2)}_{X_1, X_2\, \ldots \, X_{m+1}} \,\otimes\, \mathrm{id}_W} \Gamma(X_1 \, X_2 \, \ldots \, X_{m+1}) \, W
\end{align*}
in the case $i=0$ (symbols $\otimes$ are omitted on objects here), then
\begin{align*}
V \, \Gamma(X_1) \, \ldots \, \Gamma(X_i)\,\Gamma(X_{i+1}) \, \ldots \, \Gamma(X_{m+1}) &\xrightarrow{\mathrm{id} \,\otimes\, \Gamma^{(2)}_{X_i,X_{i+1}} \,\otimes\, \mathrm{id}} V \, \Gamma(X_1) \, \ldots \, \Gamma(X_i\,X_{i+1}) \, \ldots \, \Gamma(X_{m+1})\\
&\xrightarrow{f_{X_1,\ldots, X_i X_{i+1}, \ldots, X_{m+1}}} \Gamma(X_1 \,\ldots\, X_i \, X_{i+1} \, \ldots \, X_{m+1}) \, W
\end{align*}
for $1 \leq i \leq m$ and finally
\begin{align*}
&V \, \Gamma(X_1) \, \ldots \, \Gamma(X_m) \, \Gamma(X_{m+1}) \xrightarrow{f_{X_1,\ldots,X_m} \,\otimes\, \mathrm{id}_{\Gamma(X_{m+1})}} \Gamma(X_1 \, \ldots \, X_m) \, W \, \Gamma(X_{m+1})\\
&\xrightarrow{\mathrm{id} \,\otimes\, t^W_{X_{m+1}}} \Gamma(X_1 \, \ldots \, X_m) \, \Gamma(X_{m+1}) \, W \xrightarrow{\Gamma^{(2)}_{X_1\ldots X_m, X_{m+1}} \,\otimes\, \mathrm{id}_W} \Gamma(X_1 \, \ldots \, X_m \, X_{m+1}) \, W
\end{align*}
in the case $i = m+1$. We define
\begin{equation}\label{diffDYgeneral}
\delta^{(m)} = \sum_{i=0}^{m+1} (-1)^i\partial^{(m)}_i : \mathrm{C}^m_{\mathrm{DY}}(\Gamma; \mathsf{V},\mathsf{W}) \to \mathrm{C}^{m+1}_{\mathrm{DY}}(\Gamma; \mathsf{V},\mathsf{W}) 
\end{equation}
which is the differential of the cochain complex $\mathrm{C}^\bullet_{\mathrm{DY}}(\Gamma;\mathsf{V},\mathsf{W})$. Superscripts indicating the degree on the differential are usually omitted.

\begin{remark}
The generalization of the complex $\mathrm{C}_{\mathrm{DY}}^\bullet(\Gamma;\mathsf{V},\mathsf{W})$ to the case where the source category $\mathcal{C}$ has a non-trivial associator while the target category is still assumed strict is detailed in \cite[App.\,B.1]{willprecht}.
\end{remark}

In the case of trivial coefficients we use the shorter notations
\[ \mathrm{C}^\bullet_{\mathrm{DY}}(\Gamma) = \mathrm{C}^\bullet_{\mathrm{DY}}(\Gamma; \boldsymbol{1},\boldsymbol{1}) \quad \text{and} \quad \mathrm{H}^\bullet_{\mathrm{DY}}(\Gamma) = \mathrm{H}^\bullet_{\mathrm{DY}}(\Gamma; \boldsymbol{1},\boldsymbol{1}). \]
The version with coefficients has been introduced in \cite{GHS} while $\mathrm{C}^\bullet_{\mathrm{DY}}(\Gamma)$ was defined in \cite{CY,davydov,yetter1}.

\smallskip

\indent The {\em normalized DY subcomplex} is
\begin{equation}\label{defNormDY}
\forall \, m \geq 1, \quad \mathrm{NC}^m_{\mathrm{DY}}(\Gamma;\mathsf{V},\mathsf{W}) = \bigl\{ f \in \mathrm{C}^m_{\mathrm{DY}}(\Gamma;\mathsf{V},\mathsf{W}) \,\big|\, f_{X_1,\ldots,X_m} = 0 \text{ if some } X_i = \boldsymbol{1} \bigr\}
\end{equation}
and $\mathrm{NC}^0_{\mathrm{DY}}(\Gamma;\mathsf{V},\mathsf{W}) = \mathrm{C}^0_{\mathrm{DY}}(\Gamma;\mathsf{V},\mathsf{W}) = \Hom_{\mathcal{D}}(V,W)$. We noted in \cite[\S 3.1 and App.\,B]{FGS2} that the complexes $\mathrm{NC}^\bullet_{\mathrm{DY}}(\Gamma;\mathsf{V},\mathsf{W})$ and $\mathrm{C}^\bullet_{\mathrm{DY}}(\Gamma;\mathsf{V},\mathsf{W})$ have the same cohomology:
\[ \mathrm{NH}^\bullet_{\mathrm{DY}}(\Gamma;\mathsf{V},\mathsf{W}) \cong \mathrm{H}^\bullet_{\mathrm{DY}}(\Gamma;\mathsf{V},\mathsf{W}) \]
as a consequence of a general fact on cosimplicial objects.

\smallskip

We finally recall that the second cohomology group $\mathrm{H}^2_{\mathrm{DY}}(\Gamma)$ classifies the infinitesimal deformations of the monoidal structure $\Gamma^{(2)}$, modulo monoidal natural isomorphisms which are tangent to $\mathrm{id}_F$, see e.g.\ \cite[\S 3.1]{FGS2} for a review. In the non-trivial coefficients case, $H^1_{\mathrm{DY}}(\Gamma;\mathsf{V},\mathsf{V})$ classifies infinitesimal deformations of the half-braiding $t^V$ of $\mathsf{V} = (V,t^V)$. In \S\ref{subsubObstructionDY} we discuss the obstructions to extend such infinitesimal deformations to higher order.

\subsubsection{Monoidal structures of the representation functor}\label{subsubMonStructRho}
Let $\End(\mathcal{M})$ be the category whose objects are all endofunctors $\mathcal{M} \to \mathcal{M}$ and morphisms are natural transformations between such functors. Composition of morphisms is the vertical composition of natural transformations, defined on components by $(\alpha \circ \beta)_M = \alpha_M \circ \beta_M$. The identity morphism $\mathrm{id}_F$ has components $(\mathrm{id}_F)_M = \mathrm{id}_{F(M)}$.

Consider the functor
\begin{equation}\label{repFunctor}
\rho : \mathcal{C} \to \End(\mathcal{M}), \quad \rho(X)(M) = X \rhd M.
\end{equation}
More precisely, to an object $X \in \mathcal{C}$ we associate the functor
\[ \rho(X) : \mathcal{M} \to \mathcal{M}, \quad M \mapsto X \rhd M, \quad f \mapsto \mathrm{id}_X \rhd f \]
and to a morphism $f \in \Hom_{\mathcal{C}}(X,Y)$ we associate the natural transformation
\[ \rho(f) : \rho(X) \Rightarrow \rho(Y) \quad \text{with components } \rho(f)_M = f \rhd \mathrm{id}_M \text{ for all } M \in \mathcal{M}. \]

\indent There is naturally a strict monoidal structure on the category $\End(\mathcal{M})$. Namely, the monoidal product of objects is the composition of endofunctors (written simply by juxtaposition $FF'$) and the monoidal unit is $\mathrm{Id}_{\mathcal{M}}$. The monoidal product of two morphisms $\alpha : F \Rightarrow G$ and $\beta : F' \rightarrow G'$ is their  horizontal composition $\alpha \bullet \beta : F  F' \Rightarrow G  G'$ (a.k.a. Godement product) given on components by
\begin{equation}\label{horizontalComp}
\forall \, M \in \mathcal{M}, \quad (\alpha \bullet \beta)_M : F\bigl( F'(M) \bigr) \xrightarrow{F(\beta_M)} F\bigl( G'(M) \bigr) \xrightarrow{\alpha_{G'(M)}} G\bigl( G'(M) \bigr).
\end{equation}
By naturality this can also be written as $(\alpha \bullet \beta)_M = G(\beta_M) \circ \alpha_{F'(M)}$.

It thus makes sense to consider the set $\mathrm{Mon}(\rho)$ of monoidal structures for the functor $\rho$. An element $\theta \in \mathrm{Mon}(\rho)$ is a natural isomorphism whose component for $X,Y \in \mathcal{C}$ has the form $\theta_{X,Y} : \rho(X) \rho(Y) \!\overset{\!\sim}{\implies}\! \rho(X \otimes Y)$. Hence $\theta_{X,Y}$ is itself a natural isomorphism whose component for $M \in \mathcal{M}$ has the form
\[ (\theta_{X,Y})_M : X \rhd (Y \rhd M) \overset{\sim}{\to} (X \otimes Y) \rhd M. \]
As any monoidal structure, $\theta$ has to satisfy
\[ \theta_{X \otimes Y,Z} \circ \bigl( \theta_{X,Y} \bullet \mathrm{id}_{\rho(Z)} \bigr) = \theta_{X,Y \otimes Z} \circ \bigl( \mathrm{id}_{\rho(X)} \bullet \theta_{Y,Z} \bigr) \quad \text{ and } \quad \theta_{X,\boldsymbol{1}} = \theta_{\boldsymbol{1},X} = \mathrm{id}_{\rho(X)} \]
which by definition of $\rho$ and $\bullet$ means that for all $X,Y,Z \in \mathcal{C}$ and $M \in \mathcal{M}$ we have
\begin{align}
\begin{split}\label{defMonStructRho}
&(\theta_{X \otimes Y,Z})_M \circ (\theta_{X,Y})_{Z \rhd M} = (\theta_{X, Y\otimes Z})_M \circ \bigl( \mathrm{id}_X \rhd (\theta_{Y,Z})_M \bigr)\\[.2em]
&\qquad \qquad \text{and} \quad (\theta_{X,\boldsymbol{1}})_M = (\theta_{\boldsymbol{1},X})_M = \mathrm{id}_{X \rhd M}.
\end{split}
\end{align}

\indent Following the terminology in \cite[\S 3.1]{FGS2}, we say that two monoidal structures $\theta^1, \theta^2 \in \mathrm{Mon}(\rho)$ are {\em equivalent}, denoted by $\theta^1 \sim \theta^2$, if there is a monoidal natural isomorphism $\omega : (\rho,\theta^1) \!\overset{\!\!\sim}{\implies}\! (\rho,\theta^2)$. This can be rephrased as follows: given $\omega \in \mathrm{Aut}(\rho)$ and $\theta \in \mathrm{Mon}(\rho)$ define $\omega \cdot \theta : \rho(-)\rho(-) \overset{\sim}{\implies} \rho(- \otimes -)$ by the components
\begin{equation}\label{actionAutMon}
(\omega\cdot\theta)_{X,Y} : \rho(X)\rho(Y) \overset{\omega_X^{-1} \,\bullet\, \omega_Y^{-1}}{=\!=\!=\!=\!=\!\Rightarrow} \rho(X)\rho(Y) \overset{\theta_{X,Y}}{=\!=\!\Rightarrow} \rho(X \otimes Y) \overset{\omega_{X \otimes Y}}{=\!=\!=\!\Rightarrow} \rho(X \otimes Y).
\end{equation}
This gives an action of $\Aut(\rho)$ on $\mathrm{Mon}(\rho)$ and the equivalence classes for $\sim$ are exactly the orbits for this action.

\begin{lemma}\label{lemmaBijMixMon}
1. There is a bijection $\mathrm{Mix}(\mathcal{M}) \overset{\sim}{\longrightarrow} \mathrm{Mon}(\rho)$, given by $m \mapsto \widehat{m}$ where $\widehat{m}$ has components $(\widehat{m}_{X,Y})_M = m_{X,Y,M}$ for all $X,Y \in \mathcal{C}$ and $M \in \mathcal{M}$.
\\2. It descends to a bijection between the quotient sets $\mathrm{Mix}(\rhd)/\!\simeq$ and $\mathrm{Mon}(\rho)/\!\sim$.
\end{lemma}
\begin{proof}
1.  It suffices to compare \eqref{defMonStructRho} with \eqref{mixedAssoAxiom}--\eqref{mixedAssoUnit}.
\\2. Note that there is a bijection $\Psi : \Aut(\rhd) \to \Aut(\rho)$ defined by $\bigl(\Psi(\gamma)_X\bigr)_M = \gamma_{X,M}$. Let $\Phi$ be the bijection of item 1. It is readily seen that $\Phi(\gamma \cdot m) = \Psi(\gamma) \cdot \Phi(m)$, where we use the actions \eqref{actionAutMix} and \eqref{actionAutMon}. Hence an orbit under $\Aut(\rhd)$ is sent to an orbit under $\Aut(\rho)$, whence the claim.
\end{proof}

\subsubsection{Centralizer of the representation functor}\label{subsubCentralRepFunct}
Fix a mixed associator $m$ in $\mathcal{M}$. By Lemma \ref{lemmaBijMixMon}, this is equivalent to a monoidal structure $\widehat{m}$ on the representation functor $\rho : \mathcal{C} \to \End(\mathcal{C})$ defined in \eqref{repFunctor}. In the sequel we assume implicitly that $\rho$ is equipped with $\widehat{m}$, thus becoming a monoidal functor.

\smallskip

Let us specialize the definition of centralizer of a monoidal functor recalled in \S\ref{subsubDYCoeff} to the case of $\rho$. The category $\mathcal{Z}(\rho)$ has as 
objects the pairs $(F, b)$ where $F \in \End(\mathcal{M})$ and $b = \bigl( b_X : F \rho(X) \!\overset{\!\sim}{\implies}\! \rho(X) F \bigr)_{X \in \mathcal{C}}$ is a natural isomorphism satisfying the half-braiding property:
\begin{equation}\label{halfBraidingRepFunctor}
\xymatrix@C=4.5em@R=.6em{
F \rho(X) \rho(Y) \ar@{=>}[dd]_{\mathrm{id}_F \,\bullet\, \widehat{m}_{X,Y}} \ar@{=>}[r]^{b_X \,\bullet\, \mathrm{id}_{\rho(Y)}} & \rho(X) F \rho(Y) \ar@{=>}[r]^{\mathrm{id}_{\rho(X)} \,\bullet\, b_Y} & \rho(X) \rho(Y) F \ar@{=>}[dd]^{\widehat{m}_{X,Y} \,\bullet\, \mathrm{id}_F}\\
&\circlearrowright&\\
F \rho(X \otimes Y) \ar@{=>}[rr]_{b_{X \otimes Y}} && \rho(X \otimes Y) F
} \end{equation}
commutes for all $X,Y \in \mathcal{C}$, where we use the horizontal composition $\bullet$ recalled in \eqref{horizontalComp}. A morphism $\alpha : (F,b^F) \to (G,b^G)$ in $\mathcal{Z}(\rho)$ is a natural transformation $\alpha : F \Rightarrow G$ such that
\begin{equation}\label{morphismsInCentralizer}
\forall \, X \in \mathcal{C}, \quad \bigl( \mathrm{id}_{\rho(X)} \bullet \alpha \bigr) \circ b^F_X = b^G_X \circ \bigl( \alpha \bullet \mathrm{id}_{\rho(X)} \bigr).
\end{equation}
The category $\mathcal{Z}(\rho)$ is strict monoidal: the monoidal product of objects (written by juxtaposition) is given by $(F,b^F)(G,b^G) = (FG, b^{FG})$ where
\begin{equation}\label{monProductZrho}
\forall \, X \in \mathcal{C}, \quad b^{FG}_X : FG\rho(X) \xymatrix@C=4em{\ar@{=>}[r]^{\mathrm{id}_F \,\bullet\, b^G_X}&} F\rho(X)G \xymatrix@C=4em{\ar@{=>}[r]^{b^F_X \,\bullet\, \mathrm{id}_G}&} \rho(X)FG
\end{equation}
while the monoidal product of morphisms (=natural transformations commuting with half-braidings) is the horizontal composition $\bullet$.

\smallskip

\indent We now rephrase the monoidal category $\mathcal{Z}(\rho)$ in more familiar terms.
\begin{definition}
We denote by $\End_{\mathcal{C}}(\mathcal{M})$ the category whose objects are $\mathcal{C}$-module endofunctors of the $\mathcal{C}$-module
category $\mathcal{M}$, as defined in \eqref{defCModFunct}, and whose morphisms are $\mathcal{C}$-module natural transformations, as defined in \eqref{CmodNatTransfo}.
\end{definition}
\noindent We note that $\End_{\mathcal{C}}(\mathcal{M})$ is a strict monoidal category: the monoidal product of objects is given by $(F,\gamma^F)(G,\gamma^G) = (FG, \gamma^{FG})$ where
\begin{equation}\label{monProductEndC}
\gamma^{FG}_{X,M} : FG(X \rhd M) \xrightarrow{F(\gamma^G_{X,M})} F\bigl( X \rhd G(M) \bigr) \xrightarrow{\gamma^F_{X,G(M)}} X \rhd FG(M)
\end{equation}
for all $X \in \mathcal{C}$ and $M \in \mathcal{M}$, while the monoidal product of morphisms is the horizontal composition $\bullet$.

\smallskip

\indent Given $(F,b) \in \mathcal{Z}(\rho)$, define a natural isomorphism $\widehat{b} : F(-\rhd -) \!\overset{\sim}{\!\implies}\! - \rhd F(-)$ by
\[
\widehat{b}_{X,M} : F(X \rhd M) = \bigl(F\rho(X)\bigr)(M) \xrightarrow{(b_X)_M} \bigl( \rho(X)F \bigr)(M) = X \rhd F(M) \]
for all $X \in \mathcal{C}$ and $M \in \mathcal{M}$. Condition \eqref{halfBraidingRepFunctor} for $b$ is equivalent to condition \eqref{defCModFunct}  for $\widehat{b}$, and condition \eqref{morphismsInCentralizer} for $\alpha : (F,b^F) \to (G,b^G)$ is equivalent to condition \eqref{CmodNatTransfo} for $\alpha : \bigl(F,\widehat{b^F}\bigr) \to \bigl(G,\widehat{b^G}\bigr)$. We conclude that there is an isomorphism
\begin{equation}\label{isoCentralizerEndC}
\mathcal{Z}(\rho) \overset{\sim}{\longrightarrow} \End_{\mathcal{C}}(\mathcal{M}), \quad (F,b) \mapsto (F,\widehat{b}), \quad \alpha \mapsto \alpha.
\end{equation}
Moreover, comparing \eqref{monProductZrho} and \eqref{monProductEndC}, we see that this isomorphism is monoidal. In the sequel we use it implicitly as an identification.

\subsubsection{DY cohomology of the representation functor}\label{subsubDYRepFunct}
Here we assume that the categories $\mathcal{C}$, $\mathcal{M}$ are $\Bbbk$-linear and that the functors $\otimes_{\mathcal{C}}$ and $\rhd$ are $\Bbbk$-bilinear on morphisms. Then the representation functor $\rho : \mathcal{C} \to \End(\mathcal{M})$ takes values in the full monoidal subcategory $\End_\Bbbk(\mathcal{M})$ of $\Bbbk$-linear endofunctors.

We fix a mixed associator $m$ in $\mathcal{M}$ and equip the representation functor $\rho$ with the associated monoidal structure $\widehat{m}$ (Lemma \ref{lemmaBijMixMon}), so that it becomes a monoidal functor. It thus makes sense to consider its Davydov--Yetter (DY) cohomology with coefficients, whose definition was recalled in \S\ref{subsubDYCoeff}. 

\smallskip

\indent Given $\mathsf{F} = (F,b^F), \mathsf{G} = (G,b^G) \in \mathcal{Z}(\rho)$, denote by $\mathrm{C}^n_{\mathrm{DY}}(\rho; \mathsf{F}, \mathsf{G})$ the $n$-th DY cochain space of $\rho$ with coefficients $\mathsf{F}, \mathsf{G}$. An element $f \in \mathrm{C}^n_{\mathrm{DY}}(\rho; \mathsf{F}, \mathsf{G})$ is a natural transformation whose component at $X_1,\ldots,X_n \in \mathcal{C}$ is itself a natural transformation
\[ f_{X_1, \ldots, X_n} : F \rho(X_1) \ldots \rho(X_n) \Rightarrow \rho\bigl(X_1 \otimes \ldots \otimes X_n \bigr)G \]
which has components
\[ (f_{X_1, \ldots, X_n})_M : F\bigl(X_1 \rhd \ldots \rhd X_n \rhd M \bigr) \to \bigl( X_1 \otimes \ldots \otimes X_n \bigr) \rhd G(M) \]
for all $M \in \mathcal{M}$. Comparing with \eqref{defMixCochains}, we see that this collection of maps defines a cochain in $\mathrm{C}^n_{\mathrm{mix}}(\mathcal{M}; \mathsf{F}, \mathsf{G})$. More precisely, we have:
\begin{proposition}\label{propMixCohomDY}
Let $\mathsf{F}, \mathsf{G} \in \End_{\mathcal{C}}(\mathcal{M})$, also viewed as objects in $\mathcal{Z}(\rho)$ through the identification \eqref{isoCentralizerEndC}, and equip $\rho : \mathcal{C} \to \mathrm{End}_\Bbbk(\mathcal{M})$ 
with the monoidal structure $\widehat{m}$ defined from $m$ (Lem.\,\ref{lemmaBijMixMon}). We have an isomorphism of cosimplicial vector spaces
\[ \Phi^n : \mathrm{C}^n_{\mathrm{DY}}(\rho; \mathsf{F}, \mathsf{G}) \overset{\sim}{\longrightarrow} \mathrm{C}^n_{\mathrm{mix}}(\mathcal{M}; \mathsf{F}, \mathsf{G}) \quad \text{given by } \Phi^n(g)_{X_1,\ldots,X_n,M} = (g_{X_1,\ldots,X_n})_M \]
for all $n \geq 0$, $X_1,\ldots,X_n \in \mathcal{C}$ and $M \in \mathcal{M}$. It thus descends to isomorphisms
\[ \mathrm{H}^n_{\mathrm{DY}}(\rho; \mathsf{F}, \mathsf{G}) \overset{\sim}{\longrightarrow} \mathrm{H}^n_{\mathrm{mix}}(\mathcal{M}; \mathsf{F}, \mathsf{G}). \]
\end{proposition}
\begin{proof}
It suffices to compare the definition of the cosimplicial structure on the DY complex in \S\ref{subsubDYCoeff} for $\Gamma= \rho$ with the one of the mixed associator complex in \S\ref{subsecMixAssoCohom}.
\end{proof}

Recall the normalized subcomplexes $\mathrm{NC}^\bullet_{\mathrm{DY}}(\rho;\mathsf{F},\mathsf{G})$ and $\mathrm{NC}^\bullet_{\mathrm{mix}}(\mathcal{M}; \mathsf{F}, \mathsf{G})$ defined in \eqref{defNormDY} and \eqref{normalizedMixCochain} respectively, and which are important for the study of deformations. It is readily seen that the isomorphism $\Phi^\bullet$ in Prop.\,\ref{propMixCohomDY} restricts to
\[ \Phi^\bullet : \mathrm{NC}^\bullet_{\mathrm{DY}}(\rho;\mathsf{F},\mathsf{G}) \overset{\sim}{\longrightarrow} \mathrm{NC}^\bullet_{\mathrm{mix}}(\mathcal{M}; \mathsf{F}, \mathsf{G}). \]

\subsection{Higher order deformations and obstructions}\label{secObstructions}

In \S\ref{subsubObstructionDY} we explain how the DY cochain complex controls the obstructions for deformations of monoidal structures and of half-braidings relative to a monoidal functor. Remarkable algebraic structures on the DY complex (namely comp-algebra and cup product) play a key-role for this purpose and are detailed in App.\,\ref{appHigherOrderDef}. Then in \S\ref{sectionObstructionsMix}, using the isomorphism in Prop.\,\ref{propMixCohomDY}, we deduce obstruction results for deformations of mixed associators and of module structure of module endofunctors.

\subsubsection{Deformations of monoidal functors and of half-braidings}\label{subsubObstructionDY}
Let $\Gamma : \mathcal{C} \to \mathcal{D}$ be a $\Bbbk$-linear monoidal functor with monoidal structure $\Gamma^{(2)} : \Gamma(-) \otimes \Gamma(-) \overset{\sim}{\Longrightarrow} \Gamma(- \otimes -)$. For any integer $N \geq 1$, let $\mathcal{C}_h, \mathcal{D}_h$ and $\Gamma_h : \mathcal{C}_h \to \mathcal{D}_h$ be the extension of scalars to $\Bbbk[h]/\langle h^{N+1} \rangle$ as defined in \eqref{defExtScalCat} and \eqref{defExtLinFunct}.

\smallskip

 Let us first recall from \cite{CY,davydov,yetter1} how DY cohomology with trivial coefficients controls deformations of monoidal structures. A {\em deformation of the monoidal structure $\Gamma^{(2)}$ of order $N$} is a natural isomorphism of the form
\[\mathbf{f} = \Gamma^{(2)} + \sum_{i=1}^N f_i h^i :\:  \Gamma_h(-) \otimes \Gamma_h(-) \overset{\sim}{\Longrightarrow} \Gamma_h(- \otimes -) \]
which satisfies the monoidal structure axioms \eqref{defMonStruct} and \eqref{unitalMonStruct}. By \eqref{defExtScalNat} each $f_i$ is itself a natural transformation:
\[ \forall \, i, \quad f_i \in \mathrm{Nat}\bigl( \Gamma(-) \otimes \Gamma(-), \Gamma(- \otimes -) \bigr) =  \mathrm{C}^2_{\mathrm{DY}}(\Gamma). \]
Also $\mathbf{f}$ is automatically invertible because so is the term in degree $0$. Hence $\mathbf{f}$ is a deformation of $\Gamma^{(2)}$ if and only if the pair $(\Gamma_h, \mathbf{f})$ is a monoidal functor $\mathcal{C}_h \to \mathcal{D}_h$.

\smallskip

We say that the deformation $\mathbf{f}$ can be {\em lifted to the order $N+n$} if there exists $f'_1,\ldots, f'_n$ such that $\mathbf{f} + \sum_{j=1}^n f'_j h^{N+j}$ is a deformation of $\Gamma^{(2)}$ when scalars are extended to $\Bbbk[h]/\langle h^{N+n+1} \rangle$.

\smallskip

For any $f_1, \ldots, f_k \in \mathrm{C}^2_{\mathrm{DY}}(\Gamma)$ define the {\em obstruction} $\mathrm{obs}(f_1,\ldots,f_k) \in \mathrm{C}^3_{\mathrm{DY}}(\Gamma)$ as the natural transformation with components
\begin{align}
\begin{split}\label{defObsMonStruct}
&\mathrm{obs}(f_1,\ldots,f_k)_{X,Y,Z}\\
&=\sum_{i+j=k+1} (f_i)_{X \otimes Y,Z} \circ \bigl( (f_j)_{X,Y} \otimes \mathrm{id}_{\Gamma(Z)} \bigr) - (f_i)_{X, Y \otimes Z} \circ \bigl( \mathrm{id}_{\Gamma(X)} \otimes (f_j)_{Y,Z} \bigr)
\end{split}
\end{align}
for all $X,Y,Z \in \mathcal{C}$. Its relevance is due to the fact that $\Gamma^{(2)} + \sum_{i=1}^N f_i h^i$ is a deformation of order $N$ if and only if
\begin{equation}\label{condDeformMonStruct}
\delta(f_1) = 0 \quad \text{and} \quad \forall \, 2 \leq i \leq N, \:\: \delta(f_i) = \mathrm{obs}(f_1,\ldots,f_{i-1})
\end{equation}
where $\delta$ is the differential \eqref{diffDYgeneral} in degree $2$, and
\begin{equation}\label{condDeformMonStructUnit}
\forall \, 1 \leq i \leq N, \:\: \forall \, X \in \mathcal{C}, \:\: (f_i)_{\boldsymbol{1},X} = (f_i)_{X,\boldsymbol{1}} = 0 \qquad \text{\it i.e.\ } f_i \in \mathrm{NC}^2_{\mathrm{DY}}(\Gamma)
\end{equation}
where $\mathrm{NC}^\bullet_{\mathrm{DY}}(\Gamma)$ is the normalized subcomplex \eqref{defNormDY}. Indeed a straightforward computation reveals that \eqref{condDeformMonStruct} is equivalent to the monoidal structure axiom \eqref{defMonStruct}, while \eqref{condDeformMonStructUnit} readily corresponds to the unitality axiom \eqref{unitalMonStruct}. Item 2 in the following proposition was noted by Yetter \cite[\S 3]{yetter1}:

\begin{proposition}\label{propExtensionMonStruct} Let $\mathbf{f} = \Gamma^{(2)} + \sum_{i=1}^N f_i h^i$ be a deformation of $\Gamma^{(2)}$ of order $N$.

\smallskip

\noindent 1. Let $f' \in \mathrm{NC}^2_{\mathrm{DY}}(\Gamma)$. Then $\mathbf{f} + h^{N+1}f'$ is a deformation of order $N+1$ if and only if $\delta(f') = \mathrm{obs}(f_1,\ldots,f_N)$.

\smallskip

\noindent 2. If $\mathrm{H}^3_{\mathrm{DY}}(\Gamma) = 0$ then $\mathbf{f}$ can be lifted to any order.
\end{proposition}
\begin{proof}
1. This is just a rephrasing of \eqref{condDeformMonStruct}--\eqref{condDeformMonStructUnit}.
\\2. Lemma \ref{lemmaObstructionMonStruct} of App.\,\ref{sectionCompAlgDY} asserts that $\mathrm{obs}(f_1,\ldots,f_N)$ is a cocycle under the present assumption on $\mathbf{f}$. It is moreover readily seen from its definition that $\mathrm{obs}(f_1,\ldots,f_N)$ is a normalized cochain because $f_1,\ldots,f_N$ are themselves normalized. By \cite[\S 3.1]{FGS2} we have $\mathrm{NH}^3_{\mathrm{DY}}(\Gamma) \cong \mathrm{H}^3_{\mathrm{DY}}(\Gamma) = 0$. Hence the obstruction is always a normalized coboundary and item 1 of the present proposition can be applied recursively.
\end{proof}

\indent We now turn to DY cohomology with coefficients and explain that it controls deformations of half-braidings. Let us fix an object $\mathsf{V} = (V,t^V) \in \mathcal{Z}(\Gamma)$, where $V \in \mathcal{D}$ and $t^V : V \otimes \Gamma(-) \Rightarrow \Gamma(-) \otimes V$ is a half-braiding relative to $\Gamma$, see \eqref{axiomHalfBr}. A {\em deformation of $t^V$ of order $N$} is a natural isomorphism 
\[ \mathbf{t} = t^V + \sum_{i=1}^N c_i h^i : V \otimes \Gamma_h(-) \Rightarrow \Gamma_h(-) \otimes V \]
which satisfies the half-braiding axiom \eqref{axiomHalfBr} in $\mathcal{D}_h$. By \eqref{defExtScalNat} each $c_i$ is itself a natural transformation:
\[ \forall \, i, \quad c_i \in \mathrm{Nat}\bigl( V \otimes \Gamma(-), \Gamma(-) \otimes V \bigr) =  \mathrm{C}^1_{\mathrm{DY}}(\Gamma;\mathsf{V},\mathsf{V}). \]
Also $\mathbf{t}$ is automatically invertible because so is the term in degree $0$. Hence $\mathbf{t}$ is a deformation of $t^V$ if and only if $(V, \mathbf{t}) \in \mathcal{Z}(\Gamma_h)$.

\smallskip

For any $c_1,\ldots,c_k \in \mathrm{C}^1_{\mathrm{DY}}(\Gamma;\mathsf{V},\mathsf{V})$ define the {\em obstruction} $\mathrm{obs}(c_1,\ldots,c_k) \in \mathrm{C}^2_{\mathrm{DY}}(\Gamma;\mathsf{V},\mathsf{V})$ as the natural transformation with components
\begin{equation}\label{defObsHB}
\mathrm{obs}(c_1,\ldots,c_k)_{X,Y} = -\sum_{i+j=k+1} \bigl( \Gamma^{(2)}_{X,Y} \otimes \mathrm{id}_V \bigr) \circ \bigl( \mathrm{id}_{\Gamma(X)} \otimes (c_j)_Y \bigr) \circ \bigl( (c_i)_X \otimes \mathrm{id}_{\Gamma(Y)} \bigr)
\end{equation}
for all $X,Y \in \mathcal{C}$. Its relevance is due to the fact that $t^V + \sum_{i=1}^N c_i h^i$ is a deformation of order $N$ if and only if
\begin{equation}\label{condDeformHB}
\delta(c_1) = 0 \quad \text{and} \quad \forall \, 2 \leq i \leq N, \:\: \delta(c_i) = \mathrm{obs}(c_1,\ldots,c_{i-1})
\end{equation}
where $\delta$ is the differential \eqref{diffDYgeneral} in degree $1$. These conditions are obtained by a straightforward computation.
\begin{proposition}\label{propLiftHB} Let $\mathbf{t} = t^V + \sum_{i=1}^N c_i h^i$ be a deformation of $t^V$ of order $N$.

\smallskip

\noindent 1. Let $c' \in \mathrm{C}^1_{\mathrm{DY}}(\Gamma,\mathsf{V},\mathsf{V})$. Then $\mathbf{t} + h^{N+1}c'$ is a deformation of order $N+1$ if and only if $\delta(c') = \mathrm{obs}(c_1,\ldots,c_N)$.

\smallskip

\noindent 2. If $\mathrm{H}^2_{\mathrm{DY}}(\Gamma;\mathsf{V},\mathsf{V}) = 0$ then $\mathbf{t}$ can be lifted to any order.
\end{proposition}
\begin{proof}
1. This is just a rephrasing of  \eqref{condDeformHB}.
\\2. Lemma \ref{lemmaObstDefHB} of App.\,\ref{sectionCupProdHB} asserts that $\mathrm{obs}(c_1,\ldots,c_N)$ is a cocycle under the present assumption on $\mathbf{t}$. If $\mathrm{H}^2_{\mathrm{DY}}(\Gamma;\mathsf{V},\mathsf{V}) = 0$ then it is a coboundary and item 1 of the present proposition can be applied recursively.
\end{proof}

\subsubsection{Deformations of \texorpdfstring{$\mathcal{C}$}{}-module structures}\label{sectionObstructionsMix}
Let $\mathcal{M} = (\mathcal{M},\rhd,m)$ be a $\mathcal{C}$-module category, where $m$ is the mixed associator. Given an integer $N \geq 1$, let $\mathcal{C}_h$ and $\mathcal{M}_h$ be the extension of scalars to $\Bbbk[h]/\langle h^{N+1} \rangle$ as defined in \eqref{defExtScalCat}. The monoidal product $\otimes$ of $\mathcal{C}$ and the action $\rhd$ of $\mathcal{C}$ on $\mathcal{M}$ have trivial extensions $\otimes_h$ and $\rhd_h$ defined by $\Bbbk[h]/\langle h^{N+1} \rangle$-bilinearity, see \eqref{defExtBilinFunct}.

\smallskip

\indent A {\em deformation of $m$ of order $N$} is a natural transformation
\begin{equation*}
\mathbf{m} = m + \sum_{i=1}^N g_i h^i \,: \, - \rhd_h - \rhd_h - \Longrightarrow (- \otimes_h -) \rhd_h -
\end{equation*}
which satisfies the mixed associator axioms \eqref{mixedAssoAxiom}-\eqref{mixedAssoUnit}. By \eqref{defExtScalNat} each $g_i$ is itself a natural transformation:
\[ \forall \, i, \quad g_i \in \mathrm{Nat}\bigl(- \rhd - \rhd -, (- \otimes -) \rhd - \bigr) =  \mathrm{C}^2_{\mathrm{mix}}(\mathcal{M}). \]
Note that $\mathbf{m}$ is automatically invertible because so is the term in degree $0$. Hence $\mathbf{m}$ is a deformation of $m$ if and only if $(\mathcal{M}_h, \rhd_h, \mathbf{m})$ is a $\mathcal{C}_h$-module category.

\smallskip

Given any family $g_1,\ldots,g_k \in \mathrm{C}^2_{\mathrm{mix}}(\mathcal{M})$ let
\begin{equation}
\begin{split}\label{defObsMixedAsso}
&\mathrm{obs}(g_1,\ldots,g_k)_{X,Y,Z,M}\\
&=\sum_{i+j = k+1} (g_i)_{X \otimes Y, Z, M} \circ (g_j)_{X,Y,Z \rhd M} - (g_i)_{X,Y \otimes Z, M} \circ \bigl( \mathrm{id}_X \rhd (g_j)_{Y,Z,M} \bigr)
\end{split}
\end{equation}
for all $X,Y,Z \in \mathcal{C}$ and $M \in \mathcal{M}$. This defines a 3-cochain $\mathrm{obs}(g_1,\ldots,g_k) \in \mathrm{C}^3_{\mathrm{mix}}(\mathcal{M})$, which we call the {\em obstruction}. Its relevance comes from the following fact: $m + \sum_{i=1}^N g_i h^i$ is a deformation of $m$ of order $N$ if and only if
\begin{equation}\label{conditionDefOrderN}
d(g_1) = 0 \quad \text{and} \quad  d(g_i) = \mathrm{obs}(g_1,\ldots,g_{i-1})\ , \quad \forall \, 2 \leq i \leq N
\end{equation}
where $d$ is the differential \eqref{mixDiff} in degree $2$, and
\begin{equation}\label{unitalityCondDeformation}
\forall\, 1 \leq i \leq N+1, \:\: \forall \, X \in \mathcal{C}, \:\: \forall \, M \in \mathcal{M}, \quad (g_i)_{\boldsymbol{1},X,M} = (g_i)_{X,\boldsymbol{1},M} = 0
\end{equation}
which means that $c_i \in \mathrm{NC}^2_{\mathrm{mix}}(\mathcal{M})$, where $\mathrm{NC}^\bullet_{\mathrm{mix}}(\mathcal{M})$ is the normalized subcomplex \eqref{normalizedMixCochain}. The condition \eqref{conditionDefOrderN} is equivalent to the mixed associator axiom \eqref{mixedAssoAxiom} by a straightforward computation and the condition \eqref{unitalityCondDeformation} is readily equivalent to the unitality axiom \eqref{mixedAssoUnit}.

\begin{proposition}\label{propLiftObstruction}
Let $\mathbf{m} = m + \sum_{i=1}^N g_i h^i$ be a deformation of $m$ of order $N$.

\smallskip

\noindent 1. Let $g' \in \mathrm{NC}^2_{\mathrm{mix}}(\mathcal{M})$. Then $\mathbf{m} + h^{N+1} g'$ is a deformation of order $N+1$ if and only if $d(g') = \mathrm{obs}(g_1,\ldots,g_N)$.

\smallskip

\noindent 2. If $\mathrm{H}^3_{\mathrm{mix}}(\mathcal{M}) = 0$ then $\mathbf{m}$ can be lifted to any order.
\end{proposition}
\begin{proof}
1. This is just a rephrasing of \eqref{conditionDefOrderN} and \eqref{unitalityCondDeformation}.
\\2. The obstruction defined in \eqref{defObsMixedAsso} for deformations of the mixed associator $m$ corresponds through the isomorphism $\Phi^3 : \mathrm{C}^3_{\mathrm{DY}}(\rho) \overset{\sim}{\to} \mathrm{C}^3_{\mathrm{mix}}(\mathcal{M})$ of Prop.\,\ref{propMixCohomDY} to the obstruction defined in \eqref{defObsMonStruct} for deformations of the monoidal structure of a functor $\Gamma$, in the case where $\Gamma = \rho$ and $\Gamma^{(2)} = \widehat{m}$. Also the normalization conditions \eqref{unitalityCondDeformation} and \eqref{condDeformMonStructUnit} are equivalent through $\Phi^{2}$. In this way Prop.\,\ref{propLiftObstruction} appears as a particular case of Prop.\,\ref{propExtensionMonStruct}.
\end{proof}

Now let $\mathsf{F} = (F, \gamma^F)$ be a $\Bbbk$-linear $\mathcal{C}$-module endofunctor of $\mathcal{M}$, where $\gamma^F : F(- \rhd -) \overset{\sim}{\Longrightarrow} - \rhd F(-)$ is the $\mathcal{C}$-module structure. Let $F_h$ be the $\Bbbk[h]/\langle h^{N+1} \rangle$-linear extension of $F$, see \eqref{defExtLinFunct}. A {\em deformation of $\gamma^F$ of order $N$} is a natural transformation
\[ \boldsymbol{\gamma} = \gamma^F + \sum_{i=0}^N \gamma_i h^i : F_h(- \rhd_h -) \overset{\sim}{\Longrightarrow} - \rhd_h F_h(-)  \]
which satisfies the $\mathcal{C}$-module structure axiom \eqref{defCModFunct}.  By \eqref{defExtScalNat} each $\gamma_i$ is itself a natural transformation:
\[ \forall \, i, \quad \gamma_i \in \mathrm{Nat}\bigl( F(- \rhd -), - \rhd F(-) \bigr) =  \mathrm{C}^1_{\mathrm{mix}}(\mathcal{M};\mathsf{F},\mathsf{F}). \]
Note that $\boldsymbol{\gamma}$ is automatically invertible because so is the term in degree $0$. Hence $\boldsymbol{\gamma}$ is a deformation of $\gamma^F$ if and only if the pair $(F_h, \boldsymbol{\gamma})$ is a $\mathcal{C}_h$-module endofunctor of $(\mathcal{M}_h,\rhd_h,m)$.

\smallskip

Given any family $\gamma_1,\ldots,\gamma_k \in \mathrm{C}^1_{\mathrm{mix}}(\mathcal{M};\mathsf{F},\mathsf{F})$, let
\begin{equation}\label{defObsModuleFunct} \mathrm{obs}(\gamma_1,\ldots,\gamma_k)_{X,Y,M} = -\sum_{i+j=k+1} m_{X,Y,F(M)} \circ \bigl( \mathrm{id}_X \rhd (\gamma_i)_{Y,M}  \bigr) \circ (\gamma_j)_{X, Y \rhd M}
\end{equation}
for all $X,Y \in \mathcal{C}$ and $M \in \mathcal{M}$. This defines a 2-cochain $\mathrm{obs}(\gamma_1,\ldots,\gamma_k) \in \mathrm{C}^2_{\mathrm{mix}}(\mathcal{M};\mathsf{F},\mathsf{F})$, which we call the {\em obstruction}. Its relevance comes from the following fact: $\gamma^F + \sum_{i=0}^N \gamma_i h^i$ is a deformation of $\gamma^F$ of order $N$ if and only if 
\begin{equation}\label{obstExtCModFunct}
d(\gamma_1) = 0 \quad \text{and} \quad \forall \, 2 \leq i \leq N, \:\: d(\gamma_i) = \mathrm{obs}(\gamma_1,\ldots,\gamma_{i-1})
\end{equation}
where $d$ is the differential \eqref{mixDiff} in degree $1$.  These conditions are obtained by a straightforward computation.

\begin{proposition}\label{propExtensionCmodStruct}
Let $\boldsymbol{\gamma} = \gamma^F + \sum_{i=1}^N \gamma_i h^i$ be a deformation of $\gamma^F$ of order $N$.

\smallskip

1. Let $\gamma' \in \mathrm{C}^1_{\mathrm{mix}}(\mathcal{M}; \mathsf{F},\mathsf{F})$. Then $\boldsymbol{\gamma} + h^{N+1} \gamma'$ is a deformation of order $N+1$ if and only if $d(\gamma') = \mathrm{obs}(\gamma_1,\ldots,\gamma_N)$.

\smallskip

\noindent 2. If $\mathrm{H}^2_{\mathrm{mix}}(\mathcal{M};\mathsf{F},\mathsf{F}) = 0$ then $\boldsymbol{\gamma}$ can be lifted to any order.
\end{proposition}
\begin{proof}
1. This is just a rephrasing of \eqref{obstExtCModFunct}. 
\\2. The obstruction defined in \eqref{defObsModuleFunct} for deformations of the $\mathcal{C}$-module structure of $\mathsf{F} \in \mathrm{End}_{\mathcal{C}}(\mathcal{M})$ corresponds through the isomorphism $\Phi^2 : \mathrm{C}^2_{\mathrm{DY}}(\rho;\mathsf{F},\mathsf{F}) \overset{\sim}{\to} \mathrm{C}^2_{\mathrm{mix}}(\mathcal{M};\mathsf{F},\mathsf{F})$ of Prop.\,\ref{propMixCohomDY} to the obstruction defined in \eqref{defObsMonStruct} for deformations of the half-braiding of $\mathsf{F} \in \mathcal{Z}(\Gamma)$, in the case where $\Gamma = \rho$ and $\Gamma^{(2)} = \widehat{m}$. In this way Prop.\,\ref{propExtensionCmodStruct} appears as a particular case of Prop.\,\ref{propLiftHB}.
\end{proof}

\section{Deformations and rigidity via the adjoint algebra}
Recall from Prop.\,\ref{propMixCohomDY} that the mixed associator cohomology $\mathrm{H}^\bullet_{\mathrm{mix}}(\mathcal{M})$ of a $\mathcal{C}$-module category $\mathcal{M}$ equals the DY cohomology $\mathrm{H}^\bullet_{\mathrm{DY}}(\rho)$ of its representation functor $\rho : \mathcal{C} \to \End(\mathcal{M})$. Our main goal here, achieved in \S\ref{subsecAdjThm}, is to apply the adjunction theorem for DY cohomology \cite{FGS2} to $\mathrm{H}^\bullet_{\mathrm{DY}}(\rho)$. This gives a description of $\mathrm{H}^\bullet_{\mathrm{mix}}(\mathcal{M})$ in terms of the relative Ext groups of the adjunction $\mathcal{Z}(\mathcal{C}) \rightleftarrows \mathcal{C}$, and this description can be used for effective computations as will be demonstrated on examples in \S\ref{sectionExamples}. In this process, a special object $\mathcal{A}_{\mathcal{M}} \in \mathcal{Z}(\mathcal{C})$ naturally appears as a coefficient in relative Ext's, which is called the adjoint algebra of $\mathcal{M}$; in the case where $\mathcal{M} = \mathrm{Mod}_{\mathcal{C}}(A)$, we prove in \S\ref{subsecAdjObjFullCenter} that $\mathcal{A}_{\mathcal{M}}$ is the ``full center'' of the algebra $A \in \mathcal{C}$. From the relative Ext description of $\mathrm{H}_{\mathrm{mix}}^\bullet(\mathcal{M})$, we also establish a general Ocneanu-type rigidity theorem (\S\ref{subsubOcneanu}) for module categories over fusion categories and prove the rigidity of the regular $\mathcal{C}$-module for any finite $\mathcal{C}$ (\S\ref{subsubRigidityCmod}).
For a self-contained review of relative homological algebra in the current context, see \cite[\S 2.1]{FGS}.

\subsection{Adjunction associated to a module category}\label{subsecExtModCat}
Let $\mathcal{M} = (\mathcal{M},\rhd,m)$ be a $\mathcal{C}$-module category, where $m$ is the mixed associator. We make the following assumptions on the categories $\mathcal{C}, \mathcal{M}$ and the functor $\rhd : \mathcal{C} \times \mathcal{M} \to \mathcal{M}$
\begin{equation}\label{assumptionsActionFunct}
\begin{cases}
\: \mathcal{C} \text{ and } \mathcal{M} \text{ are }\Bbbk\text{-linear abelian categories, where } \Bbbk \text{ is a field.}\\
\: \mathcal{C} \text{ has }\Bbbk \text{-bilinear monoidal product } \otimes : \mathcal{C} \times \mathcal{C} \to \mathcal{C} \text{ (strict for simplicity).}\\
\: \mathcal{C} \text{ is rigid and finite as a }\Bbbk \text{-linear category.}\\
\: \rhd \text{ is } \Bbbk\text{-bilinear and } \rhd \text{ is right-exact in the first variable.}
\end{cases}
\end{equation}
where the last assumption means that the functor $- \rhd M : \mathcal{C} \to \mathcal{M}$ is right exact for all $M \in \mathcal{M}$. 
The concepts of {\em $\Bbbk$-linear} and {\em abelian} categories are very classical, see  for instance Def.\,1.2.2 and Def.\,1.3.1 in \cite{EGNO}. We also recall that:
\begin{itemize}[topsep=.2em, itemsep=-.2em]
\item a monoidal category is called {\em rigid} if every object has a left and right dual \cite[\S 2.10]{EGNO},
\item a $\Bbbk$-linear abelian category is called {\em finite} if it is equivalent to the category of finite-dimensional modules over a finite-dimensional associative $\Bbbk$-algebra \cite[Def.\,1.8.5]{EGNO}.
\end{itemize}
In the terminology of \cite[Def.\,4.1.1]{EGNO} the above assumptions on $\mathcal{C}$ are that it is a finite  {\em multitensor} category. In contrast to a {\em tensor} category, the monoidal unit $\boldsymbol{1}$ of $\mathcal{C}$ is not required to satisfy $\mathrm{End}_{\mathcal{C}}(\boldsymbol{1}) \cong \Bbbk$. In fact, $\boldsymbol{1}\in\mathcal{C}$ is semisimple by~\cite[Thm.\,4.3.8]{EGNO}.

\begin{remark}\label{remarkRightExactDSPS}
1. It is well-known that rigidity of $\mathcal{C}$ implies exactness in each variable of the monoidal product $\otimes_{\mathcal{C}}$ \cite[Prop.\,4.2.1]{EGNO}. The same argument implies exactness of $\rhd$ in the second variable. Indeed for all $X \in \mathcal{C}$ there are natural isomorphisms
\[ \Hom_{\mathcal{M}}(M, X \rhd N) \cong \Hom_{\mathcal{M}}(X^* \rhd M, N), \quad \Hom_{\mathcal{M}}(X \rhd M, N) \cong \Hom_{\mathcal{M}}(M, {^*\!X} \rhd N), \]
proving that the functor $X \rhd -$ has left and right adjoints, hence is left and right exact. 
\\2. If in addition to \eqref{assumptionsActionFunct} we assume that $\mathcal{M}$ is finite as a $\Bbbk$-linear category then $\rhd$ is necessarily left exact in the first variable \cite[Cor.\,2.26]{DSPS}. Hence, in the finite case, the assumptions \eqref{assumptionsActionFunct} imply that $\rhd$ is exact in each variable.
\end{remark}

\indent Let us translate the assumptions \eqref{assumptionsActionFunct} in terms of the representation functor $\rho : \mathcal{C} \to \End(\mathcal{M})$ described in \S\ref{subsecMixMon}. It is well-known that if $\mathcal{M}$ is abelian then so is $\End(\mathcal{M})$; we just recall that by definition a sequence $E \overset{\alpha}{\implies} F \overset{\beta}{\implies} G$ is exact in $\End(\mathcal{M})$ if and only the sequence $E(M) \overset{\alpha_M}{\longrightarrow} F(M) \overset{\beta_M}{\longrightarrow} G(M)$ is exact in $\mathcal{M}$ for each $M \in \mathcal{M}$.

Denote by $\mathrm{Rex}_\Bbbk(\mathcal{M})$ the full subcategory of $\End(\mathcal{M})$ consisting of right-exact and $\Bbbk$-linear endofunctors. It is a $\Bbbk$-linear abelian category. Moreover $\mathrm{Rex}_\Bbbk(\mathcal{M})$ is a monoidal subcategory, where we recall that the monoidal product is simply the composition of functors. We also note an important property:

\begin{lemma}\label{lemmaRephraseAssump}
The monoidal product of $\mathrm{Rex}_\Bbbk(\mathcal{M})$ is right-exact in each variable.
\end{lemma}
\begin{proof}
Let $E \!\overset{\alpha}{\implies}\! F \!\overset{\beta}{\implies}\! G \Rightarrow 0$ be exact in $\mathrm{Rex}_\Bbbk(\mathcal{M})$. Let $K \in \mathrm{Rex}_\Bbbk(\mathcal{M})$ and form the sequence $K \circ E \overset{\!\!\mathrm{id}_K \bullet \alpha}{=\!=\!=\!\Rightarrow} K \circ F \overset{\!\!\mathrm{id}_K \bullet \beta}{=\!=\!=\!\Rightarrow} K \circ G \Rightarrow 0$, where $\bullet$ is the monoidal product on natural transformations recalled in \eqref{horizontalComp}. We have $(\mathrm{id}_K \bullet \alpha)_M = K(\alpha_M)$ and similarly for $\beta$. Thus, on components $M \in \mathcal{M}$, the sequence is $K(E(M)) \xrightarrow{K(\alpha_M)} K(F(M)) \xrightarrow{K(\beta_M)} K(G(M)) \to 0$ and is exact because $K$ is right-exact. This shows that $K \circ -$ is a right-exact functor. Similarly for the functor $- \circ K$.
\end{proof}

Because of the last assumption in \eqref{assumptionsActionFunct} and of Remark \ref{remarkRightExactDSPS}(1), the representation functor $\rho : X \mapsto (X \rhd -)$ takes values in the subcategory $\mathrm{Rex}_\Bbbk(\mathcal{M})$ and it is $\Bbbk$-linear and right-exact.

\medskip

\noindent \textbf{Convention.} In the sequel, we consider $\rho$ as a functor $\mathcal{C} \to \mathrm{Rex}_\Bbbk(\mathcal{M})$, i.e.\ we restrict its target to $\mathrm{Rex}_\Bbbk(\mathcal{M})$ instead of the whole $\End(\mathcal{M})$.

\medskip

\indent Fix a mixed associator $m \in \mathrm{Mix}(\rhd)$ so that $\rho : \mathcal{C} \to \mathrm{Rex}_\Bbbk(\mathcal{M})$ becomes a monoidal functor with $\widehat{m} : \rho(-) \circ \rho(-) \!\overset{\sim}{\implies}\! \rho(- \otimes -)$ from Lemma \ref{lemmaBijMixMon}. Because of the convention above, objects in the category $\mathcal{Z}(\rho)$ are pairs $(F,b)$ with $F \in \mathrm{Rex}_\Bbbk(\mathcal{M})$. Thus the identification \eqref{isoCentralizerEndC} now restricts to
\[ \mathcal{Z}(\rho) \overset{\sim}{\longrightarrow} \mathrm{Rex}_{\mathcal{C}}(\mathcal{M}) \]
where $\mathrm{Rex}_{\mathcal{C}}(\mathcal{M})$ is the full subcategory of $\End_{\mathcal{C}}(\mathcal{M})$ consisting of $\Bbbk$-linear and right-exact $\mathcal{C}$-module endofunctors $(F,\gamma)$. 

\smallskip

To sum up, due to assumptions \eqref{assumptionsActionFunct} and Lemma \ref{lemmaRephraseAssump} we are in the following situation:
\begin{equation}\label{restateAssumptions}
\begin{cases}
    \: \mathcal{C} \text{ and } \mathrm{Rex}_\Bbbk(\mathcal{M}) \text{ are } \Bbbk\text{-linear abelian categories, where }\Bbbk\text{ is a field.} &\\
    \: \mathcal{C} \text{ and } \mathrm{Rex}_\Bbbk(\mathcal{M}) \text{ have } \Bbbk\text{-linear monoidal products.} &\\
    \:  \mathcal{C} \text{ is rigid and finite as a }\Bbbk\text{-linear category.} &\\
    \: \text{The monoidal product of } \mathrm{Rex}_\Bbbk(\mathcal{M}) \text{ (composition) is right-exact in each variable.} &\\
    \: \rho = (\rho,\widehat{m}): \mathcal{C} \to \mathrm{Rex}_\Bbbk(\mathcal{M}) \text{ is a } \Bbbk\text{-linear monoidal right-exact functor.}
\end{cases}
\end{equation}
By \cite[\S 3.2]{FGS2}, it follows that the forgetful functor 
\[ \mathcal{U}_\rho : \mathrm{Rex}_{\mathcal{C}}(\mathcal{M}) \to \mathrm{Rex}_\Bbbk(\mathcal{M}), \quad (F,\gamma) \mapsto F, \quad \alpha \mapsto \alpha \] has a left adjoint $\mathcal{F}_\rho \dashv \mathcal{U}_\rho$ which gives rise to a resolvent pair. For a short review on resolvent pairs and their relative Ext groups, see \cite[\S 2.1]{FGS}. In the present case, following notations from \cite{FGS,FGS2}, we denote the relative Ext groups by $\Ext^n_{\mathrm{Rex}_{\mathcal{C}}(\mathcal{M}),\, \mathrm{Rex}_\Bbbk(\mathcal{M})}(\mathsf{F}, \mathsf{G})$, where $\mathsf{F}, \mathsf{G}$ are any objects in $\mathrm{Rex}_{\mathcal{C}}(\mathcal{M})$ and are called {\em coefficients}.

\begin{proposition}\label{PropExtForMixAssoCohom}
Let $\mathcal{M} = (\mathcal{M},\rhd,m)$ be a $\mathcal{C}$-module category satisfying the assumptions \eqref{assumptionsActionFunct}. For all $n \geq 0$ and $\mathsf{F}, \mathsf{G} \in \mathrm{Rex}_{\mathcal{C}}(\mathcal{M})$, there is an isomorphism of vector spaces
\[ \mathrm{H}^n_{\mathrm{mix}}(\mathcal{M}; \mathsf{F}, \mathsf{G}) \cong \Ext^n_{\mathrm{Rex}_{\mathcal{C}}(\mathcal{M}), \,\mathrm{Rex}_\Bbbk(\mathcal{M})}(\mathsf{F}, \mathsf{G}) \]
where the left-hand side is the mixed associator cohomology defined in \S\ref{subsecMixAssoCohom}. In particular the deformation cohomology of $m$ can be computed as
\[ \mathrm{H}^n_{\mathrm{mix}}(\mathcal{M}) \cong \Ext^n_{\mathrm{Rex}_{\mathcal{C}}(\mathcal{M}), \,\mathrm{Rex}_\Bbbk(\mathcal{M})}(\mathsf{Id}, \mathsf{Id}) \]
where $\mathsf{Id} = \bigl( \mathrm{Id}_{\mathcal{M}}, \mathrm{id} \bigr)$ is the monoidal unit of $\mathrm{Rex}_{\mathcal{C}}$.
\end{proposition}
\begin{proof}
We have $\mathrm{H}^\bullet_{\mathrm{mix}}(\mathcal{M}; \mathsf{F}, \mathsf{G}) \cong \mathrm{H}^\bullet_{\mathrm{DY}}(\rho; \mathsf{F}, \mathsf{G}) \cong \Ext^n_{\mathrm{Rex}_{\mathcal{C}}(\mathcal{M}), \,\mathrm{Rex}_\Bbbk(\mathcal{M})}(\mathsf{F}, \mathsf{G})$, where the first isomorphism is Prop.\,\ref{propMixCohomDY} while the second is by \cite[Thm.\,3.3]{FGS2} together with the identification $\mathcal{Z}(\rho) \cong \mathrm{Rex}_{\mathcal{C}}(\mathcal{M})$. We note that \cite[Thm.\,3.3]{FGS2} applies thanks to the conditions \eqref{restateAssumptions}, which  are just a translation of the assumptions \eqref{assumptionsActionFunct} in terms of the functor $\rho$.
\end{proof}

\begin{remark}
Actually the isomorphisms in Prop.\,\ref{PropExtForMixAssoCohom} come from an isomorphism at the level of cochain complexes:
\[ \mathrm{C}^\bullet_{\mathrm{mix}}(\mathcal{M};\mathsf{F},\mathsf{G}) \cong \Hom_{\mathrm{Rex}_{\mathcal{C}}(\mathcal{M})}\bigl( \mathrm{Bar}^\bullet_{\mathrm{Rex}_{\mathcal{C}}(\mathcal{M}), \,\mathrm{Rex}_\Bbbk(\mathcal{M})}(\mathsf{F},\mathsf{G}) \bigr) \]
where we use the bar complex of the resolvent pair (see \cite[\S 2.1]{FGS} for the definition). This is again due to Prop.\,\ref{propMixCohomDY} and \cite[Thm.\,3.3]{FGS2}. Indeed, the latter asserts that $\mathrm{H}^\bullet_{\mathrm{DY}}(\rho; \mathsf{F}, \mathsf{G}) \cong \Ext^n_{\mathrm{Rex}_{\mathcal{C}}(\mathcal{M}), \mathrm{Rex}_\Bbbk(\mathcal{M})}(\mathsf{F}, \mathsf{G})$ holds at the level of complexes, provided the cohomology is computed through the bar complex.
\end{remark}

\begin{example}\label{exModCatVect}
Let $A$ be an associative $\Bbbk$-algebra and $\mathcal{M} = A\text{-}\mathrm{Mod}$ be its category of modules, which becomes a module category over $\mathcal{C} = \mathrm{vect}_\Bbbk$ if we define $X \rhd M = X \otimes_\Bbbk M$ with $A$-action $a \cdot (x \otimes m) = x \otimes (a \cdot m)$. Then $\mathrm{Rex}_{\mathcal{C}}(\mathcal{M}) = \mathrm{Rex}_\Bbbk(\mathcal{M})$, which means that the adjunction $\mathrm{Rex}_{\mathcal{C}}(\mathcal{M}) \rightleftarrows \mathrm{Rex}_\Bbbk(\mathcal{M})$ consists of the identity functor in both directions. It follows that any object in $\mathrm{Rex}_{\mathcal{C}}(\mathcal{M})$ is relatively projective with respect to this adjunction, and hence all relative Ext spaces $\Ext_{\mathrm{Rex}_{\mathcal{C}}(\mathcal{M}),\, \mathrm{Rex}_\Bbbk(\mathcal{M})}^n$ vanish for $n>0$. As a result $\mathrm{H}^n_{\mathrm{mix}}(\mathcal{M}) = 0$ for all $n > 0$.
\end{example}

It might be difficult in general to work with the resolvent pair $\mathrm{Rex}_{\mathcal{C}}(\mathcal{M}) \rightleftarrows \mathrm{Rex}_\Bbbk(\mathcal{M})$. The Adjunction Theorem proved in \cite[\S 3.3]{FGS2} will transport the computations within the familiar resolvent pair $\mathcal{Z}(\mathcal{C}) \rightleftarrows \mathcal{C}$, as we explain in the next section.

\subsection{Adjunction theorem for finite module categories}\label{subsecAdjThm}
Our goal here is to apply the adjunction techniques from \cite[\S 3.3]{FGS2} in order to replace the relative Ext groups $\Ext^n_{\mathrm{Rex}_{\mathcal{C}}(\mathcal{M}), \mathrm{Rex}_\Bbbk(\mathcal{M})}$ in Prop.\,\ref{PropExtForMixAssoCohom} by relative Ext groups $\Ext^n_{\mathcal{Z}(\mathcal{C}),\mathcal{C}}$. This will give a way for effective computations. However, an extra assumption is required compared to \S\ref{subsecExtModCat}:

\medskip

\noindent \textbf{Assumption.} The $\mathcal{C}$-module category $\mathcal{M} = (\mathcal{M},\rhd,m)$ satisfies \eqref{assumptionsActionFunct} and is moreover {\em finite as a $\Bbbk$-linear category}.

\medskip

\indent Let us recall the definition of the internal Hom bifunctor which is intensively used in the sequel; also see e.g.\ in \cite[\S 7.9]{EGNO}. Given an object $M \in \mathcal{M}$ consider the functor $- \rhd M :  \mathcal{C} \to \mathcal{M}$, which is right-exact by assumption \eqref{assumptionsActionFunct}. Since $\mathcal{M}$ is finite, it follows that the right adjoint of $- \rhd M$ exists \cite[Cor.\,1.9]{DSPS}. We denote this adjoint by $\underline{\Hom}(M,-) : \mathcal{M} \to \mathcal{C}$; its defining property is thus the existence of isomorphisms
\begin{equation}\label{adjIntHom}
I_{X,M,N} : \Hom_{\mathcal{M}}(X \rhd M,N) \overset{\sim}{\longrightarrow} \Hom_{\mathcal{C}}\bigl( X, \underline{\Hom}(M,N) \bigr)
\end{equation}
which are natural in $X$ and $N$. By the theorem on ``adjunctions with parameter'' \cite[\S IV.7, Thm.\,3]{MLCat}, the family of right adjoints $\underline{\Hom}(M,-)$ can be assembled into a bifunctor
\[ \underline{\Hom}(-,-) : \mathcal{M}^{\mathrm{op}} \times \mathcal{M} \to \mathcal{C} \]
such that the isomorphisms $I_{X,M,N}$ are also natural in $M$. Explicitly, naturality in $X$, $M$ and $N$ means that for $f \in \Hom_{\mathcal{M}}(M_2,M_1)$, $g \in \Hom_{\mathcal{M}}(N_1,N_2)$ and $h \in \Hom_{\mathcal{C}}(Y,X)$ we have
\begin{equation}\label{natAdjIso}
\underline{\Hom}(f,g) \circ I_{X,M_1,N_1}(\varphi) \circ h = I_{Y,M_2,N_2}\bigl( g \circ \varphi \circ (h \rhd f) \bigr)
\end{equation}
for all $\varphi \in \Hom_{\mathcal{M}}(X \rhd M_1, N_1)$; note the contravariance in $X$ and $M$.

For all $X \in \mathcal{C}$, $M,N \in \mathcal{M}$ we denote by
\[ \underline{\mathsf{ev}}_{M,N} : \underline{\Hom}(M,N) \rhd M \to N \quad \text{and} \quad \underline{\mathsf{coev}}_{X,M} : X \to \underline{\Hom}(M,X \rhd M) \]
the images of $\mathrm{id}_{\underline{\Hom}(M,N)}$ and $\mathrm{id}_{X \rhd M}$ through the adjunction isomorphisms \eqref{adjIntHom}:
\begin{align*}
&\Hom_{\mathcal{C}}\bigl( \underline{\Hom}(M,N), \underline{\Hom}(M,N) \bigr) \xrightarrow{\:I^{-1}_{\underline{\Hom}(M,N), M, N}\:} \Hom_{\mathcal{M}}\bigl( \underline{\Hom}(M,N) \rhd M, N \bigr)\\
\text{and }\: & \Hom_{\mathcal{M}}(X \rhd M, X \rhd M) \xrightarrow{\:I_{X,M,X\rhd M}\:} \Hom_{\mathcal{C}}\bigl( X, \underline{\Hom}(M, X \rhd M) \bigr).
\end{align*}
By definition, the natural families $\underline{\mathsf{ev}}_{M,-}$ and $\underline{\mathsf{coev}}_{-,M}$ are the counit and unit of the adjunction.

\smallskip

The finiteness assumption on $\mathcal{M}$ ensures the existence of a right adjoint to the representation functor $\rho : \mathcal{C} \to \mathrm{Rex}_\Bbbk(\mathcal{C})$, which has been described in \cite[\S 3]{shimizuCoend} using the internal Hom:

\begin{lemma}\label{lemmaRightAdjointRho}{\em \cite[Thm.\,3.4]{shimizuCoend}}
For all $F \in \mathrm{Rex}_\Bbbk(\mathcal{M})$, the end
\[ \int_{M \in \mathcal{M}}\underline{\Hom}\bigl(M, F(M) \bigr) \]
exists in $\mathcal{C}$ and equals $\rho^{\mathrm{ra}}(F)$, where $\rho^{\mathrm{ra}} : \mathrm{Rex}_\Bbbk(\mathcal{M}) \to \mathcal{C}$ is the right adjoint of $\rho$.
\end{lemma}
\begin{proof}
For the convenience of the reader we reproduce the main idea of Shimizu's proof, which consists in the following isomorphisms, all natural in the variable $X \in \mathcal{C}$:
\begin{align*}
&\Hom_{\mathcal{C}}\bigl( X, \rho^{\mathrm{ra}}(F) \bigr) \cong \Hom_{\mathrm{Rex}_\Bbbk(\mathcal{M})}\bigl( X \rhd -, F \bigr) =  \mathrm{Nat}(X \rhd -, F)\\
\cong\:&\textstyle \int_{M \in \mathcal{M}} \Hom_{\mathcal{M}}\bigl(X \rhd M, F(M) \bigr) \cong \int_{M \in \mathcal{M}} \Hom_{\mathcal{C}}\bigl(X, \underline{\Hom}(M,F(M)) \bigr)\\
\cong\:& \textstyle \Hom_{\mathcal{C}}\!\left( X, \int_{M \in \mathcal{M}}\underline{\Hom}(M, F(M)) \right).
\end{align*}
The Yoneda lemma thus implies that $\rho^{\mathrm{ra}}(F) \cong \int_{M \in \mathcal{M}}\underline{\Hom}(M, F(M))$.
\end{proof}

Given $F,G \in \mathrm{Rex}_\Bbbk(\mathcal{M})$ and a natural transformation $\alpha : F \Rightarrow G$, the morphism $\rho^{\mathrm{ra}}(\alpha) \in \Hom_{\mathcal{C}}\bigl( \rho^{\mathrm{ra}}(F), \rho^{\mathrm{ra}}(G) \bigr)$ is characterized by the following commutative diagram
\begin{equation*}
\xymatrix@C=4em{
\rho^{\mathrm{ra}}(F) \ar[r]^-{\pi_M^F} \ar[d]_-{\exists! \, \rho^{\mathrm{ra}}(\alpha)} & \underline{\Hom}\bigl( M, F(M) \bigr) \ar[d]^-{\underline{\Hom}(\mathrm{id}_M,\alpha_M)}\\
\rho^{\mathrm{ra}}(G) \ar[r]^-{\pi_M^G} & \underline{\Hom}\bigl( M, G(M) \bigr)
} \end{equation*}
for all $M \in \mathcal{M}$, where we denote by
\begin{equation}\label{dinatTranfoEndRa}
\pi^F = \bigl[\pi_M^F :  \rho^{\mathrm{ra}}(F) \to \underline{\Hom}(M,F(M)) \bigr]_{M \in \mathcal{M}}
\end{equation}
the universal dinatural transformation of the end in Lemma \ref{lemmaRightAdjointRho}.

\smallskip

\indent Denote by $\mathcal{Z}(\mathcal{C})$ the Drinfeld center of $\mathcal{C}$, whose objects are pairs $(V,t)$ with $V \in \mathcal{C}$ and $t : V \otimes - \!\overset{\!\sim}{\implies}\! - \otimes V$ is a half-braiding. By the general theory developped in \cite[\S 3.3]{FGS2}, we have the following diagram:
\begin{equation}\label{diagramLiftAdjunctions}
\xymatrix@C=6em@R=.7em{
\mathcal{Z}(\mathcal{C}) \hspace{-14.5em}& \phantom{XX} \ar@/^.7em/[dd]^-{\mathcal{U}_{\mathcal{C}}} \ar@/^.7em/[r]^-{\widetilde{\rho}}_{\text{\normalsize \rotatebox{270}{$\dashv$}}} & \ar@{-->}@/^.7em/[l]^-{\exists\,\widetilde{\rho}^{\mathrm{ra}}} \ar@/^.7em/[dd]^-{\mathcal{U}_\rho} \phantom{XX}& \hspace{-12.3em} \mathrm{Rex}_{\mathcal{C}}(\mathcal{M})\\
&\dashv & \dashv\\
\mathcal{C} \hspace{-15em}& \phantom{XX} \ar@/^.7em/[uu]^-{\mathcal{F}_{\mathcal{C}}} \ar@/^.7em/[r]^-{\rho}_-{\text{\normalsize \rotatebox{270}{$\dashv$}}} & \ar@/^.7em/[l]^-{\rho^{\mathrm{ra}}} \ar@/^.7em/[uu]^-{\mathcal{F}_\rho} \phantom{XX} & \hspace{-12.3em}\mathrm{Rex}_\Bbbk(\mathcal{M})
} \end{equation}
where:

\smallskip

\indent \textbullet~ $\mathcal{U}_{\mathcal{C}} : (V,t) \mapsto V$ and $\mathcal{U}_\rho : (F,\gamma) \mapsto \gamma$ are the forgetful functors. The fact that they have left adjoints is a particular case of \cite[\S 3.2]{FGS2}, thanks to the assumptions \eqref{restateAssumptions}.

\smallskip

\indent \textbullet~ $\widetilde{\rho}(V,t) = \bigl( \rho(V), \gamma(t) \bigr)$ where $\gamma(t) : \rho(V)(- \rhd -) \!\overset{\sim}{\implies}\! - \rhd \rho(V)(-)$ has components
\begin{equation}\label{CmodStructFromHalfBr}
\begin{aligned}
\gamma(t)_{X,M} : V \rhd (X \rhd M) \xrightarrow{m_{V,X,M}} (V \otimes X) \rhd M &\xrightarrow{t_X \, \rhd \, \mathrm{id}_M} (X \otimes V) \rhd M\\
&\xrightarrow{m_{X,V,M}^{-1}} X \rhd (V \rhd M)
\end{aligned}
\end{equation}
for all $X \in \mathcal{C}$ and $M \in \mathcal{M}$. For a morphism $f$ in $\mathcal{Z}(\mathcal{C})$ we let $\widetilde{\rho}(f) = \rho(f)$. Then $\widetilde{\rho}$ is a ``lift'' of $\rho$ in the sense that $\mathcal{U}_\rho \, \widetilde{\rho} = \rho \, \mathcal{U}_{\mathcal{C}}$. We note that the functor $\widetilde{\rho}$ is monoidal, because the monoidal structure of $\rho$ defined in Lemma \ref{lemmaBijMixMon} is compatible with the $\mathcal{C}$-module structures \eqref{CmodStructFromHalfBr}.

\smallskip

\indent \textbullet~  In \cite[\S 3.3]{FGS2} we proved that the lifting theorem for right adjoints \cite{johnstone} applies in situations like \eqref{diagramLiftAdjunctions}, which means that there exists a right adjoint functor $\widetilde{\rho}^{\mathrm{ra}}$ of $\widetilde{\rho}$. It satisfies $\mathcal{U}_{\mathcal{C}} \, \widetilde{\rho}^{\mathrm{ra}} = \rho^{\mathrm{ra}} \, \mathcal{U}_\rho$.

\smallskip

\indent \textbullet~  We  also recall that the left adjoint $\mathcal{F}_{\mathcal{C}}$ is given on objects by $\mathcal{F}_{\mathcal{C}}(V) = \int^{Y \in \mathcal{C}} Y^* \otimes V \otimes Y$, where this coend is equipped with the half-braiding $t^V$ uniquely characterized by
\begin{equation}\label{diagramHalfBrCoend}
\xymatrix@C=5em{
Y^* \otimes V \otimes Y \otimes X \ar[r]^-{\mathrm{coev}_X \,\otimes\, \mathrm{id}} \ar[d]_{i_Y(V) \,\otimes\, \mathrm{id}_Y} & X \otimes X^* \otimes Y^* \otimes V \otimes Y \otimes X \ar[d]^{\mathrm{id}_X \,\otimes\, i_{Y \otimes X}(V)}\\
\mathcal{F}_{\mathcal{C}}(V) \otimes X \ar[r]_-{\exists!\,t^V_X} & X \otimes \mathcal{F}_{\mathcal{C}}(V)
} \end{equation}
for all $X,Y \in \mathcal{C}$, denoting by $i_Y(V)$ the universal dinatural transformation.

\medskip

\indent In order to describe the functor $\widetilde{\rho}^{\mathrm{ra}}$ (\cite[Prop.\,3.7]{FGS2}, here we follow \cite[\S 3.6]{shimizuCoend}), consider the natural family of morphisms
\begin{equation}\label{isoJIntHom}
\begin{aligned}
J_{X,M,Y,N} : \underline{\Hom}(X \rhd M, Y \rhd N) \otimes X &\xrightarrow{B_{X,M,Y,N} \,\otimes\, \mathrm{id}_X} Y \otimes \underline{\Hom}(M,N) \otimes X^* \otimes X\\
&\xrightarrow{\mathrm{id}_{Y \otimes \underline{\Hom}(M,N)} \,\otimes\,\mathrm{ev}_X} Y \otimes \underline{\Hom}(M,N)
\end{aligned}
\end{equation}
with $X,Y \in \mathcal{C}$, $M,N \in \mathcal{M}$ and where $B_{X,M,Y,N}$ is the $\mathcal{C}$-bimodule structure of the internal Hom bifunctor, whose definition is recalled in \eqref{isosBimodStructIntHom} of App.\,\ref{appBimodStructIntHom}. Then for all $\mathsf{F} = (F,\gamma) \in \mathrm{Rex}_{\mathcal{C}}(\mathcal{M})$, we have $\widetilde{\rho}^{\mathrm{ra}}(\mathsf{F}) = \bigl( \rho^{\mathrm{ra}}(F), b^{\mathsf{F}} \bigr)$ where the half-braiding
\[ b^{\mathsf{F}}_X : \rho^{\mathrm{ra}}(F) \otimes X \overset{\sim}{\longrightarrow} X \otimes \rho^{\mathrm{ra}}(F) \]
is defined for all $X \in \mathcal{C}$ by the following commutative diagram
\begin{equation}\label{halfBrEnd}
\xymatrix@C=6em{
\rho^{\mathrm{ra}}(F) \otimes X \ar[dd]_-{\exists!\, b^{\mathsf{F}}_X} \ar[r]^-{\pi^F_{X \rhd M} \,\otimes\, \mathrm{id}_X} & \underline{\Hom}\bigl(X \rhd M, F(X \rhd M)\bigr) \otimes X \ar[d]^-{\underline{\Hom}(\mathrm{id}_{X\,\rhd\,M},\gamma_{X,M}) \,\otimes\,\mathrm{id}_X}\\
& \underline{\Hom}\bigl( X \rhd M, X \rhd F(M) \bigr) \otimes X \ar[d]^-{J_{X,M,X,F(M)}}\\
X \otimes \rho^{\mathrm{ra}}(F) \ar[r]_-{\mathrm{id}_X \,\otimes\, \pi^F_M} & X \otimes \underline{\Hom}(M, F(M))
} \end{equation}
for all $M \in \mathcal{M}$, thanks to universality of the dinatural transformation $\pi^F$ introduced in \eqref{dinatTranfoEndRa}. On morphisms $\widetilde{\rho}^{\mathrm{ra}}$ is simply equal to $\rho^{\mathrm{ra}}$.

\smallskip

\indent Let $\mathsf{Id}_{\mathcal{M}} = \bigl( \mathrm{Id}_{\mathcal{M}}, \mathrm{id} \bigr)$, which is the monoidal unit of $\mathrm{Rex}_{\mathcal{C}}(\mathcal{M})$.

\begin{definition}\label{defAdjObj}
The adjoint object of the $\mathcal{C}$-module category $\mathcal{M}$ is $\mathcal{A}_{\mathcal{M}} = \widetilde{\rho}^{\mathrm{ra}}(\mathsf{Id}_{\mathcal{M}}) \in \mathcal{Z}(\mathcal{C})$.
\end{definition}

\noindent By Lemma \ref{lemmaRightAdjointRho}, $\mathcal{A}_{\mathcal{M}}$ is $\int_M \underline{\Hom}(M,M)$ as an object in $\mathcal{C}$, equipped with the half-braiding defined by taking $F = \mathrm{Id}$ and $\gamma = \mathrm{id}$ in \eqref{halfBrEnd}.

\begin{remark}
We can also call $\mathcal{A}_{\mathcal{M}}$ the ``adjoint algebra of $\mathcal{M}$'' because it has a natural structure of an algebra in $\mathcal{Z}(\mathcal{C})$; this is explained in \S\ref{subsubAdjAlg} below. In the case $\mathcal{M} = \mathcal{C}$ it agrees with the adjoint algebra defined by Shimizu in \cite[\S 3.2]{shimizuMonCent}, \cite[\S 7.2]{shimizuCoend}, see Examples \ref{exampleAdjObjReg} and~\ref{exampleProductAC}. When moreover $\mathcal{C} = H\text{-}\mathrm{mod}$ for a finite-dimensional Hopf algebra $H$, then $\mathcal{A}_{\mathcal{C}}$ is the adjoint representation of $H$ as an object in $\mathcal{C}$ (Example \ref{exampleAdjObjRegularHopf}), whence the name.
\end{remark}

\begin{example}\label{exampleAdjObjReg}
We can regard $\mathcal{C}$ as a $\mathcal{C}$-module category by means of the monoidal product $\otimes$, and with trivial mixed associator. Then $\underline{\Hom}(M,N) = N \otimes M^*$ due to the usual duality adjunction \cite[Prop.\,2.10.8]{EGNO}. Hence $\mathcal{A}_{\mathcal{C}} = \int_{M \in \mathcal{C}} M \otimes M^*$ as an object in $\mathcal{C}$. The definition \eqref{halfBrEnd} of the half-braiding $b = b^{\mathsf{Id}}$ on $\mathcal{A}_{\mathcal{C}}$ reduces in this case to the following commutative diagram
\begin{equation}\label{diagramHalfBrEnd}
\xymatrix@C=5em{
\mathcal{A}_{\mathcal{C}} \otimes X \ar[d]_-{\exists!\,b_X} \ar[r]^-{\pi_{X \otimes M} \,\otimes\, \mathrm{id}_X} & X \otimes M \otimes M^* \otimes X^* \otimes X \ar[d]^-{\mathrm{id}_{X \otimes M \otimes M^*} \,\otimes\, \mathrm{ev}_X}\\
X \otimes \mathcal{A}_{\mathcal{C}} \ar[r]_-{\mathrm{id}_X \,\otimes\, \pi_M} & X \otimes M \otimes M^*
} \end{equation}
\end{example}

\begin{theorem}\label{thmAdjDYMod}
Let $\mathcal{M} = (\mathcal{M},\rhd,m)$ be a $\mathcal{C}$-module category which satisfies the assumptions \eqref{assumptionsActionFunct} and is moreover finite as a $\Bbbk$-linear category. There are isomorphisms of vector spaces
\[ \forall \, n \geq 0, \quad \mathrm{H}^n_{\mathrm{mix}}(\mathcal{M}) \cong \Ext_{\mathcal{Z}(\mathcal{C}),\, \mathcal{C}}^n(\boldsymbol{1}, \mathcal{A}_{\mathcal{M}}) \]
between the deformation cohomology of the mixed associator $m$ (\S\ref{subsecMixAssoCohom}) and the relative Ext spaces of the usual adjunction between $\mathcal{Z}(\mathcal{C})$ and $\mathcal{C}$, where $\boldsymbol{1}$ is the monoidal unit of $\mathcal{Z}(\mathcal{C})$ and $\mathcal{A}_{\mathcal{M}}$ is the adjoint object from Def.\,\ref{defAdjObj}.
\end{theorem}
\begin{proof}
We have
\begin{align*}
&\mathrm{H}^n_{\mathrm{mix}}(\mathcal{M}) = \mathrm{H}^n_{\mathrm{mix}}(\mathcal{M};\mathsf{Id}_{\mathcal{M}}, \mathsf{Id}_{\mathcal{M}}) \cong \Ext^n_{\mathrm{Rex}_{\mathcal{C}}(\mathcal{M}),\,\mathrm{Rex}_\Bbbk(\mathcal{M})}(\mathsf{Id}_{\mathcal{M}}, \mathsf{Id}_{\mathcal{M}})\\
=\:& \Ext^n_{\mathrm{Rex}_{\mathcal{C}}(\mathcal{M}),\,\mathrm{Rex}_\Bbbk(\mathcal{M})}\bigl(\widetilde{\rho}(\boldsymbol{1}), \mathsf{Id}_{\mathcal{M}}\bigr) \cong \Ext^n_{\mathcal{Z}(\mathcal{C}),\, \mathcal{C}}\bigl(\boldsymbol{1}, \widetilde{\rho}^{\mathrm{ra}}(\mathsf{Id}_{\mathcal{M}})\bigr) = \Ext_{\mathcal{Z}(\mathcal{C}), \,\mathcal{C}}^n(\boldsymbol{1}, \mathcal{A}_{\mathcal{M}})
\end{align*}
where the first equality is by definition in \eqref{defDeformCohom}, the second is by Prop.\,\ref{PropExtForMixAssoCohom}, the third is simply because $\widetilde{\rho}(\boldsymbol{1}) = \mathsf{Id}_{\mathcal{M}}$, the fourth is by the adjunction theorem for relative Ext \cite[\S 2.3, \S 3.3]{FGS2} and the last is by definition of $\mathcal{A}_{\mathcal{M}}$.
\end{proof}

It is worth mentionning the following dimension formula, which uses a covering of the monoidal unit $\boldsymbol{1} \in \mathcal{Z}(\mathcal{C})$ by a relatively projective object for the resolvent pair $\mathcal{Z}(\mathcal{C}) \rightleftarrows \mathcal{C}$:

\begin{corollary}\label{cor:dim-formula-Id}
Let $0 \to \mathsf{K} \to \mathsf{P} \to \boldsymbol{1} \to 0$ be an allowable short exact sequence in $\mathcal{Z}(\mathcal{C})$, with $\mathsf{P}$ relatively projective. Then for all $n \geq 1$
we can express the mixed
associator cohomology by the following representation-theoretic quantities:
\begin{align*}
\dim \mathrm{H}^n_{\mathrm{mix}}(\mathcal{M}) = \dim \Hom_{\mathcal{Z}(\mathcal{C})}\bigl(\mathsf{K}^{\otimes n},\mathcal{A}_{\mathcal{M}} \bigr) &- \dim \Hom_{\mathcal{Z}(\mathcal{C})}\bigl(\mathsf{P} \otimes \mathsf{K}^{\otimes (n-1)},\mathcal{A}_{\mathcal{M}} \bigr)\\
&+ \dim \Hom_{\mathcal{Z}(\mathcal{C})}\bigl(\mathsf{K}^{\otimes (n-1)},\mathcal{A}_{\mathcal{M}} \bigr)
\end{align*}
with the convention that $\mathsf{K}^{\otimes 0} = \boldsymbol{1}$.
\end{corollary}
\begin{proof}
For any object $\mathsf{Y} \in \mathcal{Z}(\mathcal{C})$, the dimension formula for relative Ext spaces established in \cite[Cor.\,3.5]{FGS} gives
\begin{equation}\label{dimFormulaFGS1}
\dim \Ext^n_{\mathcal{Z}(\mathcal{C}),\mathcal{C}}(\boldsymbol{1},\mathsf{Y}) = \dim \Hom_{\mathcal{Z}(\mathcal{C})}(\mathsf{K},\mathsf{M}_n) - \dim \Hom_{\mathcal{Z}(\mathcal{C})}(\mathsf{P},\mathsf{M}_n) + \dim \Hom_{\mathcal{Z}(\mathcal{C})}(\boldsymbol{1},\mathsf{M}_n)
\end{equation}
where $\mathsf{M}_n = \mathsf{Y} \otimes (\mathsf{K}^*)^{\otimes (n-1)}$.
We now take $\mathsf{Y} = \mathcal{A}_{\mathcal{M}}$, use duality in $\mathcal{Z}(\mathcal{C})$ to transform the factor $(\mathsf{K}^*)^{\otimes (n-1)}$ in the target of Hom spaces into a factor $\mathsf{K}^{\otimes (n-1)}$ in the source of Hom spaces, and combine the resulting formula with the isomorphism in Thm.\,\ref{thmAdjDYMod}.
\end{proof}

\begin{remark}
It's interesting to note that the same adjoint algebra object $\mathcal{A}_{\mathcal{M}}$, as in Thm.\,\ref{thmAdjDYMod}, appears in formulation of Hochschild cohomologies of $\mathcal{M}$ but using the usual $\mathrm{Ext}$ groups in~$\mathcal{C}$. More precisely,
assume that $\mathcal{M}$ is an exact $\mathcal{C}$-module category, i.e.\ $\mathcal{M}$ is finite and every projective object of $\mathcal{C}$ acting on any $M\in \mathcal{M}$ results in a projective object of $\mathcal{M}$. In particular, if $\mathcal{C}$ is semisimple then $\mathcal{M}$ is exact if and only if it's semisimple. Write $\mathcal{M}=A\text{-}\mathrm{mod}$ for some finite-dimensional $\Bbbk$-algebra $A$. Define Hochschild cohomologies of $\mathcal{M}$ as $\mathrm{HH}^\bullet(\mathcal{M})  := \mathrm{HH}^\bullet(A,A)$.  Then by \cite[Cor.\,7.5]{shimizuCoend} we have the chain of isomorphisms:
\begin{equation}\label{eq:Sh-rem}
\mathrm{HH}^\bullet(\mathcal{M}) 
= \mathrm{Ext}^{\bullet}_{A\text{-}\mathrm{bimod}}(A,A)
\cong \mathrm{Ext}^{\bullet}_{\mathrm{Rex}_\Bbbk(\mathcal{M})}(\mathrm{Id}_{\mathcal{M}}, \mathrm{Id}_{\mathcal{M}}) \cong \mathrm{Ext}^{\bullet}_{\mathcal{C}}(\boldsymbol{1},\mathcal{A}_{\mathcal{M}}).
\end{equation}
We emphasize that these cohomologies control deformations of linear categories or associated algebras in $\mathrm{vect}_{\Bbbk}$ without taking into account the $\mathcal{C}$-module structure, while our results in Prop.\,\ref{PropExtForMixAssoCohom} and Thm.\,\ref{thmAdjDYMod} show that considering relative $\mathrm{Ext}$'s instead takes into account such a structure.  Moreover,  the isomorphism~\eqref{eq:Sh-rem} can't hold in general for non-exact $\mathcal{C}$-modules. For example, over $\mathcal{C} = \mathrm{vect}_\Bbbk$ every non-semisimple $\Bbbk$-algebra $A$ provides $A\text{-}\mathrm{mod}$ as a non-exact $\mathcal{C}$-module and all $\mathrm{Ext}^{>0}_\mathcal{C}$ 
groups are trivial, while e.g.\ the algebra $A=\langle x\, | \, x^2 = 0 \rangle$ has non-zero Hochschild cohomologies.  In contrast, our results are applicable to any finite $\mathcal{C}$-module categories which are not required to be exact.
This is due to the fact that the diagram~\eqref{diagramLiftAdjunctions} corresponds to what we call \textsl{strong morphism of adjunctions} in~\cite[Sec.\,2.2]{FGS2}, and these always preserve
relatively projective resolutions \cite[Prop.\,2.6]{FGS2}.
\end{remark}

Let $\mathsf{V} = (V,t) \in \mathcal{Z}(\mathcal{C})$ and write $\mathsf{V} \rhd -$ as a suggestive notation for $\widetilde{\rho}(\mathsf{V})$; this is the endofunctor $V \rhd -$ of $\mathcal{M}$ equipped with the $\mathcal{C}$-module structure $\gamma(t)$ defined in \eqref{CmodStructFromHalfBr}. Theorem~\ref{thmAdjDYMod} has the following interesting generalization for coefficients of this form,
 with notations in Cor.\,\ref{cor:dim-formula-Id}:

\begin{proposition}\label{propGeneralizationAdjThm}
Under the same assumptions as in Thm.\,\ref{thmAdjDYMod}, we have for all $\mathsf{V}, \mathsf{W} \in \mathcal{Z}(\mathcal{C})$
\begin{equation}\label{eq:H-mix-V}
\forall \, n \geq 0, \quad \mathrm{H}^n_{\mathrm{mix}}(\mathcal{M};\mathsf{V} \rhd -,\, \mathsf{W} \rhd -) \cong \Ext_{\mathcal{Z}(\mathcal{C}), \,\mathcal{C}}^n(\mathsf{V}, \mathsf{W} \otimes \mathcal{A}_{\mathcal{M}} ).
\end{equation}
Moreover, if $0 \to \mathsf{K} \to \mathsf{P} \to \boldsymbol{1} \to 0$ is an allowable short exact sequence in $\mathcal{Z}(\mathcal{C})$ then
\begin{align*}
\dim \text{\rm of~\eqref{eq:H-mix-V}}
= \dim \Hom_{\mathcal{Z}(\mathcal{C})}\bigl(\mathsf{K}^{\otimes n} \otimes \mathsf{V},\mathsf{W}\otimes\mathcal{A}_{\mathcal{M}} \bigr) 
&- \dim \Hom_{\mathcal{Z}(\mathcal{C})}\bigl(\mathsf{P} \otimes \mathsf{K}^{\otimes (n-1)} \otimes \mathsf{V},\mathsf{W}\otimes\mathcal{A}_{\mathcal{M}} \bigr)\\
&+ \dim \Hom_{\mathcal{Z}(\mathcal{C})}\bigl(\mathsf{K}^{\otimes (n-1)} \otimes \mathsf{V},\mathsf{W}\otimes\mathcal{A}_{\mathcal{M}} \bigr)
\end{align*}
for all $n \geq 1$, with the convention that $\mathsf{K}^{\otimes 0} = \boldsymbol{1}$.
\end{proposition}
\begin{proof}
We noted below \eqref{CmodStructFromHalfBr} that $\widetilde{\rho}$ is monoidal. Since $\mathcal{Z}(\mathcal{C})$ is rigid, this implies that the object $\widetilde{\rho}(\mathsf{W}) = \mathsf{W} \rhd -$ has the left dual $\widetilde{\rho}(\mathsf{W}^*)$ in $\mathrm{Rex}_{\mathcal{C}}(\mathcal{M})$ by \cite[Ex.\,2.10.6]{EGNO}; also recall that the monoidal product in $\mathrm{Rex}_{\mathcal{C}}(\mathcal{M})$ is the composition of endofunctors. Using this fact together with the adjunction $\widetilde{\rho} \dashv \widetilde{\rho}^{\mathrm{ra}}$ gives the following chain of isomorphisms, all natural in the variable $\mathsf{X} \in \mathcal{Z}(\mathcal{C})$:
\begin{align*}
&\Hom_{\mathcal{Z}(\mathcal{C})}\bigl( \mathsf{X}, \widetilde{\rho}^{\mathrm{ra}}(\widetilde{\rho}(\mathsf{W})) \bigr) \cong \Hom_{\mathrm{Rex}_{\mathcal{C}}(\mathcal{M})}\bigl( \widetilde{\rho}(\mathsf{X}), \widetilde{\rho}(\mathsf{W}) \bigr)\\
\cong\:& \Hom_{\mathrm{Rex}_{\mathcal{C}}(\mathcal{M})}\bigl( \widetilde{\rho}(\mathsf{W}^*) \, \widetilde{\rho}(\mathsf{X}), \mathsf{Id}_{\mathcal{M}} \bigr) \cong \Hom_{\mathrm{Rex}_{\mathcal{C}}(\mathcal{M})}\bigl( \widetilde{\rho}(\mathsf{W}^* \otimes \mathsf{X}), \mathsf{Id}_{\mathcal{M}} \bigr)\\
\cong\:& \Hom_{\mathcal{Z}(\mathcal{C})}\bigl( \mathsf{W}^* \otimes \mathsf{X}, \widetilde{\rho}^{\mathrm{ra}}(\mathsf{Id}_{\mathcal{M}}) \bigr) = \Hom_{\mathcal{Z}(\mathcal{C})}\bigl( \mathsf{W}^* \otimes \mathsf{X}, \mathcal{A}_{\mathcal{M}} \bigr) \cong \Hom_{\mathcal{Z}(\mathcal{C})}\bigl( \mathsf{X}, \mathsf{W} \otimes \mathcal{A}_{\mathcal{M}} \bigr).
\end{align*}
It thus follows from the Yoneda lemma that $\widetilde{\rho}^{\mathrm{ra}}\bigl( \widetilde{\rho}(\mathsf{W}) \bigr) \cong \mathsf{W} \otimes \mathcal{A}_{\mathcal{M}}$. We are now ready to prove the proposition:
\begin{align*}
\mathrm{H}^n_{\mathrm{mix}}\bigl(\mathcal{M}; \mathsf{V} \rhd -, \mathsf{W} \rhd - \bigr) &= \mathrm{H}^n_{\mathrm{mix}}\bigl(\mathcal{M}; \widetilde{\rho}(\mathsf{V}), \widetilde{\rho}(\mathsf{W}) \bigr)\\
&\cong \Ext^n_{\mathrm{Rex}_{\mathcal{C}}(\mathcal{M}),\,\mathrm{Rex}_\Bbbk(\mathcal{M})}\bigl( \widetilde{\rho}(\mathsf{V}),\, \widetilde{\rho}(\mathsf{W}) \bigr) \quad \text{\footnotesize by Prop.\,\ref{PropExtForMixAssoCohom}}\\
&\cong\Ext_{\mathcal{Z}(\mathcal{C}), \,\mathcal{C}}^n\bigl( \mathsf{V},\, \widetilde{\rho}^{\mathrm{ra}}(\widetilde{\rho}(\mathsf{W})) \bigr) \quad \text{\footnotesize by adjunction thm for rel.~Ext \cite{FGS2}}\\
&\cong \Ext_{\mathcal{Z}(\mathcal{C}), \,\mathcal{C}}^n\bigl( \mathsf{V}, \mathsf{W} \otimes \mathcal{A}_{\mathcal{M}} \bigr). 
\end{align*}
To prove the last statement, note by \cite[Cor.\,3.3]{FGS} that
\[ \Ext_{\mathcal{Z}(\mathcal{C}), \,\mathcal{C}}^n(\mathsf{V}, \mathsf{W} \otimes \mathcal{A}_{\mathcal{M}}) \cong \Ext_{\mathcal{Z}(\mathcal{C}), \,\mathcal{C}}^n(\boldsymbol{1}, \mathsf{W} \otimes \mathcal{A}_{\mathcal{M}} \otimes \mathsf{V}^*) \]
and the claimed dimension formula immediately follows from \eqref{dimFormulaFGS1} and duality in $\mathcal{Z}(\mathcal{C})$.
\end{proof}

Combining Prop.\,\ref{propGeneralizationAdjThm} with Props.\,\ref{prop:H1-F}  and~\ref{propExtensionCmodStruct},
we obtain:
\begin{corollary}\label{coroInterpretDefGenralizedAdjThm}
For all $\mathsf{V} = (V,t) \in \mathcal{Z}(\mathcal{C})$, the equivalence classes (Def.\,\ref{defInfDeformationsMix}) of infinitesimal deformations of the $\mathcal{C}$-module structure $\gamma(t)$ of the endofunctor $V \rhd -$ are classified by $\Ext^1_{\mathcal{Z}(\mathcal{C}),\mathcal{C}}(\mathsf{V}, \mathsf{V} \otimes \mathcal{A}_{\mathcal{M}})$, while the obstructions for lifting to higher degrees are classified by $\Ext^2_{\mathcal{Z}(\mathcal{C}),\mathcal{C}}(\mathsf{V}, \mathsf{V} \otimes \mathcal{A}_{\mathcal{M}})$.
\end{corollary}

\begin{remark}\label{rmkInterpretDefGenralizedAdjThm}
For $\mathsf{V} = (V,t) \in \mathcal{Z}(\mathcal{C})$, the deformations of the $\mathcal{C}$-module structure $\gamma(t)$ of $V \rhd -$ must not be confused with the deformations of the half-braiding $t$, which by Prop.\,\ref{propLiftHB} are controlled by the first and second cohomology spaces in $\mathrm{H}^\bullet_{\mathrm{DY}}(\mathcal{C};\mathsf{V},\mathsf{V}) \cong \Ext^\bullet_{\mathcal{Z}(\mathcal{C}),\mathcal{C}}(\mathsf{V},\mathsf{V})$.
\end{remark}

\subsection{Rigidity results}
Here, we obtain two rigidity results: the first is that the deformation cohomology for the regular $\mathcal{C}$-module category vanishes,  as a non-trivial consequence of Theorem \ref{thmAdjDYMod}, and the second is a generalized Ocneanu rigidity stating that every $\Bbbk$-linear monoidal functor out of a fusion category (with non-zero categorical dimension) to a fairly general target category ($\Bbbk$-linear abelian and with right exact monoidal product) has trivial DY cohomologies with any coefficients, and as a consequence, we get that every module category (not necessarily semisimple or even finite) over a fusion category has vanishing mixed associator cohomologies, and even every module endofunctor has vanishing cohomologies in this case.

\subsubsection{Rigidity of the regular \texorpdfstring{$\mathcal{C}$-}{}module}\label{subsubRigidityCmod}
The following interesting result regards the regular $\mathcal{C}$-module, which is the category $\mathcal{C}$ with action $\rhd = \otimes$. Note that semisimplicity of $\mathcal{C}$ is not required in this statement:
\begin{proposition}\label{propRigidityReg}
Let $\mathcal{C}$ be a finite $\Bbbk$-linear  rigid monoidal category whose monoidal product~$\otimes$ is $\Bbbk$-bilinear, i.e.\ it is a finite multitensor category. For the regular $\mathcal{C}$-module we have
\[ \forall \, n > 0, \quad \mathrm{H}^n_{\mathrm{mix}}(\mathcal{C}) = 0. \]
\end{proposition}

In order to prove this proposition, recall from Example \ref{exampleAdjObjReg} that $\mathcal{A}_{\mathcal{C}} \in \mathcal{Z}(\mathcal{C})$ is the end $\int_{X \in \mathcal{C}} X \otimes X^*$ equipped with the half-braiding $b$ in \eqref{diagramHalfBrEnd}. A key point is the following fact:
\begin{lemma}\label{lemmaDualEndCoend}
Let $\mathcal{L}$ be the coend $\int^{X \in \mathcal{C}} X^* \otimes X$ equipped with the half-braiding obtained by taking $V = \boldsymbol{1}$ in \eqref{diagramHalfBrCoend}. Then $\mathcal{L}$ is the left dual of $\mathcal{A}_{\mathcal{C}}$ in $\mathcal{Z}(\mathcal{C})$.
\end{lemma}
\begin{proof}
Denote by $t$ the half-braiding of $\mathcal{L}$. We must define an evaluation morphism $e \in \Hom_{\mathcal{Z}(\mathcal{C}}(\mathcal{L} \otimes \mathcal{A}_{\mathcal{C}}, \boldsymbol{1})$ and a coevaluation morphism $c \in \Hom_{\mathcal{Z}(\mathcal{C})}(\boldsymbol{1}, \mathcal{A}_{\mathcal{C}} \otimes \mathcal{L})$ and check that they satisfy the zig-zag axioms. Using the universal property of the dinatural transformation $i_X : X^* \otimes X \to \mathcal{L}$, let $e$ be uniquely characterized by the commutative diagram
\begin{equation}\label{defEvEndCoend}
\xymatrix@C=2.6em{
X^* \otimes X \otimes \mathcal{A}_{\mathcal{C}} \ar[rr]^-{\mathrm{id}_{X^* \otimes X} \,\otimes\, \pi_{^*\!X}} \ar[d]_-{i_X \,\otimes\, \mathrm{id}_{\mathcal{A}_{\mathcal{C}}} } && X^* \otimes X \otimes {^*\!X} \otimes {^*(X^*)} \ar@{=}[r]&  X^* \otimes X \otimes {^*\!X} \otimes X \ar[d]^-{\mathrm{id}_{X^*} \,\otimes\, \widetilde{\mathrm{ev}}_X \,\otimes\, \mathrm{id}_X}\\
\mathcal{L} \otimes \mathcal{A}_{\mathcal{C}} \ar[rr]_{\exists! \, e} && \boldsymbol{1} &\ar[l]^-{\mathrm{ev}_X} X^* \otimes X
} \end{equation}
for all $X \in \mathcal{C}$. Using the universal property of the dinatural transformation $\pi_X : \mathcal{A}_{\mathcal{C}} \to X \otimes X^*$, let $c$ be uniquely characterized by the commutative diagram
\begin{equation}\label{defCoevEndCoend}
\xymatrix@C=7em{
\boldsymbol{1} \ar[r]^-{\widetilde{\mathrm{coev}}_X} \ar[d]_-{\exists!\,c} & {^*\!X} \otimes X \ar[r]^-{\mathrm{id}_{^*\!X} \,\otimes\, \mathrm{coev}_X \,\otimes\, \mathrm{id}_X} & {^*\!X} \otimes X \otimes X^* \otimes X \ar[d]^-{\mathrm{id}_{{^*\!X} \otimes X} \,\otimes\, i_X}\\
\mathcal{A}_{\mathcal{C}} \otimes \mathcal{L} \ar[r]_-{\pi_{^*\!X} \,\otimes\, \mathrm{id}_{\mathcal{L}}} & {^*\!X} \otimes ({^*\!X})^* \otimes \mathcal{L} \ar@{=}[r] & {^*\!X} \otimes X \otimes \mathcal{L}
} \end{equation}
for all $X \in \mathcal{C}$. It is straightforward to check the zig-zag axioms; for instance
\begin{center}
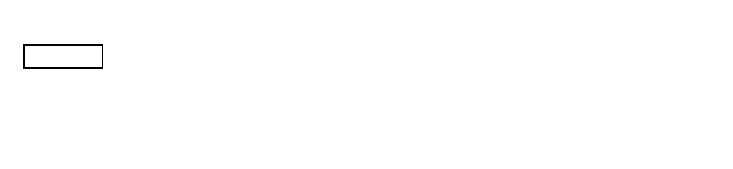
\end{center}
for all $X \in \mathcal{C}$, which implies the desired equality by the universal property of $i$. It is also not hard to check that $c$ and $e$ are morphisms in $\mathcal{Z}(\mathcal{C})$; for instance
\begin{center}
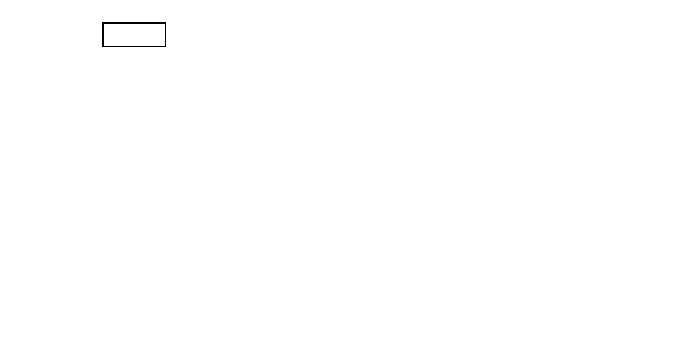
\end{center}
where the unlabelled equality is by dinaturality of $\pi$.
\end{proof}

\begin{proof}[Proof of Prop.\,\ref{propRigidityReg}.] Note first that the assumptions \eqref{assumptionsActionFunct} are fulfilled with the present choice $\mathcal{M} = \mathcal{C}$ and $\rhd = \otimes$; indeed, $\otimes$ is exact in each variable by rigidity of $\mathcal{C}$ \cite[Prop.\,4.2.1]{EGNO}. Moreover $\mathcal{C}$ is assumed to be finite so we can use Thm.\,\ref{thmAdjDYMod}:
\[ \mathrm{H}^n_{\mathrm{mix}}(\mathcal{C}) \cong \Ext_{\mathcal{Z}(\mathcal{C}),\, \mathcal{C}}^n(\boldsymbol{1}, \mathcal{A}_{\mathcal{C}}) \cong \Ext_{\mathcal{Z}(\mathcal{C}), \,\mathcal{C}}^n(\mathcal{A}_{\mathcal{C}}^*, \boldsymbol{1}) = \Ext_{\mathcal{Z}(\mathcal{C}), \,\mathcal{C}}^n(\mathcal{L}, \boldsymbol{1}) \]
where the second isomorphism uses a duality adjunction for relative Ext groups which is the obvious left dual version of \cite[Cor.\,3.3]{FGS}, and the last equality is by Lemma \ref{lemmaDualEndCoend}. But note that $\mathcal{L} = \mathcal{F}_{\mathcal{C}}(\boldsymbol{1})$ by definition, where $\mathcal{F}_{\mathcal{C}} : \mathcal{C} \to \mathcal{Z}(\mathcal{C})$ described in \eqref{diagramHalfBrCoend} is the left adjoint of the forgetful functor. Hence $\mathcal{L}$ is relatively projective (it is even a ``free object'' in the sense of relative homological algebra), and thus the functor $\Ext_{\mathcal{Z}(\mathcal{C}), \,\mathcal{C}}^n(\mathcal{L}, -)$ vanishes for all $n > 0$.
\end{proof}

\subsubsection{Generalized Ocneanu rigidity}\label{subsubOcneanu}
In order to prove a general rigidity result for the mixed associator cohomology, we first need to extend Ocneanu rigidity for DY cohomology with coefficients beyond the assumptions made in \cite[Cor.\,3.18]{GHS}. Namely, let us consider the following situation:
\begin{equation}\label{assumptionForOcneanu}
\left\{\!\!\begin{array}{l}
\mathcal{C} \text{ is a fusion category (see e.g.\ \cite[Def.\,4.1.1]{EGNO}) over a field } \Bbbk \text{ with\;}\dim(\mathcal{C}) \neq 0,\\
\mathcal{D} \text{ is a } \Bbbk\text{-linear abelian monoidal category,}\\
\otimes_{\mathcal{D}} : \mathcal{D} \times \mathcal{D} \to \mathcal{D} \text{ is } \Bbbk\text{-bilinear and right-exact in each variable,}\\
\Gamma : \mathcal{C} \to \mathcal{D} \text{ is a }\Bbbk\text{-linear monoidal functor.}
\end{array} \right.
\end{equation}
In particular, $\mathcal{C}$ is finite, rigid and semisimple with simple monoidal unit $\boldsymbol{1}$ and $\mathrm{End}_\mathcal{C}(\boldsymbol{1})\cong \Bbbk$.
The categorical dimension $\dim(\mathcal{C})\in \mathrm{End}_\mathcal{C}(\boldsymbol{1})$ is defined e.g.\ in \cite[Def.\,7.21.3]{EGNO}. Note that the assumption $\dim(\mathcal{C}) \neq 0$ is automatic if $\Bbbk$ has characteristic $0$, see~\cite[Thm.\,7.21.12]{EGNO}.
We also notice that  $\Gamma$ is \textsl{automatically} an exact functor. Indeed, any exact sequence in $\mathcal{C}$ is split by semisimplicity. Since $\Gamma$ is $\Bbbk$-linear it is in particular additive and thus it preserves split exact sequences. Hence $\Gamma$ transforms exact sequences into exact sequences, as claimed. Actually only right-exactness of $\Gamma$ is used in the sequel.

The assumptions \eqref{assumptionForOcneanu} allow us in particular to apply \cite[\S 3.2]{FGS2} where it is shown that the forgetful functor\footnote{This functor forgets the half-braidings relative to $\Gamma$, recall the definition of $\mathcal{Z}(\Gamma)$ in \S\ref{subsubDYCoeff}.} $\mathcal{U}_\Gamma : \mathcal{Z}(\Gamma) \to \mathcal{D}$ has a left adjoint $\mathcal{F}_\Gamma : \mathcal{D} \to \mathcal{Z}(\Gamma)$, thus yielding a resolvent pair between $\mathcal{Z}(\Gamma)$ and $\mathcal{D}$:
\begin{equation}\label{adjunctionWithMonads}
\xymatrix@C=6em{
\mathcal{Z}(\Gamma) \ar@/^.7em/[d]^-{\mathcal{U}_\Gamma} \ar[r]^-{\cong}& Z_\Gamma\text{-}\mathrm{mod} \ar@/^.7em/[d]^-{\mathrm{Res}_\Gamma}\\
\mathcal{D} \ar@/^.7em/[u]^-{\mathcal{F}_\Gamma}_-{\,\scalebox{1}{$\dashv$}} \ar@{=}[r] & \mathcal{D} \ar@/^.7em/[u]^-{\mathrm{Ind}_\Gamma}_-{\,\scalebox{1}{$\dashv$}}
} \end{equation}
On the right-hand side we also display the Eilenberg--Moore adjunction for the monad $Z_\Gamma$ on~$\mathcal{D}$ which is defined by $Z_\Gamma(D) = \int^X \Gamma(X^*) \otimes D \otimes \Gamma(X)$; see e.g.~\cite[\S 3.3]{GHS} for details. This isomorphic resolvent pair will facilitate the use of results in \cite{BV}. We also note that $Z_{\Gamma}$ has a structure of {\em Hopf monad}, see e.g.~\cite[\S 3.3]{GHS} and \cite[\S 3]{BV} for the general definition.

Recall that an object $\mathsf{P} \in \mathcal{Z}(\Gamma)$ is called {\em relatively projective} if it is a direct summand of $\mathcal{F}_\Gamma(D)$ for some $D \in \mathcal{D}$.\footnote{In the context of modules over monads, one also says that $\mathsf{P} \in Z_\Gamma\text{-}\mathrm{mod}$ is a {\em retract} of some {\em free module} $\mathrm{Ind}_\Gamma(D)$. This is the terminology used in \cite{BV}.} 

\begin{lemma}\label{lemmaZGammaRelProj}
Under the assumptions \eqref{assumptionForOcneanu}, every object in $\mathcal{Z}(\Gamma)$ is relatively projective.
\end{lemma}
\begin{proof}
We start with a few preliminaries. Let $T\colon \mathcal{A} \to \mathcal{A}$ be a monad on an abelian category $\mathcal{A}$; it is then  well-known that its category of modules $T\text{-}\mathrm{mod}$ is abelian as well. Assume moreover that $\mathcal{A}$ is monoidal and rigid and that $T$ is a Hopf monad on $\mathcal{A}$. A cointegral of $T$, if it exists, is a morphism $\Lambda \in \Hom_{T\text{-}\mathrm{mod}}\bigl( \boldsymbol{1}, T(\boldsymbol{1}) \bigr)$. It is called normalized if the counit $T(\boldsymbol{1})\to \boldsymbol{1}$ composed with $\Lambda$ equals $\mathrm{id}_{\boldsymbol{1}}$. The following facts are found in \cite[Rmk.\,6.2 and Thm.\,6.5]{BV}:
\begin{enumerate}[topsep=.4em, itemsep=-.1em]
\item If $T\text{-}\mathrm{mod}$ is semisimple (as an abelian category) then $T$ has a normalized cointegral.
\item If $T$ has a normalized cointegral then every object in $T\text{-}\mathrm{mod}$ is a direct summand of some free module $T(X)$ for $X\in\mathcal{A}$.
\end{enumerate}
Importantly, the Maschke type result in item 2.  holds due to \cite[Rmk.\,2.9]{BLV} for Hopf monads in arbitrary monoidal categories, not necessarily rigid.

We are now ready to prove the lemma. Note first that $\mathcal{Z}(\mathcal{C})$ is semisimple by \cite[Thm.\,9.3.2]{EGNO}; the result in \cite{EGNO} is stated for $\Bbbk$ algebraically closed of characteristic 0 but this is used only to claim that $\dim(\mathcal{C}) \neq 0$, which we take here as an assumption.\footnote{All the steps in the proof of \cite[Thm.\,9.3.2]{EGNO} are valid for any ground field $\Bbbk$, not necessarily algebraically closed. Indeed, the first step in the proof that the object $I(\boldsymbol{1})$ is projective is valid because \cite[Prop.\,6.1.3]{EGNO} is valid for any ground field, in fact it uses only bi-exactness assumption of $\otimes$. Furthermore, the second step  that the canonical isomorphism $\Phi$ is a morphism in the Drinfeld center $\mathcal{Z}(\mathcal{C})$, i.e.\ that $\Phi$ commutes with the canonical half-braiding of the end, is valid for any ground field too. Finally, the quantum trace properties from~\cite[Prop.\,4.7.3]{EGNO} also do not need $\Bbbk$ to be algebraically closed.}
Let $Z_{\mathcal{C}}\colon \mathcal{C} \to \mathcal{C}$ be the monad $Z_{\mathcal{C}}(C) = \int^X X^* \otimes C \otimes X$, so that $\mathcal{Z}(\mathcal{C}) \cong Z_{\mathcal{C}}\text{-}\mathrm{mod}$ (take $\Gamma = \mathrm{Id}_{\mathcal{C}}$ in \eqref{adjunctionWithMonads}). By item 1 above, the monad $Z_{\mathcal{C}}$ has a normalized cointegral $\Lambda \in \Hom_{Z_{\mathcal{C}}\text{-}\mathrm{mod}}(\boldsymbol{1}, Z_{\mathcal{C}}(\boldsymbol{1}))$. The same arguments as in \cite[Proof of Lem.\,3.17]{GHS} show that $\Gamma(\Lambda) \in \Hom_{Z_{\Gamma}\text{-}\mathrm{mod}}(\boldsymbol{1}, Z_{\Gamma}(\boldsymbol{1}))$ is a normalized cointegral for the monad $Z_\Gamma$ in $\mathcal{D}$;  the main point is that $\Gamma$ commutes with coends by right-exactness and hence $\Gamma Z_{\mathcal{C}} = Z_\Gamma \Gamma$. Item 2 above thus implies that any object in $Z_{\Gamma}\text{-}\mathrm{mod}$ is a direct summand of some free module $\mathrm{Ind}_\Gamma(D)$ for $D\in \mathcal{D}$. Through the isomorphism $Z_\Gamma\text{-}\mathrm{mod} \cong \mathcal{Z}(\Gamma)$, it means that any object in $\mathcal{Z}(\Gamma)$ is a direct summand of some $\mathcal{F}_\Gamma(D)$.
\end{proof}

\begin{theorem}{\em (Ocneanu rigidity with coefficients)}\label{thmGeneralizedOcneanuDY}
Assume that $\Gamma : \mathcal{C} \to \mathcal{D}$ satisfies the assumptions \eqref{assumptionForOcneanu}. Then for any coefficients $\mathsf{X}, \mathsf{Y} \in \mathcal{Z}(\Gamma)$ we have
\[ \forall \, n > 0, \quad \mathrm{H}^n_{\mathrm{DY}}(\Gamma; \mathsf{X}, \mathsf{Y}) = 0. \]
\end{theorem}
\begin{proof}
The assumptions \eqref{assumptionForOcneanu} are stronger than the assumptions used to prove \cite[Thm.\,3.3]{FGS2}. Hence we can apply this theorem, which gives
\[ \forall \, n \in \mathbb{N}, \quad \mathrm{H}^n_{\mathrm{DY}}(\Gamma; \mathsf{X}, \mathsf{Y}) \cong \Ext^n_{\mathcal{Z}(\Gamma),\mathcal{D}}(\mathsf{X},\mathsf{Y}). \]
We know from Lemma \ref{lemmaZGammaRelProj} that the object $\mathsf{X}$ is relatively projective with respect to the resolvent pair $\mathcal{Z}(\Gamma) \rightleftarrows \mathcal{D}$. It immediately follows from the definition of relative Ext groups that $\Ext^{> 0}_{\mathcal{Z}(\Gamma),\mathcal{D}}(\mathsf{X},-) = 0$.
\end{proof}

We are now ready to prove a general rigidity result for the mixed associator cohomology:
\begin{theorem}\label{thmRigidityMix}
Assume that $\mathcal{C}$ is a fusion category over a field $\Bbbk$ such that $\dim(\mathcal{C})\neq 0$, and let $\mathcal{M}$ be a $\mathcal{C}$-module category, $\Bbbk$-linear and abelian, with $\Bbbk$-bilinear action $\rhd : \mathcal{C} \times \mathcal{M} \to \mathcal{M}$. Then for any coefficients $\mathsf{F}, \mathsf{G} \in \mathrm{Rex}_{\mathcal{C}}(\mathcal{M})$ we have
\[ \forall \, n > 0, \quad \mathrm{H}^n_{\mathrm{mix}}(\mathcal{M}; \mathsf{F}, \mathsf{G}) = 0. \]
\end{theorem}
\begin{proof}
By Prop.\,\ref{propMixCohomDY} we have $\mathrm{H}^n_{\mathrm{mix}}(\mathcal{M}; \mathsf{F}, \mathsf{G}) \cong \mathrm{H}^n_{\mathrm{DY}}(\rho;\mathsf{F},\mathsf{G})$, where $\rho : \mathcal{C} \to \mathrm{End}(\mathcal{M})$ is the representation functor, given by $\rho(X) = X \rhd -$. For all $X$, the $\Bbbk$-bilinearity of $\rhd$ implies that $\rho(X)$ is a $\Bbbk$-linear endofunctor and it was noted in Rmk.\,\ref{remarkRightExactDSPS} that $\rho(X)$ is exact thanks to rigidity of $\mathcal{C}$. As a result we can look at $\rho$ as a functor $\mathcal{C} \to \mathrm{Rex}_\Bbbk(\mathcal{M})$. Bilinearity of~$\rhd$  implies that $\rho$ is a $\Bbbk$-linear functor. We noted in Lem.\,\ref{lemmaRephraseAssump} that the monoidal product of $\mathrm{Rex}_\Bbbk(\mathcal{M})$, which is the composition of endofunctors, is right-exact in each variable. Finally, we see from its definition in \eqref{horizontalComp} that the monoidal product of natural transformations, which are the morphisms in $\mathrm{Rex}_\Bbbk(\mathcal{M})$, is a $\Bbbk$-bilinear operation. Hence the assumptions \eqref{assumptionForOcneanu} are satisfied with the choice $\mathcal{D} = \mathrm{Rex}_\Bbbk(\mathcal{M})$ and $\Gamma = \rho$. It follows from Thm.\,\ref{thmGeneralizedOcneanuDY} that $\mathrm{H}^{>0}_{\mathrm{DY}}(\rho;\mathsf{F},\mathsf{G}) = 0$ and we are done.
\end{proof}

\begin{remark}
It is remarkable that in Thm.\,\ref{thmRigidityMix} there are no semisimplicity or finiteness assumptions on $\mathcal{M}$; in particular we do not have to require $\mathcal{M}$ to be exact in the sense of \cite[Def.\,7.5.1]{EGNO}. We also notice that Thm.\,\ref{thmRigidityMix} is a generalization of Example \ref{exModCatVect}. 
\end{remark}

We see from Prop.\,\ref{relDefAssoAndCohom} and Prop.\,\ref{prop:H1-F} that Thm.\,\ref{thmRigidityMix} has the following consequence:
\begin{corollary}
Under the assumptions of Thm.\,\ref{thmRigidityMix}:
\\1. The mixed associator of $\mathcal{M}$ does not have non-trivial infinitesimal deformations (in the sense of Def.\,\ref{defInfDeformationsMix}).
\\2. For all $\mathsf{F} = (F,\gamma^F) \in \mathrm{Rex}_{\mathcal{C}}(\mathcal{M})$, the $\mathcal{C}$-module structure $\gamma^F$ does not have non-trivial infinitesimal deformations (in the sense explained in Prop.\,\ref{prop:H1-F}).
\end{corollary}

\subsection{Adjoint object as full center of an algebra}\label{subsecAdjObjFullCenter}

As in \S\ref{subsecAdjThm}, we consider a finite $\mathcal{C}$-module category $\mathcal{M} = (\mathcal{M}, \rhd, m)$ fulfilling the assumptions~\eqref{assumptionsActionFunct}.

\subsubsection{Algebra structure of the adjoint object}\label{subsubAdjAlg}

\indent An {\em algebra} in $\mathcal{C}$ is a triple $A = (A,m_A,1_A)$, where $m_A \in \Hom_{\mathcal{C}}(A \otimes A, A)$ and $1_A \in \Hom_{\mathcal{C}}(\boldsymbol{1},A)$ satisfy the usual axioms:
\[ m_A \circ (1_A \otimes \mathrm{id}_A) = m_A \circ (\mathrm{id}_A \otimes 1_A) = \mathrm{id}_A \quad \text{and} \quad m_A \circ (m_A \otimes \mathrm{id}_A) = m_A \circ (\mathrm{id}_A \otimes m_A). \]
Each internal End object $\underline{\End}(M) = \underline{\Hom}(M,M)$ is in particular an algebra in $\mathcal{C}$ \cite[Ex.\,7.9.9]{EGNO} whose multiplication is
\begin{align}
\begin{split}\label{prodInternalEnd}
\underline{\End}(M) \otimes \underline{\End}(M) &\xrightarrow{\:\underline{\mathsf{coev}}_{\underline{\End}(M) \,\otimes\,\underline{\mathsf{End}}(M),M}\:} \underline{\Hom}\bigl[ M, \bigl(\underline{\End}(M) \otimes \underline{\End}(M)\bigr) \rhd M \bigr]\\
&\xrightarrow{\:\underline{\Hom}(\mathrm{id},\,m^{-1}_{\underline{\mathsf{End}}(M),\underline{\mathsf{End}}(M),M})\:} \:\,\underline{\Hom}\bigl[ M, \underline{\End}(M) \rhd \bigl( \underline{\End}(M)) \rhd M  \bigr) \bigr]\\
&\xrightarrow{\:\underline{\Hom}(\mathrm{id}, \mathrm{id}\,\rhd\,\underline{\mathsf{ev}}_{M,M})\:} \underline{\Hom}\bigl[ M, \underline{\End}(M) \rhd M  \bigr]\\
&\xrightarrow{\:\underline{\Hom}(\mathrm{id}, \underline{\mathsf{ev}}_{M,M})\:} \underline{\End}(M)
\end{split}
\end{align}
(where $m$ is the mixed associator in $\mathcal{M}$) and whose unit is $\underline{\mathsf{coev}}_{\boldsymbol{1},M}$.

\medskip

Recall the adjoint object $\mathcal{A}_{\mathcal{M}} = \widetilde{\rho}^{\mathrm{ra}}(\mathsf{Id}_{\mathcal{M}}) \in \mathcal{Z}(\mathcal{C})$ from Def.\,\ref{defAdjObj}. There is a natural algebra structure on $\mathcal{A}_{\mathcal{M}}$ in $\mathcal{Z}(\mathcal{C})$, as on any image of an algebra object under the right adjoint of a monoidal functor. To spell out this fact in the present situation, we use for simplicity the underlying object $\mathcal{A}_{\mathcal{M}} = \rho^{\mathrm{ra}}(\mathrm{Id}_ {\mathcal{M}})$ in $\mathcal{C}$, remembering that the existence of the lifts (see \eqref{diagramLiftAdjunctions}) implies that the algebra structure is compatible with the half-braiding \eqref{halfBrEnd}. Let $\eta : \mathrm{Id}_{\mathcal{C}} \Rightarrow \rho^{\mathrm{ra}}\rho$ and $\varepsilon : \rho \rho^{\mathrm{ra}} \Rightarrow \mathrm{Id}_{\mathrm{Rex}_{\Bbbk}(\mathcal{M})}$ be the unit and counit of the adjunction $\rho \dashv \rho^{\mathrm{ra}}$. Then the product on $\mathcal{A}_{\mathcal{M}}$ is
\begin{align}
\begin{split}\label{defProdAdjObj}
\mathcal{A}_{\mathcal{M}} \otimes \mathcal{A}_{\mathcal{M}} &\xrightarrow{\:\eta_{\mathcal{A}_{\mathcal{M}} \otimes \mathcal{A}_{\mathcal{M}}}\:} \rho^{\mathrm{ra}}\rho\bigl( \mathcal{A}_{\mathcal{M}} \otimes \mathcal{A}_{\mathcal{M}}\bigr)\\
&\xrightarrow{\rho^{\mathrm{ra}}\left(\widehat{m}^{-1}_{\mathcal{A}_{\mathcal{M}},\mathcal{A}_{\mathcal{M}}}\right)} \rho^{\mathrm{ra}}\bigl( \rho(\mathcal{A}_{\mathcal{M}}) \, \rho(\mathcal{A}_{\mathcal{M}}) \bigr) = \rho^{\mathrm{ra}}\bigl( (\rho\rho^{\mathrm{ra}})(\mathrm{Id}_ {\mathcal{M}}) \, (\rho\rho^{\mathrm{ra}})(\mathrm{Id}_ {\mathcal{M}}) \bigr) \\
&\xrightarrow{\:\rho^{\mathrm{ra}}\left(\varepsilon_{\mathrm{Id}_ {\mathcal{M}}} \,\bullet\, \varepsilon_{\mathrm{Id}_ {\mathcal{M}}}\right)\:} \rho^{\mathrm{ra}}( \mathrm{Id}_ {\mathcal{M}} \, \mathrm{Id}_ {\mathcal{M}} ) = \rho^{\mathrm{ra}}(\mathrm{Id}_ {\mathcal{M}}) = \mathcal{A}_{\mathcal{M}}.
\end{split}
\end{align}
where $\widehat{m}_{X,Y} : \rho(X)\,\rho(Y) \overset{\sim}{\implies} \rho(X \otimes Y)$ is the monoidal structure of $\rho$ defined thanks to the mixed associator $m$ of $\mathcal{M}$ (Lemma \ref{lemmaBijMixMon}), and we used the horizontal composition $\bullet$ from~\eqref{horizontalComp}. The unit of~$\mathcal{A}_{\mathcal{M}}$ is
\begin{equation}\label{defUnitAdjObj}
\boldsymbol{1} \xrightarrow{\:\eta_{\boldsymbol{1}}\:} \rho^{\mathrm{ra}}\rho(\boldsymbol{1}) = \rho^{\mathrm{ra}}(\boldsymbol{1} \rhd -) = \rho^{\mathrm{ra}}(\mathrm{Id}_ {\mathcal{M}}) = \mathcal{A}_{\mathcal{M}}.
\end{equation}

\smallskip

\indent One can get a more explicit idea of the algebra structure on $\mathcal{A}_{\mathcal{M}}$ by using the description of $\rho^{\mathrm{ra}}(F)$ as the end $\int_M \underline{\Hom}\bigl(M,F(M)\bigr)$ from Lemma \ref{lemmaRightAdjointRho}. Note first that
\[ \rho^{\mathrm{ra}}\rho(X) = \int_{M \in \mathcal{M}} \underline{\Hom}(M,X \rhd M), \qquad \bigl(\rho\rho^{\mathrm{ra}}(F)\bigr)(N) = \int_{M \in \mathcal{M}} \underline{\Hom}\bigl(M,F(M)\bigr) \rhd N \]
for all $X \in \mathcal{C}$, $F \in \mathrm{Rex}_\Bbbk(\mathcal{M})$ and $N \in \mathcal{M}$. Then one can check that the unit $\eta$ is characterized by 
the following commutative diagram for all $X \in \mathcal{C}$
\begin{equation}\label{descriptionUnitDinat}
\xymatrix@C=4em{
X \ar[d]^-{\exists!}_-{\eta_X} \ar[r]^-{\underline{\mathsf{coev}}_{X,M}} & \underline{\Hom}( M,X \rhd M ) \ar@{=}[d]\\
\rho^{\mathrm{ra}}\rho(X) \ar[r]_-{\pi^{\rho(X)}_M} &  \underline{\Hom}\bigl( M,\rho(X)(M) \bigr)
} \end{equation}
where we use the dinaturality of $\underline{\mathsf{coev}}$ recorded in Lemma \ref{lemmaNatEvCoev} to obtain the existence of a unique factorization through the universal dinatural transformation $\pi^{\rho(X)}$ of the end. The component $\varepsilon_F : \rho\rho^{\mathrm{ra}}(F) \Rightarrow F$ of the counit for $F \in \mathrm{Rex}_\Bbbk(\mathcal{M})$ is given on an object $N \in \mathcal{M}$ by
\[ (\varepsilon_F)_N : \bigl(\rho\rho^{\mathrm{ra}}(F)\bigr)(N) \xrightarrow{\:\pi^F_N \,\rhd\, \mathrm{id}_N\:} \underline{\Hom}\bigl(N,F(N)\bigr) \rhd N \xrightarrow{\:\underline{\mathsf{ev}}_{N,F(N)}\:} F(N). \]
With this description of $\eta$ and $\varepsilon$, it is not hard to check that the product \eqref{defProdAdjObj} on $\mathcal{A}_{\mathcal{M}}$ is described by the following commutative diagram
\begin{equation}\label{algStructAdjObj}
\xymatrix@C=5em{
\mathcal{A}_{\mathcal{M}} \otimes \mathcal{A}_{\mathcal{M}} \ar[r]^-{\pi_M \,\otimes\,\pi_M} \ar@{-->}[d]_-{\text{product \eqref{defProdAdjObj}}}^-{\exists!}& \underline{\Hom}(M,M) \otimes \underline{\Hom}(M,M) \ar[d]^-{\text{product \eqref{prodInternalEnd}}}\\
\mathcal{A}_{\mathcal{M}} \ar[r]_-{\pi_M} & \underline{\Hom}(M,M)
} \end{equation}
where $\pi_M = \pi_M^{\mathrm{Id}}$ is the universal dinatural transformation of the end $\mathcal{A}_{\mathcal{M}} = \int_{M \in \mathcal{M}} \underline{\Hom}(M,M)$. The unit \eqref{defUnitAdjObj} of $\mathcal{A}_{\mathcal{M}}$ is just described by taking $X = \boldsymbol{1}$ in \eqref{descriptionUnitDinat}.

\begin{example}\label{exampleProductAC}
Consider the regular $\mathcal{C}$-module category, i.e.\ $\mathcal{M} = \mathcal{C}$, $\rhd = \otimes$ and $m = \mathrm{id}$. Then $\underline{\Hom}(M,N) = N \otimes M^*$ with $\underline{\mathsf{ev}}_{M,N} = \mathrm{id}_N \otimes \mathrm{ev}_M$ and $\underline{\mathsf{coev}}_{X,M} = \mathrm{id}_X \otimes \mathrm{coev}_M$. Thanks to the zig-zag axiom for ev/coev, the product \eqref{prodInternalEnd} on $\underline{\End}(M)$ reduces simply to $\mathrm{id} \otimes \mathrm{ev}_M \otimes \mathrm{id}$. Hence the product \eqref{algStructAdjObj} on $\mathcal{A}_{\mathcal{C}} = \int_M M \otimes M^*$ is defined by the commutative diagram
\[ \xymatrix@C=5em{
\mathcal{A}_{\mathcal{C}} \otimes \mathcal{A}_{\mathcal{C}} \ar[r]^-{\pi_M \,\otimes\,\pi_M} \ar@{-->}[d]_-{\exists!}& M \otimes M^* \otimes M \otimes M^* \ar[d]^-{\mathrm{id} \,\otimes\, \mathrm{ev}_M \,\otimes\, \mathrm{id}}\\
\mathcal{A}_{\mathcal{C}} \ar[r]_-{\pi_M} & M \otimes M^*
} \]
\end{example}

\subsubsection{Adjoint algebra and full center}\label{subsubAdjAlgFullCent}
\indent The following categorical definition of the center of an algebra appears in \cite[Def.\,4.2.1]{Sch} and in \cite[\S 4]{davCenter}; here we use the name introduced in \cite{davCenter}.

\begin{definition}\label{defFullCenter} Let $A = (A,m_A,1_A)$ be an algebra in $\mathcal{C}$. The full center of $A$, if it exists, is a pair $(\mathsf{Z},\xi)$ such that:
\begin{itemize}[itemsep=-.2em, topsep=.2em]
\item $\mathsf{Z} = (Z,b)$ is an object in the Drinfeld center $\mathcal{Z}(\mathcal{C})$, meaning that it is an object in $\mathcal{C}$ equipped with a half-braiding $b : Z \otimes - \overset{\sim}{\implies} - \otimes Z$.
\item Center property: $\xi \in \Hom_{\mathcal{C}}(Z,A)$ satisfies the commutative diagram
\begin{equation}\label{diagramFullCenter}
 \xymatrix@C=4em{
Z \otimes A \ar[d]_-{b_A} \ar[rr]^-{\xi \,\otimes\, \mathrm{id}_A} && A \otimes A \ar[d]^-{m_A}\\
A \otimes Z \ar[r]_-{\mathrm{id}_A \,\otimes\,\xi} & A \otimes A \ar[r]_-{m_A} & A
} \end{equation}
\item Universal property: If $(\mathsf{Z}',\xi')$ also satisfies the above properties, there is a unique $u \in \Hom_{\mathcal{Z}(\mathcal{C})}(\mathsf{Z}',\mathsf{Z})$ such that $\xi' = \xi \circ u$.
\end{itemize}
\end{definition}
It clearly follows from the definition that the full center, if it exists, is unique up to unique isomorphism. The object $\mathsf{Z}$ has a natural structure of ``commutative'' algebra in $\mathcal{Z}(\mathcal{C})$, as we now recall from \cite[Prop.\,4.1]{davCenter} (also remarked in \cite[Rmk.\,4.2.2]{Sch}). One can check that the object $\mathsf{Z} \otimes \mathsf{Z} \in \mathcal{Z}(\mathcal{C})$ equipped with the morphism
\[ Z \otimes Z \xrightarrow{\:\xi \,\otimes\,\xi\:} A \otimes A \xrightarrow{\:m_A\:} A \]
satisfies the first two items in Def.\,\ref{defFullCenter}. Also, the object $\boldsymbol{1}$ equipped with the half-braiding $\mathrm{id}$ and the morphism $1_A : \boldsymbol{1} \to A$ satisfies the first two items in Def.\,\ref{defFullCenter}. As a result, by the third item in the definition, there exist unique $m_{\mathsf{Z}} \in \Hom_{\mathcal{Z}(\mathcal{C})}(\mathsf{Z} \otimes \mathsf{Z}, \mathsf{Z})$ and $1_{\mathsf{Z}} \in \Hom_{\mathcal{Z}(\mathcal{C})}(\boldsymbol{1}, \mathsf{Z})$ such that the diagrams
\begin{equation}\label{algStructFullCenter}
\xymatrix@C=4em{
Z \otimes Z \ar[d]^-{m_{\mathsf{Z}}}_-{\exists!} \ar[r]^-{\xi \,\otimes\, \xi} & A \otimes A \ar[d]^-{m_A}\\
Z \ar[r]_-{\xi} & A
} \qquad \qquad \xymatrix@C=4em{
\boldsymbol{1} \ar[rd]^-{1_A} \ar[d]^-{1_Z}_-{\exists!}& \\
Z \ar[r]_-{\xi} & A
} \end{equation}
commute. It is not hard to check that this defines an algebra structure on $\mathsf{Z}$ in $\mathcal{Z}(\mathcal{C})$. Note that by definition this product and unit on $\mathsf{Z}$ are characterized by the fact that $\xi : Z \to A$ is an algebra morphism in $\mathcal{C}$. Moreover, $\mathsf{Z}$ is commutative in the sense that $m_{\mathsf{Z}} \circ b_Z = m_{\mathsf{Z}}$.

\medskip

We continue with an algebra $(A,m_A,1_A)$ in $\mathcal{C}$ and denote by $\mathrm{Mod}_{\mathcal{C}}(A)$ the category of {\em right $A$-modules in $\mathcal{C}$}. Its objects are pairs $(M,r)$ where $M \in \mathcal{C}$ and $r \in \Hom_{\mathcal{C}}(M \otimes A, M)$ satisfies
\[ r \circ (\mathrm{id}_M \otimes 1_A) = \mathrm{id}_M \quad \text{and} \quad r \circ (\mathrm{id}_M \otimes m_A) = r \circ (r \otimes \mathrm{id}_A). \]
A morphism $f : (M,r) \to (M',r')$ in $\mathrm{Mod}_{\mathcal{C}}(A)$ is $f \in \Hom_{\mathcal{C}}(M,M')$ which commutes with $r$ and $r'$, meaning that $f \circ r = r' \circ (f \otimes \mathrm{id}_A)$. We will often use that a morphism in $\mathrm{Mod}_{\mathcal{C}}(A)$ is in particular a morphism in $\mathcal{C}$. Also note that $A$ is a right $A$-module by multiplication $m_A$.

\smallskip

There is a bifunctor $\rhd : \mathcal{C} \times \mathrm{Mod}_{\mathcal{C}}(A) \to \mathrm{Mod}_{\mathcal{C}}(A)$ given by $X \rhd (M,r) = (X \otimes M, \mathrm{id}_X \otimes r)$ on objects and $g \rhd f = g \otimes f$ on morphisms. In this way {\em $\mathcal{M} = \mathrm{Mod}_{\mathcal{C}}(A)$ becomes a $\mathcal{C}$-module category}. Note that, as underlying objects in $\mathcal{C}$, $X \rhd M$ is just $X \otimes M$. Hence {\em for the mixed associator of $\mathcal{M}$ we simply choose the identity} (which is possible since we assume that the monoidal category $\mathcal{C}$ is strict for simplicity).

\smallskip

\indent Consider the following morphism
\begin{equation*}
\xi : \mathcal{A}_{\mathcal{M}} \xrightarrow{\:\pi_A\:} \underline{\Hom}(A,A) \xrightarrow{\:\mathrm{id} \otimes 1_A\:} \underline{\Hom}(A,A) \otimes A \xrightarrow{\:\underline{\mathsf{ev}}_{A,A}\:} A.
\end{equation*}
As a side remark we note that the composition of the two last arrows in the definition of $\xi$ is an isomorphism of algebras, whose inverse is $A \xrightarrow{\: \underline{\mathsf{coev}}_{A,A} \:} \underline{\Hom}(A, A \rhd A) \xrightarrow{\: \underline{\Hom}(\mathrm{id}_A,m_A) \:} \underline{\Hom}(A, A)$.
\begin{theorem}\label{thmAdjObjFullCent}
Let $\mathcal{M} = \mathrm{Mod}_{\mathcal{C}}(A)$ and $\mathcal{A}_{\mathcal{M}} \in \mathcal{Z}(\mathcal{C})$ be the adjoint object of $\mathcal{M}$ (see Def.\,\ref{defAdjObj}). The pair $(\mathcal{A}_{\mathcal{M}},\xi)$ equipped with the algebra structure \eqref{algStructAdjObj} is the full center of the algebra $A$.
\end{theorem}
\noindent A detailed proof of this theorem is provided in Appendix \ref{appProofFullCent}.

\smallskip

We note that Thm.\,\ref{thmAdjObjFullCent} has the following immediate consequences:

\begin{corollary} Let $\mathcal{C}$ be a finite multitensor category.
\\1. Every algebra $A$ in $\mathcal{C}$ has a unique (up to isomorphism) full center, which is the adjoint algebra $\mathcal{A}_{\mathrm{Mod}_{\mathcal{C}}(A)}$.
\\2. If two algebras $A_1$ and $A_2$ in $\mathcal{C}$ are Morita equivalent\footnote{Meaning that $\mathrm{Mod}_{\mathcal{C}}(A_1) \cong \mathrm{Mod}_{\mathcal{C}}(A_2)$ as $\mathcal{C}$-module categories} then they have the same full center up to a unique isomorphism.
\end{corollary}
\noindent Item 2 was proven in \cite[\S 6]{davCenter} using another method.

\begin{remark}
The regular $\mathcal{C}$-module category $\mathcal{M} = \mathcal{C}$ is realized by the algebra $A = \boldsymbol{1}$, i.e.\ $\mathcal{C} = \mathrm{Mod}_{\mathcal{C}}(\boldsymbol{1})$. It thus follows from Thm.\,\ref{thmAdjObjFullCent} that the algebra $\mathcal{A}_{\mathcal{C}} = \int_X X \otimes X^*$ described in Examples \ref{exampleAdjObjReg} and \ref{exampleProductAC} is the full center of the algebra $\boldsymbol{1} \in \mathcal{C}$.
\end{remark}

\section{Module categories from comodule algebras}\label{sectionHopf}
Let $H = (H,\cdot,1_H,\Delta,\varepsilon,S)$ be a finite-dimensional Hopf algebra over a field $\Bbbk$. We use Sweedler's notation with implicit summation for the coproduct $\Delta$ and its iterations:
\[ \Delta(h) = h^{(1)} \otimes h^{(2)}, \quad (\Delta \otimes \mathrm{id})\bigl( \Delta(h) \bigr) = h^{(1)} \otimes h^{(2)} \otimes h^{(3)}, \quad \text{\it etc}. \]
Recall that a {\em left $H$-comodule algebra} is an associative $\Bbbk$-algebra $A$ equipped with a coaction $\Delta_A : A \to H \otimes A$ which is an algebra morphism. We write $\Delta_A(a) = a_{(1)} \otimes a_{(0)}$ with implicit summation; with this notation the properties imposed on $\Delta_A$ read
\begin{align}
&a_{(1)} \otimes a_{(0)(1)} \otimes a_{(0)(0)} = (a_{(1)})^{(1)} \otimes (a_{(1)})^{(2)} \otimes a_{(0)},\label{coactionAxiom}\\
(ab)_{(1)} \,\otimes\, &(ab)_{(0)} = a_{(1)} b_{(1)} \otimes a_{(0)} b_{(0)}, \qquad 1_{(1)} \otimes 1_{(0)} = 1_H \otimes 1_A\nonumber
\end{align}
for all $a,b \in A$. We denote by $a_{(1)} \otimes a_{(2)} \otimes a_{(0)}$ the common value in \eqref{coactionAxiom}. More generally, \eqref{coactionAxiom} allows to define an iterated coaction $A \to H^{\otimes n} \otimes A$ which we denote by $a_{(1)} \otimes \ldots \otimes a_{(n)} \otimes a_{(0)}$.

\smallskip

\indent Let $H\text{-}\mathrm{mod}$ (resp. $A\text{-}\mathrm{mod}$) be the category of finite-dimensional left modules over $H$ (resp. $A$). Then $A\text{-}\mathrm{mod}$ is a module category over $H\text{-}\mathrm{mod}$: for $X \in H\text{-}\mathrm{mod}$ and $M \in A\text{-}\mathrm{mod}$ we let $X \rhd M \in A\text{-}\mathrm{mod}$ be the vector space $X \otimes_\Bbbk M$ equipped with the $A$-action given by
\begin{equation}\label{AactionXM}
a \cdot (x \otimes m) = (a_{(1)} \cdot x) \otimes (a_{(0)} \cdot m).
\end{equation}
The co-associativity property~\eqref{coactionAxiom} of $\Delta_A$ implies that we have the associator isomorphism of $A$-modules $X \rhd (Y \rhd M) = X \otimes_{\Bbbk} (Y \otimes_{\Bbbk} M) \xrightarrow{\:\cong\:} (X \otimes Y) \rhd M = (X \otimes_{\Bbbk} Y) \otimes_{\Bbbk} M$  inherited from $\mathrm{vect}_{\Bbbk}$. This is analogous to $H\text{-}\mathrm{mod}$ whose monoidal structure using the coproduct $\Delta$ involves the associator of $\mathrm{vect}_{\Bbbk}$. We will however consider both $H\text{-}\mathrm{mod}$ and $A\text{-}\mathrm{mod}$ as strict monoidal and module categories, respectively, via replacing $\mathrm{vect}_{\Bbbk}$ by its strict (skeleton) version.\footnote{The skeletal category of $\mathrm{vect}_\Bbbk$ has as objects  natural numbers $n \in \mathbb{N}$ (thought as the space $\Bbbk^n$) and morphism spaces $\Hom_\Bbbk(m,n)$ are matrix spaces $\mathrm{Mat}_{n,m}(\Bbbk)$. Its (strict) monoidal product is $m \otimes n = mn$ for objects and is the Kronecker product of matrices for morphisms.} We thus say that the mixed associator of $A\text{-}\mathrm{mod}$ is trivial.

Our goal here is to describe the deformation complex of this trivial mixed associator as defined in \S\ref{subsecMixAssoCohom} and the corresponding adjoint algebra, see Def.\,\ref{defAdjObj}, in the comodule algebra case. This allows us to use Thm.\,\ref{thmAdjDYMod} in computing the mixed associator cohomology in quite a few examples, both for exact and non-exact module categories, in the next~\S\ref{sectionExamples}.

\subsection{Deformation complex in terms of comodule algebras}\label{subsecAlgComplex}
By the general definitions in \S\ref{subsecMixAssoCohom}, the $n$-th cochain space $\mathrm{C}^n_{\mathrm{mix}}(A\text{-}\mathrm{mod})$ in the deformation complex with trivial coefficients of $A\text{-}\mathrm{mod}$ as a module category over $H\text{-}\mathrm{mod}$ consists of natural transformations with components
\[ f_{X_1,\ldots,X_n,M} \in \End_A\bigl( X_1 \otimes \ldots \otimes X_n \otimes M \bigr), \quad \forall \, X_1,\ldots,X_n \in H\text{-}\mathrm{mod}, \:\: \forall \, M \in A\text{-}\mathrm{mod}. \]
By definition in \eqref{AactionXM}, the action of $a \in A$ on $v \in X_1 \otimes \ldots \otimes X_n \otimes M$ is the componentwise action $\Delta_A^{(n)}(a) \cdot v$, where for all $n \geq 1$
\[ \Delta^{(n)}_A : A \to H^{\otimes n} \otimes A, \quad a \mapsto a_{(1)} \otimes \ldots \otimes a_{(n)} \otimes a_{(0)} \underset{\text{e.g.}}{=} (a_{(1)})^{(1)} \otimes \ldots \otimes (a_{(1)})^{(n)} \otimes a_{(0)} \]
is the iterated coaction of $A$. By convention we put $\Delta^{(0)}_A = \mathrm{id}_A$.

\smallskip

Let $\mathrm{C}^n_{\mathrm{alg}}(H,A) \subset H^{\otimes n} \otimes A$ be the centralizer of $\Delta^{(n)}_A(A)$, \textit{i.e.}
\begin{equation}\label{defAlgCochains}
\mathrm{C}^n_{\mathrm{alg}}(H,A) = \bigl\{ r \in H^{\otimes n} \otimes A \, \big|\, \forall \, a \in A, \:\: r \,\Delta^{(n)}_A(a) = \Delta^{(n)}_A(a) \,r \bigr\}
\end{equation}
with the usual multiplication in the tensor product of algebras $H^{\otimes n} \otimes A$. Note that $\mathrm{C}^0_{\mathrm{alg}}(H,A) = \mathcal{Z}(A)$, the center of the algebra $A$.

\begin{proposition}\label{propAlgComplex}
For all $n \geq 0$ there is an isomorphism of vector spaces
\[ I^n : \mathrm{C}^n_{\mathrm{mix}}(A\text{-}\mathrm{mod}) \overset{\sim}{\longrightarrow} \mathrm{C}^n_{\mathrm{alg}}(H,A), \quad f \mapsto f_{H,\ldots,H,A}\bigl(1_H^{\otimes n} \otimes 1_A\bigr) \]
where $H$ and $A$ are regarded as modules over themselves by left multiplication. Through this collection of isomorphisms, the cosimplicial structure on $\mathrm{C}^\bullet_{\mathrm{mix}}(A\text{-}\mathrm{mod})$ from \S\ref{subsecMixAssoCohom} yields coface maps $\partial^{(n)}_i : \mathrm{C}^n_{\mathrm{alg}}(A,H) \to \mathrm{C}^{n+1}_{\mathrm{alg}}(A,H)$ given by
\[ \partial^{(n)}_i(h_1 \otimes \ldots \otimes h_n \otimes a) = \begin{cases}
1_H \otimes h_1 \otimes \ldots \otimes h_n \otimes a & \text{ for } i=0\\
h_1 \otimes \ldots \otimes h_i^{(1)} \otimes h_i^{(2)} \otimes \ldots \otimes h_n \otimes a & \text{ for } 1 \leq i \leq n\\
h_1 \otimes \ldots \otimes h_n \otimes a_{(1)} \otimes a_{(0)} & \text{ for } i=n+1\\
\end{cases} \]
(where implicit summation is understood) and codegeneracy maps $s^{(n)}_i : \mathrm{C}^n_{\mathrm{alg}}(A,H) \to \mathrm{C}^{n-1}_{\mathrm{alg}}(A,H)$ given by
\[ s^{(n)}_i(h_1 \otimes \ldots \otimes h_n \otimes a) = \varepsilon(h_i) \, h_1 \otimes \ldots \otimes h_{i-1} \otimes h_{i+1} \otimes \ldots \otimes h_n \otimes a. \]
\end{proposition}
\begin{proof}
Given $r \in \mathrm{C}^n_{\mathrm{alg}}(A,H)$ and $X_1,\ldots,X_n \in H\text{-}\mathrm{mod}$, $M \in A\text{-}\mathrm{mod}$ define
\[ J^n(r)_{X_1,\ldots,X_n,M} = \text{componentwise action of } r \in H^{\otimes n} \otimes A \text{ on } X_1 \otimes \ldots \otimes X_n \otimes M. \]
The endomorphism $J^n(r)_{X_1,\ldots,X_n,M}$ is $A$-linear, since by definition $r$ commutes with $\Delta^{(n)}_A(a)$ for all $a \in A$. Moreover the family $J^n(r)$ is clearly natural. In this way we get a linear map $J^n : \mathrm{C}^n_{\mathrm{alg}}(H,A) \to \mathrm{C}^n_{\mathrm{mix}}(A\text{-}\mathrm{mod})$. The equality $I^n \circ J^n = \mathrm{id}$ is readily seen. To prove that $J^n \circ I^n = \mathrm{id}$, let $f \in \mathrm{C}^n_{\mathrm{mix}}(A\text{-}\mathrm{mod})$, let $x_i \in X_i$ for all $1 \leq i \leq n$ and let $m \in M$. Consider the $H$-linear map $\mathrm{act}_{x_i} : H \to X_i$, $h \mapsto h \cdot x_i$ and the $A$-linear map $\mathrm{act}_m : A \to M$, $a \mapsto a \cdot m$. Note that the definition of $J^n(r)$ can be rewritten as follows:
\[ J^n(r)_{X_1,\ldots,X_n,M}(x_1 \otimes \ldots \otimes x_n \otimes m) = \bigl( \mathrm{act}_{x_1} \otimes \ldots \otimes \mathrm{act}_{x_n} \otimes \mathrm{act}_m \bigr)(r). \]
Then by naturality of $f$ we find
\begin{align*}
f_{X_1,\ldots,X_n,M}(x_1 \otimes \ldots \otimes x_n \otimes m) &= f_{X_1,\ldots,X_n,M} \bigl( \mathrm{act}_{x_1}(1_H) \otimes \ldots \otimes \mathrm{act}_{x_n}(1_H) \otimes \mathrm{act}_m(1_A) \bigr)\\
&=\bigl( \mathrm{act}_{x_1} \otimes \ldots \otimes \mathrm{act}_{x_n} \otimes \mathrm{act}_m \bigr)\bigl[ f_{H,\ldots,H,A}(1_H^{\otimes n} \otimes 1_A)  \bigr]\\
&= J^n\bigl( f_{H,\ldots,H,A}(1_H^{\otimes n} \otimes 1_A) \bigr)_{X_1,\ldots,X_n,M}(x_1 \otimes \ldots \otimes x_n \otimes m)\\
&=J^n\bigl( I^n(f) \bigr)_{X_1,\ldots,X_n,M}(x_1 \otimes \ldots \otimes x_n \otimes m).
\end{align*}
The formulas for the cosimplicial structure are easily established and details are left to the reader.
\end{proof}

It follows from Prop.\,\ref{propAlgComplex} that we have a complex of vector spaces $\mathrm{C}^\bullet_{\mathrm{alg}}(H,A) \cong \mathrm{C}^\bullet_{\mathrm{mix}}(A\text{-}\mathrm{mod})$ whose differential $\delta^n_{\mathrm{alg}} : \mathrm{C}^n_{\mathrm{alg}}(H,A) \to \mathrm{C}^{n+1}_{\mathrm{alg}}(H,A)$ is
\begin{equation}\label{differentialAlgComplex}
\delta^n_{\mathrm{alg}} = 1_H \otimes \mathrm{id}_{H^{\otimes n} \otimes A} + \sum_{i=1}^n (-1)^i \mathrm{id}_{H^{\otimes (i-1)}} \otimes \Delta_H \otimes \mathrm{id}_{H^{\otimes (n-i)} \otimes A} + (-1)^{n+1} \mathrm{id}_{H^{\otimes n}} \otimes \Delta_A.
\end{equation}
By Prop.\,\ref{relDefAssoAndCohom}, the second cohomology space $\mathrm{H}^2_{\mathrm{alg}}(H,A)$ classifies the deformations of the trivial mixed associator of the module
category $A\text{-}\mathrm{mod}$ as a module over 
the monoidal category $H\text{-}\mathrm{mod}$. For any family of cochains $r_1,\ldots,r_k \in \mathrm{C}^2_{\mathrm{alg}}(H,A)$, the obstruction defined in \eqref{defObsMixedAsso} becomes
\begin{equation}\label{obstForAlgComplex}
\begin{aligned}
&\mathrm{obs}(r_1,\ldots,r_k)
\\ =\:&\sum_{i+j=k+1} (\Delta_H \otimes \mathrm{id}_{H \otimes A})(r_i) \, (\mathrm{id}_{H \otimes H} \otimes \Delta_A)(r_j) - (\mathrm{id}_H \otimes \Delta_H \otimes \mathrm{id}_A)(r_i) \, (1_H \otimes r_j)
\end{aligned}
\end{equation}
through the isomorphism in Prop.\,\ref{propAlgComplex}.

\begin{remark}\label{rmkQuasiComod}
A {\em quasi-comodule algebra over $H$} is a triple $(C,\Delta_C, \Phi)$ where $C$ is an associative algebra, $\Delta_C : C \to H \otimes C$ is an algebra morphism such that $(\varepsilon \otimes \mathrm{id}_C) \circ \Delta_C = \mathrm{id}_C$ and $\Phi \in H^{\otimes 2} \otimes C$ is an invertible element such that
\begin{align}
&\forall \, c \in C \quad \Phi(\mathrm{id}_H \otimes \Delta_C)\bigl( \Delta_C(c) \bigr) = (\Delta_H \otimes \mathrm{id}_C)\bigl( \Delta_C(c) \bigr)\Phi,\label{commutationCoass}\\
&(\Delta_H \otimes \mathrm{id}_{H \otimes C})(\Phi) \, (\mathrm{id}_{H \otimes H} \otimes \Delta_C)(\Phi) = (\mathrm{id}_H \otimes \Delta_H \otimes \mathrm{id}_C)(\Phi) \, (1_H \otimes \Phi),\label{cocycleCoass}\\
&(\varepsilon \otimes \mathrm{id}_{H \otimes C})(\Phi) = (\mathrm{id}_H \otimes \varepsilon \otimes \mathrm{id}_C)(\Phi) = 1_{H} \otimes 1_{C}\label{unitalityCoass}.
\end{align}
The categorical interpretation is clear: $\Delta_C$ gives an action of $H\text{-}\mathrm{mod}$ on $C\text{-}\mathrm{mod}$ and the component-wise action of $\Phi$ defines a mixed associator. Indeed \eqref{commutationCoass} corresponds to $C$-linearity of the mixed associator while \eqref{cocycleCoass}-\eqref{unitalityCoass} correspond to the conditions \eqref{mixedAssoAxiom}-\eqref{mixedAssoUnit}. Now let $A$ be a usual $H$-comodule algebra ($\Phi = 1_H^{\otimes 2} \otimes 1_A$) and denote by $A_h$, $H_h$ the extensions of scalars to $\Bbbk[h]/\langle h^2 \rangle$. For $r \in H^{\otimes 2} \otimes A$, we have $r \in \mathrm{C}^2_{\mathrm{alg}}(H,A)$ if and only if $1_H^{\otimes 2} \otimes 1_A + hr$ satisfies \eqref{commutationCoass} and $r$ is a cocycle if and only if $1_H^{\otimes 2} \otimes 1_A + hr$ satisfies \eqref{cocycleCoass}. Finally up to coboundary we can ensure that $(\varepsilon \otimes \mathrm{id}_{H \otimes C})(r) = (\mathrm{id}_H \otimes \varepsilon \otimes \mathrm{id}_C)(r)=0$ so that condition \eqref{unitalityCoass} holds. Hence our deformations $(A_h,\Delta_A, 1_H^{\otimes 2} \otimes 1_A + hr)$ are in the class of quasi-comodule algebras over $H_h$. 

We furthermore note that, when expressed in a basis, the conditions \eqref{commutationCoass}-\eqref{cocycleCoass}-\eqref{unitalityCoass} are linear and quadratic equations among the coefficients of a general co-associator $\Phi$ in this basis. Hence elements $r$ such that $1_H^{\otimes 2} \otimes 1_A + hr$ satisfies these conditions form the Zariski tangent space at $1_H^{\otimes 2} \otimes 1_A$ of the affine variety of quasi-comodule algebra structures on $(A,\Delta_A)$.
\end{remark}

\subsection{Deformations of \texorpdfstring{$\mathrm{vect}_\Bbbk$}{vect} as a module category}
Let $\mathrm{vect}_\Bbbk^{H}$ be the category $\mathrm{vect}_\Bbbk$ of finite-dimensional $\Bbbk$-vector spaces viewed as a module category over $H\text{-}\mathrm{mod}$ by means of the monoidal forgetful functor $\mathcal{U}_H : H\text{-}\mathrm{mod} \to \mathrm{vect}_\Bbbk$, \textit{i.e.}
\begin{equation}\label{vectAsModCat}
X \rhd V = \mathcal{U}_H(X) \otimes_\Bbbk V.
\end{equation}
Its mixed associator cohomology is easy to describe:
\begin{corollary}\label{propVectHModCat}
We have $\mathrm{C}^\bullet_{\mathrm{mix}}\bigl( \mathrm{vect}_\Bbbk^{H} \bigr) \cong \mathrm{C}^\bullet_{\mathrm{DY}}(\mathcal{U}_H)$ as cochain complexes, and thus isomorphic to the Cartier complex of $H$ with trivial coefficients. It follows that
\[ \mathrm{H}^\bullet_{\mathrm{mix}}\bigl( \mathrm{vect}_\Bbbk^{H} \bigr) \cong \mathrm{H}^\bullet_{\mathrm{DY}}(\mathcal{U}_H). \]
\end{corollary}
\begin{proof}
The ground field $\Bbbk$ is a $H$-comodule algebra by means of the coaction $\lambda \mapsto 1_H \otimes \lambda$, and $\mathrm{vect}_\Bbbk^H$ is the module category corresponding to this comodule algebra. We have $\mathrm{C}^n_{\mathrm{alg}}(H,\Bbbk) = H^{\otimes n}$ because the centralizer condition in \eqref{defAlgCochains} is trivial for the present case. From the description of the DY complex of $\mathcal{U}_H$ in \cite[Prop.\,7]{davydov}, it is now readily seen that $\mathrm{C}^\bullet_{\mathrm{alg}}(H,\Bbbk) \cong \mathrm{C}^\bullet_{\mathrm{DY}}(\mathcal{U}_H)$ as cochain complexes. By Prop.\,\ref{propAlgComplex} we thus have $\mathrm{C}^\bullet_{\mathrm{mix}}\bigl(\mathrm{vect}_\Bbbk^{H}\bigr) \cong \mathrm{C}^\bullet_{\mathrm{DY}}(\mathcal{U}_H)$.
\end{proof}

\begin{example}\label{exampleVectOverBk}
Let
\[ B_k = \Lambda \mathbb{C}^k \rtimes \mathbb{C}[\mathbb{Z}/(2)] = \mathbb{C}\bigl\langle x_1,\ldots,x_k,g \, \big| x_i g = -gx_i, \:\: x_ix_j = -x_jx_i, \:\: g^2 = 1 \bigr\rangle \]
be the bosonization of the exterior algebra of $\mathbb{C}^k$. It follows from Cor.\,\ref{propVectHModCat} and the computations in \cite[\S 6.1.1]{FGS} that $\dim \mathrm{H}^n_{\mathrm{mix}}\bigl( \mathrm{vect}_{\mathbb{C}}^{B_k} \bigr) = \binom{k+n-1}{n} $ if $n$ is even and $0$ if $n$ is odd. In particular the vector space of equivalence classes of infinitesimal deformations of $\mathrm{vect}_{\mathbb{C}}$ as a module category over $B_k\text{-}\mathrm{mod}$ has dimension $\frac{k(k+1)}{2}$, and all these deformations can be lifted to any degree in $h$ by Prop.\,\ref{propLiftObstruction}. By \cite[Prop.\,6.2]{FGS} a basis of $\mathrm{H}^2_{\mathrm{alg}}(B_k,\mathbb{C}) \cong \mathrm{H}^2_{\mathrm{mix}}\bigl( \mathrm{vect}_{\mathbb{C}}^{B_k} \bigr)$ consists of the elements $x_i \otimes x_jg$ with $1 \leq i \leq j \leq k$.
\end{example}

\subsection{On the internal Hom for \texorpdfstring{$A\text{-}\mathrm{mod}$}{A-mod} over \texorpdfstring{$H\text{-}\mathrm{mod}$}{H-mod}}
For $\mathcal{C} = H\text{-}\mathrm{mod}$ and $\mathcal{M} = A\text{-}\mathrm{mod}$ equipped with the action \eqref{AactionXM}, the internal Hom objects $\underline{\Hom}(M,N)$ $\in$ $\mathcal{C}$ of the $\mathcal{C}$-module category $\mathcal{M}$ have the following explicit description which can be found in \cite[\S 3.1]{BM} or \cite[\S 4.4]{shimizuSerre}. For convenience we provide a proof, where the adjunction is made explicit:
\begin{lemma}\label{descriptionIntHom}
 For $M,N \in A\text{-}\mathrm{mod}$, the internal Hom object $\underline{\Hom}(M,N) \in H\text{-}\mathrm{mod}$ is the vector space $\Hom_A(H \rhd M, N)$ equipped with the $H$-action given by
\begin{equation}\label{HactionIntHom}
(k \cdot f)(h \otimes m) = f(hk \otimes m)
\end{equation}
for all $k \in H$, $f \in \Hom_A(H \rhd M, N)$ and $h \otimes m \in H \rhd M$.
\end{lemma}
\begin{proof}
Let $X \in H\text{-}\mathrm{mod}$. For all $x \in X$ and $f \in \Hom_A(X \rhd M,N)$ define a map
\[ \Psi_f(x) : H \otimes M \to N, \quad h \otimes m \mapsto f(h \cdot x \otimes m). \]
It is readily seen that $\Psi_f(x) \in \Hom_A(H \rhd M,N)$. Moreover, it is straightforward to check that $\Psi_f(k \cdot x) = k \cdot \Psi_f(x)$ for all $k \in H$, with respect to the $H$-action proposed in the lemma. Hence we get a linear map
\begin{align*}
I_{X,M,N} : \Hom_A(X \rhd M, N) &\longrightarrow \Hom_H\bigl( X, \Hom_A(H \rhd M,N) \bigr)\\
f &\longmapsto \bigl[ x \mapsto \Psi_f(x) \bigr]
\end{align*}
whose inverse is
\begin{align*}
\Hom_H\bigl( X, \Hom_A(H \rhd M,N) \bigr) &\longrightarrow \Hom_A(X \rhd M, N)\\\
g &\longmapsto \bigl[ x \otimes m \mapsto g(x)(1_H \otimes m) \bigr].
\end{align*}
It is not hard to check that the isomorphism $I_{X,M,N}$ is natural in $X,M,N$. This proves that the universal property \eqref{adjIntHom} of the internal Hom is verified by $\Hom_A(H \rhd M,N)$.
\end{proof}

Within the description of Lemma~\ref{descriptionIntHom}, the unit $\underline{\mathsf{coev}}_{X,M} = I_{X,M,X \rhd M}(\mathrm{id})$ and counit $\underline{\mathsf{ev}}_{M,N} = I^{-1}_{\underline{\Hom}(M,N),M,N}(\mathrm{id})$ of the adjunction $(- \rhd M) \dashv \underline{\Hom}(M, -) = \Hom_A(H \rhd M,-)$ are given by
\begin{align}
\begin{split}\label{CoevEvComodAlg}
&\underline{\mathsf{coev}}_{X,M} : X \longrightarrow \Hom_A(H \rhd M, X \rhd M), \quad x \longmapsto \bigl[ h \otimes m \mapsto (h \cdot x) \otimes m \bigr]\\
&\underline{\mathsf{ev}}_{M,N} : \Hom_A(H \rhd M, N) \otimes M \longrightarrow N, \quad f \otimes m \longmapsto f(1_H \otimes m).
\end{split}
\end{align}

\indent For all $M \in \mathcal{M} = A\text{-}\mathrm{mod}$, the internal End object $\underline{\End}(M) = \underline{\Hom}(M,M)$ is an algebra in $\mathcal{C} = H\text{-}\mathrm{mod}$, whose product is recalled in \eqref{prodInternalEnd} for any $\mathcal{M}$. In the present case, this product is given for $f,g \in \Hom_A(H \rhd M,M) = \underline{\End}(M)$ by
\begin{equation}\label{prodIntEndComodAlg}
\forall \, h \otimes m \in H \rhd M, \quad (fg)(h \otimes m) = f\bigl( h^{(1)} \otimes g(h^{(2)} \otimes m) \bigr)
\end{equation}
while the unit is defined by $1_{\underline{\End}(M)}(h \otimes m) = \varepsilon(h)m$.

\indent We finally note from \eqref{formulaJ} in App.\,\ref{appProofFullCent}  and \eqref{CoevEvComodAlg} that for all $X,Y \in H\text{-}\mathrm{mod}$ and $M,N \in A\text{-}\mathrm{mod}$ the morphism
\[ \xymatrix@C=1.5em@R=.7em{
J_{X,M,Y,N} : \hspace{-3.3em} & \underline{\Hom}(X \rhd M, Y \rhd N) \otimes X \ar[r]^-{\sim}\ar@{=}[d] & Y \otimes \underline{\Hom}(M,N) \ar@{=}[d]\\
&\Hom_A\bigl( (H \otimes X) \rhd M, Y \rhd N  \bigr) \otimes X & Y \otimes \Hom_A(H \rhd M,N)
} \]
introduced in~\eqref{isoJIntHom} is given by
\begin{equation}\label{isoJComodAlg}
J_{X,M,Y,N}(g \otimes x) = \sum_i y_i \otimes \Bigl[ h \otimes m \mapsto \bigl(y^i\bigl(S(h^{(1)})\cdot ?\bigr) \otimes \mathrm{id}_N\bigr)\bigl( g(h^{(2)} \otimes h^{(3)} \cdot x \otimes m) \bigr) \Bigr]
\end{equation}
where $(y_i)$ is a basis of $Y$ with dual basis $(y^i)$ and $y^i\bigl(S(h^{(1)})\cdot ?\bigr)$ is the linear form on $Y$ given by $y \mapsto y^i\bigl(S(h^{(1)})\cdot y\bigr)$.

\subsection{Description of the adjoint algebra}\label{subAdjAlgAmod}
Let $(H^*)^{\mathrm{op}}$ be the dual vector space $H^*$ endowed with the product $\alpha \star \beta = (\beta \otimes \alpha) \circ \Delta$ and unit $1_{H^*} = \varepsilon$. Denote by $D(H)$ the {\em Drinfeld double} of $H$; it is the $\Bbbk$-vector space $(H^*)^{\mathrm{op}} \otimes H$ whose product is defined by the following rules:
\begin{itemize}[itemsep=-.2em,topsep=.2em]
\item The subspaces $(H^*)^{\mathrm{op}} \otimes 1_H$ and $1_{H^*} \otimes H$ are subalgebras isomorphic to $(H^*)^{\mathrm{op}}$ and $H$. We thus write $\alpha$ and $h$ instead of $\alpha \otimes 1_H$ and $1_{H^*} \otimes h$.
\item With these notations the product of $D(H)$ satisfies
\[ \alpha  h = \alpha \otimes h, \quad h \alpha = \alpha\bigl( S(h^{(1)}) \, ? \, h^{(3)} \bigr)h^{(2)} \]
where $\alpha\bigl( S(h^{(1)}) \, ? \, h^{(3)} \bigr) \in H^*$ is defined by $x \mapsto \alpha\bigl( S(h^{(1)}) x h^{(3)} \bigr)$.
\end{itemize}
We will use the well-known monoidal isomorphism $\mathcal{Z}(H\text{-}\mathrm{mod}) \cong D(H)\text{-}\mathrm{mod}$: if $V \in H\text{-}\mathrm{mod}$ is equipped with a half-braiding $t : V \otimes - \Rightarrow - \otimes V$ then the formula
\begin{equation}\label{actionFromHalfBr}
\alpha \cdot v = (\alpha \otimes \mathrm{id}_V) \circ t_H(v \otimes 1_H)
\end{equation}
where $H$ is viewed as a left module over itself, defines a $(H^*)^{\mathrm{op}}$-action on $V$ which assembles with the already present $H$-module structure on $V$ into a $D(H)$-module structure.

\smallskip

Recall the adjoint object $\mathcal{A}_{\mathcal{M}} \in \mathcal{Z}(\mathcal{C})$ from Definition \ref{defAdjObj}, endowed with the algebra structure explained in \S\ref{subsubAdjAlg}. With the present choice $\mathcal{C} = H\text{-}\mathrm{mod}$ and $\mathcal{M} = A\text{-}\mathrm{mod}$, here is an explicit description of $\mathcal{A}_{\mathcal{M}}$ as an algebra in $\mathcal{Z}(H\text{-}\mathrm{mod}) \cong D(H)\text{-}\mathrm{mod}$. This result was stated in \cite[\S 4.2]{BM} in terms of Yetter-Drinfeld modules (and in the form of Remark \ref{rmkAdjAlgAsBimodMaps} below); we give some elements of the proof for convenience.

\begin{proposition}\label{propAdjObjComodAlg}
For $\mathcal{C} = H\text{-}\mathrm{mod}$ and $\mathcal{M} = A\text{-}\mathrm{mod}$, the adjoint algebra of the $\mathcal{C}$-module category $\mathcal{M}$ is the $\Bbbk$-vector space 
\begin{equation*}
\mathcal{A}_{H,A} = \bigl\{ \varphi \in \Hom_\Bbbk(H,A) \, \big| \, \forall \, a \in A, \:\: \forall \, h \in H, \:\: a\varphi(h) = \varphi\bigl( a_{(1)}h \bigr) a_{(0)} \bigr\}
\end{equation*}
equipped with the $D(H)$-action defined by
\begin{equation}\label{DHactionAdjObj}
(k \cdot \varphi)(h) = \varphi(hk), \quad (\alpha \cdot \varphi)(h) = \alpha\bigl(S(h^{(1)})\varphi(h^{(2)})_{(1)}h^{(3)} \bigr) \, \varphi(h^{(2)})_{(0)}
\end{equation}
for all $h,k \in H$, $\alpha \in (H^*)^{\mathrm{op}}$ and $\varphi \in \mathcal{A}_{H,A}$. The algebra structure is given by
\[\forall \, \varphi,\psi \in \mathcal{A}_{H,A}, \:\: \forall \, h \in H, \quad (\varphi\psi)(h) = \varphi(h^{(1)})\psi(h^{(2)}), \quad 1_{\mathcal{A}_{H,A}}(h) = \varepsilon(h)1_A. \]
\end{proposition}
\begin{proof}
By the end description of the adjoint object explained after Definition \ref{defAdjObj} and using the description of the internal Hom in Lemma \ref{descriptionIntHom} we have to show that 
\[ \textstyle \mathcal{A}_{H,A} = \int_{M \in A\text{-}\mathrm{mod}} \Hom_A(H \rhd M,M) \] \textit{i.e.} we must provide a universal dinatural transformation with values in the $H$-module $\mathcal{A}_{H,A}$. For all $M \in A\text{-}\mathrm{mod}$, let
\begin{equation}\label{defDinatUnivComodAlg}
\pi_M : \mathcal{A}_{H,A} \longrightarrow \Hom_A(H \rhd M, M), \quad \varphi \longmapsto \bigl[ h \otimes m \mapsto \varphi(h) \cdot m \bigr].
\end{equation}
We first note that the $\Bbbk$-linear map $\pi_M(\varphi)$ indeed takes values in $\Hom_A(H \rhd M, M)$, {\it i.e.} it is $A$-linear, thanks to the condition defining $\mathcal{A}_{H,A}$:
\begin{align*}
\pi_M(\varphi)\bigl( a \cdot (h \otimes m) \bigr) &= \pi_M(\varphi)\bigl( (a_{(1)} h) \otimes (a_{(0)} \cdot m) \bigr)\\
&= \varphi(a_{(1)} h) a_{(0)} \cdot m = a\varphi(h) \cdot m = a \cdot \pi_M(\varphi)(h \otimes m).
\end{align*}
It is moreover clear that the family $\pi = (\pi_M)_{M \in \mathcal{M}}$ is dinatural, which in the present case means that
\[ \forall \, f \in \Hom_A(M,N), \:\: \forall \, \varphi \in \mathcal{A}_{H,A}, \quad f \circ \pi_M(\varphi) = \pi_N(\varphi) \circ (\mathrm{id}_H \otimes f). \]
Finally, $\pi_M$ is easily seen to be $H$-linear for all $M$, with respect to the actions \eqref{DHactionAdjObj} and \eqref{HactionIntHom}. We now check that $\pi$ is universal. Let $d = \bigl( d_M : C \to \Hom_A(H \rhd M,M) \bigr)_{M \in \mathcal{M}}$ be a dinatural family, with $C \in H\text{-}\mathrm{mod}$. For all $c \in C$ define a $\Bbbk$-linear map
\[ u(c) : H \to A, \quad h \mapsto d_A(c)(h \otimes 1_A). \]
Then by $A$-linearity of $d_A(c)$ and dinaturality of $d$ we have
\begin{align*}
&a u(c)(h) = d_A(c)\bigl( a \cdot (h \otimes 1_A) \bigr) = d_A(c)\bigl( a_{(1)}h \otimes a_{(0)} \bigr)\\
=\:& d_A(c) \bigl[ (\mathrm{id}_H \otimes r_{a_{(0)}})\bigl( a_{(1)}h \otimes 1_A \bigr) \bigr] = r_{a_{(0)}} \bigl[ d_A(c)\bigl( a_{(1)}h \otimes 1_A \bigr) \bigr] = u(c)\bigl( a_{(1)}h \bigr) a_{(0)}
\end{align*}
where $r_a \in \End_A(A)$ is the right multiplication $x \mapsto xa$. Hence $u(c) \in \mathcal{A}_{H,A}$, which means that we have constructed a map $u : C \to \mathcal{A}_{H,A}$. Using the $H$-linearity of $d_A$ with respect to \eqref{HactionIntHom}, it is easy to see that $u$ is $H$-linear when $\mathcal{A}_{H,A}$ is equipped with the $H$-action in \eqref{DHactionAdjObj}. In order to prove that $d_M = \pi_M \circ u$, consider for all $M \in A\text{-}\mathrm{mod}$ and $m \in M$ the $A$-linear map $r_m : A \to M$, $a \mapsto a \cdot m$. Then for all $c \in C$ the dinaturality of $d$ yields
\begin{align*}
d_M(c)(h \otimes m) &= d_M(c) \bigl[ (\mathrm{id}_H \otimes r_m)(h \otimes 1_A) \bigr] = r_m \bigl[ d_A(c)(h \otimes 1_A) \bigr]\\
&= d_A(c)(h \otimes 1_A) \cdot m = u(c)(h) \cdot m = \pi_M\bigl( u(c) \bigr)(h \otimes m)
\end{align*}
as desired.

To obtain the announced $(H^*)^{\mathrm{op}}$-action on $\mathcal{A}_{H,A}$ we have to compute the half-braiding $b$ on it, characterized by the universal property \eqref{halfBrEnd} which uses the natural isomorphism $J_{X,M,Y,N}$ described in \eqref{isoJComodAlg} in the present case $\mathcal{M} = A\text{-}\mathrm{mod}$. Fix a basis $(h_j)$ of $H$ with dual basis $(h^j)$ and let $X \in H\text{-}\mathrm{mod}$. Then using \eqref{defDinatUnivComodAlg} it is not hard to find
\[ J_{X,M,X,M} \bigl( \pi_{X \rhd M}(\varphi) \otimes x \bigr) = \sum_{j}  \bigl(S(h_j^{(1)})\varphi(h_j^{(2)})_{(1)}h_j^{(3)} \cdot x\bigr) \otimes \pi_M\bigl( \hspace{-2.4em}\underbrace{\varphi(h_j^{(2)})_{(0)}\otimes h^j}_{\qquad\qquad\in A \otimes H^* \,\cong\, \Hom_\Bbbk(H,A)} \hspace{-2.4em} \bigr) \]
for all $\varphi \in \mathcal{A}_{H,A}$ and $x \in X$. It follows from \eqref{halfBrEnd} that $b_X : \mathcal{A}_{H,A} \otimes X \to X \otimes \mathcal{A}_{H,A}$ is given by
\[ b_X(\varphi \otimes x) = \sum_j \bigl(S(h_j^{(1)})\varphi(h_j^{(2)})_{(1)}h_j^{(3)} \cdot x\bigr) \otimes \bigl(\hspace{-2.4em}\underbrace{\varphi(h_j^{(2)})_{(0)}\otimes h^j}_{\qquad\qquad\in A \otimes H^* \,\cong\, \Hom_\Bbbk(H,A)}\hspace{-2.4em}\bigr). \]
The expression of the $(H^*)^{\mathrm{op}}$-action on $\mathcal{A}_{H,A}$ now immediately follows from \eqref{actionFromHalfBr}.

Finally recall that the product in the adjoint algebra of $\mathcal{M}$ is described in \eqref{algStructAdjObj} through the product on internal End objects $\underline{\End}(M)$, which is itself given in \eqref{prodIntEndComodAlg} in the present case $\mathcal{M} = A\text{-}\mathrm{mod}$. To use this description, note from \eqref{defDinatUnivComodAlg} that $\varphi(h) = \pi_A(\varphi)(h \otimes 1_A)$ for all $\varphi \in \mathcal{A}_{H,A}$ and $h \in H$. We thus find
\begin{align*}
&(\varphi \psi)(h) = \pi_A(\varphi \psi)(h \otimes 1_A) \overset{\eqref{algStructAdjObj}}{=} \bigl( \pi_A(\varphi) \pi_A(\psi) \bigr)(h \otimes 1_A)\\
 \overset{\eqref{prodIntEndComodAlg}}{=} \:& \pi_A(\varphi)\bigl( h^{(1)} \otimes \pi_A(\psi)(h^{(2)} \otimes 1_A) \bigr) = \pi_A(\varphi)\bigl( h^{(1)} \otimes \psi(h^{(2)}) \bigr) \overset{\eqref{defDinatUnivComodAlg}}{=} \varphi(h^{(1)})\psi(h^{(2)}).
\end{align*}
The formula for the unit is deduced similarly.
\end{proof}

\begin{example}\label{exampleAdjObjRegularHopf}
Take $A=H$, viewed as a $H$-comodule by mean of the coproduct. Then $\mathcal{A}_{H,H}$ consists of linear maps $\varphi : H \to H$ satisfying $k\varphi(h) = \varphi(k_{(1)}h)k_{(2)}$ for all $h,k \in H$, which we rewrite as $\varphi(kh) = k_{(1)}\varphi(h)S(k_{(2)}) = \mathrm{ad}(k)\bigl(\varphi(h)\bigr)$ introducing the adjoint action. In particular $\varphi(k) = \mathrm{ad}(k)\bigl( \varphi(1_H) \bigr)$ and we get the isomorphism of vector spaces
\[ \mathcal{A}_{H,H} \overset{\sim}{\longrightarrow} H, \quad \varphi \longmapsto \varphi(1_H) \]
whose inverse is $h \mapsto \varphi_h$ with $\varphi_h(k) = \mathrm{ad}(k)(h)$ for all $k \in H$. It is easy to see that, through this isomorphism, the $D(H)$-module structure on $H$ induced by \eqref{DHactionAdjObj} is given by
\[ \forall k,h \in H, \:\: \forall \, \alpha \in (H^*)^{\mathrm{op}}, \quad k \cdot h = \mathrm{ad}(k)(h) \quad \text{and} \quad \alpha \cdot h = \alpha(h_{(1)})h_{(2)}. \]
Also, through the isomorphism, the algebra structure of $\mathcal{A}_{H,H}$ is just that of $H$.
\end{example}

\begin{remark}\label{rmkAdjAlgAsBimodMaps}
Equip the vector space $H \otimes A$ with the $A$-bimodule structure given by
\[ a \cdot (h \otimes a') = (a_{(1)}h) \otimes (a_{(0)}a') \quad \text{and} \quad (h \otimes a') \cdot a = h \otimes (a'a). \]
Also look at $A$ as a bimodule over itself (regular bimodule) and denote by $\Hom_{A\text{-}A}(H \otimes A,A)$ the space of $A$-bimodule morphisms. It is easily seen that the map
\[ \Hom_{A\text{-}A}(H \otimes A,A) \overset{\sim}{\longrightarrow} \mathcal{A}_{H,A}, \quad f \mapsto f(? \otimes 1_A) \]
is an isomorphism of vector spaces whose inverse is $\varphi \mapsto \bigl[ h \otimes a \mapsto \varphi(h)a \bigr]$. The $D(H)$-module algebra structure on $\mathcal{A}_{H,A}$ from  Prop.\,\ref{propAdjObjComodAlg} can easily be translated through this iso.
\end{remark}

\indent Still using the isomorphism $\mathcal{Z}(H\text{-}\mathrm{mod}) \cong D(H)\text{-}\mathrm{mod}$, we can replace the resolvent pair $\mathcal{Z}(H\text{-}\mathrm{mod}) \rightleftarrows H\text{-}\mathrm{mod}$ by the resolvent pair $D(H)\text{-}\mathrm{mod} \rightleftarrows H\text{-}\mathrm{mod}$, where the forgetful functor is the pullback of the inclusion $H \hookrightarrow D(H)$ and its left adjoint is the usual induction from $H$ to $D(H)$. Then Thm.\,\ref{thmAdjDYMod} is rephrased as follows:

\begin{corollary}\label{coroComputeMixCohomComodAlg}Let $A$ be an $H$-comodule algebra and consider
$A\text{-}\mathrm{mod}$ as a module category over $H\text{-}\mathrm{mod}$ by \eqref{AactionXM}. Then
\[ \mathrm{H}^\bullet_{\mathrm{mix}}(A\text{-}\mathrm{mod}) \cong \mathrm{Ext}^\bullet_{D(H),H}(\Bbbk, \mathcal{A}_{H,A}) \]
where $\Bbbk$ is the trivial $D(H)$-module and $\mathcal{A}_{H,A}$ is the $D(H)$-module described in Prop.\,\ref{propAdjObjComodAlg}.
\end{corollary}

As another corollary of Prop.\,\ref{propAdjObjComodAlg}, we note that the dimension formula provided in Cor.\,\ref{cor:dim-formula-Id} applies to $\dim \mathrm{H}^\bullet_{\mathrm{mix}}(A\text{-}\mathrm{mod})$ with $\mathcal{A}_{\mathcal{M}} = \mathcal{A}_{H,A}$ and $\Hom_{\mathcal{Z}(\mathcal{C})} = \Hom_{D(H)}$.

\section{Examples}\label{sectionExamples}
The examples we present are based on the Sweedler Hopf algebra
\[ \mathsf{Sw} = \mathbb{C}\bigl\langle x,g \, \big|\,  gx = -xg, \quad g^2 = 1, \quad x^2 = 0 \bigr\rangle. \]
A basis of $\mathsf{Sw}$ is formed by the monomials $1, g, x, gx$. Its Hopf structure is defined by
\begin{equation*}
\begin{array}{c}
\Delta(g) = g \otimes g, \:\: \Delta(x) = 1 \otimes x + x \otimes g,\\
\varepsilon(g) = 1, \:\: \varepsilon(x) = 0, \qquad S(x) = gx, \:\: S(g) = g^{-1} = g.
\end{array}
\end{equation*}
The Drinfeld double $D(\mathsf{Sw})$ is generated by $x,g,y,f$ modulo the relations
\begin{equation}\label{relationsDSw}
\begin{array}{c}
gx = -xg, \quad gy = - yg, \quad fx = -xf, \quad fy = -yf, \quad gf = fg,\\[.3em]
xy + yx = (1 - fg), \quad x^2 = 0, \quad y^2 = 0, \quad g^2 = 1, \quad f^2 = 1.
\end{array}
\end{equation}
The elements $f$ and $y$ correspond to the linear forms on $\mathsf{Sw}$ defined by
\begin{equation*}
\forall \, i,j \in \{0,1\}, \quad f\bigl( g^ix^j \bigr) = (-1)^i \delta_{j,0} \quad \text{and} \quad  y\bigl(g^ix^j\bigr) = \delta_{j,1}
\end{equation*}
and $(\mathsf{Sw}^*)^{\mathrm{op}}$ is generated by $f,y$ modulo the relations \eqref{relationsDSw} among them.

\smallskip

The algebra $\mathsf{Sw}$ is a particular case of bosonized exterior algebra $B_k = \Lambda \mathbb{C}^k \rtimes \mathbb{C}[\mathbb{Z}/(2)]$ for $k=1$. We can thus use the relatively projective resolution of $\mathbb{C} \in D(B_k)\text{-}\mathrm{mod}$ constructed in \cite[\S 5]{GHS}, which for $k=1$ is as follows: let $\mathcal{C}_\pm$ be the indecomposable 2-dimensional $D(\mathsf{Sw})$-module with basis $(t_\pm,b_\pm)$ such that

\begin{equation}\label{defModCpm} \xymatrix@R=1em{
\overset{(\pm,\pm)}{t_{\pm}} \ar[d]^-y_-{ \text{\normalsize $\begin{array}{lll}
g \cdot t_\pm = f \cdot t_\pm = \pm t_\pm, & x \cdot t_\pm = 0, & y \cdot t_\pm = b_\pm,\\
g \cdot b_\pm = f \cdot b_\pm = \mp b_\pm, & x \cdot b_\pm = 0, & y \cdot b_\pm = 0.
\end{array}$\hspace{6em}}} \\
\underset{(\mp,\mp)}{b_{\pm}}
} \end{equation}
where the pairs of signs in the picture indicate the weights for the pair $(g,f)$ of group-like elements. Then the 2-periodic sequence
{
\arraycolsep=.15em
\begin{equation}\label{relProjResTrivSw}
\begin{array}{ccccccccccccc}
&&&& t_+ & \longmapsto & b_- && t_+ & \longmapsto & 1 &  &  \\[-.3em]
\ldots & \longrightarrow & \mathcal{C}_- & \overset{d_-}{\longrightarrow} & \mathcal{C}_+ & \underset{d_+}{\longrightarrow} & \mathcal{C}_- & \overset{d_-}{\longrightarrow} & \mathcal{C}_+ & \longrightarrow & \mathbb{C} & \longrightarrow & 0\\[-.5em]
&& t_- & \longmapsto & b_+ && t_- & \longmapsto & b_+ & & &  &
\end{array}
\end{equation}}
is a relatively projective resolution with respect to the resolvent pair $D(\mathsf{Sw})\text{-}\mathrm{mod} \rightleftarrows \mathsf{Sw}\text{-}\mathrm{mod}$.

\smallskip

We will also encounter the simple 2-dimensional $D(\mathsf{Sw})$-module $\mathcal{S}_\pm$ with basis $(v_\pm,w_\pm)$ such that
\begin{equation}\label{defSimpleModSpm} \xymatrix@R=1em{
\overset{(\pm,\mp)}{v_{\pm}} \ar@<-1ex>[d]_-x_-{ \text{\normalsize $\begin{array}{llll}
g \cdot v_\pm = \pm v_\pm, & f \cdot v_\pm = \mp v_\pm, & x \cdot v_\pm = w_{\pm}, & y \cdot v_\pm = 0,\\
g \cdot w_\pm = \mp w_\pm, & f \cdot w_\pm = \pm w_\pm, & x \cdot w_\pm = 0, & y \cdot w_\pm = 2v_\pm.
\end{array}$\hspace{4em}}} \\
\underset{(\mp,\pm)}{w_{\pm}} \ar@<-1ex>[u]_-y
} \end{equation}
where the pairs of signs in the picture indicate the weights for the pair $(g,f)$.

\subsection{Deformations of \texorpdfstring{$\mathrm{vect}_{\mathbb{C}}$}{vect} over \texorpdfstring{$\mathsf{Sw}\text{-}\mathrm{mod}$}{Sw-mod}}\label{subDeformVectSw}
Let $\mathrm{vect}_{\mathbb{C}}$ be the category of finite-dimensional complex vector spaces. It is a finite module category over $\mathsf{Sw}\text{-}\mathrm{mod}$ by means of the forgetful functor $\mathcal{U} : \mathsf{Sw}\text{-}\mathrm{mod} \to \mathrm{vect}_{\mathbb{C}}$, see \eqref{vectAsModCat}. It follows from Example \ref{exampleVectOverBk} (take $k=1$ there) that 
\[ \dim \mathrm{H}^p_{\mathrm{mix}}(\mathrm{vect}_{\mathbb{C}}) = \begin{cases}
1 & \text{if } p \text{ is even,}\\
0 & \text{if } p \text{ is odd}
\end{cases} \]
and a basis of $\mathrm{H}^2_{\mathrm{mix}}(\mathrm{vect}_{\mathbb{C}})$ consists of the natural transformation $\alpha$ given by
\[ \alpha_{V,W,M} : \mathcal{U}(V) \otimes \mathcal{U}(W) \otimes M \to \mathcal{U}(V \otimes W) \otimes M, \quad v \otimes w \otimes m \mapsto x \cdot v \otimes xg \cdot w \otimes m \]
for all $V,W \in \mathsf{Sw}\text{-}\mathrm{mod}$ and $M \in \mathrm{vect}_{\mathbb{C}}$. Hence $\mathrm{id} + h\alpha$ is an infinitesimal deformation of the trivial mixed associator. Since $\mathrm{H}^3_{\mathrm{mix}}(\mathrm{vect}_{\mathbb{C}}) = 0$ we know from Prop.\,\ref{propLiftObstruction} that it can be lifted to any degree. One can check that $\mathrm{id} +h\alpha$ is again a mixed associator over $\mathbb{C}[h]/\langle h^{N+1} \rangle$ for any $N$, and thus $\mathrm{id} +h\alpha$ can be seen as a formal deformation. In particular this deformation can be evaluated at any complex number $\lambda$ and we get in this way a 1-parameter family $\mathrm{vect}_{\mathbb{C}}^\lambda$ of module categories over $\mathsf{Sw}\text{-}\mathrm{mod}$, each of which is $\mathrm{vect}_{\mathbb{C}}$ as a linear category and whose mixed associators are $\mathrm{id} + \lambda \alpha$.

It is proven in \cite[Prop.\,5.20]{willprecht} that $\mathrm{vect}_{\mathbb{C}}^{\lambda_1}$ and $\mathrm{vect}_{\mathbb{C}}^{\lambda_2}$ are equivalent as modules categories if and only if $\lambda_1 = \lambda_2$.

\subsection{Deformations of \texorpdfstring{$\mathrm{svect}_{\mathbb{C}}$}{svect} over \texorpdfstring{$\mathsf{Sw}\text{-}\mathrm{mod}$}{Sw}}\label{subDeformSvectSw}
We have the Hopf subalgebra $\mathbb{C}\langle g \rangle \subset \mathsf{Sw}$, which in particular is a $\mathsf{Sw}$-comodule algebra. Since $\mathbb{C}\langle g \rangle \cong \mathbb{C}[\mathbb{Z}/(2)]$, the category $\mathbb{C}[\mathbb{Z}/(2)]\text{-}\mathrm{mod} = \mathrm{svect}_\mathbb{C}$ of super-vector spaces (we do not imply any braiding in this identification) is a module category over $\mathsf{Sw}\text{-}\mathrm{mod}$.

\smallskip

An element in the adjoint object $\mathcal{A} = \mathcal{A}_{\mathsf{Sw},\mathbb{C}\langle g \rangle}$ described in Prop.\,\ref{propAdjObjComodAlg} is a $\mathbb{C}$-linear map $\varphi : \mathsf{Sw} \to \mathbb{C}\langle g \rangle = \mathrm{span}_{\mathbb{C}}(1,g)$ such that $\varphi(gh) = g\varphi(h)g$ for all $h \in \mathsf{Sw}$. It thus entirely determined by the values $\varphi(1_{\mathsf{Sw}})$ and $\varphi(x)$ in $\mathbb{C}\langle g \rangle$. As a result $\dim(\mathcal{A}) = 4$. Define $\varphi_1,\varphi_g \in \mathcal{A}$:
\[ \varphi_1(1_{\mathsf{Sw}}) = 0, \quad \varphi_1(x) = 1 \qquad \text{and} \qquad \varphi_g(1_{\mathsf{Sw}}) = 0, \quad \varphi_g(x) = g. \]
Then clearly $\bigl\{ \varphi_1,\, \varphi_g,\, x \cdot \varphi_1,\, x \cdot \varphi_g \bigr\}$ is a basis of $\mathcal{A}$, where we use the action \eqref{DHactionAdjObj} on $\mathcal{A}$. Straightforward computations reveal moreover that
\begin{align*}
&g \cdot \varphi_1 = -\varphi_1, \quad f \cdot \varphi_1 = -\varphi_1, \quad y \cdot \varphi_1 = 0 \quad \text{ whence } y \cdot (x \cdot \varphi_1) = 0 \text{ by \eqref{relationsDSw},}\\
&g \cdot \varphi_g = - \varphi_g, \quad f \cdot \varphi_g = \varphi_g, \quad y \cdot \varphi_g = 0 \quad \text{ whence } y \cdot (x \cdot \varphi_g) = 2 \varphi_g \text{ by \eqref{relationsDSw}}.
\end{align*}
Let us represent schematically the structure of $\mathcal{A}$ as a $D(\mathsf{Sw})$-module:
\[ \xymatrix@R=.1em@C=.3em{
&\overset{(-,-)}{\varphi_1} \ar[dd]^-x & & \overset{(-,+)}{\varphi_g} \ar@<-1ex>[dd]_-x & \\
\mathcal{A} = & &\oplus\:\: & & = \mathcal{V} \oplus \mathcal{S}_-\\
&\underset{(+,+)}{x \cdot \varphi_1} &  & \underset{(+,-)}{x \cdot \varphi_g}\ar@<-1ex>[uu]_-y &
}\]
where the couples of signs indicate the weights for the pair $(g,f)$, we introduce the notation $\mathcal{V}$ for the first summand and $\mathcal{S}_-$ is the simple module \eqref{defSimpleModSpm}.

\begin{proposition}\label{propDefSvectSw}
The deformation cohomology of the (trivial) mixed associator of $\mathrm{svect}_{\mathbb{C}} = \mathbb{C}\langle g \rangle\text{-}\mathrm{mod}$ as a module category over $\mathsf{Sw}\text{-}\mathrm{mod}$ satisfies
\[ \dim \mathrm{H}_{\mathrm{mix}}^p(\mathrm{svect}_{\mathbb{C}}) = \begin{cases}
1 & \text{if } p \text{ is even,}\\
0 & \text{if } p \text{ is odd.}
\end{cases} \]
The cocycle $x \otimes xg \otimes 1$ forms a basis of $\mathrm{H}_{\mathrm{alg}}^2(\mathsf{Sw},\mathbb{C}\langle g \rangle) \cong \mathrm{H}_{\mathrm{mix}}^2(\mathrm{svect}_{\mathbb{C}})$, see \S\ref{subsecAlgComplex} for the definition of $\mathrm{H}_{\mathrm{alg}}^\bullet$.
\end{proposition}
\begin{proof}
We use Cor.\,\ref{coroComputeMixCohomComodAlg} together with the relative resolution \eqref{relProjResTrivSw} of $\mathbb{C}$ in $D(\mathsf{Sw})\text{-}\mathrm{mod}$. Clearly $\Hom_{D(\mathsf{Sw})}(\mathcal{C}_\pm,\mathcal{S}_-)=0$, whence $\Ext^\bullet_{D(\mathsf{Sw}),\mathsf{Sw}}(\mathbb{C}, \mathcal{A}) = \Ext^\bullet_{D(\mathsf{Sw}),\mathsf{Sw}}(\mathbb{C}, \mathcal{V})$ by additivity of Ext in its second variable (see e.g.\ \cite[Prop.\,3.3.4]{weibel}). It is readily seen that $\Hom_{D(\mathsf{Sw})}(\mathcal{C}_+,\mathcal{V})$ has dimension $1$ with basis $t_+ \mapsto x \cdot \varphi_1$ and that $\Hom_{D(\mathsf{Sw})}(\mathcal{C}_-,\mathcal{V}) = 0$. This gives the announced dimensions.

For the second claim, note from \eqref{defAlgCochains} that $\mathrm{C}_{\mathrm{alg}}^n(\mathsf{Sw},\mathbb{C}\langle g \rangle)$ is the centralizer of $g^{\otimes n+1}$ in $\mathsf{Sw}^{\otimes n} \otimes \mathbb{C}\langle g \rangle$. It is easy to check that $x \otimes xg \otimes 1$ fulfills this condition for $n=2$  and is a cocycle for the differential \eqref{differentialAlgComplex}. Moreover $\mathrm{C}_{\mathrm{alg}}^1(\mathsf{Sw},\mathbb{C}\langle g \rangle) = \mathrm{span}_{\mathbb{C}}\bigl\{ g^i \otimes g^j \bigr\}_{i,j \in \{0,1\}}$ so that the 2-coboundaries belong to the subspace $\mathrm{span}_{\mathbb{C}}\bigl\{ g^i \otimes g^j \otimes g^k \bigr\}_{i,j,k \in \{0,1\}}$, and thus $x \otimes xg \otimes 1$ is not a coboundary.
\end{proof}

Let $\beta$ be the natural transformation given by
\[ \beta_{V,W,M} : V \rhd W \rhd M \to (V \otimes W) \rhd M, \quad v \otimes w \otimes m \mapsto x \cdot v \otimes xg \cdot w \otimes m \]
for all $V,W \in \mathsf{Sw}\text{-}\mathrm{mod}$ and $M \in \mathrm{svect}_{\mathbb{C}}$. Then $\beta$ is a basis of $\mathrm{H}_{\mathrm{mix}}^2(\mathrm{svect}_{\mathbb{C}})$ and $\mathrm{id} + h\beta$ is an infinitesimal deformation of the trivial mixed associator. One can check that $\mathrm{id} +h\beta$ is again a mixed associator over $\mathbb{C}[h]/\langle h^{N+1} \rangle$ for any $N$, and thus $\mathrm{id} +h\beta$ can be seen as a formal deformation. In particular this deformation can be evaluated at any complex number $\lambda$ and we get in this way a 1-parameter family of module categories $\mathrm{svect}_{\mathbb{C}}^\lambda$, each of which is $\mathrm{svect}_{\mathbb{C}}$ as a linear category and whose mixed associators are $\mathrm{id}+\lambda \beta$. It is proven in \cite[Prop.\,5.21]{willprecht} that $\mathrm{svect}_{\mathbb{C}}^{\lambda_1}$ and $\mathrm{svect}_{\mathbb{C}}^{\lambda_2}$ are equivalent as modules categories if and only if $\lambda_1 = \lambda_2$.

\subsection{A continuous family of comodule algebras over \texorpdfstring{$\mathsf{Sw}$}{Sw}}
For an arbitrary $\xi \in \mathbb{C}$ consider the 4-dimensional algebra
\[ A_\xi = \mathbb{C}\bigl\langle \hat{x}, \hat{g} \,\big|\, \hat{g}\hat{x} = - \hat{x}\hat{g}, \quad \hat{x}^2 = 1, \quad \hat{g}^2 = 1 \bigr\rangle \]
equipped with $\mathsf{Sw}$-comodule algebra structure on $A_\xi$ given by
\[ \hat{x}_{(1)} \otimes \hat{x}_{(0)} = 1_{\mathsf{Sw}} \otimes \hat{x} + \xi x \otimes \hat{g}, \qquad \hat{g}_{(1)} \otimes \hat{g}_{(0)} = g \otimes \hat{g}. \]
Note that $A_\xi$ depends on $\xi$ as a $\mathsf{Sw}$-comodule but not as an algebra. 

\begin{remark}\label{remarkVectLambda}
There is an isomorphism of algebras $A_\xi \overset{\sim}{\longrightarrow} \End_{\mathbb{C}}(\mathbb{C}^2)$ given by $\hat{x} \mapsto \left( \begin{smallmatrix} 0 & 1\\ 1 & 0\end{smallmatrix} \right)$ and $\hat{g} \mapsto \left( \begin{smallmatrix} 1 & 0\\ 0 & -1\end{smallmatrix} \right)$, which yields the equivalence $\mathrm{vect}_{\mathbb{C}} \overset{\sim}{\longrightarrow} A_\xi\text{-}\mathrm{mod}$, $\mathbb{C} \mapsto \mathbb{C}^2$ {\em as linear categories}. Recall the module categories $\mathrm{vect}^\lambda_{\mathbb{C}}$ introduced \S\ref{subDeformVectSw}. It is proven in \cite[Prop.\,5.23]{willprecht} that $\mathrm{vect}_{\mathbb{C}}^{\xi^2/4} \overset{\sim}{\longrightarrow} A_\xi\text{-}\mathrm{mod}$ {\em as module categories} over $\mathsf{Sw}\text{-}\mathrm{mod}$. In particular, we have 
$A_\xi\text{-}\mathrm{mod}\cong A_{-\xi}\text{-}\mathrm{mod}$ as module categories.
\end{remark}

Let $\mathcal{A} = \mathcal{A}_{\mathsf{Sw},A_\xi} \subset \Hom_{\mathbb{C}}(\mathsf{Sw},A_\xi)$ be the adjoint object described in Prop.\,\ref{propAdjObjComodAlg}. By definition we have $\varphi(gh) = \widehat{g}\varphi(h)\widehat{g}$ for all $h \in \mathsf{Sw}$, thus $\varphi$ is entirely determined by the values $\varphi(1_{\mathsf{Sw}})$ and $\varphi(x)$. A simple computation reveals that these values have the form
\[ \varphi(1_{\mathsf{Sw}}) = \lambda_1 \hat{1} + \frac{\xi}{2}\lambda_2 \, \hat{g} + \lambda_3 \,\hat{x} + \frac{\xi}{2}\lambda_4 \,\hat{x}\hat{g}, \quad \varphi(x) = \lambda_4 \hat{1} + \lambda_2 \, \hat{x} \]
for arbitrary $\lambda_1,\ldots,\lambda_4 \in \mathbb{C}$ and where $\hat{1} = 1_{A_\xi}$, so that $\dim\mathcal{A} = 4$. Let $\varphi_1,\varphi_2 \in \mathcal{A}$ defined by
\[ \varphi_1(1_{\mathsf{Sw}}) = \frac{\xi}{2}\hat{g}, \:\: \varphi_1(x) = \hat{x} \quad \text{and} \quad \varphi_2(1_{\mathsf{Sw}}) = \frac{\xi}{2}\hat{x}\hat{g}, \:\: \varphi_2(x) = \hat{1}. \]
Using the action \eqref{DHactionAdjObj} of $x \in  D(\mathsf{Sw})$ on $\mathcal{A}$ we get the new elements $x \cdot \varphi_i$, which are given by
\[ (x \cdot \varphi_1)(1_{\mathsf{Sw}}) = \hat{x}, \:\: (x \cdot \varphi_1)(x) = 0 \quad \text{and} \quad (x \cdot \varphi_2)(1_{\mathsf{Sw}}) = \hat{1}, \:\: (x \cdot \varphi_2)(x) = 0. \]
Hence $\bigl\{ \varphi_1, \,\varphi_2,\, x \cdot \varphi_1,\, x \cdot \varphi_2 \bigr\}$ is a basis of $\mathcal{A}$. The actions of the other generators of $D(\mathsf{Sw})$ on $\varphi_i$ are obtained by straightforward computations:
\begin{align*}
&g \cdot \varphi_1 = \varphi_1, \quad \quad f \cdot \varphi_1 = -\varphi_1, \quad y \cdot \varphi_1 = 0,\\
&g \cdot \varphi_2 = -\varphi_2, \quad f \cdot \varphi_2 = -\varphi_2, \quad y \cdot \varphi_2 = -\frac{\xi^2}{2} x \cdot \varphi_2
\end{align*}
while the actions on $x \cdot \varphi_i$ are immediately deduced from the defining relations \eqref{relationsDSw} of $D(\mathsf{Sw})$; in particular $y \cdot (x \cdot \varphi_1) = 2\varphi_1$ and $y \cdot (x \cdot \varphi_2) = 0$. The structure of $\mathcal{A}$ as a $D(\mathsf{Sw})$-module is now clear:
\begin{equation}\label{decompAdjObjAxi}
\xymatrix@R=.1em@C=.3em{
&\overset{(+,-)}{\varphi_1} \ar@<-1ex>[dd]_-x & & \overset{(-,-)}{\varphi_2} \ar@<1ex>[dd]^-y \ar@<-1ex>[dd]_-x & \\
\mathcal{A} = & &\oplus\:\: & & = \mathcal{S}_+ \oplus \mathcal{V}_\xi\\
&\underset{(-,+)}{x \cdot \varphi_1} \ar@<-1ex>[uu]_-y &  & \underset{(+,+)}{x \cdot \varphi_2} &
}\end{equation}
where the couples of signs indicate the weights for the pair $(g,f)$, $\mathcal{S}_+$ is the simple module \eqref{defSimpleModSpm} and we introduce the notation $\mathcal{V}_\xi$ for the second summand.

\begin{proposition}
The deformation cohomology of the (trivial) mixed associator of $A_\xi\text{-}\mathrm{mod}$ as a module category over $\mathsf{Sw}\text{-}\mathrm{mod}$ satisfies
\[ \dim \mathrm{H}_{\mathrm{mix}}^p(A_\xi\text{-}\mathrm{mod}) = \begin{cases}
1 & \text{if } p \text{ is even,}\\
0 & \text{if } p \text{ is odd.}
\end{cases} \]
The cocycle $x \otimes xg \otimes \hat{1}$ forms a basis of $\mathrm{H}_{\mathrm{alg}}^2(\mathsf{Sw},A_\xi) \cong \mathrm{H}_{\mathrm{mix}}^2(A_\xi\text{-}\mathrm{mod})$, see \S\ref{subsecAlgComplex} for the definition of $\mathrm{H}_{\mathrm{alg}}^\bullet$.
\end{proposition}
\begin{proof}
We use Cor.\,\ref{coroComputeMixCohomComodAlg} together with the relative resolution \eqref{relProjResTrivSw} of $\mathbb{C}$ in $D(\mathsf{Sw})\text{-}\mathrm{mod}$. As for Prop.\,\ref{propDefSvectSw} we first note that $\Ext^\bullet_{D(\mathsf{Sw}), \mathsf{Sw}}(\mathbb{C},\mathcal{A}) = \Ext^\bullet_{D(\mathsf{Sw}), \mathsf{Sw}}(\mathbb{C},\mathcal{V}_\xi)$. It is readily seen from the pictures in \eqref{defModCpm} and \eqref{decompAdjObjAxi} that $\Hom_{D(\mathsf{Sw})}(\mathcal{C}_+,\mathcal{V}_\xi)$ has dimension $1$ with basis $t_+ \mapsto x \cdot \varphi_2$ and that $\Hom_{D(\mathsf{Sw})}(\mathcal{C}_-,\mathcal{V}_\xi) = 0$. This gives the announced dimensions.

For the second claim, one can check from the definitions \eqref{defAlgCochains} and \eqref{differentialAlgComplex} that $x \otimes xg \otimes \hat{1}$ is indeed a 2-cocycle in $\mathrm{C}_{\mathrm{alg}}^\bullet(\mathsf{Sw},A_\xi)$. A straightforward centralizer computation reveals that
\[ \mathrm{C}_{\mathrm{alg}}^1(\mathsf{Sw},A_\xi) = \mathrm{span}_{\mathbb{C}}\bigl( 1_{\mathsf{Sw}} \otimes \hat{1},\, x \otimes \hat{x},\, xg \otimes \hat{x},\, g \otimes \hat{1} - \xi xg \otimes \hat{x}\hat{g} \bigr). \]
From this we can compute the subspace of 2-coboundaries (which is 4-dimensional again) and see that it does not contain $x \otimes xg \otimes \hat{1}$.
\end{proof}

\begin{remark}
Let $\alpha'$ be the natural transformation given by
\[ \alpha'_{V,W,M} : V \rhd W \rhd M \to (V \otimes W) \rhd M, \quad v \otimes w \otimes m \mapsto x \cdot v \otimes xg \cdot w \otimes m \]
for all $V,W \in \mathsf{Sw}\text{-}\mathrm{mod}$ and $M \in A_\xi\text{-}\mathrm{mod}$. Then $\alpha'$ is a basis of $\mathrm{H}_{\mathrm{mix}}^2(A_\xi\text{-}\mathrm{mod})$ and $\mathrm{id} + h\alpha'$ is an infinitesimal deformation of the trivial mixed associator. One can check that $\mathrm{id} +h\alpha'$ is again a mixed associator over $\mathbb{C}[h]/\langle h^{N+1} \rangle$ for any $N$, and thus $\mathrm{id} +h\alpha'$ can be seen as a formal deformation. In particular this deformation can be evaluated at any complex number $\nu$ and we get in this way a 1-parameter family of module categories $(A_\xi\text{-}\mathrm{mod})^\nu$ which are $A_\xi\text{-}\mathrm{mod}$ equipped with the mixed associators $\mathrm{id}+\nu \alpha'$. However, it is proven in \cite[Prop.\,5.23]{willprecht} that this does not yield new examples because $(A_\xi\text{-}\mathrm{mod})^\nu \cong \mathrm{vect}_{\mathbb{C}}^{\lambda} \cong A_{2\sqrt{\lambda}}\text{-}\mathrm{mod}$ as module-categories, where $\lambda = \frac{\xi^2}{4} + \nu$ and the second equivalence uses Rmk.\,\ref{remarkVectLambda}. Hence the 1-parameter family of module categories $\mathrm{vect}^\lambda_{\mathbb{C}}$ (or equivalently $A_\xi\text{-}\mathrm{mod}$) is ``closed under deformation''. The same is true for the family of module categories $\mathrm{svect}_{\mathbb{C}}^\lambda$ defined at the end of \S\ref{subDeformSvectSw}, see \cite[\S 5.3.4]{willprecht}.
\end{remark}

\begin{remark}
Let $D = \mathrm{span}_{\mathbb{C}}(1,gx) \subset \mathsf{Sw}$ be the left coideal subalgebra generated by $gx$, \textit{i.e.} $\Delta(D) \subset \mathsf{Sw} \otimes D$. Recall that the mixed associator cohomology of $\mathcal{C}=\mathsf{Sw}\text{-}\mathrm{mod}$ vanishes by Prop.\,\ref{propRigidityReg}. Similarly to Lemma~\ref{lemmaDualEndCoend}, one can show that the adjoint algebra for  $D\text{-}\mathrm{mod}$ is relatively projective, and therefore $ \mathrm{H}_{\mathrm{mix}}^{> 0}(D\text{-}\mathrm{mod})$ vanishes too, see also \cite[Prop.\,5.13]{willprecht}. Thanks to the tangent space interpretation of the 2nd cohomology space in Prop.\,\ref{relDefAssoAndCohom} we deduce that these two module categories cannot have non-trivial continuous families of deformations in contrast to $\mathrm{vect}_{\mathbb{C}} = \mathbb{C}\langle 1_{\mathsf{Sw}} \rangle\text{-}\mathrm{mod}$ and $\mathrm{svect}_{\mathbb{C}} = \mathbb{C}\langle g \rangle\text{-}\mathrm{mod}$.

It is interesting to note that the families of pairwise inequivalent module categories $\mathrm{vect}^{\lambda}_{\mathbb{C}}$ and $\mathrm{svect}^{\lambda}_{\mathbb{C}}$, from \S\ref{subDeformVectSw} and \S\ref{subDeformSvectSw}, together with $\mathsf{Sw}\text{-}\mathrm{mod}$ and $D\text{-}\mathrm{mod}$ form the complete list of indecomposable {\em exact} module categories over $\mathcal{C}$ \cite[\S 8]{Mombelli}, recall Remark \ref{remarkVectLambda}. We thus see in this example that all exact $\mathcal{C}$-modules appear as deformations of 4 basic ones corresponding to the 4 coideal subalgebras in $\mathsf{Sw}$. It would be interesting to see how far this goes in general.
\end{remark}

\subsection{Discrete family of comodule algebras over \texorpdfstring{$\mathsf{Sw}$}{Sweedler}}

For any pair of positive integers $m,n \geq 1$ consider the algebra
\[ A_{m,n} = \mathbb{C}\bigl\langle \tilde{g}, \tilde{x} \, \big|\, \tilde{g}\tilde{x} = -\tilde{x}\tilde{g}, \: \tilde{g}^{2m} = 1, \: \tilde{x}^{2n} = 0 \bigr\rangle \]
which has dimension $4mn$, with basis $\tilde{g}^i\tilde{x}^j$, $0 \leq i \leq 2m-1$, $0 \leq j \leq 2n-1$. It is easily checked that $A_{m,n}$ is a $\mathsf{Sw}$-comodule algebra by means of the coaction $A_{m,n} \to \mathsf{Sw} \otimes A_{m,n}$ defined by
\begin{equation*}
\tilde{g}_{(1)} \otimes \tilde{g}_{(0)} = g \otimes \tilde{g}, \qquad \tilde{x}_{(1)} \otimes \tilde{x}_{(0)} = 1 \otimes \tilde{x} + x \otimes \tilde{g}.
\end{equation*}
These values extend to
\begin{equation}\label{defCoactSw}
(\tilde{g}^i\tilde{x}^j)_{(1)} \otimes (\tilde{g}^i\tilde{x}^j)_{(0)} = \begin{cases}
g^i \otimes \tilde{g}^i \tilde{x}^j & \text{if } j \text{ is even,}\\
g^i \otimes \tilde{g}^i \tilde{x}^j + g^ix \otimes \tilde{g}^{i+1} \tilde{x}^{j-1}& \text{if } j \text{ is odd.}
\end{cases} \end{equation}
It is readily seen that $A_{m,n} \cong \mathsf{Sw}^{\oplus \, mn}$ as a comodule, thanks to the decomposition
\[ A_{m,n} = \bigoplus_{i=0}^{m-1}\bigoplus_{j=0}^{n-1} \mathrm{span}_{\mathbb{C}}\bigl( 1_{A_{m,n}}, \tilde{g}, \tilde{x}, \tilde{g}\tilde{x} \bigr)\tilde{g}^{2i}\tilde{x}^{2j} \] and because the central elements $\tilde{g}^{2i}\tilde{x}^{2j}$ are coinvariant. Despite this deceptively simple comodule structure, we will see that the category $A_{m,n}\text{-}\mathrm{mod}$ has non-trivial cohomology as a module category over the monoidal category $\mathsf{Sw}\text{-}\mathrm{mod}$.

\begin{remark}
Let us show that $A_{m,n}\text{-}\mathrm{mod}$ is not an exact module category (in the sense of \cite[Def.\,7.5.1]{EGNO}) over $\mathsf{Sw}\text{-}\mathrm{mod}$ as far as $n > 1$. Let $q = e^{\mathbf{i}\pi/m} \in \mathbb{C}$, so that $q^{2m} = 1$. For all $1 \leq d \leq 2n$ and $0 \leq k \leq 2m-1$ there is an indecomposable $A_{m,n}$-module $V_{d,k}$ with basis $(v_1, \ldots, v_d)$ such that
\[ \tilde{x} \cdot v_i = v_{i+1}, \quad \tilde{x} \cdot v_{d} = 0, \qquad \tilde{g} \cdot v_i = (-1)^{i+1}q^kv_i. \]
Note that $V_{d,k}$ and $V_{d',k'}$ are isomorphic if and only if $(d,k) = (d',k')$. Any finite-dimensional $A_{m,n}$-module is isomorphic to a direct sum of some $V_{d,k}$'s.\footnote{The proof of this fact is an adaptation of the classification theorem of nilpotent endomorphisms.} Define $\Phi_k = \sum_{j=0}^{2m-1} q^{-jk} \tilde{g}^j \in A_{m,n}$, such that $\tilde{g}\Phi_k = q^k\Phi_k$. Then the regular module decomposes as
\[ \textstyle A_{m,n} = \bigoplus_{k=0}^{2m-1} \mathrm{span}_{\mathbb{C}}\bigl\{ \tilde{x}^i\Phi_k \bigr\}_{0 \leq i \leq 2n-1} \cong \bigoplus_{k=0}^{2m-1} V_{2n,k} \]
proving that the indecomposable projective $A_{m,n}$-modules are the $V_{2n,k}$'s. Now note that $\mathsf{Sw}$ has two 2-dimensional projective modules $P_{\pm}$, hence $P_{\pm} \rhd V_{1,k} \in A_{m,n}\text{-}\mathrm{mod}$ is 2-dimensional as well and is not projective if $n > 1$. This shows that the module category $A_{m,n}\text{-}\mathrm{mod}$ is not exact, in contrast with the examples studied in the previous subsections.
\end{remark}

Let us describe the adjoint object $\mathcal{A}_{\mathsf{Sw},A_{m,n}} \in D(\mathsf{Sw})\text{-}\mathrm{mod}$. Recall from Prop.\,\ref{propAdjObjComodAlg} that $\mathcal{A}_{\mathsf{Sw},A_{m,n}}$ is a subspace of $\Hom_{\mathbb{C}}(\mathsf{Sw},A_{m,n})$.
\begin{lemma}\label{lemmaAdjObjSw}
1. There is an isomorphism of vector spaces $I : \mathcal{A}_{\mathsf{Sw},A_{m,n}} \overset{\sim}{\to} A_{m,n}$, $\varphi \mapsto \varphi(1_{\mathsf{Sw}})$.
\\2. Through this identification we get a $D(\mathsf{Sw})$-action on $A_{m,n}$ which is given by
\[ g \cdot a = \tilde{g} a \tilde{g}^{-1}, \quad x \cdot a = [\tilde{x},a]\tilde{g}^{-1}, \quad f \cdot a = f(a_{(1)})a_{(0)}, \quad y \cdot a = y(a_{(1)})a_{(0)}, \]
for all $a \in A_{m,n}$, where $[\text{-},\text{-}]$ is the commutator and $a_{(1)} \otimes a_{(0)}$ denotes the coaction \eqref{defCoactSw}.
\end{lemma}
\begin{proof}
1. For $a \in A_{m,n}$ let $\varphi_a \in \Hom_{\mathbb{C}}(\mathsf{Sw},A_{m,n})$ be defined by
\[ \varphi_a(1_{\mathsf{Sw}}) = a, \quad \varphi_a(g) = \tilde{g}a\tilde{g}^{-1}, \quad \varphi_a(x) = [\tilde{x},a]\tilde{g}^{-1}, \quad \varphi_a(gx) = \tilde{g}[\tilde{x},a]\tilde{g}^{-2}. \]
It is straightforward to check that $\varphi_a \in \mathcal{A}_{\mathsf{Sw},A_{m,n}}$. Thus we have a linear map $J : A_{m,n} \to \mathcal{A}_{\mathsf{Sw},A_{m,n}}$, $a \mapsto \varphi_a$. We claim that $J$ is the inverse of $I$. On the one hand if $\varphi \in \mathcal{A}_{\mathsf{Sw},A_{m,n}}$ then by definition $a\varphi(h) = \varphi(a_{(1)}h)a_{(0)}$ for all $h \in \mathsf{Sw}$ and $a \in A_{m,n}$. Hence $\tilde{g}\varphi(1_{\mathsf{Sw}}) = \varphi(g)\tilde{g}$, so that $\varphi(g) = \tilde{g}\varphi(1_{\mathsf{Sw}})\tilde{g}^{-1}$. Also, $\tilde{x}\varphi(1_{\mathsf{Sw}}) = \varphi(1_{\mathsf{Sw}})\tilde{x} + \varphi(x)\tilde{g}$, which gives $\varphi(x) = [\tilde{x},\varphi(1_{\mathsf{Sw}})]\tilde{g}^{-1}$. Finally, $\varphi(gx) = \varphi(gx)\tilde{g}\tilde{g}^{-1} = \tilde{g}\varphi(x)\tilde{g}^{-1} = \tilde{g}[\tilde{x},\varphi(1_{\mathsf{Sw}})]\tilde{g}^{-2}$. As a result $\varphi = J\bigl( \varphi(1_{\mathsf{Sw}}) \bigr)$. On the other hand it is obvious that $I(\varphi_a) = a$.
\\2. Let $a \in A_{m,n}$ and $\varphi_a$ as defined above. By \eqref{DHactionAdjObj} we have
\begin{align*}
(g \cdot \varphi_a)(1_{\mathsf{Sw}}) &= \varphi_a(g) = \tilde{g}a\tilde{g}^{-1} = \varphi_{\tilde{g}a\tilde{g}^{-1}}(1_{\mathsf{Sw}}),\\
(x \cdot \varphi_a)(1_{\mathsf{Sw}}) &= \varphi_a(x) = [\tilde{x},a]\tilde{g}^{-1} = \varphi_{[\tilde{x},a]\tilde{g}^{-1}}(1_{\mathsf{Sw}}),\\
(\alpha \cdot \varphi_a)(1_{\mathsf{Sw}}) &= \alpha\bigl( \varphi_a(1_{\mathsf{Sw}})_{(1)} \bigr) \varphi_a(1_{\mathsf{Sw}})_{(0)} = \alpha(a_{(1)})a_{(0)} = \varphi_{\alpha(a_{(1)})a_{(0)}}(1_{\mathsf{Sw}})
\end{align*}
for all $\alpha \in (\mathsf{Sw}^*)^{\mathrm{op}}$. This proves the claim because we have seen that an element in $\mathcal{A}_{\mathsf{Sw},A_{m,n}}$ is characterized by its value on $1_{\mathsf{Sw}}$.
\end{proof}

Recall the simple $D(\mathsf{Sw})$-module $\mathcal{S}_+$ in \eqref{defSimpleModSpm} and let $\mathcal{W}_n$ be the indecomposable $D(\mathsf{Sw})$-module with basis $\bigl( v_k, w_k \bigr)_{0 \leq k \leq n-1}$ such that
\[ \begin{array}{lllll}
g \cdot v_k = -v_k, & f \cdot v_k = -v_k,& x \cdot v_k = w_{k+1} \text{ and }x \cdot v_{n-1} = 0, & y \cdot v_k = w_k,\\
g \cdot w_k = w_k, & f \cdot w_k = w_k, & x \cdot w_k = 0, &  y \cdot w_k = 0.
\end{array} \]
Schematically, $\mathcal{W}_n$ is depicted as
\[
\xymatrix@C=.5em@R=1em{
& \overset{(-,-)}{v_0} \ar[dl]_y \ar[rd]^x & & \overset{(-,-)}{v_1} \ar[dl]_y \ar[dr]^x && \ldots\quad \ar[dr]^x && \overset{(-,-)}{v_{n-1}} \ar[dl]_y\\
\underset{(+,+)}{w_0} & & \underset{(+,+)}{w_1} & & \quad\ldots & & \underset{(+,+)}{w_{n-1}} &
}
\]
where the superscripts/subscripts indicate the weights for the pair $(g, f)$.

\begin{lemma}\label{propDecAdjObjSw}
With these notations, we have an isomorphism of $D(\mathsf{Sw})$-modules
\[ \mathcal{A}_{\mathsf{Sw},A_{m,n}} \cong mn \,\mathcal{S}_+ \oplus m\,\mathcal{W}_n \]
where the factors are the multiplicities.
\end{lemma}
\begin{proof}
Let us denote $|i,j\rangle = \tilde{g}^i\tilde{x}^j$ for the monomial basis elements of $A_{m,n}$, with $0 \leq i \leq 2m-1$ and $0 \leq j \leq 2n-1$. Easy computations using Lemma \ref{lemmaAdjObjSw}(2) and \eqref{defCoactSw} reveal that 
\begin{align*}
&g \cdot |i,j\rangle = (-1)^j |i,j\rangle, \qquad x \cdot |i,j\rangle = \begin{cases}
0 & \text{if } i \text{ is even}\\
2(-1)^j|i-1,j+1\rangle & \text{if } i \text{ is odd}
\end{cases}\\
&f \cdot |i,j\rangle = (-1)^i|i,j\rangle, \qquad y \cdot |i,j\rangle = \begin{cases}
0 & \text{if } j \text{ is even}\\
|i+1,j-1\rangle & \text{if } j \text{ is odd}
\end{cases}
\end{align*}
with the convention that the index $i$ is seen modulo $2m$ and $|i,2n\rangle = 0$ for all $i$. The decomposition of $\mathcal{A}_{\mathsf{Sw},A_{m,n}}$ is obtained as follows:
\begin{itemize}[topsep=.2em,itemsep=0em]
\item For $i$ even and $j$ odd, set $v = |i,j\rangle$ and $w=|i+1,j-1 \rangle$. Then $\mathrm{span}_{\mathbb{C}}\{v,w\} \cong \mathcal{S}_+$. Moreover, there are $mn$ such pairs of indices $(i,j)$, whence the summand $mn \, \mathcal{S}_+$. Note that this summand contains all vectors $|a,b \rangle$ where $a$ and $b$ have different parity.
\item For $i$ even, set $w_k = 2^k|i - 2k, 2k \rangle$ and $v_k = 2^k|i - 2k-1, 2k + 1 \rangle$ for $0 \leq k \leq n-1$. Then $\mathrm{span}_{\mathbb{C}}\bigl\{ v_k,w_k \bigr\}_{0 \leq k \leq n-1} \cong \mathcal{W}_n$. Moreover, there are $m$ such indices $i$, whence the summand $m \, \mathcal{W}_n$. Note that this summand contains all vectors $|a,b \rangle$ where $a$ and $b$ have the same parity.
\end{itemize}
\end{proof}

\begin{proposition}\label{propCohomAmn}
1. The deformation cohomology of the (trivial) mixed associator of $A_{m,n}\text{-}\mathrm{mod}$ as a module category over $\mathsf{Sw}\text{-}\mathrm{mod}$ satisfies
\[ \dim \mathrm{H}_{\mathrm{mix}}^p(A_{m,n}\text{-}\mathrm{mod}) = \begin{cases}
mn & \text{if } p = 0,\\
m(n-1) & \text{if } p > 0 \text{ is even,}\\
0 & \text{if } p \text{ is odd.}
\end{cases} \]
2. The cocycles $x \otimes gx \otimes \tilde{g}^{2i}\tilde{x}^{2j}$ with $0 \leq i \leq m-1$ and $0 \leq j \leq n-2$ form a basis of $\mathrm{H}_{\mathrm{alg}}^2(\mathsf{Sw},A_{m,n}) \cong \mathrm{H}_{\mathrm{mix}}^2(A_{m,n}\text{-}\mathrm{mod})$, see \S\ref{subsecAlgComplex} for the definition of $\mathrm{H}_{\mathrm{alg}}^\bullet$.
\end{proposition}
\begin{proof}
1. We use Corollary \ref{coroComputeMixCohomComodAlg} with the help of the relatively projective resolution of $\mathbb{C} \in D(\mathsf{Sw})\text{-}\mathrm{mod}$ in \eqref{relProjResTrivSw}. Note that $\Hom_{D(\mathsf{Sw})}(\mathcal{C}_\pm,\mathcal{S}_+) = 0$, so that $\Ext^\bullet_{D(\mathsf{Sw}),\mathsf{Sw}}(\mathbb{C},\mathcal{S}_+) = 0$. Hence Prop.\,\ref{propDecAdjObjSw} and additivity of Ext yield $\Ext^\bullet_{D(\mathsf{Sw}),\mathsf{Sw}}(\mathbb{C},\mathcal{A}_{\mathsf{Sw},A_{m,n}}) = m\Ext^\bullet_{D(\mathsf{Sw}),\mathsf{Sw}}(\mathbb{C}, \mathcal{W}_n)$. We are thus left to compute the cohomology of the complex
\[ 0 \longrightarrow \Hom_{D(\mathsf{Sw})}(\mathcal{C}_+, \mathcal{W}_n) \xrightarrow{\:(d_-)^*\:} \Hom_{D(\mathsf{Sw})}(\mathcal{C}_-, \mathcal{W}_n) \xrightarrow{\:(d_+)^*\:} \Hom_{D(\mathsf{Sw})}(\mathcal{C}_+, \mathcal{W}_n) \xrightarrow{\:(d_-)^*\:} \ldots \]
Using the basis $(t_\pm,b_\pm)$ of $\mathcal{C}_\pm$ described in \eqref{defModCpm}, let $\pi_k : \mathcal{C}_+ \to \mathcal{W}_n$ defined by $\pi_k(t_+) = w_k$ and $\pi_k(b_+) = 0$ for all $0 \leq k \leq n-1$. Also let $\sigma : \mathcal{C}_- \to \mathcal{W}_n$ defined by $\sigma(t_-) = v_{n-1}$ and $\sigma(b_-) = w_{n-1}$. Then $\Hom_{D(\mathsf{Sw})}(\mathcal{C}_+, \mathcal{W}_n) = \mathrm{span}_{\mathbb{C}}\{ \pi_k\}_{0 \leq k \leq n-1}$ and $\Hom_{D(\mathsf{Sw})}(\mathcal{C}_-, \mathcal{W}_n) = \mathbb{C}\sigma$. Moreover $(d_-)^*(\pi_k) = 0$ and $(d_+)^*(\sigma) = \pi_{n-1}$. Hence we have $\ker\bigl( (d_-)^* \bigr) \cong \mathbb{C}^n$ and $\ker\bigl( (d_+)^* \bigr) \cong \{ 0 \}$, and this implies $\Ext^0_{D(\mathsf{Sw}),\mathsf{Sw}}(\mathbb{C}, \mathcal{W}_n) \cong \mathbb{C}^n$, and $\Ext^{\text{even}>0}_{D(\mathsf{Sw}),\mathsf{Sw}}(\mathbb{C}, \mathcal{W}_n) \cong \mathbb{C}^n/\mathbb{C} \cong \mathbb{C}^{n-1}$ and that $\Ext^{\text{odd}}_{D(\mathsf{Sw}),\mathsf{Sw}}(\mathbb{C}, \mathcal{W}_n) =0$. The announced dimensions then follow.

2. It is straightforward to check that these elements are indeed 2-cocycles, \textit{i.e.} they belong to the 2-cochain space defined in \eqref{defAlgCochains} and they are annihilated by the differential \eqref{differentialAlgComplex}. By item~1.\ of the present proposition, the proposed family has the correct cardinal for a basis. It thus suffices to check that these cocycles remain linearly independent in the cohomology space. Let us thus analyze the subspace of 2-coboundaries.

By definition in \eqref{defAlgCochains}, $\mathrm{C}^1_{\mathrm{alg}}(\mathsf{Sw},A_{m,n})$ consists of those elements in $\mathsf{Sw} \otimes A_{m,n}$ commuting with $\Delta_{A_{m,n}}(\tilde{g}) = g \otimes \tilde{g}$ and $\Delta_{A_{m,n}}(\tilde{x}) = 1 \otimes \tilde{x} + x \otimes \tilde{g}$. Since $\tilde{g}^{2i}\tilde{x}^{2j}$ is central for all $i,j$, it is helpful to write an element in $A_{m,n}$ as a linear combination of $\tilde{g}^s \tilde{x}^t\bigl(\tilde{g}^{2i} \tilde{x}^{2j}\bigr)$ with $0 \leq s,t \leq 1$, $0 \leq i \leq m-1$ and $0 \leq j \leq n-1$. The centralizer of $g \otimes \tilde{g}$ is spanned by $g^r x^t \otimes \tilde{g}^s \tilde{x}^t \bigl(\tilde{g}^{2i} \tilde{x}^{2j} \bigr)$ with $0 \leq r,s,t \leq 1$ and $i,j$ as above. Looking at the effect of the commutator $\bigl[-,\Delta_{A_{m,n}}(\tilde{x})\bigr]$ on a linear combination of such elements, we find after computation that a basis of $\mathrm{C}^1_{\mathrm{alg}}(\mathsf{Sw},A_{m,n})$ consists of
\begin{equation}\label{basis1CochainsAmn}
\left.\begin{array}{l}
1 \otimes \tilde{g}^{2i} \tilde{x}^{2j}, \quad x \otimes \tilde{x}\bigl(\tilde{g}^{2i} \tilde{x}^{2j}\bigr), \quad gx \otimes \tilde{x}\bigl(\tilde{g}^{2i} \tilde{x}^{2j}\bigr) \qquad \text{\footnotesize for all $0 \leq i \leq m-1$ and $0 \leq j \leq n-1$}\\[.3em]
x \otimes \tilde{g}\tilde{x}\bigl(\tilde{g}^{2i} \tilde{x}^{2(n-1)}\bigr), \quad gx \otimes \tilde{g}\tilde{x}\bigl(\tilde{g}^{2i} \tilde{x}^{2(n-1)}\bigr) \qquad \text{\footnotesize for all $0 \leq i \leq m-1$}\\[.3em]
g \otimes \tilde{g}^{2i} \tilde{x}^{2j} - gx \otimes \tilde{g}\tilde{x}\bigl(\tilde{g}^{2i} \tilde{x}^{2(j-1)}\bigr) \qquad \text{\footnotesize for all $0 \leq i \leq m-1$ and $1 \leq j \leq n-1$}
\end{array} \right\}
\end{equation}
so that $\dim \mathrm{C}^1_{\mathrm{alg}}(\mathsf{Sw},A_{m,n}) = m(4n+1)$. Since the elements $\tilde{g}^{2i} \tilde{x}^{2j}$ are $\mathsf{Sw}$-coinvariant, the differential \eqref{differentialAlgComplex} satisfies $\delta^1_{\mathrm{alg}}\bigl(h \otimes a (\tilde{g}^{2i} \tilde{x}^{2j})\bigr) = \delta^1_{\mathrm{alg}}(h \otimes a) \bigl( 1 \otimes \tilde{g}^{2i} \tilde{x}^{2j} \bigr)$. From this, one checks easily that the value of $\delta_{\mathrm{alg}}^1$ on any of the 1-cochains in \eqref{basis1CochainsAmn} is a linear combination of elements of the form $g^rx^s \otimes g^tx^u \otimes a$ with $g^rx^s \otimes g^tx^u \neq x \otimes gx$ and $a\in A_{m,n}$. The only exception is  for
$x \otimes \tilde{g}\tilde{x}\bigl(\tilde{g}^{2i} \tilde{x}^{2(n-1)}\bigr)$, on which $\delta_{\mathrm{alg}}^1$ has the value $x \otimes xg \otimes \tilde{g}^{2(i+1)}\tilde{x}^{2(n-1)}$. This proves that the proposed 2-cocycles are linearly independent in $\mathrm{H}_{\mathrm{alg}}^2(\mathsf{Sw},A_{m,n})$.
\end{proof}

We now promote the cocycles obtained in Prop.\,\ref{propCohomAmn}(2) to finite deformations of the trivial associator of $A_{m,n}\text{-}\mathrm{mod}$. For any $0 \leq s < m$ and $0 \leq t < n-1$ consider the natural transformation
\[ \alpha^{s,t}_{V,W,M} : V \rhd W \rhd M \to (V \otimes W) \rhd M, \quad v \otimes w \otimes m \mapsto x \cdot v \otimes gx \cdot w \otimes \tilde{g}^{2s}\tilde{x}^{2t} \cdot m \]
for all $V,W \in \mathsf{Sw}\text{-}\mathrm{mod}$ and $M \in A_{m,n}\text{-}\mathrm{mod}$. Then $\alpha^{s,t}$ is a 2-cocycle in $\mathrm{C}^\bullet_{\mathrm{mix}}(A_{m,n}\text{-}\mathrm{mod})$ by Prop.\,\ref{propCohomAmn}(2) and Prop.\,\ref{propAlgComplex}. It follows that $\mathrm{id} + h\sum_{s,t} \lambda_{s,t} \alpha^{s,t}$ is an infinitesimal deformation of the trivial mixed associator for any scalars $\lambda_{s,t} \in \mathbb{C}$. Using formula \eqref{obstForAlgComplex} for the obstruction \eqref{defObsMixedAsso} it is straightforward to check that $\mathrm{obs}\Bigl( \sum_{s,t} \lambda_{s,t} \alpha^{s,t} \Bigr) = 0$. Moreover we have $\mathrm{obs}(-,0,\ldots,0) = 0$ for any cochain in the first place. As a result $\mathrm{id} + h\sum_{s,t} \lambda_{s,t} \alpha^{s,t}$ is a deformation over $\mathbb{C}[h]/\langle h^{N+1} \rangle$ for any $N \geq 1$ by Prop.\,\ref{propExtensionMonStruct}(1), and thus it can be seen as a formal deformation. In particular we can specialize $h$ at any complex number $\mu$, which gives a mixed associator $\mathrm{id} + \sum_{s,t} \lambda'_{s,t} \alpha^{s,t}$ on $A_{m,n}\text{-}\mathrm{mod}$, where $\lambda'_{s,t} = \mu\lambda_{s,t}$. These form a continuous family of mixed associators, indexed by $\mathbb{C}^{m(n-1)}$. 
It would be interesting to find a family of comodule algebras realizing these module categories, which is possible by \cite{AM}. 
Another interesting problem would be to see which $\mathcal{C}$-modules in these deformed continuous families are quasi-Frobenius in the sense of~\cite{shimizuFrob}; this is a family closed under relative Deligne product~\cite{GM}, and it thus provides interesting examples of (in this case, symmetric) monoidal 2-categories, extending the one calculated in~\cite[\S8.4]{GL} based on a subclass of quasi-Frobenius formed by perfect module categories over $\mathsf{Sw}\text{-}\mathrm{mod}$.

\appendix

\section{Products on DY cochain complexes}\label{appHigherOrderDef}
The goal of this appendix is to explain the algebraic structures on the DY complex which govern the obstructions associated to deformations of half-braidings and of monoidal structures. We also describe the corresponding products for the mixed associator complex.

We denote by $\Gamma : \mathcal{C} \to \mathcal{D}$ a $\Bbbk$-linear monoidal functor; the definition of DY cochain complexes $\mathrm{C}^\bullet_{\mathrm{DY}}(\Gamma; \mathsf{V},\mathsf{W})$ was recalled in \S\ref{subsubDYCoeff}.

\subsection{Comp-algebra structure}\label{sectionCompAlgDY}
It was noted by Yetter \cite[\S 3]{yetter1} that the cochain complex $\mathrm{C}^\bullet_{\mathrm{DY}}(\Gamma) = \mathrm{C}^\bullet_{\mathrm{DY}}(\Gamma; \boldsymbol{1}, \boldsymbol{1})$ carries a structure of comp-algebra which is completely similar to that on the Hochschild cochain complex of an associative algebra \cite{gerst}. This directly implies that obstructions are cocycles \cite[Prop.\,3.4]{yetter1}, a fact which is used in the proof of Prop.\,\ref{propExtensionMonStruct}. For convenience we review this structure in detail here.

\begin{definition}\label{defCompiYetter} Let $f \in \mathrm{C}^m_{\mathrm{DY}}(\Gamma)$, $g \in \mathrm{C}^n_{\mathrm{DY}}(\Gamma)$ and $1 \leq i \leq m$. For all objects $X_j,Y_k \in \mathcal{C}$ define $(f \:\comp_i\: g)_{X_1,\ldots,X_{i-1},Y_1,\ldots,Y_n,X_i,\ldots,X_{m-1}}$ to be the following composition
\begin{align*}
&\Gamma(X_1) \,\ldots\, \Gamma(X_{i-1})\,\Gamma(Y_1) \, \ldots \, \Gamma(Y_n) \, \Gamma(X_i) \,\ldots\, \Gamma(X_{m-1})\\
&\xrightarrow{\mathrm{id} \,\otimes\, g_{Y_1,\ldots,Y_n} \,\otimes\, \mathrm{id}} \Gamma(X_1) \,\ldots\, \Gamma(X_{i-1}) \, \Gamma(Y_1 \, \ldots \, Y_n) \, \Gamma(X_i) \,\ldots\, \Gamma(X_{m-1})\\
&\xrightarrow{f_{X_1,\ldots,X_{i-1},Y_1 \otimes \ldots \otimes Y_n, X_i, \ldots, X_{m-1}}} \Gamma(X_1 \,\ldots\, X_{i-1} \, Y_1 \,\ldots\,Y_n \, X_i \,\ldots\, X_{m-1})
\end{align*}
where the symbol $\otimes$ is omitted between objects. This gives the components of the natural transformation $f \:\comp_i\: g \in \mathrm{C}^{m+n-1}_{\mathrm{DY}}(\Gamma)$.
\end{definition}
\noindent These operations satisfy the axioms of a ``pre--Lie system'' in the sense of \cite[\S 5]{gerst}.\footnote{But with a shift in degree, see \cite[p.\,279]{gerst}, so our convention is that of \cite[\S 4]{GS}.} Moreover we note that the monoidal structure $\Gamma^{(2)} : \Gamma(-) \otimes \Gamma(-) \Rightarrow \Gamma(- \otimes -)$ can be seen as a 2-cochain, and its defining property \eqref{defMonStruct} can be rewritten as $\Gamma^{(2)} \:\comp_1\: \Gamma^{(2)} = \Gamma^{(2)} \:\comp_2\: \Gamma^{(2)}$ which is the definition of a {\em distinguished element} \cite[\S 4]{GS}.

As a result the operations $\comp_i$ together with the element $\Gamma^{(2)} \in \mathrm{C}^2_{\mathrm{DY}}(\Gamma)$ equip $\mathrm{C}^\bullet_{\mathrm{DY}}(\Gamma)$ with a {\em comp-algebra structure}; we refer to \cite[\S 4]{GS} for an overview of the general theory of comp-algebras. The comp-product of $f \in \mathrm{C}^m_{\mathrm{DY}}(\Gamma)$ with $g \in \mathrm{C}^n_{\mathrm{DY}}(\Gamma)$ is defined to be
\begin{equation}\label{defCompProd}
f \:\comp\: g = \sum_{i=1}^m (-1)^{(n-1)(i-1)} f \:\comp_i\: g \in \mathrm{C}^{m+n-1}_{\mathrm{DY}}(\Gamma) \,\,.
\end{equation}
More importantly, their {\em Gerstenhaber bracket} is then defined as
\begin{equation}\label{defGerstBracket}
[f,g] = f \:\comp\: g -(-1)^{(m-1)(n-1)} g \:\comp\: f.
\end{equation}
It endows $\mathrm{C}^\bullet_{\mathrm{DY}}(\Gamma)$ with the structure of a graded Lie superalgebra:
\begin{equation}\label{grJacobiComp}
[f,g] = -(-1)^{(m-1)(n-1)}[g,f], \qquad \bigl[ f, [g,h] \bigr] = \bigl[ [f,g],h \bigr] + (-1)^{(m-1)(n-1)} \bigl[ g, [f,h] \bigr].
\end{equation}
Using moreover the distinguished element $\Gamma^{(2)} \in \mathrm{C}^2_{\mathrm{DY}}(\Gamma)$, define
\begin{equation*}
\forall \, f \in \mathrm{C}^m_{\mathrm{DY}}(\Gamma), \quad \delta'(f) = -\bigl[ f, \Gamma^{(2)} \bigr] \in \mathrm{C}^{m+1}_{\mathrm{DY}}(\Gamma).
\end{equation*}
By the general theory of comp-algebras this is a differential, \textit{i.e.} $\delta' \circ \delta' = 0$; this follows from the fact that $\Gamma^{(2)} \:\comp\: \Gamma^{(2)} = 0$, because $\Gamma^{(2)}$ is a distinguished element.

\begin{proposition}
$\delta'$ equals the DY differential $\delta$ defined in \eqref{diffDYgeneral} (with $\mathsf{V} = \mathsf{W} = \boldsymbol{1}$ there). It follows that
\begin{equation}\label{DiffOnGerstBracket}
\delta\bigl( [f,g] \bigr) = (-1)^n[\delta(f),g] + [f,\delta(g)]
\end{equation}
for all $f \in \mathrm{C}^m_{\mathrm{DY}}(\Gamma)$ and $g \in \mathrm{C}^n_{\mathrm{DY}}(\Gamma)$.
\end{proposition}
\begin{proof}
Let $\partial_i : \mathrm{C}^m_{\mathrm{DY}}(\Gamma) \to \mathrm{C}^{m+1}_{\mathrm{DY}}(\Gamma)$ be the partial DY differentials defined in \S\,\ref{subsubDYCoeff}. A simple comparison of formulas reveals that $f \:\comp_i\: \Gamma^{(2)} = \partial_i(f)$ for $1 \leq i \leq m$ and $\Gamma^{(2)} \:\comp_1\: f = \partial_{m+1}(f)$, $\Gamma^{(2)} \:\comp_2\: f = \partial_0(f)$. Hence
\begin{align*}
-[f,\Gamma^{(2)}] &= -\biggl(\sum_{i=1}^m (-1)^{i-1} f \:\comp_i\: \Gamma^{(2)} \biggr) + (-1)^{m-1}\bigl( \Gamma_{(2)} \:\comp_1\:f + (-1)^{m-1} \Gamma^{(2)} \:\comp_2\:f \bigr)\\
&= \biggl(\sum_{i=1}^m (-1)^i \partial_i(f) \biggr) + (-1)^{m+1} \partial_{m+1}(f) + \partial_0(f) = \delta(f).
\end{align*}
The second claim follows immediately from the properties \eqref{grJacobiComp} of the bracket applied to the definition of $\delta'$.
\end{proof}

\begin{remark}
In \cite{BBK} a ``weak comp-algebra'' structure is obtained on $\mathrm{C}^\bullet_{\mathrm{DY}}(\Gamma; C, \boldsymbol{1})$, where $C$ is a coalgebra object in $\mathcal{Z}(\Gamma)$, as well as two different cup products on such complexes. When $C = \boldsymbol{1}$ all their operations reduce to the definitions above.
\end{remark}

Let $f_1, \ldots, f_k \in \mathrm{C}^2_{\mathrm{DY}}(\Gamma)$. By definition of $\comp$ in \eqref{defCompProd} we have $f_i \:\comp\: f_j = f_i \:\comp_1\: f_j - f_i \:\comp_2\: f_j$ and Def.\,\ref{defCompiYetter} gives in this case
\[ (f_i \:\comp\: f_j)_{X,Y,Z} = (f_i)_{X \otimes Y,Z} \circ \bigl( (f_j)_{X,Y} \otimes \mathrm{id}_{\Gamma(Z)} \bigr) - (f_i)_{X, Y \otimes Z} \circ \bigl( \mathrm{id}_{\Gamma(X)} \otimes (f_j)_{Y,Z} \bigr) \]
for all $X,Y,Z \in \mathcal{C}$. It follows that the obstruction cochain $\mathrm{obs}(f_1,\ldots,f_k) \in \mathrm{C}^3_{\mathrm{DY}}(\Gamma)$ defined in \eqref{defObsMonStruct} can be rewritten as
\begin{equation}\label{defObsMonStructGerst}
\mathrm{obs}(f_1,\ldots,f_k) = \sum_{i+j=k+1} f_i \:\comp\: f_j = \frac{1}{2}\sum_{i+j=k+1} [f_i, f_j ]
\end{equation}
where the second equality uses symmetry of the sum and the fact that $[f_i,f_j] = f_i \:\comp\: f_j + f_j \:\comp\: f_i$ according to \eqref{defGerstBracket}.

\begin{lemma}\label{lemmaObstructionMonStruct} If a family of cochains  $f_1, \ldots, f_N \in \mathrm{C}^2_{\mathrm{DY}}(\Gamma)$ satisfies
\[ \delta(f_1) = 0 \quad \text{and} \quad \forall \, 2 \leq i \leq N, \:\: \delta(f_i) = \mathrm{obs}(f_1,\ldots,f_{i-1}) \]
then $\mathrm{obs}(f_1,\ldots,f_N)$ is a cocycle.
\end{lemma}
\begin{proof}
This is the same computation as in Hochschild cohomology:
\begin{align*}
&\delta \bigl(\mathrm{obs}(f_1,\ldots,f_N) \bigr) = \frac{1}{2}\sum_{i+j=N+1} [\delta(f_i),f_j] + [f_i,\delta(f_j)] \quad \text{\footnotesize by  \eqref{defObsMonStructGerst} and \eqref{DiffOnGerstBracket}}\\
=\:\,&\frac{1}{2}\sum_{i+j=N+1} \bigl[ \mathrm{obs}(f_1,\ldots,f_{i-1}), f_j \bigr] + \bigl[ f_i,\mathrm{obs}(f_1,\ldots,f_{j-1}) \bigr] \quad \text{\footnotesize by assumption}\\
=\:\,& \frac{1}{4}\sum_{i+j=N+1}\sum_{p+q=i} \bigl[ [f_p, f_q],f_j \bigr] + \frac{1}{4}\sum_{i+j=N+1}\sum_{r+s = j} \bigl[ f_i, [f_r, f_s] \bigr] \quad \text{\footnotesize by \eqref{defObsMonStructGerst}}\\
=\:\,& \frac{1}{4}\sum_{a+b+c=N+1} \bigl[ [f_a, f_b],f_c \bigr] + \bigl[ f_c, [f_a, f_b] \bigr] = 0 \quad \text{\footnotesize by super-antisymmetry \eqref{grJacobiComp}.} \qedhere
\end{align*}
\end{proof}

\subsection{Cup product}\label{sectionCupProdHB}
In this subsection we work with a DY cochain complex of the form $\mathrm{C}^\bullet_{\mathrm{DY}}(\Gamma;\mathsf{V},\mathsf{V})$, where $\mathsf{V} = (V,t^V) \in \mathcal{Z}(\Gamma)$ is arbitrary. Recall from \S\ref{subsubDYCoeff} that $V \in \mathcal{D}$ and $t^V : V \otimes \Gamma(-) \Rightarrow \Gamma(-) \otimes V$ is a half-braiding relative to $\Gamma : \mathcal{C} \to \mathcal{D}$.

\begin{definition}\label{defCupProdDY}
Let $f \in \mathrm{C}^m_{\mathrm{DY}}(\Gamma;\mathsf{V},\mathsf{V})$ and $g \in \mathrm{C}^n_{\mathrm{DY}}(\Gamma;\mathsf{V},\mathsf{V})$. For all objects $X_i,Y_j \in \mathcal{C}$, define $(f \cup g)_{X_1,\ldots,X_m,Y_1,\ldots,Y_n}$ to be the following composition
\begin{align*}
V \, \Gamma(X_1) \,\ldots \, \Gamma(X_m) \, \Gamma(Y_1) \, \ldots \, \Gamma(Y_n) &\xrightarrow{f_{X_1,\ldots,X_m} \,\otimes\, \mathrm{id}} \Gamma(X_1 \,\ldots\, X_m) \, V \, \Gamma(Y_1) \, \ldots \, \Gamma(Y_n)\\
&\xrightarrow{\mathrm{id} \,\otimes\, g_{Y_1,\ldots,Y_n}} \Gamma(X_1 \,\ldots\, X_m) \, \Gamma(Y_1 \, \ldots \, Y_n) \, V\\
&\xrightarrow{\Gamma^{(2)}_{X_1 \ldots X_m, Y_1 \ldots Y_n} \,\otimes\, \mathrm{id}_V} \Gamma(X_1 \,\ldots\, X_m \, Y_1 \, \ldots \, Y_n) \, V
\end{align*}
where the symbol $\otimes$ is omitted between objects. This gives the components of the natural transformation $f \cup g \in \mathrm{C}^{m+n}_{\mathrm{DY}}(\Gamma;\mathsf{V},\mathsf{V})$.
\end{definition}

\begin{remark}
We have $f \circ g = (-1)^{mn} g \cup f$, where $\circ$ is the Yoneda product computed in \cite[\S 4.4]{FGS} through the isomorphism between DY cohomology and relative Ext groups. In the case of trivial coefficients, \textit{i.e.} $\mathsf{V} = \boldsymbol{1}$, the cup product was defined in \cite[\S 3]{yetter1} and is further studied in \cite[\S 3.1]{BD}.
\end{remark}

The operation $\cup$ satisfies the following properties, which as we will see later are key-points for its application to the  deformation theory of half-braidings relative to $\Gamma$.

\begin{proposition}\label{propCupProdDG}
1. The operation $\cup$ on $\mathrm{C}^\bullet_{\mathrm{DY}}(\Gamma; \mathsf{V},\mathsf{V})$ is $\Bbbk$-bilinear and associative.
\\2. For all $f \in \mathrm{C}^m_{\mathrm{DY}}(\Gamma;\mathsf{V},\mathsf{V})$ and $g \in \mathrm{C}^n_{\mathrm{DY}}(\Gamma;\mathsf{V},\mathsf{V})$ it holds
\[ \delta(f \cup g) = \delta(f) \cup g + (-1)^m f \cup \delta(g). \]
\end{proposition}
\begin{proof}
1. Bilinearity is obvious. Associativity is an easy computation which requires to use the monoidal structure axiom \eqref{defMonStruct} of $\Gamma^{(2)}$.
\\2. Using diagrammatic calculus, it is straightforward to check that
\begin{equation}\label{partialCup}
\partial_i(f \cup g) = \begin{cases}
\partial_i(f) \cup g & \text{if } 0 \leq i \leq m\\
f \cup \partial_{i-m}(g) & \text{if } m < i \leq m+n+1
\end{cases}
\end{equation}
and that
\begin{equation}\label{eqEg}
\partial_{m+1}(f) \cup g = f \cup \partial_0(g).
\end{equation}
The proof of \eqref{eqEg} and of the cases $i=0$ and $i=m+n+1$ in \eqref{partialCup} require to use the axiom \eqref{defMonStruct} for the monoidal structure of a functor. The other cases are immediate. Hence we have
{\small \begin{align*}
&\delta(f \cup g) = \sum_{i=0}^{m+n+1} (-1)^i \partial_i(f \cup g) \overset{\eqref{partialCup}}{=} \left(\sum_{i=0}^m (-1)^i\partial_i(f) \cup g\right) + \left(\sum_{i=m+1}^{m+n+1} (-1)^i f \cup \partial_{i-m}(g)\right)\\
&\overset{\eqref{eqEg}}{=} \left(\sum_{i=0}^m (-1)^i\partial_i(f) \cup g \right) + (-1)^{m+1}\partial_{m+1}(f) \cup g + (-1)^mf \cup \partial_0(g)  + \left(\sum_{i=1}^{n+1} (-1)^{m+i} f \cup \partial_{i}(g)\right)\\
&= \left(\sum_{i=0}^{m+1} (-1)^i\partial_i(f) \cup g\right)  + \left(\sum_{i=0}^{n+1} (-1)^{m+i} f \cup \partial_{i}(g)\right) = \delta(f) \cup g + (-1)^mf \cup \delta(g). \qedhere
\end{align*}}
\end{proof}

Let $c_1,\ldots,c_k \in \mathrm{C}^1_{\mathrm{DY}}(\Gamma;\mathsf{V},\mathsf{V})$. By definition of $\cup$ we have
\[ (c_i \cup c_j)_{X,Y} = \bigl( \Gamma^{(2)}_{X,Y} \otimes \mathrm{id}_V \bigr) \circ \bigl( \mathrm{id}_{\Gamma(X)} \otimes (c_j)_Y \bigr) \circ \bigl( (c_i)_X \otimes \mathrm{id}_{\Gamma(Y)} \bigr) \]
for all $X,Y \in \mathcal{C}$. Hence the obstruction defined in \eqref{defObsHB} can be rewritten as
\begin{equation}\label{obsCup1Cochains}
\mathrm{obs}(c_1,\ldots,c_k) = -\sum_{i+j = k+1} c_i \cup c_j.
\end{equation}

\begin{lemma}\label{lemmaObstDefHB}
If a family of cochains $c_1,\ldots,c_N \in \mathrm{C}^1_{\mathrm{DY}}(\Gamma;\mathsf{V},\mathsf{V})$ satisfies
\[ \delta(c_1) = 0 \quad \text{and} \quad \forall \, 2 \leq i \leq N, \:\: \delta(c_i) = \mathrm{obs}(c_1,\ldots,c_{i-1}) \]
then $\mathrm{obs}(c_1,\ldots,c_N)$ is a cocycle.
\end{lemma}
\begin{proof}
It suffices to use Prop.\,\ref{propCupProdDG}(2) together with the assumption:
\begin{align*}
\delta \bigl( \mathrm{obs}(c_1,\ldots,c_N) \bigr) &\overset{\eqref{obsCup1Cochains}}{=} - \sum_{i+j=N+1} \delta(c_i \cup c_j) = -\sum_{i + j = N+1} \delta(c_i) \cup c_j + \sum_{i + j = N+1} c_i \cup \delta(c_j)\\
&= \sum_{i + j = N+1} \sum_{k+l = i} c_k \cup c_l \cup c_j  - \sum_{i + j = N+1}\sum_{k+l = j} c_i \cup c_k \cup c_l = 0. \qedhere
\end{align*}
\end{proof}

\subsection{Products on mixed associator cochains}\label{subsecProofObsMix}

Let $\mathcal{M}$ be a $\Bbbk$-linear $\mathcal{C}$-module category with mixed associator $m_{X,Y,M} : X \rhd Y \rhd M \overset{\sim}{\to} (X \otimes Y) \rhd M$. Denote by $\rho : \mathcal{M} \to \mathrm{End}_\Bbbk(\mathcal{M})$ the representation functor, \textit{i.e.} $\rho(X)(M) = X \rhd M$. Recall from Lemma \ref{lemmaBijMixMon} that $m$ produces a monoidal structure $\widehat{m}$ for $\rho$. In Prop.\,\ref{propMixCohomDY} we established an isomorphism of cochain complexes
\[ \Phi^\bullet : \mathrm{C}^\bullet_{\mathrm{DY}}(\rho;\mathsf{F},\mathsf{G}) \to \mathrm{C}^\bullet_{\mathrm{mix}}(\mathcal{M};\mathsf{F},\mathsf{G}) \]
for all $\mathsf{F},\mathsf{G} \in \mathrm{End}_{\mathcal{C}}(\mathcal{M}) = \mathcal{Z}(\rho)$. We can thus transport the comp-algebra operations $\comp_i$ on $\mathrm{C}^\bullet_{\mathrm{DY}}(\rho)$ from Def.\,\ref{defCompiYetter} and the cup product $\cup$ on $\mathrm{C}^\bullet_{\mathrm{DY}}(\rho;\mathsf{F},\mathsf{F})$  from Def.\,\ref{defCupProdDY} through this isomorphism, yielding analogous structures on the mix cochain spaces. Here are the resulting formulas:

\medskip

\indent \textbullet~Let $f \in \mathrm{C}^p_{\mathrm{mix}}(\mathcal{M})$, $g \in \mathrm{C}^q_{\mathrm{mix}}(\mathcal{M})$ and $1 \leq i \leq p$. For all objects $X_k \in \mathcal{C}$ and $M \in \mathcal{M}$ define $(f \:\comp_i\: g)_{X_1,\ldots,X_{p+q-1},M}$ to be the following composition
\[ \xymatrix{
\scalebox{.94}{$X_1 \rhd \ldots \rhd X_i \rhd \ldots \rhd X_{i+q-1} \rhd \ldots \rhd X_{p+q-1} \rhd M$}
\ar[d]^{\mathrm{id}_{X_1 \rhd \ldots \rhd X_{i-1}} \,\rhd\, g_{X_i, \ldots, X_{i+q-1}, X_{i+q} \rhd \ldots \rhd X_{p+q-1} \rhd M}}\\
\scalebox{.94}{$X_1 \rhd \ldots \rhd X_{i-1} \rhd (X_i \otimes \ldots \otimes X_{i+q-1}) \rhd X_{i+q} \rhd \ldots \rhd X_{p+q-1} \rhd M$} \ar[d]^{f_{X_1,\ldots,X_{i-1}, X_i \otimes \ldots \otimes X_{i+q-1}, X_{i+q}, \ldots, X_{p+q-1},M}}\\
\scalebox{.94}{$( X_1  \otimes \ldots \otimes X_{p+q-1} ) \rhd M$}
} \]
This gives the components of the natural transformation $f \:\comp_i\: g \in \mathrm{C}^{p+q-1}_{\mathrm{mix}}(\mathcal{M})$. These operations define a comp-algebra structure on $\mathrm{C}^\bullet_{\mathrm{mix}}(\mathcal{M})$, whose distinguished element is the 2-cochain $m$ (the mixed associator of $\mathcal{M}$).

\medskip

\indent \textbullet~Write $\mathsf{F} = (F,\gamma^F)$ where $\gamma^F : F(- \rhd -) \Rightarrow - \rhd F(-)$ is the $\mathcal{C}$-module structure. Let $f \in \mathrm{C}^p_{\mathrm{mix}}(\mathcal{M};\mathsf{F},\mathsf{F})$ and $g \in \mathrm{C}^q_{\mathrm{mix}}(\mathcal{M};\mathsf{F},\mathsf{F})$. For all objects $X_i,Y_j \in \mathcal{C}$ and $M \in \mathcal{M}$, define $(f \cup g)_{X_1,\ldots,X_p,Y_1,\ldots,Y_q,M}$ to be the following composition
\begin{align*}
&F( X_1 \rhd \ldots \rhd X_p \rhd Y_1 \rhd \ldots \rhd Y_q \rhd M )\\
&\xrightarrow{f_{X_1,\ldots,X_p, Y_1 \rhd \ldots \rhd Y_q \rhd M}} (X_1 \otimes \ldots \otimes X_p) \rhd F(Y_1 \rhd \ldots \rhd Y_q \rhd M)\\
&\xrightarrow{\mathrm{id}_{X_1 \otimes\ldots\otimes X_p} \,\rhd\, g_{Y_1 \rhd \ldots \rhd Y_q \rhd M}} (X_1 \otimes \ldots \otimes X_p) \rhd (Y_1 \otimes \ldots \otimes Y_q) \rhd F(M)\\
&\xrightarrow{m_{X_1 \otimes\ldots\otimes X_p, Y_1  \otimes\ldots\otimes Y_q,F(M)}} (X_1 \otimes \ldots \otimes X_p \otimes Y_1 \otimes \ldots \otimes Y_q) \rhd F(M).
\end{align*}
This gives the components of the natural transformation $f \cup g \in \mathrm{C}^{p+q}_{\mathrm{mix}}(\mathcal{M};\mathsf{F},\mathsf{F})$.

\section{Bimodule structure of the internal Hom}\label{appBimodStructIntHom}
Let $\mathcal{C}$ be a rigid strict monoidal category and $\mathcal{M} = (\mathcal{M},\rhd,m)$ be a $\mathcal{C}$-module category. We assume that for all $M \in \mathcal{M}$ the functor $- \rhd M : \mathcal{C} \to \mathcal{M}$ has a right adjoint, denoted by $\underline{\Hom}(M,-) : \mathcal{M} \to \mathcal{C}$. As recalled in \S\ref{subsecAdjThm}, such right adjoints assemble into a bifunctor $\underline{\Hom} : \mathcal{M}^{\mathrm{op}} \times \mathcal{M} \to \mathcal{C}$ which is called the {\em internal Hom}. The internal Hom exists for instance under the assumptions used in \S\ref{subsecAdjThm}.

\smallskip

For any objects $X,Y \in \mathcal{C}$ and $M,N \in \mathcal{M}$ we have the following chain of natural isomorphisms in the variable $Z \in \mathcal{C}$:
\begin{equation}\label{isosBimodStructIntHom}
\begin{aligned}
&\Hom_{\mathcal{C}}\bigl( Z, \underline{\Hom}(X \rhd M, Y \rhd N) \bigr) \cong \Hom_{\mathcal{M}}\bigl(Z \rhd X \rhd M, Y \rhd N \bigr)\\
\cong\:& \Hom_{\mathcal{M}}\bigl(Y^* \rhd Z \rhd X \rhd M, N \bigr) \cong \Hom_{\mathcal{M}}\bigl( (Y^* \otimes Z \otimes X) \rhd M, N \bigr)\\
\cong\:& \Hom_{\mathcal{C}}\bigl( Y^* \otimes Z \otimes X, \underline{\Hom}(M,N) \bigr) \cong \Hom_{\mathcal{C}}\bigl( Z, Y \otimes \underline{\Hom}(M,N) \otimes X^* \bigr).
\end{aligned}
\end{equation}
By the Yoneda lemma, this defines an isomorphism in $\mathcal{C}$
\[ B_{X,M,Y,N} : \underline{\Hom}(X \rhd M, Y \rhd N) \overset{\sim}{\longrightarrow} Y \otimes \underline{\Hom}(M,N) \otimes X^*. \]
More explicitly, $B_{X,M,Y,N}$ is the image of $\mathrm{id}_Z$ through the chain of isomorphisms \eqref{isosBimodStructIntHom} with the choice $Z = \underline{\Hom}(X \rhd M, Y \rhd N)$. The family of isomorphisms $B$ is natural in all its variables and is called the {\em $\mathcal{C}$-bimodule structure of the internal Hom bifunctor} because of the compatibility properties explained below.

\smallskip

As in \S\ref{subsecAdjThm}, denote by $\underline{\mathsf{coev}}_{-,M}$ and $\underline{\mathsf{ev}}_{M,-}$ the unit and counit of the adjunction $(- \rhd M) \dashv \underline{\Hom}(M,-)$. Like any adjunction, the isomorphism
\begin{equation}\label{adjIsoIntHom2}
I_{X,M,N} : \Hom_{\mathcal{M}}(X \rhd M,N) \overset{\sim}{\longrightarrow} \Hom_{\mathcal{C}}\bigl( X, \underline{\Hom}(M,N) \bigr)
\end{equation}
and its inverse can be expressed in terms of the (co)unit: given $\varphi \in \Hom_{\mathcal{M}}(X \rhd M,N)$ and $\psi \in \Hom_{\mathcal{C}}\bigl(X, \underline{\Hom}(M,N)\bigr)$ we have
\begin{align}
\begin{split}\label{adjIsoCoUnit}
&I_{X,M,N}(\varphi) = \Bigl[ X \xrightarrow{\:\underline{\mathsf{coev}}_{X,M}\:} \underline{\Hom}(M,X \rhd M) \xrightarrow{\:\underline{\Hom}(\mathrm{id}_M,\varphi)\:} \underline{\Hom}(M,N) \Bigr],\\
&I_{X,M,N}^{-1}(\psi) = \Bigl[ X \rhd M \xrightarrow{\:\psi \,\rhd\,\mathrm{id}_M\:} \underline{\Hom}(M,N) \rhd M \xrightarrow{\:\underline{\mathsf{ev}}_{M,N}\:} N \Bigr].
\end{split}
\end{align}
Using \eqref{adjIsoCoUnit} we can give a more explicit description of the $\mathcal{C}$-bimodule structure $B_{X,M,Y,N}$ of $\underline{\Hom}$, as the following composition of arrows ($\otimes$ is omitted between objects of $\mathcal{C}$)
\begin{equation}\label{formulaBimodStructHom}
\begin{aligned}
&\underline{\Hom}(X\rhd M,Y \rhd N) \xrightarrow{\mathrm{coev}_Y \,\otimes\, \mathrm{id}\,\otimes\, \mathrm{coev}_X} YY^* \underline{\Hom}(X\rhd M,Y \rhd N) XX^*\\
&\xrightarrow{\mathrm{id}_{Y^*} \,\otimes\, \underline{\mathsf{coev}}_{Y^* \underline{\Hom}(X \rhd M,Y \rhd N) X,M} \,\otimes\, \mathrm{id}_{X^*}} Y \underline{\Hom}\bigl[ M, \bigl(Y^* \underline{\Hom}(X \rhd M,Y \rhd N) X \bigr) \rhd M \bigr] X^*\\
&\xrightarrow{\text{mixed assoc. in } \underline{\Hom}(M,\ldots)} Y \underline{\Hom}\bigl[ M, Y^* \rhd \underline{\Hom}(X \rhd M,Y \rhd N) \rhd X \rhd M \bigr] X^*\\
&\xrightarrow{\mathrm{id}_Y \,\otimes\, \underline{\Hom}(\mathrm{id}_M,\mathrm{id}_{Y^*} \,\rhd\, \underline{\mathsf{ev}}_{X\rhd M,Y \rhd N}) \,\otimes\,\mathrm{id}_{X^*}} Y \underline{\Hom}\bigl( M,Y^* \rhd Y \rhd N \bigr) X^*\\
&\xrightarrow{\text{mixed asso. in } \underline{\Hom}(M,\ldots)} Y \underline{\Hom}\bigl(M,(Y^*  Y) \rhd N\bigr) X^*\\
&\xrightarrow{\mathrm{id}_Y \,\otimes\, \underline{\Hom}(\mathrm{id}_M,\mathsf{ev}_Y \,\rhd\,\mathrm{id}_N) \,\otimes\,\mathrm{id}_{X^*}} Y \underline{\Hom}(M,N) X^*.
\end{aligned}
\end{equation}

Introduce the following natural isomorphisms
\begin{equation*}
\begin{aligned}
&L^M_{Y,N} = B_{\boldsymbol{1},M,Y,N} : \underline{\Hom}(M,Y \rhd N) \overset{\sim}{\longrightarrow} Y \otimes \underline{\Hom}(M,N),\\
&R^N_{X,M} = B_{X,M,\boldsymbol{1},N} : \underline{\Hom}(X \rhd M,N) \overset{\sim}{\longrightarrow} \underline{\Hom}(M,N) \otimes X^*.
\end{aligned}
\end{equation*}
Then the pairs $\bigl( \underline{\Hom}(M,-), L^M \bigr)$ and $\bigl( \underline{\Hom}(-,N), R^N \bigr)$ are $\mathcal{C}$-module functors in the sense of~\eqref{defCModFunct}, for all $M,N \in \mathcal{M}$. Moreover, we have a commutative diagram
\[ \xymatrix@C=7em{
\underline{\Hom}(X \rhd M, Y \rhd N) \ar[r]^-{R^{Y \rhd N}_{X,M}} \ar[d]_-{L^{X \rhd M}_{Y,N}} \ar[dr]^-{B_{X,M,Y,N}} & \underline{\Hom}(M, Y \rhd N) \otimes X^* \ar[d]^-{L^M_{Y,N} \,\otimes\, \mathrm{id}_{X^* }}\\
Y \otimes \underline{\Hom}(X \rhd M, N) \ar[r]_-{\mathrm{id}_Y \,\otimes\, R_{X,M}^N}& Y \otimes \underline{\Hom}(M,N) \otimes X^* 
} \] 
which gives a factorization of $B$ in terms of left and right $\mathcal{C}$-module structures.

\section{Proof of Thm. \ref{thmAdjObjFullCent}}\label{appProofFullCent}

In this appendix we denote by $\mathcal{M}$ the $\mathcal{C}$-module category $\mathrm{Mod}_{\mathcal{C}}(A)$ for some associative
unital algebra $A\in\mathcal{C}$. Let us collect some technical preliminaries before proving the theorem (these preliminary facts are completely general and true for any $\mathcal{C}$-module category $\mathcal{M}$).

\smallskip

\indent As in \S\ref{subsecAdjObjFullCenter}, we denote the unit and counit of the adjunction $(- \rhd M) \dashv \underline{\Hom}(M,-)$ by $\underline{\mathsf{coev}}_{-,M} : - \Rightarrow \underline{\Hom}(M, - \rhd M)$ and $\underline{\mathsf{ev}}_{M,-} : \underline{\Hom}(M,-) \rhd M \Rightarrow -$. As for any adjunction, these families of morphisms are natural, which means that for all $f \in \Hom_{\mathcal{M}}(N_1,N_2)$ and $g \in \Hom_{\mathcal{C}}(X_1,X_2)$ we have
\begin{align}
\begin{split}\label{natEvInt}
&f \circ \underline{\mathsf{ev}}_{M,N_1} = \underline{\mathsf{ev}}_{M,N_2} \circ \bigl( \underline{\Hom}(\mathrm{id}_M,f) \rhd \mathrm{id}_M \bigr)\\
\text{and }\:&\underline{\Hom}(\mathrm{id}_M,g \rhd \mathrm{id}_M) \circ \underline{\mathsf{coev}}_{X_1,M} = \underline{\mathsf{coev}}_{X_2,M} \circ g.
\end{split}
\end{align}
It is also a general fact that the ``zig-zag properties'' hold, \textit{i.e.} the compositions
\begin{align}
\begin{split}\label{zigZagPropAdj}
&X \rhd M \xrightarrow{\:\underline{\mathsf{coev}}_{X,M} \,\rhd\, \mathrm{id}_M \:} \underline{\Hom}(M, X \rhd M) \rhd M \xrightarrow{\:\underline{\mathsf{ev}}_{M,X \rhd M}\:} X \rhd M\\
&\underline{\Hom}(M,N) \xrightarrow{\:\underline{\mathsf{coev}}_{\underline{\Hom}(M,N),M}\:} \underline{\Hom}\bigl(M, \underline{\Hom}(M,N) \rhd M \bigr) \xrightarrow{\:\underline{\Hom}(\mathrm{id}_M, \underline{\mathsf{ev}}_{M,N})\:} \underline{\Hom}(M,N)
\end{split}
\end{align}
are equal to $\mathrm{id}_{X \rhd M}$ and $\mathrm{id}_{\underline{\Hom}(M,N)}$ respectively \cite[\S IV.1, Thm.\,1]{MLCat}. Also recall from  \eqref{adjIsoCoUnit} that the adjunction isomorphism \eqref{adjIsoIntHom2} and its inverse can be expressed in term of the (co)unit. We finally note the following dinaturality properties of $\underline{\mathsf{ev}}_{M,N}$ and $\underline{\mathsf{coev}}_{X,M}$ in the variable $M$; this is true for any ``adjunction with parameter'':

\begin{lemma}\label{lemmaNatEvCoev}
For all $X \in \mathcal{C}$, $N \in \mathcal{M}$ and $f \in \Hom_{\mathcal{M}}(M_1,M_2)$ we have commutative diagrams
\[ \xymatrix@C=7.5em{
\underline{\Hom}(M_2,N) \rhd M_1 \ar[d]_-{\mathrm{id} \,\rhd\,f} \ar[r]^-{\underline{\Hom}(f,\,\mathrm{id}_N) \,\rhd\, \mathrm{id}_{M_1}}& \underline{\Hom}(M_1,N) \rhd M_1 \ar[d]^-{\underline{\mathsf{ev}}_{M_1,N}}\\
\underline{\Hom}(M_2,N) \rhd M_2 \ar[r]_-{\underline{\mathsf{ev}}_{M_2,N}} & N
} \]
and
\[ \xymatrix@C=5em{
X \ar[r]^-{\underline{\mathsf{coev}}_{X,M_1}} \ar[d]_-{\underline{\mathsf{coev}}_{X,M_2}} & \underline{\Hom}(M_1, X \rhd M_1) \ar[d]^-{\underline{\Hom}(\mathrm{id}, \mathrm{id}_X \rhd f)}\\
\underline{\Hom}(M_2, X \rhd M_2) \ar[r]_-{\underline{\Hom}(f, \mathrm{id})} & \underline{\Hom}(M_1, X \rhd M_2)
} \]
\end{lemma}
\begin{proof}
By \eqref{natAdjIso} and the definition of $\underline{\mathsf{ev}}$ we have
\[ I_{\underline{\Hom}(M_2,N), M_1, N}\bigl( \underline{\mathsf{ev}}_{M_2,N} \circ (\mathrm{id} \rhd f) \bigr) = \underline{\Hom}(f,\mathrm{id}_N) \circ I_{\underline{\Hom}(M_2,N), M_2, N}(\underline{\mathsf{ev}}_{M_2,N}) = \underline{\Hom}(f,\mathrm{id}_N). \]
It suffices to combine this with \eqref{adjIsoCoUnit}:
\[ \underline{\mathsf{ev}}_{M_2,N} \circ (\mathrm{id} \rhd f) = I_{\underline{\Hom}(M_2,N), M_1, N}^{-1}\bigl( \underline{\Hom}(f,\mathrm{id}_N) \bigr) = \underline{\mathsf{ev}}_{M_1,N} \circ \bigl( \underline{\Hom}(f,\mathrm{id}_N) \rhd \mathrm{id}_{M_1} \bigr). \]
The second diagram is established similarly.
\end{proof}

Recall the morphism $J_{X,M,Y,N}$ in $\mathcal{C}$ introduced in \eqref{isoJIntHom}. We see from \eqref{formulaBimodStructHom} that it equals the following composition of arrows, where we use strictness of $\mathrm{Mod}_{\mathcal{C}}(A)$ to suppress the symbols $\otimes$, $\rhd$ and the mixed associator,
\begin{align}
\begin{split}\label{formulaJ}
J_{X,M,Y,N} : \:\, &\underline{\Hom}(X M, Y  N)  X \xrightarrow{\mathrm{coev}_Y \otimes \mathrm{id}} YY^*\,\underline{\Hom}(X M, Y  N)X\\
&\xrightarrow{\mathrm{id}_Y \otimes \underline{\mathsf{coev}}_{Y^* \underline{\Hom}(X M, Y N) X, M}} Y\,\underline{\Hom}\bigl( M, Y^*\, \underline{\Hom}(X  M, Y N) X M \bigr)\\
&\xrightarrow{\mathrm{id}_Y \otimes \underline{\Hom}(\mathrm{id}_{M}, \mathrm{id}_{Y^*} \otimes \underline{\mathsf{ev}}_{XM,YN})} Y\,\underline{\Hom}(M,Y^*\,YN)\\
&\xrightarrow{\mathrm{id}_Y \otimes \underline{\Hom}(\mathrm{id}_M, \mathrm{ev}_Y \otimes \mathrm{id}_N)} Y\,\underline{\Hom}(M,N).
\end{split}
\end{align}
This expression allows us to prove the following lemma (note that morphisms in $\mathrm{Mod}_{\mathcal{C}}(A)$ are in particular morphisms in $\mathcal{C}$, so the lemma makes sense):
\begin{lemma}\label{lemmaTechnicalJ}
In $\mathcal{C}$ it holds $(\mathrm{id}_Y \otimes \underline{\mathsf{ev}}_{M,N}) \circ (J_{X,M,Y,N} \otimes \mathrm{id}_M) = \underline{\mathsf{ev}}_{X \rhd M, Y \rhd N}$.
\end{lemma}
\begin{proof}
The fact that $X \rhd M = X \otimes M$ as objects in $\mathcal{C}$ for all $X,M$ allows us to use diagrammatic calculus in $\mathcal{C}$ and we find (diagrams are read from bottom to top)
\begin{center}
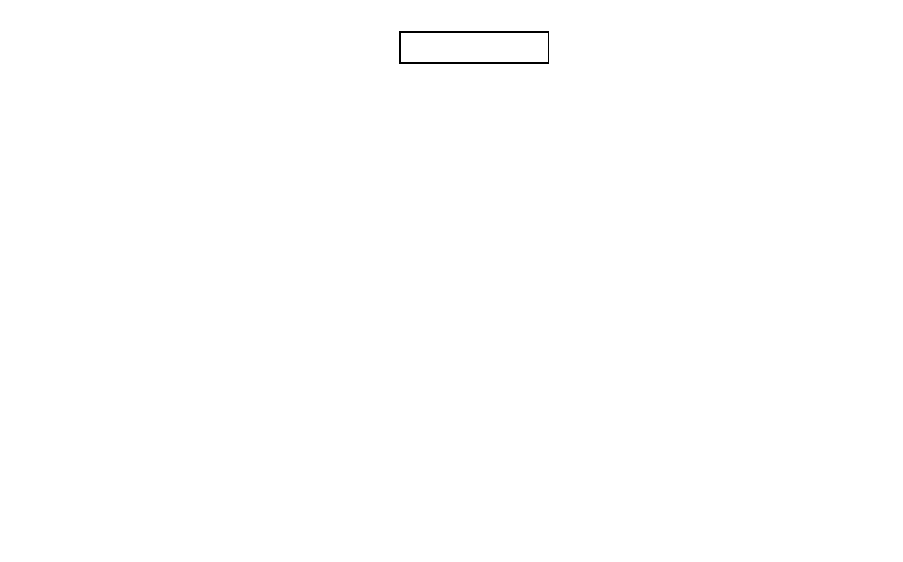
\end{center}
where the first equality is by definition of $J$, the second and third equalities use \eqref{natEvInt}, the fourth equality is by the zig-zag property \eqref{zigZagPropAdj} relating the unit and counit of an adjunction, and the last equality is the zig-zag axiom for duality in $\mathcal{C}$.
\end{proof}

We are now ready to prove Thm.\,\ref{thmAdjObjFullCent}. Recall that $\mathcal{A}_{\mathcal{M}} = \int_{M \in \mathcal{M}} \underline{\Hom}(M,M)$ is equipped with the half-braiding $b^{\mathsf{Id}}$ from \eqref{halfBrEnd} and the algebra structure described in \eqref{algStructAdjObj}. We denote by $\bigl( \pi_M : \mathcal{A}_{\mathcal{M}} \to \underline{\Hom}(M,M) \bigr)_{M \in \mathcal{M}}$ the universal dinatural transformation of this end.

\noindent Also recall that the morphism $\xi \in \Hom_{\mathcal{C}}(\mathcal{A}_{\mathcal{M}},A)$ is given by
\begin{equation}\label{defXiAdjObj}
\xi : \mathcal{A}_{\mathcal{M}} \xrightarrow{\:\pi_A\:} \underline{\Hom}(A,A) \xrightarrow{\:\mathrm{id} \otimes 1_A\:} \underline{\Hom}(A,A) \otimes A \xrightarrow{\:\underline{\mathsf{ev}}_{A,A}\:} A.
\end{equation} 
The proof is cut in the several facts which have to be checked.

\medskip

\noindent \textbf{Claim 1.} {\em The morphism $\xi : \mathcal{A}_{\mathcal{M}} \to A$ satisfies the center property \eqref{diagramFullCenter}}.

\begin{proof} We use the notation $\underline{\End}(M) = \underline{\Hom}(M,M)$. On the one hand we have
\begin{center}
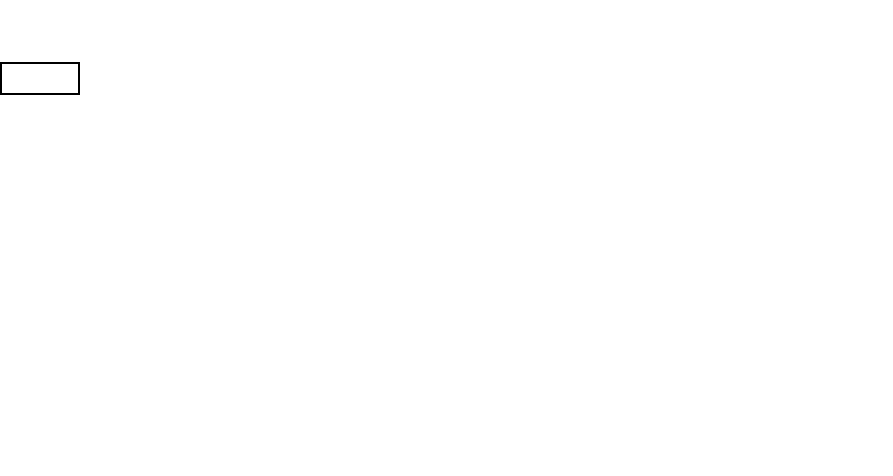
\end{center}
where the first equality is by definition of $\xi$ in \eqref{defXiAdjObj}, the second is by definition of the half-braiding $b^{\mathsf{Id}}$ in \eqref{halfBrEnd}, the third is by Lemma \ref{lemmaTechnicalJ}, the fourth is by naturality of $\underline{\mathsf{ev}}$ in its second variable \eqref{natEvInt}, the fifth is by dinaturality of $\pi$, the sixth is by Lemma \ref{lemmaNatEvCoev} and the last is by unitality. On the other hand we have
\begin{center}
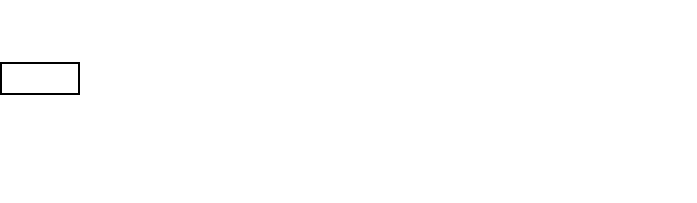
\end{center}
as desired, where the first equality is by definition of $\xi$ in \eqref{defXiAdjObj}, the second uses that $\underline{\mathsf{ev}}_{A,A}$ is an $A$-module morphism $\underline{\Hom}(A,A) \rhd A \to A$ (by definition) and the third is by unitality.
\end{proof}

\noindent \textbf{Claim 2.} {\em Let $(C,h) \in \mathcal{Z}(\mathcal{C})$ be equipped with $\zeta \in \Hom_{\mathcal{C}}(C,A)$ which satisfies the center property \eqref{diagramFullCenter}. Then there is a unique $u \in \Hom_{\mathcal{C}}(C, \mathcal{A}_{\mathcal{M}})$ such that $\zeta = \xi \circ u$.}

\begin{proof}
Here $h = \bigl(h_X : C \otimes X \to X \otimes C\bigr)_{X \in \mathcal{C}}$ denotes the half-braiding on $C$. For all right $A$-module $M=(M,r) \in \mathcal{M}$ consider the following morphism in $\mathcal{C}$
\begin{equation}\label{defBeta}
\beta_M : C \otimes M \xrightarrow{\:h_M\:} M \otimes C \xrightarrow{\:\mathrm{id}_M \,\otimes\, \zeta} M \otimes A \xrightarrow{\:r\:} M.
\end{equation}
We claim that $\beta_M \in \Hom_{\mathcal{M}}(C \rhd M,M)$, which means that it commutes with the right actions of $A$. Indeed $\beta_M \circ (\mathrm{id}_C \otimes r)$ equals
\begin{center}
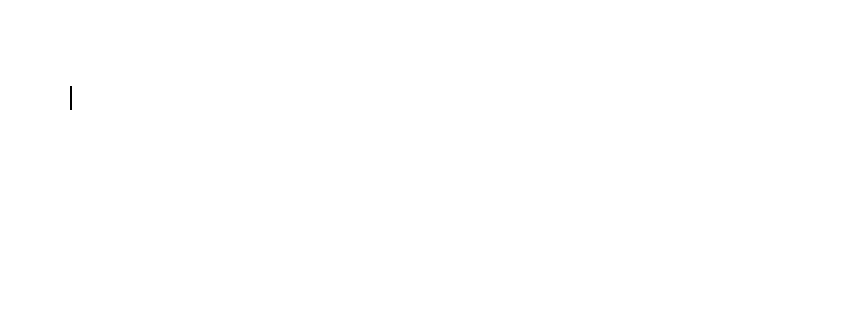
\end{center}
and the last term equals $r \circ (\beta_M \otimes \mathrm{id}_A)$, as desired. The first equality is by naturality of $h$, the second is by the $A$-module axiom, the third is by the half-braiding property, the fourth is by the center property \eqref{diagramFullCenter} and the fifth is by the $A$-module axiom. It is readily seen that the collection of morphisms $\{\beta_M\}_{M \in \mathcal{M}}$ is natural: for all $f \in \Hom_{\mathcal{M}}(M_1,M_2)$ we have
\begin{equation}\label{natBetaM}
f \circ \beta_{M_1} = \beta_{M_2} \circ (\mathrm{id}_C \rhd f).
\end{equation}
Using the adjunction isomorphism $I_{C,M,M} : \Hom_{\mathcal{M}}(C \rhd M, M) \overset{\sim}{\longrightarrow} \Hom_{\mathcal{C}}\bigl( C, \underline{\Hom}(M,M) \bigr)$, consider
\begin{equation*}
\widetilde{\beta}_M  = I_{C,M,M}(\beta_M) = \underline{\Hom}(\mathrm{id}_M,\beta_M) \circ \underline{\mathsf{coev}}_{C,M} : C \to \underline{\Hom}(M,M)
\end{equation*}
where the second equality uses the expression of $I$ given in \eqref{adjIsoCoUnit}. The family $\{ \widetilde{\beta}_M \}$ is dinatural: indeed, for $f \in \Hom_{\mathcal{M}}(M_1,M_2)$ we have
\begin{align*}
&\underline{\Hom}(\mathrm{id}_{M_1},f) \circ \widetilde{\beta}_{M_1} = \underline{\Hom}(\mathrm{id}_{M_1},f) \circ I_{C,M_1,M_1}(\beta_{M_1}) \overset{\eqref{natAdjIso}}{=} I_{C,M_2,M_2}\bigl( f \circ \beta_{M_1} \bigr)\\
\overset{\eqref{natBetaM}}{=}\:\,& I_{C,M_2,M_2}\bigl( \beta_{M_2} \circ (\mathrm{id}_C \rhd f) \bigr) \overset{\eqref{natAdjIso}}{=} \underline{\Hom}(f,\mathrm{id}_{M_2}) \circ I_{C,M_2,M_2}(\beta_{M_2}) = \underline{\Hom}(f,\mathrm{id}_{M_2}) \circ \widetilde{\beta}_{M_2}.
\end{align*}
Hence by the universal property of the end $\mathcal{A}_{\mathcal{M}}$ we get a commutative diagram in $\mathcal{C}$
\begin{equation}\label{defFactU}
\xymatrix@C=4em{
C \ar[dr]_-{\widetilde{\beta}_M} \ar[d]_-{\exists!\,u} \ar[r]^-{\underline{\mathsf{coev}}_{C,M}}& \underline{\Hom}(M,C \rhd M) \ar[d]^-{\underline{\Hom}(\mathrm{id}_M,\beta_M)}\\
\mathcal{A}_{\mathcal{M}} \ar[r]_-{\pi_M} & \underline{\Hom}(M,M)
} \end{equation}
for all $M \in \mathcal{M}$. In this way we get a {\em unique} morphism $u$. Let us prove that $\xi \circ u = \zeta$. By definition of $\xi$ in \eqref{defXiAdjObj}, $\xi \circ u$ equals
\begin{center}
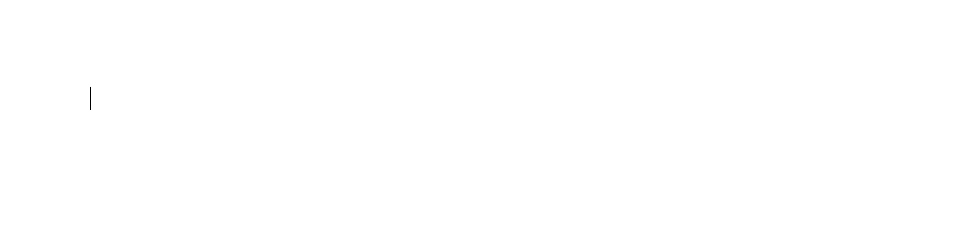
\end{center}
and the last term equals $\zeta$, as desired. For the first equality we used the definition of $u$ in \eqref{defFactU}, the second equality is by naturality of $\underline{\mathsf{ev}}$ as written in \eqref{natEvInt}, the third uses the zig-zag property \eqref{zigZagPropAdj}, the fourth is by definition of $\beta_A$ in \eqref{defBeta} and the fifth is by naturality of $h$ and the fact that $h_{\boldsymbol{1}} = \mathrm{id}_C$.
\end{proof}

\noindent \textbf{Claim 3.} {\em The unique morphism $u \in \Hom_{\mathcal{C}}(C, \mathcal{A}_{\mathcal{M}})$ obtained in Claim 2 is actually a morphism $(C,h) \to (\mathcal{A}_{\mathcal{M}}, b^{\mathsf{Id}})$ in $\mathcal{Z}(\mathcal{C})$. Hence the pair $(\mathcal{A}_{\mathcal{M}},\xi)$ satisfies the universal property of the full center of $A$ (Definition \ref{defFullCenter}).}
\begin{proof}
In Figure \ref{figProofHalfBr} on page \pageref{figProofHalfBr} we prove that for any given $X \in \mathcal{C}$ it holds
\begin{equation}\label{formulaHalfBrUniv}
\forall \, M \in \mathcal{M}, \quad b^{\mathsf{Id}}_X \circ (u \otimes \mathrm{id}_X) \circ (\mathrm{id}_X \otimes \pi_M) = (\mathrm{id}_X \otimes u) \circ h_X \circ (\mathrm{id}_X \otimes \pi_M).
\end{equation}
The desired property then follows by universality of $\mathrm{id}_X \otimes \pi$. Indeed, the family $\mathrm{id}_X \otimes \pi$ is universal because the functor $X \otimes -$ preserves limits as does any right adjoint \cite[\S V.5]{MLCat} ($X \otimes -$ being the right adjoint of $X^* \otimes -$), thus in particular it preserves ends \cite[\S IX.5]{MLCat}.
\end{proof}

\noindent \textbf{Claim 4.} {\em If $\mathcal{A}_{\mathcal{M}}$ is equipped with the algebra structure from \eqref{algStructAdjObj} then $\xi : \mathcal{A}_{\mathcal{M}} \to A$ is a morphism of algebras. It follows that the algebra structure on $\mathcal{A}_{\mathcal{M}}$ described in \eqref{algStructAdjObj} coincides with the canonical algebra structure \eqref{algStructFullCenter} on the full center of $A$.}
\begin{proof}
As a preliminary remark, note that by definition $\underline{\mathsf{ev}}_{M,N} : \underline{\Hom}(M,N) \rhd M \to N$ is a morphism in $\mathcal{M}$. In terms of right $A$-modules $M,N \in \mathrm{Mod}_{\mathcal{C}}(A)$, it means that
\begin{equation}\label{AlinEv}
\underline{\mathsf{ev}}_{M,N} \circ \bigl( \mathrm{id}_{\underline{\Hom}(M,N)} \otimes r_M \bigr) = r_N \circ \bigl( \underline{\mathsf{ev}}_{M,N} \otimes \mathrm{id}_A \bigr)
\end{equation}
where $r_M : M \otimes A \to M$ and $r_N : N \otimes A \to N$ are the actions of $A$. Let us now prove compatibility of $\xi$ with the multiplications $m_{\mathcal{A}_{\mathcal{M}}}$ of $\mathcal{A}_{\mathcal{M}}$ and $m_A$ of $A$:
\begin{center}
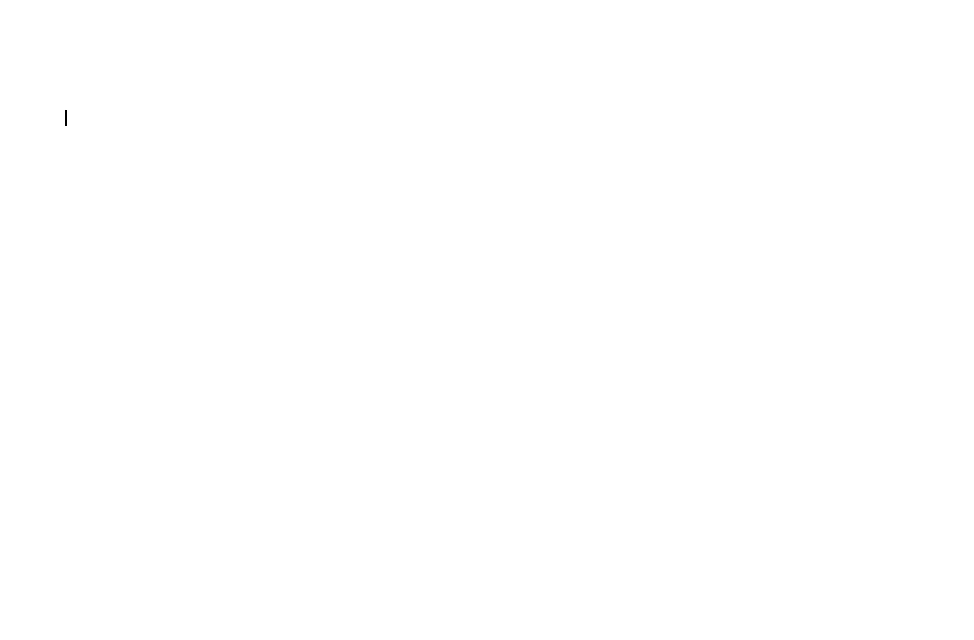
\end{center}
where the first equality is by definition of $\xi$ in \eqref{defXiAdjObj}, the second is by the description of the multiplication of $\mathcal{A}_{\mathcal{M}}$ in \eqref{algStructAdjObj}, the third is by naturality of $\underline{\mathsf{ev}}$ explicited in \eqref{natEvInt} and used here two times, the fourth is by the zig-zag property \eqref{zigZagPropAdj}, the fifth is a trick using unitality of $m_A$, the sixth is by \eqref{AlinEv} with $M=N=A$ and the last is by definition of $\xi$. 
\\The compatibility of $\xi$ with units easily follows from \eqref{descriptionUnitDinat} and the zig-zag property \eqref{zigZagPropAdj}.
\end{proof}

\newpage

\begin{figure}[h!]
\centering
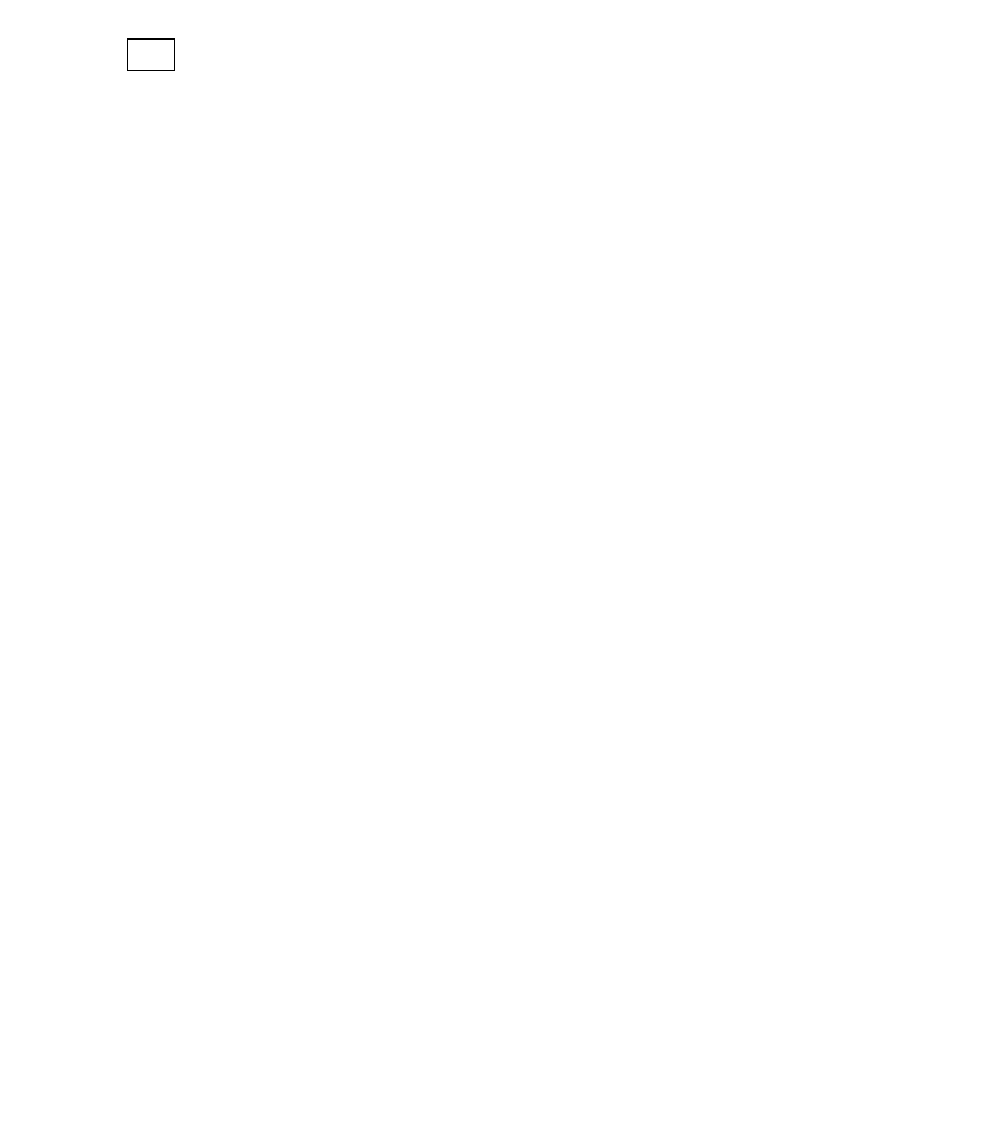
\caption{Diagrammatic proof of eq.\,\eqref{formulaHalfBrUniv}. All symbols $\rhd$ and $\otimes$ are omitted thanks to strictness; for morphisms we write for instance $M$ instead of $\mathrm{id}_M$. The first equality is by definition of $b^{\mathsf{Id}}$ in \eqref{halfBrEnd}, the second is by definition of $u$ in \eqref{defFactU}, the third uses the fact that $\beta_{X \rhd M} = (\mathrm{id}_X \rhd \beta_M) \circ (h_X \rhd \mathrm{id}_M)$ which easily follows from the definition of $\beta$ in \eqref{defBeta} and the half-braiding property of $h$, the fourth is by naturality of $J$ (here in the last variable), the fifth is by the formula for $J$ in \eqref{formulaJ}, the sixth is by naturality of $\underline{\mathsf{coev}}$ \eqref{natEvInt}, the seventh is by naturality of $\underline{\mathsf{ev}}$ applied within the functor $\underline{\Hom}(M,-)$ and by naturality of $\underline{\mathsf{coev}}$ applied in the two bottom boxes \eqref{natEvInt}, the eighth is by the zig-zag property \eqref{zigZagPropAdj} used within the functor $\underline{\Hom}(M,-)$, the ninth is by naturality of $\underline{\mathsf{coev}}$ \eqref{natEvInt} used here two times and the last is by definition of $u$ in \eqref{defFactU}.}
\label{figProofHalfBr}
\end{figure}

\end{document}

%% file: reminderGraphicalCalculus.pdf_tex
\begingroup%
  \makeatletter%
  \providecommand\color[2][]{%
    \errmessage{(Inkscape) Color is used for the text in Inkscape, but the package 'color.sty' is not loaded}%
    \renewcommand\color[2][]{}%
  }%
  \providecommand\transparent[1]{%
    \errmessage{(Inkscape) Transparency is used (non-zero) for the text in Inkscape, but the package 'transparent.sty' is not loaded}%
    \renewcommand\transparent[1]{}%
  }%
  \providecommand\rotatebox[2]{#2}%
  \newcommand*\fsize{\dimexpr\f@size pt\relax}%
  \newcommand*\lineheight[1]{\fontsize{\fsize}{#1\fsize}\selectfont}%
  \ifx\svgwidth\undefined%
    \setlength{\unitlength}{442.17477126bp}%
    \ifx\svgscale\undefined%
      \relax%
    \else%
      \setlength{\unitlength}{\unitlength * \real{\svgscale}}%
    \fi%
  \else%
    \setlength{\unitlength}{\svgwidth}%
  \fi%
  \global\let\svgwidth\undefined%
  \global\let\svgscale\undefined%
  \makeatother%
  \begin{picture}(1,0.15308769)%
    \lineheight{1}%
    \setlength\tabcolsep{0pt}%
    \put(0,0){\includegraphics[width=\unitlength,page=1]{reminderGraphicalCalculus.pdf}}%
    \put(0.02186648,0.07146984){\color[rgb]{0,0,0}\makebox(0,0)[lt]{\lineheight{1.25}\smash{\begin{tabular}[t]{l}$= \mathrm{id}_X$\end{tabular}}}}%
    \put(0.00537085,0.10297199){\color[rgb]{0,0,0}\makebox(0,0)[lt]{\lineheight{1.25}\smash{\begin{tabular}[t]{l}$_X$\end{tabular}}}}%
    \put(0,0){\includegraphics[width=\unitlength,page=2]{reminderGraphicalCalculus.pdf}}%
    \put(0.16308922,0.07293899){\color[rgb]{0,0,0}\makebox(0,0)[lt]{\lineheight{1.25}\smash{\begin{tabular}[t]{l}$f$\end{tabular}}}}%
    \put(0,0){\includegraphics[width=\unitlength,page=3]{reminderGraphicalCalculus.pdf}}%
    \put(0.16184799,0.03472027){\color[rgb]{0,0,0}\makebox(0,0)[lt]{\lineheight{1.25}\smash{\begin{tabular}[t]{l}$_X$\end{tabular}}}}%
    \put(0.16481383,0.12124171){\color[rgb]{0,0,0}\makebox(0,0)[lt]{\lineheight{1.25}\smash{\begin{tabular}[t]{l}$_Y$\end{tabular}}}}%
    \put(0.2002775,0.07136696){\color[rgb]{0,0,0}\makebox(0,0)[lt]{\lineheight{1.25}\smash{\begin{tabular}[t]{l}$= (f : X \to Y)$\end{tabular}}}}%
    \put(0,0){\includegraphics[width=\unitlength,page=4]{reminderGraphicalCalculus.pdf}}%
    \put(0.45598004,0.03992442){\color[rgb]{0,0,0}\makebox(0,0)[lt]{\lineheight{1.25}\smash{\begin{tabular}[t]{l}$f$\end{tabular}}}}%
    \put(0,0){\includegraphics[width=\unitlength,page=5]{reminderGraphicalCalculus.pdf}}%
    \put(0.45473881,0.00170568){\color[rgb]{0,0,0}\makebox(0,0)[lt]{\lineheight{1.25}\smash{\begin{tabular}[t]{l}$_X$\end{tabular}}}}%
    \put(0,0){\includegraphics[width=\unitlength,page=6]{reminderGraphicalCalculus.pdf}}%
    \put(0.45598004,0.09929008){\color[rgb]{0,0,0}\makebox(0,0)[lt]{\lineheight{1.25}\smash{\begin{tabular}[t]{l}$g$\end{tabular}}}}%
    \put(0,0){\includegraphics[width=\unitlength,page=7]{reminderGraphicalCalculus.pdf}}%
    \put(0.45770465,0.1475928){\color[rgb]{0,0,0}\makebox(0,0)[lt]{\lineheight{1.25}\smash{\begin{tabular}[t]{l}$_Z$\end{tabular}}}}%
    \put(0.46875339,0.07366891){\color[rgb]{0,0,0}\makebox(0,0)[lt]{\lineheight{1.25}\smash{\begin{tabular}[t]{l}$_Y$\end{tabular}}}}%
    \put(0,0){\includegraphics[width=\unitlength,page=8]{reminderGraphicalCalculus.pdf}}%
    \put(0.54086828,0.06990287){\color[rgb]{0,0,0}\makebox(0,0)[lt]{\lineheight{1.25}\smash{\begin{tabular}[t]{l}$g \circ f$\end{tabular}}}}%
    \put(0,0){\includegraphics[width=\unitlength,page=9]{reminderGraphicalCalculus.pdf}}%
    \put(0.56044605,0.03051323){\color[rgb]{0,0,0}\makebox(0,0)[lt]{\lineheight{1.25}\smash{\begin{tabular}[t]{l}$_X$\end{tabular}}}}%
    \put(0,0){\includegraphics[width=\unitlength,page=10]{reminderGraphicalCalculus.pdf}}%
    \put(0.56341188,0.11703469){\color[rgb]{0,0,0}\makebox(0,0)[lt]{\lineheight{1.25}\smash{\begin{tabular}[t]{l}$_Z$\end{tabular}}}}%
    \put(0.49948814,0.07124856){\color[rgb]{0,0,0}\makebox(0,0)[lt]{\lineheight{1.25}\smash{\begin{tabular}[t]{l}$=$\end{tabular}}}}%
    \put(0,0){\includegraphics[width=\unitlength,page=11]{reminderGraphicalCalculus.pdf}}%
    \put(0.67796038,0.07327709){\color[rgb]{0,0,0}\makebox(0,0)[lt]{\lineheight{1.25}\smash{\begin{tabular}[t]{l}$g$\end{tabular}}}}%
    \put(0,0){\includegraphics[width=\unitlength,page=12]{reminderGraphicalCalculus.pdf}}%
    \put(0.67690911,0.03198799){\color[rgb]{0,0,0}\makebox(0,0)[lt]{\lineheight{1.25}\smash{\begin{tabular}[t]{l}$_{X'}$\end{tabular}}}}%
    \put(0.67987485,0.11850943){\color[rgb]{0,0,0}\makebox(0,0)[lt]{\lineheight{1.25}\smash{\begin{tabular}[t]{l}$_{Y'}$\end{tabular}}}}%
    \put(0,0){\includegraphics[width=\unitlength,page=13]{reminderGraphicalCalculus.pdf}}%
    \put(0.72055431,0.07020672){\color[rgb]{0,0,0}\makebox(0,0)[lt]{\lineheight{1.25}\smash{\begin{tabular}[t]{l}$f$\end{tabular}}}}%
    \put(0,0){\includegraphics[width=\unitlength,page=14]{reminderGraphicalCalculus.pdf}}%
    \put(0.71931303,0.03198799){\color[rgb]{0,0,0}\makebox(0,0)[lt]{\lineheight{1.25}\smash{\begin{tabular}[t]{l}$_X$\end{tabular}}}}%
    \put(0.72227891,0.11850943){\color[rgb]{0,0,0}\makebox(0,0)[lt]{\lineheight{1.25}\smash{\begin{tabular}[t]{l}$_Y$\end{tabular}}}}%
    \put(0,0){\includegraphics[width=\unitlength,page=15]{reminderGraphicalCalculus.pdf}}%
    \put(0.80367603,0.07033591){\color[rgb]{0,0,0}\makebox(0,0)[lt]{\lineheight{1.25}\smash{\begin{tabular}[t]{l}$g \otimes f$\end{tabular}}}}%
    \put(0.75821154,0.07124579){\color[rgb]{0,0,0}\makebox(0,0)[lt]{\lineheight{1.25}\smash{\begin{tabular}[t]{l}$=$\end{tabular}}}}%
    \put(0,0){\includegraphics[width=\unitlength,page=16]{reminderGraphicalCalculus.pdf}}%
    \put(0.80412126,0.03198799){\color[rgb]{0,0,0}\makebox(0,0)[lt]{\lineheight{1.25}\smash{\begin{tabular}[t]{l}$_{X'}$\end{tabular}}}}%
    \put(0,0){\includegraphics[width=\unitlength,page=17]{reminderGraphicalCalculus.pdf}}%
    \put(0.84652542,0.03198799){\color[rgb]{0,0,0}\makebox(0,0)[lt]{\lineheight{1.25}\smash{\begin{tabular}[t]{l}$_X$\end{tabular}}}}%
    \put(0,0){\includegraphics[width=\unitlength,page=18]{reminderGraphicalCalculus.pdf}}%
    \put(0.80708709,0.11850943){\color[rgb]{0,0,0}\makebox(0,0)[lt]{\lineheight{1.25}\smash{\begin{tabular}[t]{l}$_{Y'}$\end{tabular}}}}%
    \put(0,0){\includegraphics[width=\unitlength,page=19]{reminderGraphicalCalculus.pdf}}%
    \put(0.8494914,0.11850943){\color[rgb]{0,0,0}\makebox(0,0)[lt]{\lineheight{1.25}\smash{\begin{tabular}[t]{l}$_Y$\end{tabular}}}}%
    \put(0,0){\includegraphics[width=\unitlength,page=20]{reminderGraphicalCalculus.pdf}}%
    \put(0.93088827,0.07033591){\color[rgb]{0,0,0}\makebox(0,0)[lt]{\lineheight{1.25}\smash{\begin{tabular}[t]{l}$g \otimes f$\end{tabular}}}}%
    \put(0.88671099,0.07154867){\color[rgb]{0,0,0}\makebox(0,0)[lt]{\lineheight{1.25}\smash{\begin{tabular}[t]{l}$=$\end{tabular}}}}%
    \put(0,0){\includegraphics[width=\unitlength,page=21]{reminderGraphicalCalculus.pdf}}%
    \put(0.93216638,0.03024639){\color[rgb]{0,0,0}\makebox(0,0)[lt]{\lineheight{1.25}\smash{\begin{tabular}[t]{l}$_{X' \otimes X}$\end{tabular}}}}%
    \put(0,0){\includegraphics[width=\unitlength,page=22]{reminderGraphicalCalculus.pdf}}%
    \put(0.93142178,0.11941809){\color[rgb]{0,0,0}\makebox(0,0)[lt]{\lineheight{1.25}\smash{\begin{tabular}[t]{l}$_{Y' \otimes Y}$\end{tabular}}}}%
  \end{picture}%
\endgroup%

%% file: diagramsDuality.pdf_tex
\begingroup%
  \makeatletter%
  \providecommand\color[2][]{%
    \errmessage{(Inkscape) Color is used for the text in Inkscape, but the package 'color.sty' is not loaded}%
    \renewcommand\color[2][]{}%
  }%
  \providecommand\transparent[1]{%
    \errmessage{(Inkscape) Transparency is used (non-zero) for the text in Inkscape, but the package 'transparent.sty' is not loaded}%
    \renewcommand\transparent[1]{}%
  }%
  \providecommand\rotatebox[2]{#2}%
  \newcommand*\fsize{\dimexpr\f@size pt\relax}%
  \newcommand*\lineheight[1]{\fontsize{\fsize}{#1\fsize}\selectfont}%
  \ifx\svgwidth\undefined%
    \setlength{\unitlength}{391.75223123bp}%
    \ifx\svgscale\undefined%
      \relax%
    \else%
      \setlength{\unitlength}{\unitlength * \real{\svgscale}}%
    \fi%
  \else%
    \setlength{\unitlength}{\svgwidth}%
  \fi%
  \global\let\svgwidth\undefined%
  \global\let\svgscale\undefined%
  \makeatother%
  \begin{picture}(1,0.07357272)%
    \lineheight{1}%
    \setlength\tabcolsep{0pt}%
    \put(0,0){\includegraphics[width=\unitlength,page=1]{diagramsDuality.pdf}}%
    \put(0.22530651,0.06544694){\color[rgb]{0,0,0}\makebox(0,0)[lt]{\lineheight{1.25}\smash{\begin{tabular}[t]{l}$_X$\end{tabular}}}}%
    \put(0.27997376,0.06515563){\color[rgb]{0,0,0}\makebox(0,0)[lt]{\lineheight{1.25}\smash{\begin{tabular}[t]{l}$_{X^*}$\end{tabular}}}}%
    \put(0,0){\includegraphics[width=\unitlength,page=2]{diagramsDuality.pdf}}%
    \put(-0.00067761,0.00198519){\color[rgb]{0,0,0}\makebox(0,0)[lt]{\lineheight{1.25}\smash{\begin{tabular}[t]{l}$_{X^*}$\end{tabular}}}}%
    \put(0.06196706,0.00192521){\color[rgb]{0,0,0}\makebox(0,0)[lt]{\lineheight{1.25}\smash{\begin{tabular}[t]{l}$_X$\end{tabular}}}}%
    \put(0,0){\includegraphics[width=\unitlength,page=3]{diagramsDuality.pdf}}%
    \put(0.76592259,0.06589471){\color[rgb]{0,0,0}\makebox(0,0)[lt]{\lineheight{1.25}\smash{\begin{tabular}[t]{l}$_{^*\!X}$\end{tabular}}}}%
    \put(0.8285068,0.06579148){\color[rgb]{0,0,0}\makebox(0,0)[lt]{\lineheight{1.25}\smash{\begin{tabular}[t]{l}$_X$\end{tabular}}}}%
    \put(0.5497279,0.00286036){\color[rgb]{0,0,0}\makebox(0,0)[lt]{\lineheight{1.25}\smash{\begin{tabular}[t]{l}$_X$\end{tabular}}}}%
    \put(0.60439519,0.00256901){\color[rgb]{0,0,0}\makebox(0,0)[lt]{\lineheight{1.25}\smash{\begin{tabular}[t]{l}$_{^*\!X}$\end{tabular}}}}%
    \put(0.08705566,0.03410017){\color[rgb]{0,0,0}\makebox(0,0)[lt]{\lineheight{1.25}\smash{\begin{tabular}[t]{l}$= \mathrm{ev}_X$\end{tabular}}}}%
    \put(0.30750945,0.03391078){\color[rgb]{0,0,0}\makebox(0,0)[lt]{\lineheight{1.25}\smash{\begin{tabular}[t]{l}$= \mathrm{coev}_X$\end{tabular}}}}%
    \put(0.6330177,0.03419744){\color[rgb]{0,0,0}\makebox(0,0)[lt]{\lineheight{1.25}\smash{\begin{tabular}[t]{l}$= \widetilde{\mathrm{ev}}_X$\end{tabular}}}}%
    \put(0.85282675,0.03427569){\color[rgb]{0,0,0}\makebox(0,0)[lt]{\lineheight{1.25}\smash{\begin{tabular}[t]{l}$= \widetilde{\mathrm{coev}}_X$)\end{tabular}}}}%
    \put(0.46547987,0.03275879){\color[rgb]{0,0,0}\makebox(0,0)[lt]{\lineheight{1.25}\smash{\begin{tabular}[t]{l}(resp.\end{tabular}}}}%
  \end{picture}%
\endgroup%

%% file: dualEndCoend.pdf_tex
\begingroup%
  \makeatletter%
  \providecommand\color[2][]{%
    \errmessage{(Inkscape) Color is used for the text in Inkscape, but the package 'color.sty' is not loaded}%
    \renewcommand\color[2][]{}%
  }%
  \providecommand\transparent[1]{%
    \errmessage{(Inkscape) Transparency is used (non-zero) for the text in Inkscape, but the package 'transparent.sty' is not loaded}%
    \renewcommand\transparent[1]{}%
  }%
  \providecommand\rotatebox[2]{#2}%
  \newcommand*\fsize{\dimexpr\f@size pt\relax}%
  \newcommand*\lineheight[1]{\fontsize{\fsize}{#1\fsize}\selectfont}%
  \ifx\svgwidth\undefined%
    \setlength{\unitlength}{351.40979686bp}%
    \ifx\svgscale\undefined%
      \relax%
    \else%
      \setlength{\unitlength}{\unitlength * \real{\svgscale}}%
    \fi%
  \else%
    \setlength{\unitlength}{\svgwidth}%
  \fi%
  \global\let\svgwidth\undefined%
  \global\let\svgscale\undefined%
  \makeatother%
  \begin{picture}(1,0.2295918)%
    \lineheight{1}%
    \setlength\tabcolsep{0pt}%
    \put(0,0){\includegraphics[width=\unitlength,page=1]{dualEndCoend.pdf}}%
    \put(0.08026396,0.14640207){\color[rgb]{0,0,0}\makebox(0,0)[lt]{\lineheight{1.25}\smash{\begin{tabular}[t]{l}$e$\end{tabular}}}}%
    \put(0.23603178,0.11457921){\color[rgb]{0,0,0}\makebox(0,0)[lt]{\lineheight{1.25}\smash{\begin{tabular}[t]{l}$\overset{\eqref{defEvEndCoend}}{=}$\end{tabular}}}}%
    \put(0,0){\includegraphics[width=\unitlength,page=2]{dualEndCoend.pdf}}%
    \put(0.05006787,0.11075255){\color[rgb]{0,0,0}\makebox(0,0)[lt]{\lineheight{1.25}\smash{\begin{tabular}[t]{l}$_{\mathcal{L}}$\end{tabular}}}}%
    \put(0,0){\includegraphics[width=\unitlength,page=3]{dualEndCoend.pdf}}%
    \put(0.03562946,0.06560008){\color[rgb]{0,0,0}\makebox(0,0)[lt]{\lineheight{1.25}\smash{\begin{tabular}[t]{l}$i_X$\end{tabular}}}}%
    \put(0,0){\includegraphics[width=\unitlength,page=4]{dualEndCoend.pdf}}%
    \put(0.01464264,0.01369989){\color[rgb]{0,0,0}\makebox(0,0)[lt]{\lineheight{1.25}\smash{\begin{tabular}[t]{l}$_{X^*}$\end{tabular}}}}%
    \put(0,0){\includegraphics[width=\unitlength,page=5]{dualEndCoend.pdf}}%
    \put(0.06799906,0.01369994){\color[rgb]{0,0,0}\makebox(0,0)[lt]{\lineheight{1.25}\smash{\begin{tabular}[t]{l}$_X$\end{tabular}}}}%
    \put(0,0){\includegraphics[width=\unitlength,page=6]{dualEndCoend.pdf}}%
    \put(0.16013024,0.07086627){\color[rgb]{0,0,0}\makebox(0,0)[lt]{\lineheight{1.25}\smash{\begin{tabular}[t]{l}$c$\end{tabular}}}}%
    \put(0,0){\includegraphics[width=\unitlength,page=7]{dualEndCoend.pdf}}%
    \put(0.1354383,0.11075255){\color[rgb]{0,0,0}\makebox(0,0)[lt]{\lineheight{1.25}\smash{\begin{tabular}[t]{l}$_{\mathcal{A}_{\mathcal{C}}}$\end{tabular}}}}%
    \put(0,0){\includegraphics[width=\unitlength,page=8]{dualEndCoend.pdf}}%
    \put(0.19908916,0.19849259){\color[rgb]{0,0,0}\makebox(0,0)[lt]{\lineheight{1.25}\smash{\begin{tabular}[t]{l}$_{\mathcal{L}}$\end{tabular}}}}%
    \put(0,0){\includegraphics[width=\unitlength,page=9]{dualEndCoend.pdf}}%
    \put(0.39149635,0.12021114){\color[rgb]{0,0,0}\makebox(0,0)[lt]{\lineheight{1.25}\smash{\begin{tabular}[t]{l}$\pi_{^*\!X}$\end{tabular}}}}%
    \put(0,0){\includegraphics[width=\unitlength,page=10]{dualEndCoend.pdf}}%
    \put(0.30276781,0.00302855){\color[rgb]{0,0,0}\makebox(0,0)[lt]{\lineheight{1.25}\smash{\begin{tabular}[t]{l}$_{X^*}$\end{tabular}}}}%
    \put(0,0){\includegraphics[width=\unitlength,page=11]{dualEndCoend.pdf}}%
    \put(0.34545292,0.0030286){\color[rgb]{0,0,0}\makebox(0,0)[lt]{\lineheight{1.25}\smash{\begin{tabular}[t]{l}$_X$\end{tabular}}}}%
    \put(0,0){\includegraphics[width=\unitlength,page=12]{dualEndCoend.pdf}}%
    \put(0.44825545,0.03885234){\color[rgb]{0,0,0}\makebox(0,0)[lt]{\lineheight{1.25}\smash{\begin{tabular}[t]{l}$c$\end{tabular}}}}%
    \put(0,0){\includegraphics[width=\unitlength,page=13]{dualEndCoend.pdf}}%
    \put(0.42343339,0.08125931){\color[rgb]{0,0,0}\makebox(0,0)[lt]{\lineheight{1.25}\smash{\begin{tabular}[t]{l}$_{\mathcal{A}_{\mathcal{C}}}$\end{tabular}}}}%
    \put(0,0){\includegraphics[width=\unitlength,page=14]{dualEndCoend.pdf}}%
    \put(0.48721429,0.20916391){\color[rgb]{0,0,0}\makebox(0,0)[lt]{\lineheight{1.25}\smash{\begin{tabular}[t]{l}$_{\mathcal{L}}$\end{tabular}}}}%
    \put(0,0){\includegraphics[width=\unitlength,page=15]{dualEndCoend.pdf}}%
    \put(0.59089302,0.00302855){\color[rgb]{0,0,0}\makebox(0,0)[lt]{\lineheight{1.25}\smash{\begin{tabular}[t]{l}$_{X^*}$\end{tabular}}}}%
    \put(0,0){\includegraphics[width=\unitlength,page=16]{dualEndCoend.pdf}}%
    \put(0.63357813,0.0030286){\color[rgb]{0,0,0}\makebox(0,0)[lt]{\lineheight{1.25}\smash{\begin{tabular}[t]{l}$_X$\end{tabular}}}}%
    \put(0,0){\includegraphics[width=\unitlength,page=17]{dualEndCoend.pdf}}%
    \put(0.78019568,0.16083101){\color[rgb]{0,0,0}\makebox(0,0)[lt]{\lineheight{1.25}\smash{\begin{tabular}[t]{l}$i_X$\end{tabular}}}}%
    \put(0,0){\includegraphics[width=\unitlength,page=18]{dualEndCoend.pdf}}%
    \put(0.78601084,0.2198352){\color[rgb]{0,0,0}\makebox(0,0)[lt]{\lineheight{1.25}\smash{\begin{tabular}[t]{l}$_{\mathcal{L}}$\end{tabular}}}}%
    \put(0.85542509,0.11192053){\color[rgb]{0,0,0}\makebox(0,0)[lt]{\lineheight{1.25}\smash{\begin{tabular}[t]{l}$=$\end{tabular}}}}%
    \put(0,0){\includegraphics[width=\unitlength,page=19]{dualEndCoend.pdf}}%
    \put(0.94026524,0.10747448){\color[rgb]{0,0,0}\makebox(0,0)[lt]{\lineheight{1.25}\smash{\begin{tabular}[t]{l}$i_X$\end{tabular}}}}%
    \put(0,0){\includegraphics[width=\unitlength,page=20]{dualEndCoend.pdf}}%
    \put(0.94608029,0.16647866){\color[rgb]{0,0,0}\makebox(0,0)[lt]{\lineheight{1.25}\smash{\begin{tabular}[t]{l}$_{\mathcal{L}}$\end{tabular}}}}%
    \put(0,0){\includegraphics[width=\unitlength,page=21]{dualEndCoend.pdf}}%
    \put(0.92170343,0.05638506){\color[rgb]{0,0,0}\makebox(0,0)[lt]{\lineheight{1.25}\smash{\begin{tabular}[t]{l}$_{X^*}$\end{tabular}}}}%
    \put(0,0){\includegraphics[width=\unitlength,page=22]{dualEndCoend.pdf}}%
    \put(0.96438837,0.05638506){\color[rgb]{0,0,0}\makebox(0,0)[lt]{\lineheight{1.25}\smash{\begin{tabular}[t]{l}$_X$\end{tabular}}}}%
    \put(0.52712821,0.11436827){\color[rgb]{0,0,0}\makebox(0,0)[lt]{\lineheight{1.25}\smash{\begin{tabular}[t]{l}$\overset{\eqref{defCoevEndCoend}}{=}$\end{tabular}}}}%
  \end{picture}%
\endgroup%

%% file: EvEndCoendZC.pdf_tex
\begingroup%
  \makeatletter%
  \providecommand\color[2][]{%
    \errmessage{(Inkscape) Color is used for the text in Inkscape, but the package 'color.sty' is not loaded}%
    \renewcommand\color[2][]{}%
  }%
  \providecommand\transparent[1]{%
    \errmessage{(Inkscape) Transparency is used (non-zero) for the text in Inkscape, but the package 'transparent.sty' is not loaded}%
    \renewcommand\transparent[1]{}%
  }%
  \providecommand\rotatebox[2]{#2}%
  \newcommand*\fsize{\dimexpr\f@size pt\relax}%
  \newcommand*\lineheight[1]{\fontsize{\fsize}{#1\fsize}\selectfont}%
  \ifx\svgwidth\undefined%
    \setlength{\unitlength}{322.58661875bp}%
    \ifx\svgscale\undefined%
      \relax%
    \else%
      \setlength{\unitlength}{\unitlength * \real{\svgscale}}%
    \fi%
  \else%
    \setlength{\unitlength}{\svgwidth}%
  \fi%
  \global\let\svgwidth\undefined%
  \global\let\svgscale\undefined%
  \makeatother%
  \begin{picture}(1,0.53202034)%
    \lineheight{1}%
    \setlength\tabcolsep{0pt}%
    \put(0,0){\includegraphics[width=\unitlength,page=1]{EvEndCoendZC.pdf}}%
    \put(0.19231022,0.47380169){\color[rgb]{0,0,0}\makebox(0,0)[lt]{\lineheight{1.25}\smash{\begin{tabular}[t]{l}$e$\end{tabular}}}}%
    \put(0,0){\includegraphics[width=\unitlength,page=2]{EvEndCoendZC.pdf}}%
    \put(0.17148513,0.43462465){\color[rgb]{0,0,0}\makebox(0,0)[lt]{\lineheight{1.25}\smash{\begin{tabular}[t]{l}$_{\mathcal{L}}$\end{tabular}}}}%
    \put(0,0){\includegraphics[width=\unitlength,page=3]{EvEndCoendZC.pdf}}%
    \put(0.11260548,0.38538249){\color[rgb]{0,0,0}\makebox(0,0)[lt]{\lineheight{1.25}\smash{\begin{tabular}[t]{l}$t_X$\end{tabular}}}}%
    \put(0,0){\includegraphics[width=\unitlength,page=4]{EvEndCoendZC.pdf}}%
    \put(0.08600794,0.29243951){\color[rgb]{0,0,0}\makebox(0,0)[lt]{\lineheight{1.25}\smash{\begin{tabular}[t]{l}$i_Y$\end{tabular}}}}%
    \put(0,0){\includegraphics[width=\unitlength,page=5]{EvEndCoendZC.pdf}}%
    \put(0.06314593,0.23590204){\color[rgb]{0,0,0}\makebox(0,0)[lt]{\lineheight{1.25}\smash{\begin{tabular}[t]{l}$_{Y^*}$\end{tabular}}}}%
    \put(0,0){\includegraphics[width=\unitlength,page=6]{EvEndCoendZC.pdf}}%
    \put(0.10964497,0.23590211){\color[rgb]{0,0,0}\makebox(0,0)[lt]{\lineheight{1.25}\smash{\begin{tabular}[t]{l}$_Y$\end{tabular}}}}%
    \put(0,0){\includegraphics[width=\unitlength,page=7]{EvEndCoendZC.pdf}}%
    \put(0.10173641,0.3416264){\color[rgb]{0,0,0}\makebox(0,0)[lt]{\lineheight{1.25}\smash{\begin{tabular}[t]{l}$_{\mathcal{L}}$\end{tabular}}}}%
    \put(0,0){\includegraphics[width=\unitlength,page=8]{EvEndCoendZC.pdf}}%
    \put(0.17148515,0.3416264){\color[rgb]{0,0,0}\makebox(0,0)[lt]{\lineheight{1.25}\smash{\begin{tabular}[t]{l}$_X$\end{tabular}}}}%
    \put(0,0){\includegraphics[width=\unitlength,page=9]{EvEndCoendZC.pdf}}%
    \put(0.19615123,0.38735212){\color[rgb]{0,0,0}\makebox(0,0)[lt]{\lineheight{1.25}\smash{\begin{tabular}[t]{l}$_{\mathcal{A}_{\mathcal{C}}}$\end{tabular}}}}%
    \put(0,0){\includegraphics[width=\unitlength,page=10]{EvEndCoendZC.pdf}}%
    \put(0.18602075,0.29194631){\color[rgb]{0,0,0}\makebox(0,0)[lt]{\lineheight{1.25}\smash{\begin{tabular}[t]{l}$b_X$\end{tabular}}}}%
    \put(0,0){\includegraphics[width=\unitlength,page=11]{EvEndCoendZC.pdf}}%
    \put(0.15641162,0.23590208){\color[rgb]{0,0,0}\makebox(0,0)[lt]{\lineheight{1.25}\smash{\begin{tabular}[t]{l}$_{\mathcal{A}_{\mathcal{C}}}$\end{tabular}}}}%
    \put(0,0){\includegraphics[width=\unitlength,page=12]{EvEndCoendZC.pdf}}%
    \put(0.22589283,0.23590211){\color[rgb]{0,0,0}\makebox(0,0)[lt]{\lineheight{1.25}\smash{\begin{tabular}[t]{l}$_X$\end{tabular}}}}%
    \put(0,0){\includegraphics[width=\unitlength,page=13]{EvEndCoendZC.pdf}}%
    \put(0.08970102,0.5185795){\color[rgb]{0,0,0}\makebox(0,0)[lt]{\lineheight{1.25}\smash{\begin{tabular}[t]{l}$_X$\end{tabular}}}}%
    \put(0,0){\includegraphics[width=\unitlength,page=14]{EvEndCoendZC.pdf}}%
    \put(0.4945546,0.47380169){\color[rgb]{0,0,0}\makebox(0,0)[lt]{\lineheight{1.25}\smash{\begin{tabular}[t]{l}$e$\end{tabular}}}}%
    \put(0,0){\includegraphics[width=\unitlength,page=15]{EvEndCoendZC.pdf}}%
    \put(0.43885522,0.43462464){\color[rgb]{0,0,0}\makebox(0,0)[lt]{\lineheight{1.25}\smash{\begin{tabular}[t]{l}$_{\mathcal{L}}$\end{tabular}}}}%
    \put(0,0){\includegraphics[width=\unitlength,page=16]{EvEndCoendZC.pdf}}%
    \put(0.40231262,0.38667416){\color[rgb]{0,0,0}\makebox(0,0)[lt]{\lineheight{1.25}\smash{\begin{tabular}[t]{l}$i_{Y \otimes X}$\end{tabular}}}}%
    \put(0,0){\includegraphics[width=\unitlength,page=17]{EvEndCoendZC.pdf}}%
    \put(0.40026471,0.23590199){\color[rgb]{0,0,0}\makebox(0,0)[lt]{\lineheight{1.25}\smash{\begin{tabular}[t]{l}$_{Y^*}$\end{tabular}}}}%
    \put(0,0){\includegraphics[width=\unitlength,page=18]{EvEndCoendZC.pdf}}%
    \put(0.44676378,0.23590209){\color[rgb]{0,0,0}\makebox(0,0)[lt]{\lineheight{1.25}\smash{\begin{tabular}[t]{l}$_Y$\end{tabular}}}}%
    \put(0,0){\includegraphics[width=\unitlength,page=19]{EvEndCoendZC.pdf}}%
    \put(0.50860396,0.34162637){\color[rgb]{0,0,0}\makebox(0,0)[lt]{\lineheight{1.25}\smash{\begin{tabular}[t]{l}$_X$\end{tabular}}}}%
    \put(0,0){\includegraphics[width=\unitlength,page=20]{EvEndCoendZC.pdf}}%
    \put(0.53216176,0.38583333){\color[rgb]{0,0,0}\makebox(0,0)[lt]{\lineheight{1.25}\smash{\begin{tabular}[t]{l}$_{\mathcal{A}_{\mathcal{C}}}$\end{tabular}}}}%
    \put(0,0){\includegraphics[width=\unitlength,page=21]{EvEndCoendZC.pdf}}%
    \put(0.52396456,0.29164666){\color[rgb]{0,0,0}\makebox(0,0)[lt]{\lineheight{1.25}\smash{\begin{tabular}[t]{l}$b_X$\end{tabular}}}}%
    \put(0,0){\includegraphics[width=\unitlength,page=22]{EvEndCoendZC.pdf}}%
    \put(0.49353041,0.23590204){\color[rgb]{0,0,0}\makebox(0,0)[lt]{\lineheight{1.25}\smash{\begin{tabular}[t]{l}$_{\mathcal{A}_{\mathcal{C}}}$\end{tabular}}}}%
    \put(0,0){\includegraphics[width=\unitlength,page=23]{EvEndCoendZC.pdf}}%
    \put(0.56301158,0.23590209){\color[rgb]{0,0,0}\makebox(0,0)[lt]{\lineheight{1.25}\smash{\begin{tabular}[t]{l}$_X$\end{tabular}}}}%
    \put(0,0){\includegraphics[width=\unitlength,page=24]{EvEndCoendZC.pdf}}%
    \put(0.32219675,0.51857951){\color[rgb]{0,0,0}\makebox(0,0)[lt]{\lineheight{1.25}\smash{\begin{tabular}[t]{l}$_X$\end{tabular}}}}%
    \put(0,0){\includegraphics[width=\unitlength,page=25]{EvEndCoendZC.pdf}}%
    \put(0.26249779,0.38796485){\color[rgb]{0,0,0}\makebox(0,0)[lt]{\lineheight{1.25}\smash{\begin{tabular}[t]{l}$\overset{\eqref{diagramHalfBrCoend}}{=}$\end{tabular}}}}%
    \put(0,0){\includegraphics[width=\unitlength,page=26]{EvEndCoendZC.pdf}}%
    \put(0.85747621,0.38617002){\color[rgb]{0,0,0}\makebox(0,0)[lt]{\lineheight{1.25}\smash{\begin{tabular}[t]{l}$\pi_{^*\!X \,\otimes \,{^*Y}}$\end{tabular}}}}%
    \put(0,0){\includegraphics[width=\unitlength,page=27]{EvEndCoendZC.pdf}}%
    \put(0.82247316,0.3416264){\color[rgb]{0,0,0}\makebox(0,0)[lt]{\lineheight{1.25}\smash{\begin{tabular}[t]{l}$_X$\end{tabular}}}}%
    \put(0,0){\includegraphics[width=\unitlength,page=28]{EvEndCoendZC.pdf}}%
    \put(0.84825774,0.29109819){\color[rgb]{0,0,0}\makebox(0,0)[lt]{\lineheight{1.25}\smash{\begin{tabular}[t]{l}$b_X$\end{tabular}}}}%
    \put(0,0){\includegraphics[width=\unitlength,page=29]{EvEndCoendZC.pdf}}%
    \put(0.81902429,0.23590208){\color[rgb]{0,0,0}\makebox(0,0)[lt]{\lineheight{1.25}\smash{\begin{tabular}[t]{l}$_{\mathcal{A}_{\mathcal{C}}}$\end{tabular}}}}%
    \put(0,0){\includegraphics[width=\unitlength,page=30]{EvEndCoendZC.pdf}}%
    \put(0.23778494,0.00340642){\color[rgb]{0,0,0}\makebox(0,0)[lt]{\lineheight{1.25}\smash{\begin{tabular}[t]{l}$_{\mathcal{A}_{\mathcal{C}}}$\end{tabular}}}}%
    \put(0,0){\includegraphics[width=\unitlength,page=31]{EvEndCoendZC.pdf}}%
    \put(0.88850555,0.23590216){\color[rgb]{0,0,0}\makebox(0,0)[lt]{\lineheight{1.25}\smash{\begin{tabular}[t]{l}$_X$\end{tabular}}}}%
    \put(0,0){\includegraphics[width=\unitlength,page=32]{EvEndCoendZC.pdf}}%
    \put(0.91530318,0.34184155){\color[rgb]{0,0,0}\makebox(0,0)[lt]{\lineheight{1.25}\smash{\begin{tabular}[t]{l}$_{\mathcal{A}_{\mathcal{C}}}$\end{tabular}}}}%
    \put(0,0){\includegraphics[width=\unitlength,page=33]{EvEndCoendZC.pdf}}%
    \put(0.65931557,0.5185795){\color[rgb]{0,0,0}\makebox(0,0)[lt]{\lineheight{1.25}\smash{\begin{tabular}[t]{l}$_X$\end{tabular}}}}%
    \put(0,0){\includegraphics[width=\unitlength,page=34]{EvEndCoendZC.pdf}}%
    \put(0.72525802,0.23579477){\color[rgb]{0,0,0}\makebox(0,0)[lt]{\lineheight{1.25}\smash{\begin{tabular}[t]{l}$_{Y^*}$\end{tabular}}}}%
    \put(0,0){\includegraphics[width=\unitlength,page=35]{EvEndCoendZC.pdf}}%
    \put(0.77225779,0.23590204){\color[rgb]{0,0,0}\makebox(0,0)[lt]{\lineheight{1.25}\smash{\begin{tabular}[t]{l}$_Y$\end{tabular}}}}%
    \put(0.5993203,0.38771077){\color[rgb]{0,0,0}\makebox(0,0)[lt]{\lineheight{1.25}\smash{\begin{tabular}[t]{l}$\overset{\eqref{defEvEndCoend}}{=}$\end{tabular}}}}%
    \put(0,0){\includegraphics[width=\unitlength,page=36]{EvEndCoendZC.pdf}}%
    \put(0.15641164,0.06080967){\color[rgb]{0,0,0}\makebox(0,0)[lt]{\lineheight{1.25}\smash{\begin{tabular}[t]{l}$\pi_{X \,\otimes\,{^*\! X} \,\otimes\, {^*Y}}$\end{tabular}}}}%
    \put(0,0){\includegraphics[width=\unitlength,page=37]{EvEndCoendZC.pdf}}%
    \put(0.292877,0.19243862){\color[rgb]{0,0,0}\makebox(0,0)[lt]{\lineheight{1.25}\smash{\begin{tabular}[t]{l}$_X$\end{tabular}}}}%
    \put(0,0){\includegraphics[width=\unitlength,page=38]{EvEndCoendZC.pdf}}%
    \put(0.36539007,0.00340645){\color[rgb]{0,0,0}\makebox(0,0)[lt]{\lineheight{1.25}\smash{\begin{tabular}[t]{l}$_X$\end{tabular}}}}%
    \put(0,0){\includegraphics[width=\unitlength,page=39]{EvEndCoendZC.pdf}}%
    \put(0.05102039,0.00329916){\color[rgb]{0,0,0}\makebox(0,0)[lt]{\lineheight{1.25}\smash{\begin{tabular}[t]{l}$_{Y^*}$\end{tabular}}}}%
    \put(0,0){\includegraphics[width=\unitlength,page=40]{EvEndCoendZC.pdf}}%
    \put(0.09802021,0.00340645){\color[rgb]{0,0,0}\makebox(0,0)[lt]{\lineheight{1.25}\smash{\begin{tabular}[t]{l}$_Y$\end{tabular}}}}%
    \put(-0.0002989,0.074661){\color[rgb]{0,0,0}\makebox(0,0)[lt]{\lineheight{1.25}\smash{\begin{tabular}[t]{l}$\overset{\eqref{diagramHalfBrEnd}}{=}$\end{tabular}}}}%
    \put(0.12780327,0.10534008){\color[rgb]{0,0,0}\makebox(0,0)[lt]{\lineheight{1.25}\smash{\begin{tabular}[t]{l}$_X$\end{tabular}}}}%
    \put(0,0){\includegraphics[width=\unitlength,page=41]{EvEndCoendZC.pdf}}%
    \put(0.57490365,0.00340635){\color[rgb]{0,0,0}\makebox(0,0)[lt]{\lineheight{1.25}\smash{\begin{tabular}[t]{l}$_{\mathcal{A}_{\mathcal{C}}}$\end{tabular}}}}%
    \put(0,0){\includegraphics[width=\unitlength,page=42]{EvEndCoendZC.pdf}}%
    \put(0.56417373,0.06224934){\color[rgb]{0,0,0}\makebox(0,0)[lt]{\lineheight{1.25}\smash{\begin{tabular}[t]{l}$\pi_{^*Y}$\end{tabular}}}}%
    \put(0,0){\includegraphics[width=\unitlength,page=43]{EvEndCoendZC.pdf}}%
    \put(0.7141336,0.00340649){\color[rgb]{0,0,0}\makebox(0,0)[lt]{\lineheight{1.25}\smash{\begin{tabular}[t]{l}$_X$\end{tabular}}}}%
    \put(0,0){\includegraphics[width=\unitlength,page=44]{EvEndCoendZC.pdf}}%
    \put(0.43463832,0.00329916){\color[rgb]{0,0,0}\makebox(0,0)[lt]{\lineheight{1.25}\smash{\begin{tabular}[t]{l}$_{Y^*}$\end{tabular}}}}%
    \put(0,0){\includegraphics[width=\unitlength,page=45]{EvEndCoendZC.pdf}}%
    \put(0.50488775,0.00340645){\color[rgb]{0,0,0}\makebox(0,0)[lt]{\lineheight{1.25}\smash{\begin{tabular}[t]{l}$_Y$\end{tabular}}}}%
    \put(0,0){\includegraphics[width=\unitlength,page=46]{EvEndCoendZC.pdf}}%
    \put(0.62444108,0.16983604){\color[rgb]{0,0,0}\makebox(0,0)[lt]{\lineheight{1.25}\smash{\begin{tabular}[t]{l}$_X$\end{tabular}}}}%
    \put(0.39969857,0.07518556){\color[rgb]{0,0,0}\makebox(0,0)[lt]{\lineheight{1.25}\smash{\begin{tabular}[t]{l}$=$\end{tabular}}}}%
    \put(0,0){\includegraphics[width=\unitlength,page=47]{EvEndCoendZC.pdf}}%
    \put(0.85324365,0.05994376){\color[rgb]{0,0,0}\makebox(0,0)[lt]{\lineheight{1.25}\smash{\begin{tabular}[t]{l}$i_Y$\end{tabular}}}}%
    \put(0,0){\includegraphics[width=\unitlength,page=48]{EvEndCoendZC.pdf}}%
    \put(0.83038167,0.0034063){\color[rgb]{0,0,0}\makebox(0,0)[lt]{\lineheight{1.25}\smash{\begin{tabular}[t]{l}$_{Y^*}$\end{tabular}}}}%
    \put(0,0){\includegraphics[width=\unitlength,page=49]{EvEndCoendZC.pdf}}%
    \put(0.87688079,0.0034064){\color[rgb]{0,0,0}\makebox(0,0)[lt]{\lineheight{1.25}\smash{\begin{tabular}[t]{l}$_Y$\end{tabular}}}}%
    \put(0,0){\includegraphics[width=\unitlength,page=50]{EvEndCoendZC.pdf}}%
    \put(0.86897219,0.10913065){\color[rgb]{0,0,0}\makebox(0,0)[lt]{\lineheight{1.25}\smash{\begin{tabular}[t]{l}$_{\mathcal{L}}$\end{tabular}}}}%
    \put(0,0){\includegraphics[width=\unitlength,page=51]{EvEndCoendZC.pdf}}%
    \put(0.88800487,0.14830774){\color[rgb]{0,0,0}\makebox(0,0)[lt]{\lineheight{1.25}\smash{\begin{tabular}[t]{l}$e$\end{tabular}}}}%
    \put(0,0){\includegraphics[width=\unitlength,page=52]{EvEndCoendZC.pdf}}%
    \put(0.92364736,0.00340628){\color[rgb]{0,0,0}\makebox(0,0)[lt]{\lineheight{1.25}\smash{\begin{tabular}[t]{l}$_{\mathcal{A}_{\mathcal{C}}}$\end{tabular}}}}%
    \put(0,0){\includegraphics[width=\unitlength,page=53]{EvEndCoendZC.pdf}}%
    \put(0.98150369,0.00340645){\color[rgb]{0,0,0}\makebox(0,0)[lt]{\lineheight{1.25}\smash{\begin{tabular}[t]{l}$_X$\end{tabular}}}}%
    \put(0.9848096,0.19308554){\color[rgb]{0,0,0}\makebox(0,0)[lt]{\lineheight{1.25}\smash{\begin{tabular}[t]{l}$_X$\end{tabular}}}}%
    \put(0.74706438,0.07495786){\color[rgb]{0,0,0}\makebox(0,0)[lt]{\lineheight{1.25}\smash{\begin{tabular}[t]{l}$\overset{\eqref{defEvEndCoend}}{=}$\end{tabular}}}}%
  \end{picture}%
\endgroup%

%% file: proofJ.pdf_tex
\begingroup%
  \makeatletter%
  \providecommand\color[2][]{%
    \errmessage{(Inkscape) Color is used for the text in Inkscape, but the package 'color.sty' is not loaded}%
    \renewcommand\color[2][]{}%
  }%
  \providecommand\transparent[1]{%
    \errmessage{(Inkscape) Transparency is used (non-zero) for the text in Inkscape, but the package 'transparent.sty' is not loaded}%
    \renewcommand\transparent[1]{}%
  }%
  \providecommand\rotatebox[2]{#2}%
  \newcommand*\fsize{\dimexpr\f@size pt\relax}%
  \newcommand*\lineheight[1]{\fontsize{\fsize}{#1\fsize}\selectfont}%
  \ifx\svgwidth\undefined%
    \setlength{\unitlength}{433.43940825bp}%
    \ifx\svgscale\undefined%
      \relax%
    \else%
      \setlength{\unitlength}{\unitlength * \real{\svgscale}}%
    \fi%
  \else%
    \setlength{\unitlength}{\svgwidth}%
  \fi%
  \global\let\svgwidth\undefined%
  \global\let\svgscale\undefined%
  \makeatother%
  \begin{picture}(1,0.62167348)%
    \lineheight{1}%
    \setlength\tabcolsep{0pt}%
    \put(0,0){\includegraphics[width=\unitlength,page=1]{proofJ.pdf}}%
    \put(0.4899069,0.56559323){\color[rgb]{0,0,0}\makebox(0,0)[lt]{\lineheight{1.25}\smash{\begin{tabular}[t]{l}${\scriptstyle \underline{\mathsf{ev}}_{M,N}}$\end{tabular}}}}%
    \put(0,0){\includegraphics[width=\unitlength,page=2]{proofJ.pdf}}%
    \put(0.45749668,0.52930571){\color[rgb]{0,0,0}\makebox(0,0)[lt]{\lineheight{1.25}\smash{\begin{tabular}[t]{l}${\scriptscriptstyle \underline{\Hom}(M,N)}$\end{tabular}}}}%
    \put(0,0){\includegraphics[width=\unitlength,page=3]{proofJ.pdf}}%
    \put(0.34662306,0.49646097){\color[rgb]{0,0,0}\makebox(0,0)[lt]{\lineheight{1.25}\smash{\begin{tabular}[t]{l}${\scriptstyle \underline{\Hom}(\mathrm{id}_M,\mathrm{ev}_Y \rhd \mathrm{id}_N)}$\end{tabular}}}}%
    \put(0,0){\includegraphics[width=\unitlength,page=4]{proofJ.pdf}}%
    \put(0.51749903,0.61376332){\color[rgb]{0,0,0}\makebox(0,0)[lt]{\lineheight{1.25}\smash{\begin{tabular}[t]{l}$_N$\end{tabular}}}}%
    \put(0.58967556,0.29664266){\color[rgb]{0,0,0}\makebox(0,0)[lt]{\lineheight{1.25}\smash{\begin{tabular}[t]{l}$_{M}$\end{tabular}}}}%
    \put(0.37309644,0.29715701){\color[rgb]{0,0,0}\makebox(0,0)[lt]{\lineheight{1.25}\smash{\begin{tabular}[t]{l}$_{\underline{\Hom}(XM,YN)}$\end{tabular}}}}%
    \put(0,0){\includegraphics[width=\unitlength,page=5]{proofJ.pdf}}%
    \put(0.6220559,0.45342244){\color[rgb]{0,0,0}\makebox(0,0)[lt]{\lineheight{1.25}\smash{\begin{tabular}[t]{l}$=$\end{tabular}}}}%
    \put(0,0){\includegraphics[width=\unitlength,page=6]{proofJ.pdf}}%
    \put(0.35139221,0.38845281){\color[rgb]{0,0,0}\makebox(0,0)[lt]{\lineheight{1.25}\smash{\begin{tabular}[t]{l}${\scriptscriptstyle \underline{\Hom}(M, \,Y^* \underline{\Hom}(XM,YN) X M)}$\end{tabular}}}}%
    \put(0,0){\includegraphics[width=\unitlength,page=7]{proofJ.pdf}}%
    \put(0.33171671,0.35766704){\color[rgb]{0,0,0}\makebox(0,0)[lt]{\lineheight{1.25}\smash{\begin{tabular}[t]{l}${\scriptstyle \underline{\mathsf{coev}}_{Y^* \underline{\Hom}(XM,YN) X, M}}$\end{tabular}}}}%
    \put(0,0){\includegraphics[width=\unitlength,page=8]{proofJ.pdf}}%
    \put(0.44019772,0.46009188){\color[rgb]{0,0,0}\makebox(0,0)[lt]{\lineheight{1.25}\smash{\begin{tabular}[t]{l}${\scriptscriptstyle \underline{\Hom}(M,Y^*YN)}$\end{tabular}}}}%
    \put(0,0){\includegraphics[width=\unitlength,page=9]{proofJ.pdf}}%
    \put(0.32436186,0.42636991){\color[rgb]{0,0,0}\makebox(0,0)[lt]{\lineheight{1.25}\smash{\begin{tabular}[t]{l}${\scriptstyle \underline{\Hom}(\mathrm{id}_M,\mathrm{id}_Y \rhd \underline{\mathsf{ev}}_{XM,YN})}$\end{tabular}}}}%
    \put(0,0){\includegraphics[width=\unitlength,page=10]{proofJ.pdf}}%
    \put(0.2925541,0.61376331){\color[rgb]{0,0,0}\makebox(0,0)[lt]{\lineheight{1.25}\smash{\begin{tabular}[t]{l}$_Y$\end{tabular}}}}%
    \put(0,0){\includegraphics[width=\unitlength,page=11]{proofJ.pdf}}%
    \put(0.85825042,0.50518122){\color[rgb]{0,0,0}\makebox(0,0)[lt]{\lineheight{1.25}\smash{\begin{tabular}[t]{l}${\scriptstyle \underline{\mathsf{ev}}_{M,Y^*YN}}$\end{tabular}}}}%
    \put(0,0){\includegraphics[width=\unitlength,page=12]{proofJ.pdf}}%
    \put(0.94143373,0.5705047){\color[rgb]{0,0,0}\makebox(0,0)[lt]{\lineheight{1.25}\smash{\begin{tabular}[t]{l}$_N$\end{tabular}}}}%
    \put(0.96114445,0.30504743){\color[rgb]{0,0,0}\makebox(0,0)[lt]{\lineheight{1.25}\smash{\begin{tabular}[t]{l}$_{M}$\end{tabular}}}}%
    \put(0.74382809,0.30474189){\color[rgb]{0,0,0}\makebox(0,0)[lt]{\lineheight{1.25}\smash{\begin{tabular}[t]{l}$_{\underline{\Hom}(XM,YN)}$\end{tabular}}}}%
    \put(0,0){\includegraphics[width=\unitlength,page=13]{proofJ.pdf}}%
    \put(0.72285996,0.39692016){\color[rgb]{0,0,0}\makebox(0,0)[lt]{\lineheight{1.25}\smash{\begin{tabular}[t]{l}${\scriptscriptstyle \underline{\Hom}(M, \,Y^* \underline{\Hom}(XM,YN) X M)}$\end{tabular}}}}%
    \put(0,0){\includegraphics[width=\unitlength,page=14]{proofJ.pdf}}%
    \put(0.70245982,0.36546972){\color[rgb]{0,0,0}\makebox(0,0)[lt]{\lineheight{1.25}\smash{\begin{tabular}[t]{l}${\scriptstyle \underline{\mathsf{coev}}_{Y^* \underline{\Hom}(XM,YN) X, M}}$\end{tabular}}}}%
    \put(0,0){\includegraphics[width=\unitlength,page=15]{proofJ.pdf}}%
    \put(0.82952552,0.4687436){\color[rgb]{0,0,0}\makebox(0,0)[lt]{\lineheight{1.25}\smash{\begin{tabular}[t]{l}${\scriptscriptstyle \underline{\Hom}(M,Y^*YN)}$\end{tabular}}}}%
    \put(0,0){\includegraphics[width=\unitlength,page=16]{proofJ.pdf}}%
    \put(0.69632707,0.43483642){\color[rgb]{0,0,0}\makebox(0,0)[lt]{\lineheight{1.25}\smash{\begin{tabular}[t]{l}${\scriptstyle \underline{\Hom}(\mathrm{id}_M,\mathrm{id}_Y \rhd \underline{\mathsf{ev}}_{XM,YN})}$\end{tabular}}}}%
    \put(0,0){\includegraphics[width=\unitlength,page=17]{proofJ.pdf}}%
    \put(0.66457843,0.57050467){\color[rgb]{0,0,0}\makebox(0,0)[lt]{\lineheight{1.25}\smash{\begin{tabular}[t]{l}$_Y$\end{tabular}}}}%
    \put(0,0){\includegraphics[width=\unitlength,page=18]{proofJ.pdf}}%
    \put(0.89690064,0.54765355){\color[rgb]{0,0,0}\makebox(0,0)[lt]{\lineheight{1.25}\smash{\begin{tabular}[t]{l}$_Y$\end{tabular}}}}%
    \put(0,0){\includegraphics[width=\unitlength,page=19]{proofJ.pdf}}%
    \put(0.14101756,0.49628452){\color[rgb]{0,0,0}\makebox(0,0)[lt]{\lineheight{1.25}\smash{\begin{tabular}[t]{l}${\scriptstyle \underline{\mathsf{ev}}_{M,N}}$\end{tabular}}}}%
    \put(0,0){\includegraphics[width=\unitlength,page=20]{proofJ.pdf}}%
    \put(0.12030006,0.46009191){\color[rgb]{0,0,0}\makebox(0,0)[lt]{\lineheight{1.25}\smash{\begin{tabular}[t]{l}${\scriptscriptstyle \underline{\Hom}(M,N)}$\end{tabular}}}}%
    \put(0,0){\includegraphics[width=\unitlength,page=21]{proofJ.pdf}}%
    \put(0.16277818,0.54454951){\color[rgb]{0,0,0}\makebox(0,0)[lt]{\lineheight{1.25}\smash{\begin{tabular}[t]{l}$_N$\end{tabular}}}}%
    \put(0,0){\includegraphics[width=\unitlength,page=22]{proofJ.pdf}}%
    \put(0.06051594,0.42572326){\color[rgb]{0,0,0}\makebox(0,0)[lt]{\lineheight{1.25}\smash{\begin{tabular}[t]{l}${\scriptstyle J_{X,M,Y,N}}$\end{tabular}}}}%
    \put(0,0){\includegraphics[width=\unitlength,page=23]{proofJ.pdf}}%
    \put(0.52252841,0.29668362){\color[rgb]{0,0,0}\makebox(0,0)[lt]{\lineheight{1.25}\smash{\begin{tabular}[t]{l}$_X$\end{tabular}}}}%
    \put(0,0){\includegraphics[width=\unitlength,page=24]{proofJ.pdf}}%
    \put(-0.00039651,0.37462465){\color[rgb]{0,0,0}\makebox(0,0)[lt]{\lineheight{1.25}\smash{\begin{tabular}[t]{l}$_{\underline{\Hom}(XM,YN)}$\end{tabular}}}}%
    \put(0,0){\includegraphics[width=\unitlength,page=25]{proofJ.pdf}}%
    \put(0.14269772,0.37454918){\color[rgb]{0,0,0}\makebox(0,0)[lt]{\lineheight{1.25}\smash{\begin{tabular}[t]{l}$_X$\end{tabular}}}}%
    \put(0,0){\includegraphics[width=\unitlength,page=26]{proofJ.pdf}}%
    \put(0.05895744,0.54454949){\color[rgb]{0,0,0}\makebox(0,0)[lt]{\lineheight{1.25}\smash{\begin{tabular}[t]{l}$_Y$\end{tabular}}}}%
    \put(0,0){\includegraphics[width=\unitlength,page=27]{proofJ.pdf}}%
    \put(0.21064902,0.37458557){\color[rgb]{0,0,0}\makebox(0,0)[lt]{\lineheight{1.25}\smash{\begin{tabular}[t]{l}$_M$\end{tabular}}}}%
    \put(0.24973863,0.45312274){\color[rgb]{0,0,0}\makebox(0,0)[lt]{\lineheight{1.25}\smash{\begin{tabular}[t]{l}$=$\end{tabular}}}}%
    \put(0,0){\includegraphics[width=\unitlength,page=28]{proofJ.pdf}}%
    \put(0.89455277,0.30533534){\color[rgb]{0,0,0}\makebox(0,0)[lt]{\lineheight{1.25}\smash{\begin{tabular}[t]{l}$_X$\end{tabular}}}}%
    \put(0.0250883,0.15083045){\color[rgb]{0,0,0}\makebox(0,0)[lt]{\lineheight{1.25}\smash{\begin{tabular}[t]{l}$=$\end{tabular}}}}%
    \put(0,0){\includegraphics[width=\unitlength,page=29]{proofJ.pdf}}%
    \put(0.18203788,0.20174974){\color[rgb]{0,0,0}\makebox(0,0)[lt]{\lineheight{1.25}\smash{\begin{tabular}[t]{l}${\scriptstyle \underline{\mathsf{ev}}_{XM,YN}}$\end{tabular}}}}%
    \put(0,0){\includegraphics[width=\unitlength,page=30]{proofJ.pdf}}%
    \put(0.26659897,0.2674757){\color[rgb]{0,0,0}\makebox(0,0)[lt]{\lineheight{1.25}\smash{\begin{tabular}[t]{l}$_N$\end{tabular}}}}%
    \put(0.36417683,0.00245539){\color[rgb]{0,0,0}\makebox(0,0)[lt]{\lineheight{1.25}\smash{\begin{tabular}[t]{l}$_{M}$\end{tabular}}}}%
    \put(0.14541764,0.00288499){\color[rgb]{0,0,0}\makebox(0,0)[lt]{\lineheight{1.25}\smash{\begin{tabular}[t]{l}$_{\underline{\Hom}(XM,YN)}$\end{tabular}}}}%
    \put(0,0){\includegraphics[width=\unitlength,page=31]{proofJ.pdf}}%
    \put(0.12589235,0.09432815){\color[rgb]{0,0,0}\makebox(0,0)[lt]{\lineheight{1.25}\smash{\begin{tabular}[t]{l}${\scriptscriptstyle \underline{\Hom}(M, \,Y^* \underline{\Hom}(XM,YN) X M)}$\end{tabular}}}}%
    \put(0,0){\includegraphics[width=\unitlength,page=32]{proofJ.pdf}}%
    \put(0.1054922,0.06287773){\color[rgb]{0,0,0}\makebox(0,0)[lt]{\lineheight{1.25}\smash{\begin{tabular}[t]{l}${\scriptstyle \underline{\mathsf{coev}}_{Y^* \underline{\Hom}(XM,YN) X, M}}$\end{tabular}}}}%
    \put(0,0){\includegraphics[width=\unitlength,page=33]{proofJ.pdf}}%
    \put(0.23255791,0.16615157){\color[rgb]{0,0,0}\makebox(0,0)[lt]{\lineheight{1.25}\smash{\begin{tabular}[t]{l}${\scriptscriptstyle \underline{\Hom}(XM,YN)XM}$\end{tabular}}}}%
    \put(0,0){\includegraphics[width=\unitlength,page=34]{proofJ.pdf}}%
    \put(0.15079886,0.13472436){\color[rgb]{0,0,0}\makebox(0,0)[lt]{\lineheight{1.25}\smash{\begin{tabular}[t]{l}${\scriptstyle \underline{\mathsf{ev}}_{M, Y^*\underline{\Hom}(XM,YN)XM}}$\end{tabular}}}}%
    \put(0,0){\includegraphics[width=\unitlength,page=35]{proofJ.pdf}}%
    \put(0.06761082,0.26791264){\color[rgb]{0,0,0}\makebox(0,0)[lt]{\lineheight{1.25}\smash{\begin{tabular}[t]{l}$_Y$\end{tabular}}}}%
    \put(0,0){\includegraphics[width=\unitlength,page=36]{proofJ.pdf}}%
    \put(0.18746055,0.24506154){\color[rgb]{0,0,0}\makebox(0,0)[lt]{\lineheight{1.25}\smash{\begin{tabular}[t]{l}$_Y$\end{tabular}}}}%
    \put(0,0){\includegraphics[width=\unitlength,page=37]{proofJ.pdf}}%
    \put(0.29758516,0.00274334){\color[rgb]{0,0,0}\makebox(0,0)[lt]{\lineheight{1.25}\smash{\begin{tabular}[t]{l}$_X$\end{tabular}}}}%
    \put(0,0){\includegraphics[width=\unitlength,page=38]{proofJ.pdf}}%
    \put(0.4195363,0.15086647){\color[rgb]{0,0,0}\makebox(0,0)[lt]{\lineheight{1.25}\smash{\begin{tabular}[t]{l}$=$\end{tabular}}}}%
    \put(0,0){\includegraphics[width=\unitlength,page=39]{proofJ.pdf}}%
    \put(0.71024601,0.06301751){\color[rgb]{0,0,0}\makebox(0,0)[lt]{\lineheight{1.25}\smash{\begin{tabular}[t]{l}$_{M}$\end{tabular}}}}%
    \put(0.53474542,0.06344711){\color[rgb]{0,0,0}\makebox(0,0)[lt]{\lineheight{1.25}\smash{\begin{tabular}[t]{l}$_{\underline{\Hom}(XM,YN)}$\end{tabular}}}}%
    \put(0,0){\includegraphics[width=\unitlength,page=40]{proofJ.pdf}}%
    \put(0.60988867,0.14933661){\color[rgb]{0,0,0}\makebox(0,0)[lt]{\lineheight{1.25}\smash{\begin{tabular}[t]{l}${\scriptstyle \underline{\mathsf{ev}}_{X M, Y N}}$\end{tabular}}}}%
    \put(0,0){\includegraphics[width=\unitlength,page=41]{proofJ.pdf}}%
    \put(0.4655887,0.21578378){\color[rgb]{0,0,0}\makebox(0,0)[lt]{\lineheight{1.25}\smash{\begin{tabular}[t]{l}$_Y$\end{tabular}}}}%
    \put(0,0){\includegraphics[width=\unitlength,page=42]{proofJ.pdf}}%
    \put(0.67826116,0.06330546){\color[rgb]{0,0,0}\makebox(0,0)[lt]{\lineheight{1.25}\smash{\begin{tabular}[t]{l}$_X$\end{tabular}}}}%
    \put(0,0){\includegraphics[width=\unitlength,page=43]{proofJ.pdf}}%
    \put(0.60274352,0.19315117){\color[rgb]{0,0,0}\makebox(0,0)[lt]{\lineheight{1.25}\smash{\begin{tabular}[t]{l}$_Y$\end{tabular}}}}%
    \put(0,0){\includegraphics[width=\unitlength,page=44]{proofJ.pdf}}%
    \put(0.69918535,0.21578378){\color[rgb]{0,0,0}\makebox(0,0)[lt]{\lineheight{1.25}\smash{\begin{tabular}[t]{l}$_N$\end{tabular}}}}%
    \put(0,0){\includegraphics[width=\unitlength,page=45]{proofJ.pdf}}%
    \put(0.95249274,0.08875419){\color[rgb]{0,0,0}\makebox(0,0)[lt]{\lineheight{1.25}\smash{\begin{tabular}[t]{l}$_{M}$\end{tabular}}}}%
    \put(0.77699222,0.08918379){\color[rgb]{0,0,0}\makebox(0,0)[lt]{\lineheight{1.25}\smash{\begin{tabular}[t]{l}$_{\underline{\Hom}(XM,YN)}$\end{tabular}}}}%
    \put(0,0){\includegraphics[width=\unitlength,page=46]{proofJ.pdf}}%
    \put(0.85203798,0.14933661){\color[rgb]{0,0,0}\makebox(0,0)[lt]{\lineheight{1.25}\smash{\begin{tabular}[t]{l}${\scriptstyle \underline{\mathsf{ev}}_{X M, Y N}}$\end{tabular}}}}%
    \put(0,0){\includegraphics[width=\unitlength,page=47]{proofJ.pdf}}%
    \put(0.9205079,0.08904214){\color[rgb]{0,0,0}\makebox(0,0)[lt]{\lineheight{1.25}\smash{\begin{tabular}[t]{l}$_X$\end{tabular}}}}%
    \put(0,0){\includegraphics[width=\unitlength,page=48]{proofJ.pdf}}%
    \put(0.94143368,0.21578378){\color[rgb]{0,0,0}\makebox(0,0)[lt]{\lineheight{1.25}\smash{\begin{tabular}[t]{l}$_N$\end{tabular}}}}%
    \put(0,0){\includegraphics[width=\unitlength,page=49]{proofJ.pdf}}%
    \put(0.8462647,0.21578378){\color[rgb]{0,0,0}\makebox(0,0)[lt]{\lineheight{1.25}\smash{\begin{tabular}[t]{l}$_Y$\end{tabular}}}}%
    \put(0.76199627,0.15009464){\color[rgb]{0,0,0}\makebox(0,0)[lt]{\lineheight{1.25}\smash{\begin{tabular}[t]{l}$=$\end{tabular}}}}%
  \end{picture}%
\endgroup%

%% file: proofFullCenter.pdf_tex
\begingroup%
  \makeatletter%
  \providecommand\color[2][]{%
    \errmessage{(Inkscape) Color is used for the text in Inkscape, but the package 'color.sty' is not loaded}%
    \renewcommand\color[2][]{}%
  }%
  \providecommand\transparent[1]{%
    \errmessage{(Inkscape) Transparency is used (non-zero) for the text in Inkscape, but the package 'transparent.sty' is not loaded}%
    \renewcommand\transparent[1]{}%
  }%
  \providecommand\rotatebox[2]{#2}%
  \newcommand*\fsize{\dimexpr\f@size pt\relax}%
  \newcommand*\lineheight[1]{\fontsize{\fsize}{#1\fsize}\selectfont}%
  \ifx\svgwidth\undefined%
    \setlength{\unitlength}{428.25000897bp}%
    \ifx\svgscale\undefined%
      \relax%
    \else%
      \setlength{\unitlength}{\unitlength * \real{\svgscale}}%
    \fi%
  \else%
    \setlength{\unitlength}{\svgwidth}%
  \fi%
  \global\let\svgwidth\undefined%
  \global\let\svgscale\undefined%
  \makeatother%
  \begin{picture}(1,0.53342521)%
    \lineheight{1}%
    \setlength\tabcolsep{0pt}%
    \put(0,0){\includegraphics[width=\unitlength,page=1]{proofFullCenter.pdf}}%
    \put(0.02882909,0.44142158){\color[rgb]{0,0,0}\makebox(0,0)[lt]{\lineheight{1.25}\smash{\begin{tabular}[t]{l}$m_A$\end{tabular}}}}%
    \put(0,0){\includegraphics[width=\unitlength,page=2]{proofFullCenter.pdf}}%
    \put(0.07694488,0.41100422){\color[rgb]{0,0,0}\makebox(0,0)[lt]{\lineheight{1.25}\smash{\begin{tabular}[t]{l}$_A$\end{tabular}}}}%
    \put(0,0){\includegraphics[width=\unitlength,page=3]{proofFullCenter.pdf}}%
    \put(0.04109566,0.49039294){\color[rgb]{0,0,0}\makebox(0,0)[lt]{\lineheight{1.25}\smash{\begin{tabular}[t]{l}$_A$\end{tabular}}}}%
    \put(0,0){\includegraphics[width=\unitlength,page=4]{proofFullCenter.pdf}}%
    \put(0.06530029,0.37877885){\color[rgb]{0,0,0}\makebox(0,0)[lt]{\lineheight{1.25}\smash{\begin{tabular}[t]{l}$\xi$\end{tabular}}}}%
    \put(0,0){\includegraphics[width=\unitlength,page=5]{proofFullCenter.pdf}}%
    \put(0.00927083,0.27393872){\color[rgb]{0,0,0}\makebox(0,0)[lt]{\lineheight{1.25}\smash{\begin{tabular}[t]{l}$_{\mathcal{A}_{\mathcal{M}}}$\end{tabular}}}}%
    \put(0,0){\includegraphics[width=\unitlength,page=6]{proofFullCenter.pdf}}%
    \put(0.06455469,0.27454978){\color[rgb]{0,0,0}\makebox(0,0)[lt]{\lineheight{1.25}\smash{\begin{tabular}[t]{l}$_A$\end{tabular}}}}%
    \put(0,0){\includegraphics[width=\unitlength,page=7]{proofFullCenter.pdf}}%
    \put(0.12170434,0.38203092){\color[rgb]{0,0,0}\makebox(0,0)[lt]{\lineheight{1.25}\smash{\begin{tabular}[t]{l}$=$\end{tabular}}}}%
    \put(0,0){\includegraphics[width=\unitlength,page=8]{proofFullCenter.pdf}}%
    \put(0.02859475,0.31421026){\color[rgb]{0,0,0}\makebox(0,0)[lt]{\lineheight{1.25}\smash{\begin{tabular}[t]{l}$b_A^{\mathsf{Id}}$\end{tabular}}}}%
    \put(0,0){\includegraphics[width=\unitlength,page=9]{proofFullCenter.pdf}}%
    \put(0.07679075,0.35210384){\color[rgb]{0,0,0}\makebox(0,0)[lt]{\lineheight{1.25}\smash{\begin{tabular}[t]{l}$_{\mathcal{A}_{\mathcal{M}}}$\end{tabular}}}}%
    \put(0.02436814,0.37653865){\color[rgb]{0,0,0}\makebox(0,0)[lt]{\lineheight{1.25}\smash{\begin{tabular}[t]{l}$_A$\end{tabular}}}}%
    \put(0,0){\includegraphics[width=\unitlength,page=10]{proofFullCenter.pdf}}%
    \put(0.21291406,0.47648083){\color[rgb]{0,0,0}\makebox(0,0)[lt]{\lineheight{1.25}\smash{\begin{tabular}[t]{l}$m_A$\end{tabular}}}}%
    \put(0,0){\includegraphics[width=\unitlength,page=11]{proofFullCenter.pdf}}%
    \put(0.24331964,0.38473452){\color[rgb]{0,0,0}\makebox(0,0)[lt]{\lineheight{1.25}\smash{\begin{tabular}[t]{l}${\scriptscriptstyle \underline{\mathrm{End}}(A)}$\end{tabular}}}}%
    \put(0,0){\includegraphics[width=\unitlength,page=12]{proofFullCenter.pdf}}%
    \put(0.22498358,0.5254192){\color[rgb]{0,0,0}\makebox(0,0)[lt]{\lineheight{1.25}\smash{\begin{tabular}[t]{l}$_A$\end{tabular}}}}%
    \put(0,0){\includegraphics[width=\unitlength,page=13]{proofFullCenter.pdf}}%
    \put(0.22021527,0.35299349){\color[rgb]{0,0,0}\makebox(0,0)[lt]{\lineheight{1.25}\smash{\begin{tabular}[t]{l}$\pi_A$\end{tabular}}}}%
    \put(0,0){\includegraphics[width=\unitlength,page=14]{proofFullCenter.pdf}}%
    \put(0.24305147,0.32583413){\color[rgb]{0,0,0}\makebox(0,0)[lt]{\lineheight{1.25}\smash{\begin{tabular}[t]{l}$_{\mathcal{A}_{\mathcal{M}}}$\end{tabular}}}}%
    \put(0.18287171,0.38647385){\color[rgb]{0,0,0}\makebox(0,0)[lt]{\lineheight{1.25}\smash{\begin{tabular}[t]{l}$_A$\end{tabular}}}}%
    \put(0,0){\includegraphics[width=\unitlength,page=15]{proofFullCenter.pdf}}%
    \put(0.26682788,0.41515536){\color[rgb]{0,0,0}\makebox(0,0)[lt]{\lineheight{1.25}\smash{\begin{tabular}[t]{l}$\underline{\mathsf{ev}}_{A,A}$\end{tabular}}}}%
    \put(0,0){\includegraphics[width=\unitlength,page=16]{proofFullCenter.pdf}}%
    \put(0.34839846,0.38473452){\color[rgb]{0,0,0}\makebox(0,0)[lt]{\lineheight{1.25}\smash{\begin{tabular}[t]{l}$_A$\end{tabular}}}}%
    \put(0,0){\includegraphics[width=\unitlength,page=17]{proofFullCenter.pdf}}%
    \put(0.32476023,0.35267668){\color[rgb]{0,0,0}\makebox(0,0)[lt]{\lineheight{1.25}\smash{\begin{tabular}[t]{l}$1_A$\end{tabular}}}}%
    \put(0,0){\includegraphics[width=\unitlength,page=18]{proofFullCenter.pdf}}%
    \put(0.2871025,0.4460305){\color[rgb]{0,0,0}\makebox(0,0)[lt]{\lineheight{1.25}\smash{\begin{tabular}[t]{l}$_A$\end{tabular}}}}%
    \put(0,0){\includegraphics[width=\unitlength,page=19]{proofFullCenter.pdf}}%
    \put(0.16689045,0.247669){\color[rgb]{0,0,0}\makebox(0,0)[lt]{\lineheight{1.25}\smash{\begin{tabular}[t]{l}$_{\mathcal{A}_{\mathcal{M}}}$\end{tabular}}}}%
    \put(0,0){\includegraphics[width=\unitlength,page=20]{proofFullCenter.pdf}}%
    \put(0.22217291,0.24828005){\color[rgb]{0,0,0}\makebox(0,0)[lt]{\lineheight{1.25}\smash{\begin{tabular}[t]{l}$_A$\end{tabular}}}}%
    \put(0,0){\includegraphics[width=\unitlength,page=21]{proofFullCenter.pdf}}%
    \put(0.18621295,0.28794054){\color[rgb]{0,0,0}\makebox(0,0)[lt]{\lineheight{1.25}\smash{\begin{tabular}[t]{l}$b_A^{\mathsf{Id}}$\end{tabular}}}}%
    \put(0,0){\includegraphics[width=\unitlength,page=22]{proofFullCenter.pdf}}%
    \put(0.49347404,0.47609457){\color[rgb]{0,0,0}\makebox(0,0)[lt]{\lineheight{1.25}\smash{\begin{tabular}[t]{l}$m_A$\end{tabular}}}}%
    \put(0,0){\includegraphics[width=\unitlength,page=23]{proofFullCenter.pdf}}%
    \put(0.4623279,0.32343853){\color[rgb]{0,0,0}\makebox(0,0)[lt]{\lineheight{1.25}\smash{\begin{tabular}[t]{l}${\scriptscriptstyle \underline{\mathrm{End}}(A\rhd A)}$\end{tabular}}}}%
    \put(0,0){\includegraphics[width=\unitlength,page=24]{proofFullCenter.pdf}}%
    \put(0.50519369,0.52541919){\color[rgb]{0,0,0}\makebox(0,0)[lt]{\lineheight{1.25}\smash{\begin{tabular}[t]{l}$_A$\end{tabular}}}}%
    \put(0,0){\includegraphics[width=\unitlength,page=25]{proofFullCenter.pdf}}%
    \put(0.42698903,0.29262278){\color[rgb]{0,0,0}\makebox(0,0)[lt]{\lineheight{1.25}\smash{\begin{tabular}[t]{l}$\pi_{A \rhd A}$\end{tabular}}}}%
    \put(0,0){\includegraphics[width=\unitlength,page=26]{proofFullCenter.pdf}}%
    \put(0.46308185,0.38647384){\color[rgb]{0,0,0}\makebox(0,0)[lt]{\lineheight{1.25}\smash{\begin{tabular}[t]{l}$_A$\end{tabular}}}}%
    \put(0,0){\includegraphics[width=\unitlength,page=27]{proofFullCenter.pdf}}%
    \put(0.54703804,0.41515536){\color[rgb]{0,0,0}\makebox(0,0)[lt]{\lineheight{1.25}\smash{\begin{tabular}[t]{l}$\underline{\mathsf{ev}}_{A,A}$\end{tabular}}}}%
    \put(0,0){\includegraphics[width=\unitlength,page=28]{proofFullCenter.pdf}}%
    \put(0.62860863,0.32343854){\color[rgb]{0,0,0}\makebox(0,0)[lt]{\lineheight{1.25}\smash{\begin{tabular}[t]{l}$_A$\end{tabular}}}}%
    \put(0,0){\includegraphics[width=\unitlength,page=29]{proofFullCenter.pdf}}%
    \put(0.60497042,0.29138072){\color[rgb]{0,0,0}\makebox(0,0)[lt]{\lineheight{1.25}\smash{\begin{tabular}[t]{l}$1_A$\end{tabular}}}}%
    \put(0,0){\includegraphics[width=\unitlength,page=30]{proofFullCenter.pdf}}%
    \put(0.56731261,0.44603047){\color[rgb]{0,0,0}\makebox(0,0)[lt]{\lineheight{1.25}\smash{\begin{tabular}[t]{l}$_A$\end{tabular}}}}%
    \put(0,0){\includegraphics[width=\unitlength,page=31]{proofFullCenter.pdf}}%
    \put(0.44710057,0.247669){\color[rgb]{0,0,0}\makebox(0,0)[lt]{\lineheight{1.25}\smash{\begin{tabular}[t]{l}$_{\mathcal{A}_{\mathcal{M}}}$\end{tabular}}}}%
    \put(0,0){\includegraphics[width=\unitlength,page=32]{proofFullCenter.pdf}}%
    \put(0.55492228,0.24828006){\color[rgb]{0,0,0}\makebox(0,0)[lt]{\lineheight{1.25}\smash{\begin{tabular}[t]{l}$_A$\end{tabular}}}}%
    \put(0,0){\includegraphics[width=\unitlength,page=33]{proofFullCenter.pdf}}%
    \put(0.46221831,0.35172118){\color[rgb]{0,0,0}\makebox(0,0)[lt]{\lineheight{1.25}\smash{\begin{tabular}[t]{l}$J_{A,A,A,A}$\end{tabular}}}}%
    \put(0,0){\includegraphics[width=\unitlength,page=34]{proofFullCenter.pdf}}%
    \put(0.51477322,0.38473451){\color[rgb]{0,0,0}\makebox(0,0)[lt]{\lineheight{1.25}\smash{\begin{tabular}[t]{l}${\scriptscriptstyle \underline{\mathrm{End}}(A)}$\end{tabular}}}}%
    \put(0,0){\includegraphics[width=\unitlength,page=35]{proofFullCenter.pdf}}%
    \put(0.79119728,0.44982488){\color[rgb]{0,0,0}\makebox(0,0)[lt]{\lineheight{1.25}\smash{\begin{tabular}[t]{l}$m_A$\end{tabular}}}}%
    \put(0,0){\includegraphics[width=\unitlength,page=36]{proofFullCenter.pdf}}%
    \put(0.74253809,0.35846482){\color[rgb]{0,0,0}\makebox(0,0)[lt]{\lineheight{1.25}\smash{\begin{tabular}[t]{l}${\scriptscriptstyle \underline{\mathrm{End}}(A\rhd A)}$\end{tabular}}}}%
    \put(0,0){\includegraphics[width=\unitlength,page=37]{proofFullCenter.pdf}}%
    \put(0.80291696,0.4991495){\color[rgb]{0,0,0}\makebox(0,0)[lt]{\lineheight{1.25}\smash{\begin{tabular}[t]{l}$_A$\end{tabular}}}}%
    \put(0,0){\includegraphics[width=\unitlength,page=38]{proofFullCenter.pdf}}%
    \put(0.70719925,0.32764904){\color[rgb]{0,0,0}\makebox(0,0)[lt]{\lineheight{1.25}\smash{\begin{tabular}[t]{l}$\pi_{A \rhd A}$\end{tabular}}}}%
    \put(0,0){\includegraphics[width=\unitlength,page=39]{proofFullCenter.pdf}}%
    \put(0.76085471,0.38886745){\color[rgb]{0,0,0}\makebox(0,0)[lt]{\lineheight{1.25}\smash{\begin{tabular}[t]{l}$\underline{\mathsf{ev}}_{A\rhd A,A \rhd A}$\end{tabular}}}}%
    \put(0,0){\includegraphics[width=\unitlength,page=40]{proofFullCenter.pdf}}%
    \put(0.89130572,0.35846482){\color[rgb]{0,0,0}\makebox(0,0)[lt]{\lineheight{1.25}\smash{\begin{tabular}[t]{l}$_A$\end{tabular}}}}%
    \put(0,0){\includegraphics[width=\unitlength,page=41]{proofFullCenter.pdf}}%
    \put(0.86766741,0.32640698){\color[rgb]{0,0,0}\makebox(0,0)[lt]{\lineheight{1.25}\smash{\begin{tabular}[t]{l}$1_A$\end{tabular}}}}%
    \put(0,0){\includegraphics[width=\unitlength,page=42]{proofFullCenter.pdf}}%
    \put(0.85627935,0.41976079){\color[rgb]{0,0,0}\makebox(0,0)[lt]{\lineheight{1.25}\smash{\begin{tabular}[t]{l}$_A$\end{tabular}}}}%
    \put(0,0){\includegraphics[width=\unitlength,page=43]{proofFullCenter.pdf}}%
    \put(0.72731073,0.28269526){\color[rgb]{0,0,0}\makebox(0,0)[lt]{\lineheight{1.25}\smash{\begin{tabular}[t]{l}$_{\mathcal{A}_{\mathcal{M}}}$\end{tabular}}}}%
    \put(0,0){\includegraphics[width=\unitlength,page=44]{proofFullCenter.pdf}}%
    \put(0.83513245,0.28330636){\color[rgb]{0,0,0}\makebox(0,0)[lt]{\lineheight{1.25}\smash{\begin{tabular}[t]{l}$_A$\end{tabular}}}}%
    \put(0,0){\includegraphics[width=\unitlength,page=45]{proofFullCenter.pdf}}%
    \put(0.76871358,0.41976079){\color[rgb]{0,0,0}\makebox(0,0)[lt]{\lineheight{1.25}\smash{\begin{tabular}[t]{l}$_A$\end{tabular}}}}%
    \put(0.38689196,0.38034352){\color[rgb]{0,0,0}\makebox(0,0)[lt]{\lineheight{1.25}\smash{\begin{tabular}[t]{l}$=$\end{tabular}}}}%
    \put(0.66516465,0.38007268){\color[rgb]{0,0,0}\makebox(0,0)[lt]{\lineheight{1.25}\smash{\begin{tabular}[t]{l}$=$\end{tabular}}}}%
    \put(0,0){\includegraphics[width=\unitlength,page=46]{proofFullCenter.pdf}}%
    \put(0.12551456,0.16983431){\color[rgb]{0,0,0}\makebox(0,0)[lt]{\lineheight{1.25}\smash{\begin{tabular}[t]{l}$\underline{\mathsf{ev}}_{A\rhd A, A}$\end{tabular}}}}%
    \put(0,0){\includegraphics[width=\unitlength,page=47]{proofFullCenter.pdf}}%
    \put(0.15493105,0.21893934){\color[rgb]{0,0,0}\makebox(0,0)[lt]{\lineheight{1.25}\smash{\begin{tabular}[t]{l}$_A$\end{tabular}}}}%
    \put(0,0){\includegraphics[width=\unitlength,page=48]{proofFullCenter.pdf}}%
    \put(0.09501241,0.04743887){\color[rgb]{0,0,0}\makebox(0,0)[lt]{\lineheight{1.25}\smash{\begin{tabular}[t]{l}$\pi_{A \rhd A}$\end{tabular}}}}%
    \put(0,0){\includegraphics[width=\unitlength,page=49]{proofFullCenter.pdf}}%
    \put(0.04804344,0.10868748){\color[rgb]{0,0,0}\makebox(0,0)[lt]{\lineheight{1.25}\smash{\begin{tabular}[t]{l}${\scriptstyle \underline{\Hom}(\mathrm{id}_{A\rhd A},m_A)}$\end{tabular}}}}%
    \put(0,0){\includegraphics[width=\unitlength,page=50]{proofFullCenter.pdf}}%
    \put(0.26499821,0.07780415){\color[rgb]{0,0,0}\makebox(0,0)[lt]{\lineheight{1.25}\smash{\begin{tabular}[t]{l}$_A$\end{tabular}}}}%
    \put(0,0){\includegraphics[width=\unitlength,page=51]{proofFullCenter.pdf}}%
    \put(0.24265534,0.04619681){\color[rgb]{0,0,0}\makebox(0,0)[lt]{\lineheight{1.25}\smash{\begin{tabular}[t]{l}$1_A$\end{tabular}}}}%
    \put(0,0){\includegraphics[width=\unitlength,page=52]{proofFullCenter.pdf}}%
    \put(0.09683785,0.00248515){\color[rgb]{0,0,0}\makebox(0,0)[lt]{\lineheight{1.25}\smash{\begin{tabular}[t]{l}$_{\mathcal{A}_{\mathcal{M}}}$\end{tabular}}}}%
    \put(0,0){\includegraphics[width=\unitlength,page=53]{proofFullCenter.pdf}}%
    \put(0.21341612,0.00309621){\color[rgb]{0,0,0}\makebox(0,0)[lt]{\lineheight{1.25}\smash{\begin{tabular}[t]{l}$_A$\end{tabular}}}}%
    \put(0,0){\includegraphics[width=\unitlength,page=54]{proofFullCenter.pdf}}%
    \put(0.1032145,0.13955062){\color[rgb]{0,0,0}\makebox(0,0)[lt]{\lineheight{1.25}\smash{\begin{tabular}[t]{l}${\scriptscriptstyle \underline{\Hom}(A\rhd A,A)}$\end{tabular}}}}%
    \put(0.00299945,0.10026883){\color[rgb]{0,0,0}\makebox(0,0)[lt]{\lineheight{1.25}\smash{\begin{tabular}[t]{l}$=$\end{tabular}}}}%
    \put(0,0){\includegraphics[width=\unitlength,page=55]{proofFullCenter.pdf}}%
    \put(0.11197109,0.07825463){\color[rgb]{0,0,0}\makebox(0,0)[lt]{\lineheight{1.25}\smash{\begin{tabular}[t]{l}${\scriptscriptstyle \underline{\mathrm{End}}(A\rhd A)}$\end{tabular}}}}%
    \put(0,0){\includegraphics[width=\unitlength,page=56]{proofFullCenter.pdf}}%
    \put(0.41448129,0.16983428){\color[rgb]{0,0,0}\makebox(0,0)[lt]{\lineheight{1.25}\smash{\begin{tabular}[t]{l}$\underline{\mathsf{ev}}_{A\rhd A, A}$\end{tabular}}}}%
    \put(0,0){\includegraphics[width=\unitlength,page=57]{proofFullCenter.pdf}}%
    \put(0.44389779,0.21893934){\color[rgb]{0,0,0}\makebox(0,0)[lt]{\lineheight{1.25}\smash{\begin{tabular}[t]{l}$_A$\end{tabular}}}}%
    \put(0,0){\includegraphics[width=\unitlength,page=58]{proofFullCenter.pdf}}%
    \put(0.39371781,0.04721926){\color[rgb]{0,0,0}\makebox(0,0)[lt]{\lineheight{1.25}\smash{\begin{tabular}[t]{l}$\pi_A$\end{tabular}}}}%
    \put(0,0){\includegraphics[width=\unitlength,page=59]{proofFullCenter.pdf}}%
    \put(0.34702403,0.10832901){\color[rgb]{0,0,0}\makebox(0,0)[lt]{\lineheight{1.25}\smash{\begin{tabular}[t]{l}${\scriptstyle \underline{\Hom}(m_A,\mathrm{id}_A)}$\end{tabular}}}}%
    \put(0,0){\includegraphics[width=\unitlength,page=60]{proofFullCenter.pdf}}%
    \put(0.55396492,0.07780415){\color[rgb]{0,0,0}\makebox(0,0)[lt]{\lineheight{1.25}\smash{\begin{tabular}[t]{l}$_A$\end{tabular}}}}%
    \put(0,0){\includegraphics[width=\unitlength,page=61]{proofFullCenter.pdf}}%
    \put(0.53162207,0.04619681){\color[rgb]{0,0,0}\makebox(0,0)[lt]{\lineheight{1.25}\smash{\begin{tabular}[t]{l}$1_A$\end{tabular}}}}%
    \put(0,0){\includegraphics[width=\unitlength,page=62]{proofFullCenter.pdf}}%
    \put(0.38580455,0.00248515){\color[rgb]{0,0,0}\makebox(0,0)[lt]{\lineheight{1.25}\smash{\begin{tabular}[t]{l}$_{\mathcal{A}_{\mathcal{M}}}$\end{tabular}}}}%
    \put(0,0){\includegraphics[width=\unitlength,page=63]{proofFullCenter.pdf}}%
    \put(0.50238287,0.00309617){\color[rgb]{0,0,0}\makebox(0,0)[lt]{\lineheight{1.25}\smash{\begin{tabular}[t]{l}$_A$\end{tabular}}}}%
    \put(0,0){\includegraphics[width=\unitlength,page=64]{proofFullCenter.pdf}}%
    \put(0.39218123,0.1395506){\color[rgb]{0,0,0}\makebox(0,0)[lt]{\lineheight{1.25}\smash{\begin{tabular}[t]{l}${\scriptscriptstyle \underline{\Hom}(A\rhd A,A)}$\end{tabular}}}}%
    \put(0,0){\includegraphics[width=\unitlength,page=65]{proofFullCenter.pdf}}%
    \put(0.4009378,0.07825463){\color[rgb]{0,0,0}\makebox(0,0)[lt]{\lineheight{1.25}\smash{\begin{tabular}[t]{l}${\scriptscriptstyle \underline{\mathrm{End}}(A)}$\end{tabular}}}}%
    \put(0,0){\includegraphics[width=\unitlength,page=66]{proofFullCenter.pdf}}%
    \put(0.67767217,0.16983428){\color[rgb]{0,0,0}\makebox(0,0)[lt]{\lineheight{1.25}\smash{\begin{tabular}[t]{l}$\underline{\mathsf{ev}}_{A, A}$\end{tabular}}}}%
    \put(0,0){\includegraphics[width=\unitlength,page=67]{proofFullCenter.pdf}}%
    \put(0.69783816,0.21893934){\color[rgb]{0,0,0}\makebox(0,0)[lt]{\lineheight{1.25}\smash{\begin{tabular}[t]{l}$_A$\end{tabular}}}}%
    \put(0,0){\includegraphics[width=\unitlength,page=68]{proofFullCenter.pdf}}%
    \put(0.63543212,0.05578521){\color[rgb]{0,0,0}\makebox(0,0)[lt]{\lineheight{1.25}\smash{\begin{tabular}[t]{l}$\pi_A$\end{tabular}}}}%
    \put(0,0){\includegraphics[width=\unitlength,page=69]{proofFullCenter.pdf}}%
    \put(0.75005491,0.10768355){\color[rgb]{0,0,0}\makebox(0,0)[lt]{\lineheight{1.25}\smash{\begin{tabular}[t]{l}$m_A$\end{tabular}}}}%
    \put(0,0){\includegraphics[width=\unitlength,page=70]{proofFullCenter.pdf}}%
    \put(0.63974515,0.01124171){\color[rgb]{0,0,0}\makebox(0,0)[lt]{\lineheight{1.25}\smash{\begin{tabular}[t]{l}$_{\mathcal{A}_{\mathcal{M}}}$\end{tabular}}}}%
    \put(0,0){\includegraphics[width=\unitlength,page=71]{proofFullCenter.pdf}}%
    \put(0.76871363,0.1395506){\color[rgb]{0,0,0}\makebox(0,0)[lt]{\lineheight{1.25}\smash{\begin{tabular}[t]{l}$_A$\end{tabular}}}}%
    \put(0,0){\includegraphics[width=\unitlength,page=72]{proofFullCenter.pdf}}%
    \put(0.65191816,0.08789562){\color[rgb]{0,0,0}\makebox(0,0)[lt]{\lineheight{1.25}\smash{\begin{tabular}[t]{l}${\scriptscriptstyle \underline{\mathrm{End}}(A)}$\end{tabular}}}}%
    \put(0,0){\includegraphics[width=\unitlength,page=73]{proofFullCenter.pdf}}%
    \put(0.79498333,0.07825463){\color[rgb]{0,0,0}\makebox(0,0)[lt]{\lineheight{1.25}\smash{\begin{tabular}[t]{l}$_A$\end{tabular}}}}%
    \put(0,0){\includegraphics[width=\unitlength,page=74]{proofFullCenter.pdf}}%
    \put(0.77264056,0.04664731){\color[rgb]{0,0,0}\makebox(0,0)[lt]{\lineheight{1.25}\smash{\begin{tabular}[t]{l}$1_A$\end{tabular}}}}%
    \put(0,0){\includegraphics[width=\unitlength,page=75]{proofFullCenter.pdf}}%
    \put(0.73005369,0.01185273){\color[rgb]{0,0,0}\makebox(0,0)[lt]{\lineheight{1.25}\smash{\begin{tabular}[t]{l}$_A$\end{tabular}}}}%
    \put(0,0){\includegraphics[width=\unitlength,page=76]{proofFullCenter.pdf}}%
    \put(0.91409942,0.14356461){\color[rgb]{0,0,0}\makebox(0,0)[lt]{\lineheight{1.25}\smash{\begin{tabular}[t]{l}$\underline{\mathsf{ev}}_{A, A}$\end{tabular}}}}%
    \put(0,0){\includegraphics[width=\unitlength,page=77]{proofFullCenter.pdf}}%
    \put(0.93426547,0.19266964){\color[rgb]{0,0,0}\makebox(0,0)[lt]{\lineheight{1.25}\smash{\begin{tabular}[t]{l}$_A$\end{tabular}}}}%
    \put(0,0){\includegraphics[width=\unitlength,page=78]{proofFullCenter.pdf}}%
    \put(0.88061609,0.07329835){\color[rgb]{0,0,0}\makebox(0,0)[lt]{\lineheight{1.25}\smash{\begin{tabular}[t]{l}$\pi_A$\end{tabular}}}}%
    \put(0,0){\includegraphics[width=\unitlength,page=79]{proofFullCenter.pdf}}%
    \put(0.88492911,0.02875484){\color[rgb]{0,0,0}\makebox(0,0)[lt]{\lineheight{1.25}\smash{\begin{tabular}[t]{l}$_{\mathcal{A}_{\mathcal{M}}}$\end{tabular}}}}%
    \put(0,0){\includegraphics[width=\unitlength,page=80]{proofFullCenter.pdf}}%
    \put(0.89710204,0.10540874){\color[rgb]{0,0,0}\makebox(0,0)[lt]{\lineheight{1.25}\smash{\begin{tabular}[t]{l}${\scriptscriptstyle \underline{\mathrm{End}}(A)}$\end{tabular}}}}%
    \put(0,0){\includegraphics[width=\unitlength,page=81]{proofFullCenter.pdf}}%
    \put(0.97523754,0.02936585){\color[rgb]{0,0,0}\makebox(0,0)[lt]{\lineheight{1.25}\smash{\begin{tabular}[t]{l}$_A$\end{tabular}}}}%
    \put(0.28783155,0.10083347){\color[rgb]{0,0,0}\makebox(0,0)[lt]{\lineheight{1.25}\smash{\begin{tabular}[t]{l}$=$\end{tabular}}}}%
    \put(0.57867222,0.10070293){\color[rgb]{0,0,0}\makebox(0,0)[lt]{\lineheight{1.25}\smash{\begin{tabular}[t]{l}$=$\end{tabular}}}}%
    \put(0.82310778,0.1002144){\color[rgb]{0,0,0}\makebox(0,0)[lt]{\lineheight{1.25}\smash{\begin{tabular}[t]{l}$=$\end{tabular}}}}%
  \end{picture}%
\endgroup%

%% file: proofFullCenter2.pdf_tex
\begingroup%
  \makeatletter%
  \providecommand\color[2][]{%
    \errmessage{(Inkscape) Color is used for the text in Inkscape, but the package 'color.sty' is not loaded}%
    \renewcommand\color[2][]{}%
  }%
  \providecommand\transparent[1]{%
    \errmessage{(Inkscape) Transparency is used (non-zero) for the text in Inkscape, but the package 'transparent.sty' is not loaded}%
    \renewcommand\transparent[1]{}%
  }%
  \providecommand\rotatebox[2]{#2}%
  \newcommand*\fsize{\dimexpr\f@size pt\relax}%
  \newcommand*\lineheight[1]{\fontsize{\fsize}{#1\fsize}\selectfont}%
  \ifx\svgwidth\undefined%
    \setlength{\unitlength}{334.50003921bp}%
    \ifx\svgscale\undefined%
      \relax%
    \else%
      \setlength{\unitlength}{\unitlength * \real{\svgscale}}%
    \fi%
  \else%
    \setlength{\unitlength}{\svgwidth}%
  \fi%
  \global\let\svgwidth\undefined%
  \global\let\svgscale\undefined%
  \makeatother%
  \begin{picture}(1,0.29055108)%
    \lineheight{1}%
    \setlength\tabcolsep{0pt}%
    \put(0,0){\includegraphics[width=\unitlength,page=1]{proofFullCenter2.pdf}}%
    \put(0.03690897,0.17276169){\color[rgb]{0,0,0}\makebox(0,0)[lt]{\lineheight{1.25}\smash{\begin{tabular}[t]{l}$m_A$\end{tabular}}}}%
    \put(0,0){\includegraphics[width=\unitlength,page=2]{proofFullCenter2.pdf}}%
    \put(0.03124558,0.1338193){\color[rgb]{0,0,0}\makebox(0,0)[lt]{\lineheight{1.25}\smash{\begin{tabular}[t]{l}$_A$\end{tabular}}}}%
    \put(0,0){\includegraphics[width=\unitlength,page=3]{proofFullCenter2.pdf}}%
    \put(0.05261353,0.2354582){\color[rgb]{0,0,0}\makebox(0,0)[lt]{\lineheight{1.25}\smash{\begin{tabular}[t]{l}$_A$\end{tabular}}}}%
    \put(0,0){\includegraphics[width=\unitlength,page=4]{proofFullCenter2.pdf}}%
    \put(0.01633731,0.09256215){\color[rgb]{0,0,0}\makebox(0,0)[lt]{\lineheight{1.25}\smash{\begin{tabular}[t]{l}$\xi$\end{tabular}}}}%
    \put(0,0){\includegraphics[width=\unitlength,page=5]{proofFullCenter2.pdf}}%
    \put(0.01186918,0.03681393){\color[rgb]{0,0,0}\makebox(0,0)[lt]{\lineheight{1.25}\smash{\begin{tabular}[t]{l}$_{\mathcal{A}_{\mathcal{M}}}$\end{tabular}}}}%
    \put(0,0){\includegraphics[width=\unitlength,page=6]{proofFullCenter2.pdf}}%
    \put(0.08264737,0.03759626){\color[rgb]{0,0,0}\makebox(0,0)[lt]{\lineheight{1.25}\smash{\begin{tabular}[t]{l}$_A$\end{tabular}}}}%
    \put(0.13260277,0.11813702){\color[rgb]{0,0,0}\makebox(0,0)[lt]{\lineheight{1.25}\smash{\begin{tabular}[t]{l}$=$\end{tabular}}}}%
    \put(0,0){\includegraphics[width=\unitlength,page=7]{proofFullCenter2.pdf}}%
    \put(0.22197506,0.10018701){\color[rgb]{0,0,0}\makebox(0,0)[lt]{\lineheight{1.25}\smash{\begin{tabular}[t]{l}${\scriptscriptstyle \underline{\mathrm{End}}(A)}$\end{tabular}}}}%
    \put(0,0){\includegraphics[width=\unitlength,page=8]{proofFullCenter2.pdf}}%
    \put(0.19224869,0.05954993){\color[rgb]{0,0,0}\makebox(0,0)[lt]{\lineheight{1.25}\smash{\begin{tabular}[t]{l}$\pi_A$\end{tabular}}}}%
    \put(0,0){\includegraphics[width=\unitlength,page=9]{proofFullCenter2.pdf}}%
    \put(0.2407146,0.13913386){\color[rgb]{0,0,0}\makebox(0,0)[lt]{\lineheight{1.25}\smash{\begin{tabular}[t]{l}$\underline{\mathsf{ev}}_{A,A}$\end{tabular}}}}%
    \put(0,0){\includegraphics[width=\unitlength,page=10]{proofFullCenter2.pdf}}%
    \put(0.33393615,0.10018702){\color[rgb]{0,0,0}\makebox(0,0)[lt]{\lineheight{1.25}\smash{\begin{tabular}[t]{l}$_A$\end{tabular}}}}%
    \put(0,0){\includegraphics[width=\unitlength,page=11]{proofFullCenter2.pdf}}%
    \put(0.30587748,0.06008226){\color[rgb]{0,0,0}\makebox(0,0)[lt]{\lineheight{1.25}\smash{\begin{tabular}[t]{l}$1_A$\end{tabular}}}}%
    \put(0,0){\includegraphics[width=\unitlength,page=12]{proofFullCenter2.pdf}}%
    \put(0.27788235,0.17866237){\color[rgb]{0,0,0}\makebox(0,0)[lt]{\lineheight{1.25}\smash{\begin{tabular}[t]{l}$_A$\end{tabular}}}}%
    \put(0,0){\includegraphics[width=\unitlength,page=13]{proofFullCenter2.pdf}}%
    \put(0.30596723,0.21760475){\color[rgb]{0,0,0}\makebox(0,0)[lt]{\lineheight{1.25}\smash{\begin{tabular}[t]{l}$m_A$\end{tabular}}}}%
    \put(0,0){\includegraphics[width=\unitlength,page=14]{proofFullCenter2.pdf}}%
    \put(0.33288259,0.28030124){\color[rgb]{0,0,0}\makebox(0,0)[lt]{\lineheight{1.25}\smash{\begin{tabular}[t]{l}$_A$\end{tabular}}}}%
    \put(0,0){\includegraphics[width=\unitlength,page=15]{proofFullCenter2.pdf}}%
    \put(0.38533794,0.00396396){\color[rgb]{0,0,0}\makebox(0,0)[lt]{\lineheight{1.25}\smash{\begin{tabular}[t]{l}$_A$\end{tabular}}}}%
    \put(0,0){\includegraphics[width=\unitlength,page=16]{proofFullCenter2.pdf}}%
    \put(0.20245211,0.00318166){\color[rgb]{0,0,0}\makebox(0,0)[lt]{\lineheight{1.25}\smash{\begin{tabular}[t]{l}$_{\mathcal{A}_{\mathcal{M}}}$\end{tabular}}}}%
    \put(0,0){\includegraphics[width=\unitlength,page=17]{proofFullCenter2.pdf}}%
    \put(0.53587644,0.16745159){\color[rgb]{0,0,0}\makebox(0,0)[lt]{\lineheight{1.25}\smash{\begin{tabular}[t]{l}${\scriptscriptstyle \underline{\mathrm{End}}(A)}$\end{tabular}}}}%
    \put(0,0){\includegraphics[width=\unitlength,page=18]{proofFullCenter2.pdf}}%
    \put(0.50615006,0.12681451){\color[rgb]{0,0,0}\makebox(0,0)[lt]{\lineheight{1.25}\smash{\begin{tabular}[t]{l}$\pi_A$\end{tabular}}}}%
    \put(0,0){\includegraphics[width=\unitlength,page=19]{proofFullCenter2.pdf}}%
    \put(0.57703746,0.20639844){\color[rgb]{0,0,0}\makebox(0,0)[lt]{\lineheight{1.25}\smash{\begin{tabular}[t]{l}$\underline{\mathsf{ev}}_{A,A}$\end{tabular}}}}%
    \put(0,0){\includegraphics[width=\unitlength,page=20]{proofFullCenter2.pdf}}%
    \put(0.61420521,0.24592694){\color[rgb]{0,0,0}\makebox(0,0)[lt]{\lineheight{1.25}\smash{\begin{tabular}[t]{l}$_A$\end{tabular}}}}%
    \put(0,0){\includegraphics[width=\unitlength,page=21]{proofFullCenter2.pdf}}%
    \put(0.51635344,0.01439241){\color[rgb]{0,0,0}\makebox(0,0)[lt]{\lineheight{1.25}\smash{\begin{tabular}[t]{l}$_{\mathcal{A}_{\mathcal{M}}}$\end{tabular}}}}%
    \put(0,0){\includegraphics[width=\unitlength,page=22]{proofFullCenter2.pdf}}%
    \put(0.65904824,0.10018704){\color[rgb]{0,0,0}\makebox(0,0)[lt]{\lineheight{1.25}\smash{\begin{tabular}[t]{l}$_A$\end{tabular}}}}%
    \put(0,0){\includegraphics[width=\unitlength,page=23]{proofFullCenter2.pdf}}%
    \put(0.63098955,0.06008226){\color[rgb]{0,0,0}\makebox(0,0)[lt]{\lineheight{1.25}\smash{\begin{tabular}[t]{l}$1_A$\end{tabular}}}}%
    \put(0,0){\includegraphics[width=\unitlength,page=24]{proofFullCenter2.pdf}}%
    \put(0.66471163,0.13912941){\color[rgb]{0,0,0}\makebox(0,0)[lt]{\lineheight{1.25}\smash{\begin{tabular}[t]{l}$m_A$\end{tabular}}}}%
    \put(0,0){\includegraphics[width=\unitlength,page=25]{proofFullCenter2.pdf}}%
    \put(0.71044989,0.01517473){\color[rgb]{0,0,0}\makebox(0,0)[lt]{\lineheight{1.25}\smash{\begin{tabular}[t]{l}$_A$\end{tabular}}}}%
    \put(0.43471604,0.11704638){\color[rgb]{0,0,0}\makebox(0,0)[lt]{\lineheight{1.25}\smash{\begin{tabular}[t]{l}$=$\end{tabular}}}}%
    \put(0,0){\includegraphics[width=\unitlength,page=26]{proofFullCenter2.pdf}}%
    \put(0.89002406,0.18380132){\color[rgb]{0,0,0}\makebox(0,0)[lt]{\lineheight{1.25}\smash{\begin{tabular}[t]{l}$\underline{\mathsf{ev}}_{A, A}$\end{tabular}}}}%
    \put(0,0){\includegraphics[width=\unitlength,page=27]{proofFullCenter2.pdf}}%
    \put(0.91584204,0.24666896){\color[rgb]{0,0,0}\makebox(0,0)[lt]{\lineheight{1.25}\smash{\begin{tabular}[t]{l}$_A$\end{tabular}}}}%
    \put(0,0){\includegraphics[width=\unitlength,page=28]{proofFullCenter2.pdf}}%
    \put(0.84715634,0.09384162){\color[rgb]{0,0,0}\makebox(0,0)[lt]{\lineheight{1.25}\smash{\begin{tabular}[t]{l}$\pi_A$\end{tabular}}}}%
    \put(0,0){\includegraphics[width=\unitlength,page=29]{proofFullCenter2.pdf}}%
    \put(0.85267825,0.03681394){\color[rgb]{0,0,0}\makebox(0,0)[lt]{\lineheight{1.25}\smash{\begin{tabular}[t]{l}$_{\mathcal{A}_{\mathcal{M}}}$\end{tabular}}}}%
    \put(0,0){\includegraphics[width=\unitlength,page=30]{proofFullCenter2.pdf}}%
    \put(0.8682628,0.13495153){\color[rgb]{0,0,0}\makebox(0,0)[lt]{\lineheight{1.25}\smash{\begin{tabular}[t]{l}${\scriptscriptstyle \underline{\mathrm{End}}(A)}$\end{tabular}}}}%
    \put(0,0){\includegraphics[width=\unitlength,page=31]{proofFullCenter2.pdf}}%
    \put(0.96829731,0.03759616){\color[rgb]{0,0,0}\makebox(0,0)[lt]{\lineheight{1.25}\smash{\begin{tabular}[t]{l}$_A$\end{tabular}}}}%
    \put(0.76435311,0.11752541){\color[rgb]{0,0,0}\makebox(0,0)[lt]{\lineheight{1.25}\smash{\begin{tabular}[t]{l}$=$\end{tabular}}}}%
  \end{picture}%
\endgroup%

%% file: proofBetaLin.pdf_tex
\begingroup%
  \makeatletter%
  \providecommand\color[2][]{%
    \errmessage{(Inkscape) Color is used for the text in Inkscape, but the package 'color.sty' is not loaded}%
    \renewcommand\color[2][]{}%
  }%
  \providecommand\transparent[1]{%
    \errmessage{(Inkscape) Transparency is used (non-zero) for the text in Inkscape, but the package 'transparent.sty' is not loaded}%
    \renewcommand\transparent[1]{}%
  }%
  \providecommand\rotatebox[2]{#2}%
  \newcommand*\fsize{\dimexpr\f@size pt\relax}%
  \newcommand*\lineheight[1]{\fontsize{\fsize}{#1\fsize}\selectfont}%
  \ifx\svgwidth\undefined%
    \setlength{\unitlength}{409.49998864bp}%
    \ifx\svgscale\undefined%
      \relax%
    \else%
      \setlength{\unitlength}{\unitlength * \real{\svgscale}}%
    \fi%
  \else%
    \setlength{\unitlength}{\svgwidth}%
  \fi%
  \global\let\svgwidth\undefined%
  \global\let\svgscale\undefined%
  \makeatother%
  \begin{picture}(1,0.36586036)%
    \lineheight{1}%
    \setlength\tabcolsep{0pt}%
    \put(0,0){\includegraphics[width=\unitlength,page=1]{proofBetaLin.pdf}}%
    \put(0.0896255,0.24699159){\color[rgb]{0,0,0}\makebox(0,0)[lt]{\lineheight{1.25}\smash{\begin{tabular}[t]{l}$_A$\end{tabular}}}}%
    \put(0,0){\includegraphics[width=\unitlength,page=2]{proofBetaLin.pdf}}%
    \put(0.07373018,0.21259529){\color[rgb]{0,0,0}\makebox(0,0)[lt]{\lineheight{1.25}\smash{\begin{tabular}[t]{l}$\zeta$\end{tabular}}}}%
    \put(0,0){\includegraphics[width=\unitlength,page=3]{proofBetaLin.pdf}}%
    \put(0.04257154,0.27709537){\color[rgb]{0,0,0}\makebox(0,0)[lt]{\lineheight{1.25}\smash{\begin{tabular}[t]{l}$r$\end{tabular}}}}%
    \put(0,0){\includegraphics[width=\unitlength,page=4]{proofBetaLin.pdf}}%
    \put(0.05213488,0.33001527){\color[rgb]{0,0,0}\makebox(0,0)[lt]{\lineheight{1.25}\smash{\begin{tabular}[t]{l}$_M$\end{tabular}}}}%
    \put(0,0){\includegraphics[width=\unitlength,page=5]{proofBetaLin.pdf}}%
    \put(0.03645977,0.14807508){\color[rgb]{0,0,0}\makebox(0,0)[lt]{\lineheight{1.25}\smash{\begin{tabular}[t]{l}$h_M$\end{tabular}}}}%
    \put(0,0){\includegraphics[width=\unitlength,page=6]{proofBetaLin.pdf}}%
    \put(0.0896255,0.11878644){\color[rgb]{0,0,0}\makebox(0,0)[lt]{\lineheight{1.25}\smash{\begin{tabular}[t]{l}$_A$\end{tabular}}}}%
    \put(0,0){\includegraphics[width=\unitlength,page=7]{proofBetaLin.pdf}}%
    \put(0.075862,0.08520169){\color[rgb]{0,0,0}\makebox(0,0)[lt]{\lineheight{1.25}\smash{\begin{tabular}[t]{l}$r$\end{tabular}}}}%
    \put(0,0){\includegraphics[width=\unitlength,page=8]{proofBetaLin.pdf}}%
    \put(0.05906926,0.03922899){\color[rgb]{0,0,0}\makebox(0,0)[lt]{\lineheight{1.25}\smash{\begin{tabular}[t]{l}$_M$\end{tabular}}}}%
    \put(0,0){\includegraphics[width=\unitlength,page=9]{proofBetaLin.pdf}}%
    \put(0.0956993,0.03922899){\color[rgb]{0,0,0}\makebox(0,0)[lt]{\lineheight{1.25}\smash{\begin{tabular}[t]{l}$_A$\end{tabular}}}}%
    \put(0,0){\includegraphics[width=\unitlength,page=10]{proofBetaLin.pdf}}%
    \put(0.08945347,0.18288904){\color[rgb]{0,0,0}\makebox(0,0)[lt]{\lineheight{1.25}\smash{\begin{tabular}[t]{l}$_C$\end{tabular}}}}%
    \put(0,0){\includegraphics[width=\unitlength,page=11]{proofBetaLin.pdf}}%
    \put(0.0253509,0.21036157){\color[rgb]{0,0,0}\makebox(0,0)[lt]{\lineheight{1.25}\smash{\begin{tabular}[t]{l}$_M$\end{tabular}}}}%
    \put(0,0){\includegraphics[width=\unitlength,page=12]{proofBetaLin.pdf}}%
    \put(0.01256542,0.04018661){\color[rgb]{0,0,0}\makebox(0,0)[lt]{\lineheight{1.25}\smash{\begin{tabular}[t]{l}$_C$\end{tabular}}}}%
    \put(0.12352682,0.17981022){\color[rgb]{0,0,0}\makebox(0,0)[lt]{\lineheight{1.25}\smash{\begin{tabular}[t]{l}$=$\end{tabular}}}}%
    \put(0,0){\includegraphics[width=\unitlength,page=13]{proofBetaLin.pdf}}%
    \put(0.26361819,0.21036155){\color[rgb]{0,0,0}\makebox(0,0)[lt]{\lineheight{1.25}\smash{\begin{tabular}[t]{l}$_A$\end{tabular}}}}%
    \put(0,0){\includegraphics[width=\unitlength,page=14]{proofBetaLin.pdf}}%
    \put(0.24990329,0.17566842){\color[rgb]{0,0,0}\makebox(0,0)[lt]{\lineheight{1.25}\smash{\begin{tabular}[t]{l}$\zeta$\end{tabular}}}}%
    \put(0,0){\includegraphics[width=\unitlength,page=15]{proofBetaLin.pdf}}%
    \put(0.21656424,0.24046534){\color[rgb]{0,0,0}\makebox(0,0)[lt]{\lineheight{1.25}\smash{\begin{tabular}[t]{l}$r$\end{tabular}}}}%
    \put(0,0){\includegraphics[width=\unitlength,page=16]{proofBetaLin.pdf}}%
    \put(0.22612757,0.29338522){\color[rgb]{0,0,0}\makebox(0,0)[lt]{\lineheight{1.25}\smash{\begin{tabular}[t]{l}$_M$\end{tabular}}}}%
    \put(0,0){\includegraphics[width=\unitlength,page=17]{proofBetaLin.pdf}}%
    \put(0.19112811,0.1113225){\color[rgb]{0,0,0}\makebox(0,0)[lt]{\lineheight{1.25}\smash{\begin{tabular}[t]{l}$h_{M \otimes A}$\end{tabular}}}}%
    \put(0,0){\includegraphics[width=\unitlength,page=18]{proofBetaLin.pdf}}%
    \put(0.1995156,0.21036155){\color[rgb]{0,0,0}\makebox(0,0)[lt]{\lineheight{1.25}\smash{\begin{tabular}[t]{l}$_M$\end{tabular}}}}%
    \put(0,0){\includegraphics[width=\unitlength,page=19]{proofBetaLin.pdf}}%
    \put(0.18575213,0.17677679){\color[rgb]{0,0,0}\makebox(0,0)[lt]{\lineheight{1.25}\smash{\begin{tabular}[t]{l}$r$\end{tabular}}}}%
    \put(0,0){\includegraphics[width=\unitlength,page=20]{proofBetaLin.pdf}}%
    \put(0.22390445,0.06670153){\color[rgb]{0,0,0}\makebox(0,0)[lt]{\lineheight{1.25}\smash{\begin{tabular}[t]{l}$_M$\end{tabular}}}}%
    \put(0,0){\includegraphics[width=\unitlength,page=21]{proofBetaLin.pdf}}%
    \put(0.26053451,0.06670153){\color[rgb]{0,0,0}\makebox(0,0)[lt]{\lineheight{1.25}\smash{\begin{tabular}[t]{l}$_A$\end{tabular}}}}%
    \put(0,0){\includegraphics[width=\unitlength,page=22]{proofBetaLin.pdf}}%
    \put(0.26344613,0.14625901){\color[rgb]{0,0,0}\makebox(0,0)[lt]{\lineheight{1.25}\smash{\begin{tabular}[t]{l}$_C$\end{tabular}}}}%
    \put(0,0){\includegraphics[width=\unitlength,page=23]{proofBetaLin.pdf}}%
    \put(0.17740061,0.06765915){\color[rgb]{0,0,0}\makebox(0,0)[lt]{\lineheight{1.25}\smash{\begin{tabular}[t]{l}$_C$\end{tabular}}}}%
    \put(0,0){\includegraphics[width=\unitlength,page=24]{proofBetaLin.pdf}}%
    \put(0.18120056,0.14625901){\color[rgb]{0,0,0}\makebox(0,0)[lt]{\lineheight{1.25}\smash{\begin{tabular}[t]{l}$_M$\end{tabular}}}}%
    \put(0,0){\includegraphics[width=\unitlength,page=25]{proofBetaLin.pdf}}%
    \put(0.21783065,0.14625901){\color[rgb]{0,0,0}\makebox(0,0)[lt]{\lineheight{1.25}\smash{\begin{tabular}[t]{l}$_A$\end{tabular}}}}%
    \put(0.29751945,0.1798102){\color[rgb]{0,0,0}\makebox(0,0)[lt]{\lineheight{1.25}\smash{\begin{tabular}[t]{l}$=$\end{tabular}}}}%
    \put(0,0){\includegraphics[width=\unitlength,page=26]{proofBetaLin.pdf}}%
    \put(0.43761085,0.18288903){\color[rgb]{0,0,0}\makebox(0,0)[lt]{\lineheight{1.25}\smash{\begin{tabular}[t]{l}$_A$\end{tabular}}}}%
    \put(0,0){\includegraphics[width=\unitlength,page=27]{proofBetaLin.pdf}}%
    \put(0.42389597,0.1481959){\color[rgb]{0,0,0}\makebox(0,0)[lt]{\lineheight{1.25}\smash{\begin{tabular}[t]{l}$\zeta$\end{tabular}}}}%
    \put(0,0){\includegraphics[width=\unitlength,page=28]{proofBetaLin.pdf}}%
    \put(0.36308439,0.27709537){\color[rgb]{0,0,0}\makebox(0,0)[lt]{\lineheight{1.25}\smash{\begin{tabular}[t]{l}$r$\end{tabular}}}}%
    \put(0,0){\includegraphics[width=\unitlength,page=29]{proofBetaLin.pdf}}%
    \put(0.37264773,0.33001525){\color[rgb]{0,0,0}\makebox(0,0)[lt]{\lineheight{1.25}\smash{\begin{tabular}[t]{l}$_M$\end{tabular}}}}%
    \put(0,0){\includegraphics[width=\unitlength,page=30]{proofBetaLin.pdf}}%
    \put(0.35596329,0.08384998){\color[rgb]{0,0,0}\makebox(0,0)[lt]{\lineheight{1.25}\smash{\begin{tabular}[t]{l}$h_{M \otimes A}$\end{tabular}}}}%
    \put(0,0){\includegraphics[width=\unitlength,page=31]{proofBetaLin.pdf}}%
    \put(0.41013832,0.24699164){\color[rgb]{0,0,0}\makebox(0,0)[lt]{\lineheight{1.25}\smash{\begin{tabular}[t]{l}$_A$\end{tabular}}}}%
    \put(0,0){\includegraphics[width=\unitlength,page=32]{proofBetaLin.pdf}}%
    \put(0.38303953,0.21361563){\color[rgb]{0,0,0}\makebox(0,0)[lt]{\lineheight{1.25}\smash{\begin{tabular}[t]{l}$m_A$\end{tabular}}}}%
    \put(0,0){\includegraphics[width=\unitlength,page=33]{proofBetaLin.pdf}}%
    \put(0.38873964,0.03922897){\color[rgb]{0,0,0}\makebox(0,0)[lt]{\lineheight{1.25}\smash{\begin{tabular}[t]{l}$_M$\end{tabular}}}}%
    \put(0,0){\includegraphics[width=\unitlength,page=34]{proofBetaLin.pdf}}%
    \put(0.42536968,0.03922897){\color[rgb]{0,0,0}\makebox(0,0)[lt]{\lineheight{1.25}\smash{\begin{tabular}[t]{l}$_A$\end{tabular}}}}%
    \put(0,0){\includegraphics[width=\unitlength,page=35]{proofBetaLin.pdf}}%
    \put(0.43743879,0.11878648){\color[rgb]{0,0,0}\makebox(0,0)[lt]{\lineheight{1.25}\smash{\begin{tabular}[t]{l}$_C$\end{tabular}}}}%
    \put(0,0){\includegraphics[width=\unitlength,page=36]{proofBetaLin.pdf}}%
    \put(0.34223585,0.04018665){\color[rgb]{0,0,0}\makebox(0,0)[lt]{\lineheight{1.25}\smash{\begin{tabular}[t]{l}$_C$\end{tabular}}}}%
    \put(0,0){\includegraphics[width=\unitlength,page=37]{proofBetaLin.pdf}}%
    \put(0.34603571,0.11878649){\color[rgb]{0,0,0}\makebox(0,0)[lt]{\lineheight{1.25}\smash{\begin{tabular}[t]{l}$_M$\end{tabular}}}}%
    \put(0,0){\includegraphics[width=\unitlength,page=38]{proofBetaLin.pdf}}%
    \put(0.3826658,0.11878649){\color[rgb]{0,0,0}\makebox(0,0)[lt]{\lineheight{1.25}\smash{\begin{tabular}[t]{l}$_A$\end{tabular}}}}%
    \put(0,0){\includegraphics[width=\unitlength,page=39]{proofBetaLin.pdf}}%
    \put(0.61160344,0.21036155){\color[rgb]{0,0,0}\makebox(0,0)[lt]{\lineheight{1.25}\smash{\begin{tabular}[t]{l}$_A$\end{tabular}}}}%
    \put(0,0){\includegraphics[width=\unitlength,page=40]{proofBetaLin.pdf}}%
    \put(0.59788859,0.1756684){\color[rgb]{0,0,0}\makebox(0,0)[lt]{\lineheight{1.25}\smash{\begin{tabular}[t]{l}$\zeta$\end{tabular}}}}%
    \put(0,0){\includegraphics[width=\unitlength,page=41]{proofBetaLin.pdf}}%
    \put(0.53707699,0.30456788){\color[rgb]{0,0,0}\makebox(0,0)[lt]{\lineheight{1.25}\smash{\begin{tabular}[t]{l}$r$\end{tabular}}}}%
    \put(0,0){\includegraphics[width=\unitlength,page=42]{proofBetaLin.pdf}}%
    \put(0.54664034,0.35748777){\color[rgb]{0,0,0}\makebox(0,0)[lt]{\lineheight{1.25}\smash{\begin{tabular}[t]{l}$_M$\end{tabular}}}}%
    \put(0,0){\includegraphics[width=\unitlength,page=43]{proofBetaLin.pdf}}%
    \put(0.51745189,0.04745051){\color[rgb]{0,0,0}\makebox(0,0)[lt]{\lineheight{1.25}\smash{\begin{tabular}[t]{l}$h_M$\end{tabular}}}}%
    \put(0,0){\includegraphics[width=\unitlength,page=44]{proofBetaLin.pdf}}%
    \put(0.58413092,0.27446414){\color[rgb]{0,0,0}\makebox(0,0)[lt]{\lineheight{1.25}\smash{\begin{tabular}[t]{l}$_A$\end{tabular}}}}%
    \put(0,0){\includegraphics[width=\unitlength,page=45]{proofBetaLin.pdf}}%
    \put(0.56036981,0.24194138){\color[rgb]{0,0,0}\makebox(0,0)[lt]{\lineheight{1.25}\smash{\begin{tabular}[t]{l}$m_A$\end{tabular}}}}%
    \put(0,0){\includegraphics[width=\unitlength,page=46]{proofBetaLin.pdf}}%
    \put(0.54441727,0.00259894){\color[rgb]{0,0,0}\makebox(0,0)[lt]{\lineheight{1.25}\smash{\begin{tabular}[t]{l}$_M$\end{tabular}}}}%
    \put(0,0){\includegraphics[width=\unitlength,page=47]{proofBetaLin.pdf}}%
    \put(0.59936232,0.00259894){\color[rgb]{0,0,0}\makebox(0,0)[lt]{\lineheight{1.25}\smash{\begin{tabular}[t]{l}$_A$\end{tabular}}}}%
    \put(0,0){\includegraphics[width=\unitlength,page=48]{proofBetaLin.pdf}}%
    \put(0.61143142,0.14625901){\color[rgb]{0,0,0}\makebox(0,0)[lt]{\lineheight{1.25}\smash{\begin{tabular}[t]{l}$_C$\end{tabular}}}}%
    \put(0,0){\includegraphics[width=\unitlength,page=49]{proofBetaLin.pdf}}%
    \put(0.50707103,0.0035566){\color[rgb]{0,0,0}\makebox(0,0)[lt]{\lineheight{1.25}\smash{\begin{tabular}[t]{l}$_C$\end{tabular}}}}%
    \put(0,0){\includegraphics[width=\unitlength,page=50]{proofBetaLin.pdf}}%
    \put(0.5200283,0.14625901){\color[rgb]{0,0,0}\makebox(0,0)[lt]{\lineheight{1.25}\smash{\begin{tabular}[t]{l}$_M$\end{tabular}}}}%
    \put(0,0){\includegraphics[width=\unitlength,page=51]{proofBetaLin.pdf}}%
    \put(0.5566584,0.14625901){\color[rgb]{0,0,0}\makebox(0,0)[lt]{\lineheight{1.25}\smash{\begin{tabular}[t]{l}$_A$\end{tabular}}}}%
    \put(0,0){\includegraphics[width=\unitlength,page=52]{proofBetaLin.pdf}}%
    \put(0.56669702,0.11177272){\color[rgb]{0,0,0}\makebox(0,0)[lt]{\lineheight{1.25}\smash{\begin{tabular}[t]{l}$h_A$\end{tabular}}}}%
    \put(0,0){\includegraphics[width=\unitlength,page=53]{proofBetaLin.pdf}}%
    \put(0.55648637,0.08215644){\color[rgb]{0,0,0}\makebox(0,0)[lt]{\lineheight{1.25}\smash{\begin{tabular}[t]{l}$_C$\end{tabular}}}}%
    \put(0,0){\includegraphics[width=\unitlength,page=54]{proofBetaLin.pdf}}%
    \put(0.75812363,0.18288901){\color[rgb]{0,0,0}\makebox(0,0)[lt]{\lineheight{1.25}\smash{\begin{tabular}[t]{l}$_A$\end{tabular}}}}%
    \put(0,0){\includegraphics[width=\unitlength,page=55]{proofBetaLin.pdf}}%
    \put(0.74440867,0.14819584){\color[rgb]{0,0,0}\makebox(0,0)[lt]{\lineheight{1.25}\smash{\begin{tabular}[t]{l}$\zeta$\end{tabular}}}}%
    \put(0,0){\includegraphics[width=\unitlength,page=56]{proofBetaLin.pdf}}%
    \put(0.72938467,0.27709534){\color[rgb]{0,0,0}\makebox(0,0)[lt]{\lineheight{1.25}\smash{\begin{tabular}[t]{l}$r$\end{tabular}}}}%
    \put(0,0){\includegraphics[width=\unitlength,page=57]{proofBetaLin.pdf}}%
    \put(0.73894803,0.33001524){\color[rgb]{0,0,0}\makebox(0,0)[lt]{\lineheight{1.25}\smash{\begin{tabular}[t]{l}$_M$\end{tabular}}}}%
    \put(0,0){\includegraphics[width=\unitlength,page=58]{proofBetaLin.pdf}}%
    \put(0.70975958,0.08408056){\color[rgb]{0,0,0}\makebox(0,0)[lt]{\lineheight{1.25}\smash{\begin{tabular}[t]{l}$h_M$\end{tabular}}}}%
    \put(0,0){\includegraphics[width=\unitlength,page=59]{proofBetaLin.pdf}}%
    \put(0.78089731,0.24648507){\color[rgb]{0,0,0}\makebox(0,0)[lt]{\lineheight{1.25}\smash{\begin{tabular}[t]{l}$_A$\end{tabular}}}}%
    \put(0,0){\includegraphics[width=\unitlength,page=60]{proofBetaLin.pdf}}%
    \put(0.75685702,0.21446882){\color[rgb]{0,0,0}\makebox(0,0)[lt]{\lineheight{1.25}\smash{\begin{tabular}[t]{l}$m_A$\end{tabular}}}}%
    \put(0,0){\includegraphics[width=\unitlength,page=61]{proofBetaLin.pdf}}%
    \put(0.74588237,0.03922897){\color[rgb]{0,0,0}\makebox(0,0)[lt]{\lineheight{1.25}\smash{\begin{tabular}[t]{l}$_M$\end{tabular}}}}%
    \put(0,0){\includegraphics[width=\unitlength,page=62]{proofBetaLin.pdf}}%
    \put(0.79166999,0.03922897){\color[rgb]{0,0,0}\makebox(0,0)[lt]{\lineheight{1.25}\smash{\begin{tabular}[t]{l}$_A$\end{tabular}}}}%
    \put(0,0){\includegraphics[width=\unitlength,page=63]{proofBetaLin.pdf}}%
    \put(0.69022114,0.04018661){\color[rgb]{0,0,0}\makebox(0,0)[lt]{\lineheight{1.25}\smash{\begin{tabular}[t]{l}$_C$\end{tabular}}}}%
    \put(0,0){\includegraphics[width=\unitlength,page=64]{proofBetaLin.pdf}}%
    \put(0.70317847,0.11878649){\color[rgb]{0,0,0}\makebox(0,0)[lt]{\lineheight{1.25}\smash{\begin{tabular}[t]{l}$_M$\end{tabular}}}}%
    \put(0,0){\includegraphics[width=\unitlength,page=65]{proofBetaLin.pdf}}%
    \put(0.75795152,0.11878648){\color[rgb]{0,0,0}\makebox(0,0)[lt]{\lineheight{1.25}\smash{\begin{tabular}[t]{l}$_C$\end{tabular}}}}%
    \put(0,0){\includegraphics[width=\unitlength,page=66]{proofBetaLin.pdf}}%
    \put(0.94127392,0.18288898){\color[rgb]{0,0,0}\makebox(0,0)[lt]{\lineheight{1.25}\smash{\begin{tabular}[t]{l}$_A$\end{tabular}}}}%
    \put(0,0){\includegraphics[width=\unitlength,page=67]{proofBetaLin.pdf}}%
    \put(0.92755891,0.14819582){\color[rgb]{0,0,0}\makebox(0,0)[lt]{\lineheight{1.25}\smash{\begin{tabular}[t]{l}$\zeta$\end{tabular}}}}%
    \put(0,0){\includegraphics[width=\unitlength,page=68]{proofBetaLin.pdf}}%
    \put(0.93755072,0.27728017){\color[rgb]{0,0,0}\makebox(0,0)[lt]{\lineheight{1.25}\smash{\begin{tabular}[t]{l}$r$\end{tabular}}}}%
    \put(0,0){\includegraphics[width=\unitlength,page=69]{proofBetaLin.pdf}}%
    \put(0.94041327,0.33001523){\color[rgb]{0,0,0}\makebox(0,0)[lt]{\lineheight{1.25}\smash{\begin{tabular}[t]{l}$_M$\end{tabular}}}}%
    \put(0,0){\includegraphics[width=\unitlength,page=70]{proofBetaLin.pdf}}%
    \put(0.89290987,0.08408056){\color[rgb]{0,0,0}\makebox(0,0)[lt]{\lineheight{1.25}\smash{\begin{tabular}[t]{l}$h_M$\end{tabular}}}}%
    \put(0,0){\includegraphics[width=\unitlength,page=71]{proofBetaLin.pdf}}%
    \put(0.91380145,0.24699155){\color[rgb]{0,0,0}\makebox(0,0)[lt]{\lineheight{1.25}\smash{\begin{tabular}[t]{l}$_M$\end{tabular}}}}%
    \put(0,0){\includegraphics[width=\unitlength,page=72]{proofBetaLin.pdf}}%
    \put(0.90139757,0.21258262){\color[rgb]{0,0,0}\makebox(0,0)[lt]{\lineheight{1.25}\smash{\begin{tabular}[t]{l}$r$\end{tabular}}}}%
    \put(0,0){\includegraphics[width=\unitlength,page=73]{proofBetaLin.pdf}}%
    \put(0.92903255,0.03922894){\color[rgb]{0,0,0}\makebox(0,0)[lt]{\lineheight{1.25}\smash{\begin{tabular}[t]{l}$_M$\end{tabular}}}}%
    \put(0,0){\includegraphics[width=\unitlength,page=74]{proofBetaLin.pdf}}%
    \put(0.97482017,0.03922894){\color[rgb]{0,0,0}\makebox(0,0)[lt]{\lineheight{1.25}\smash{\begin{tabular}[t]{l}$_A$\end{tabular}}}}%
    \put(0,0){\includegraphics[width=\unitlength,page=75]{proofBetaLin.pdf}}%
    \put(0.87337132,0.04018661){\color[rgb]{0,0,0}\makebox(0,0)[lt]{\lineheight{1.25}\smash{\begin{tabular}[t]{l}$_C$\end{tabular}}}}%
    \put(0,0){\includegraphics[width=\unitlength,page=76]{proofBetaLin.pdf}}%
    \put(0.88632875,0.11878646){\color[rgb]{0,0,0}\makebox(0,0)[lt]{\lineheight{1.25}\smash{\begin{tabular}[t]{l}$_M$\end{tabular}}}}%
    \put(0,0){\includegraphics[width=\unitlength,page=77]{proofBetaLin.pdf}}%
    \put(0.9411017,0.11878646){\color[rgb]{0,0,0}\makebox(0,0)[lt]{\lineheight{1.25}\smash{\begin{tabular}[t]{l}$_C$\end{tabular}}}}%
    \put(0.46736107,0.17824663){\color[rgb]{0,0,0}\makebox(0,0)[lt]{\lineheight{1.25}\smash{\begin{tabular}[t]{l}$=$\end{tabular}}}}%
    \put(0.64410092,0.17857526){\color[rgb]{0,0,0}\makebox(0,0)[lt]{\lineheight{1.25}\smash{\begin{tabular}[t]{l}$=$\end{tabular}}}}%
    \put(0.82769538,0.17851669){\color[rgb]{0,0,0}\makebox(0,0)[lt]{\lineheight{1.25}\smash{\begin{tabular}[t]{l}$=$\end{tabular}}}}%
  \end{picture}%
\endgroup%

%% file: proofUnivEnd.pdf_tex
\begingroup%
  \makeatletter%
  \providecommand\color[2][]{%
    \errmessage{(Inkscape) Color is used for the text in Inkscape, but the package 'color.sty' is not loaded}%
    \renewcommand\color[2][]{}%
  }%
  \providecommand\transparent[1]{%
    \errmessage{(Inkscape) Transparency is used (non-zero) for the text in Inkscape, but the package 'transparent.sty' is not loaded}%
    \renewcommand\transparent[1]{}%
  }%
  \providecommand\rotatebox[2]{#2}%
  \newcommand*\fsize{\dimexpr\f@size pt\relax}%
  \newcommand*\lineheight[1]{\fontsize{\fsize}{#1\fsize}\selectfont}%
  \ifx\svgwidth\undefined%
    \setlength{\unitlength}{458.24993856bp}%
    \ifx\svgscale\undefined%
      \relax%
    \else%
      \setlength{\unitlength}{\unitlength * \real{\svgscale}}%
    \fi%
  \else%
    \setlength{\unitlength}{\svgwidth}%
  \fi%
  \global\let\svgwidth\undefined%
  \global\let\svgscale\undefined%
  \makeatother%
  \begin{picture}(1,0.26079211)%
    \lineheight{1}%
    \setlength\tabcolsep{0pt}%
    \put(0,0){\includegraphics[width=\unitlength,page=1]{proofUnivEnd.pdf}}%
    \put(0.10052016,0.15439373){\color[rgb]{0,0,0}\makebox(0,0)[lt]{\lineheight{1.25}\smash{\begin{tabular}[t]{l}$_A$\end{tabular}}}}%
    \put(0,0){\includegraphics[width=\unitlength,page=2]{proofUnivEnd.pdf}}%
    \put(0.08134065,0.12378963){\color[rgb]{0,0,0}\makebox(0,0)[lt]{\lineheight{1.25}\smash{\begin{tabular}[t]{l}$1_A$\end{tabular}}}}%
    \put(0,0){\includegraphics[width=\unitlength,page=3]{proofUnivEnd.pdf}}%
    \put(0.0309914,0.18330129){\color[rgb]{0,0,0}\makebox(0,0)[lt]{\lineheight{1.25}\smash{\begin{tabular}[t]{l}$\underline{\mathsf{ev}}_{A,A}$\end{tabular}}}}%
    \put(0,0){\includegraphics[width=\unitlength,page=4]{proofUnivEnd.pdf}}%
    \put(0.05883456,0.22858508){\color[rgb]{0,0,0}\makebox(0,0)[lt]{\lineheight{1.25}\smash{\begin{tabular}[t]{l}$_A$\end{tabular}}}}%
    \put(0,0){\includegraphics[width=\unitlength,page=5]{proofUnivEnd.pdf}}%
    \put(0.01529134,0.02687239){\color[rgb]{0,0,0}\makebox(0,0)[lt]{\lineheight{1.25}\smash{\begin{tabular}[t]{l}$_C$\end{tabular}}}}%
    \put(0,0){\includegraphics[width=\unitlength,page=6]{proofUnivEnd.pdf}}%
    \put(0.0243491,0.15440747){\color[rgb]{0,0,0}\makebox(0,0)[lt]{\lineheight{1.25}\smash{\begin{tabular}[t]{l}${\scriptscriptstyle \underline{\mathrm{End}}(A)}$\end{tabular}}}}%
    \put(0,0){\includegraphics[width=\unitlength,page=7]{proofUnivEnd.pdf}}%
    \put(0.00801211,0.12488746){\color[rgb]{0,0,0}\makebox(0,0)[lt]{\lineheight{1.25}\smash{\begin{tabular}[t]{l}$\pi_A$\end{tabular}}}}%
    \put(0,0){\includegraphics[width=\unitlength,page=8]{proofUnivEnd.pdf}}%
    \put(0.02633524,0.09908997){\color[rgb]{0,0,0}\makebox(0,0)[lt]{\lineheight{1.25}\smash{\begin{tabular}[t]{l}$_{\mathcal{A}_{\mathcal{M}}}$\end{tabular}}}}%
    \put(0,0){\includegraphics[width=\unitlength,page=9]{proofUnivEnd.pdf}}%
    \put(0.01436037,0.06741962){\color[rgb]{0,0,0}\makebox(0,0)[lt]{\lineheight{1.25}\smash{\begin{tabular}[t]{l}$u$\end{tabular}}}}%
    \put(0,0){\includegraphics[width=\unitlength,page=10]{proofUnivEnd.pdf}}%
    \put(0.37056928,0.09711059){\color[rgb]{0,0,0}\makebox(0,0)[lt]{\lineheight{1.25}\smash{\begin{tabular}[t]{l}$_A$\end{tabular}}}}%
    \put(0,0){\includegraphics[width=\unitlength,page=11]{proofUnivEnd.pdf}}%
    \put(0.35138977,0.06650648){\color[rgb]{0,0,0}\makebox(0,0)[lt]{\lineheight{1.25}\smash{\begin{tabular}[t]{l}$1_A$\end{tabular}}}}%
    \put(0,0){\includegraphics[width=\unitlength,page=12]{proofUnivEnd.pdf}}%
    \put(0.27270014,0.18375393){\color[rgb]{0,0,0}\makebox(0,0)[lt]{\lineheight{1.25}\smash{\begin{tabular}[t]{l}$\underline{\mathsf{ev}}_{A,A}$\end{tabular}}}}%
    \put(0,0){\includegraphics[width=\unitlength,page=13]{proofUnivEnd.pdf}}%
    \put(0.29615048,0.22858511){\color[rgb]{0,0,0}\makebox(0,0)[lt]{\lineheight{1.25}\smash{\begin{tabular}[t]{l}$_A$\end{tabular}}}}%
    \put(0,0){\includegraphics[width=\unitlength,page=14]{proofUnivEnd.pdf}}%
    \put(0.24529837,0.15440748){\color[rgb]{0,0,0}\makebox(0,0)[lt]{\lineheight{1.25}\smash{\begin{tabular}[t]{l}${\scriptscriptstyle \underline{\mathrm{End}}(A)}$\end{tabular}}}}%
    \put(0,0){\includegraphics[width=\unitlength,page=15]{proofUnivEnd.pdf}}%
    \put(0.17167269,0.12273277){\color[rgb]{0,0,0}\makebox(0,0)[lt]{\lineheight{1.25}\smash{\begin{tabular}[t]{l}{\small $\underline{\mathrm{Hom}}(\mathrm{id}_A,\beta_A)$}\end{tabular}}}}%
    \put(0,0){\includegraphics[width=\unitlength,page=16]{proofUnivEnd.pdf}}%
    \put(0.20037177,0.06836266){\color[rgb]{0,0,0}\makebox(0,0)[lt]{\lineheight{1.25}\smash{\begin{tabular}[t]{l}$\underline{\mathsf{coev}}_{C,A}$\end{tabular}}}}%
    \put(0,0){\includegraphics[width=\unitlength,page=17]{proofUnivEnd.pdf}}%
    \put(0.24447665,0.09690857){\color[rgb]{0,0,0}\makebox(0,0)[lt]{\lineheight{1.25}\smash{\begin{tabular}[t]{l}${\scriptscriptstyle \underline{\mathrm{Hom}}(A, C\rhd A)}$\end{tabular}}}}%
    \put(0,0){\includegraphics[width=\unitlength,page=18]{proofUnivEnd.pdf}}%
    \put(0.23624066,0.02687239){\color[rgb]{0,0,0}\makebox(0,0)[lt]{\lineheight{1.25}\smash{\begin{tabular}[t]{l}$_C$\end{tabular}}}}%
    \put(0,0){\includegraphics[width=\unitlength,page=19]{proofUnivEnd.pdf}}%
    \put(0.59970191,0.09711063){\color[rgb]{0,0,0}\makebox(0,0)[lt]{\lineheight{1.25}\smash{\begin{tabular}[t]{l}$_A$\end{tabular}}}}%
    \put(0,0){\includegraphics[width=\unitlength,page=20]{proofUnivEnd.pdf}}%
    \put(0.58052242,0.06650646){\color[rgb]{0,0,0}\makebox(0,0)[lt]{\lineheight{1.25}\smash{\begin{tabular}[t]{l}$1_A$\end{tabular}}}}%
    \put(0,0){\includegraphics[width=\unitlength,page=21]{proofUnivEnd.pdf}}%
    \put(0.4901205,0.12600354){\color[rgb]{0,0,0}\makebox(0,0)[lt]{\lineheight{1.25}\smash{\begin{tabular}[t]{l}$\underline{\mathsf{ev}}_{A,C \rhd A}$\end{tabular}}}}%
    \put(0,0){\includegraphics[width=\unitlength,page=22]{proofUnivEnd.pdf}}%
    \put(0.42950434,0.06836269){\color[rgb]{0,0,0}\makebox(0,0)[lt]{\lineheight{1.25}\smash{\begin{tabular}[t]{l}$\underline{\mathsf{coev}}_{C,A}$\end{tabular}}}}%
    \put(0,0){\includegraphics[width=\unitlength,page=23]{proofUnivEnd.pdf}}%
    \put(0.47360921,0.09690857){\color[rgb]{0,0,0}\makebox(0,0)[lt]{\lineheight{1.25}\smash{\begin{tabular}[t]{l}${\scriptscriptstyle \underline{\mathrm{Hom}}(A, C\rhd A)}$\end{tabular}}}}%
    \put(0,0){\includegraphics[width=\unitlength,page=24]{proofUnivEnd.pdf}}%
    \put(0.46537326,0.02687239){\color[rgb]{0,0,0}\makebox(0,0)[lt]{\lineheight{1.25}\smash{\begin{tabular}[t]{l}$_C$\end{tabular}}}}%
    \put(0,0){\includegraphics[width=\unitlength,page=25]{proofUnivEnd.pdf}}%
    \put(0.54241876,0.15439376){\color[rgb]{0,0,0}\makebox(0,0)[lt]{\lineheight{1.25}\smash{\begin{tabular}[t]{l}$_{C \rhd A}$\end{tabular}}}}%
    \put(0,0){\includegraphics[width=\unitlength,page=26]{proofUnivEnd.pdf}}%
    \put(0.52323934,0.18107277){\color[rgb]{0,0,0}\makebox(0,0)[lt]{\lineheight{1.25}\smash{\begin{tabular}[t]{l}$\beta_A$\end{tabular}}}}%
    \put(0,0){\includegraphics[width=\unitlength,page=27]{proofUnivEnd.pdf}}%
    \put(0.52528302,0.22858511){\color[rgb]{0,0,0}\makebox(0,0)[lt]{\lineheight{1.25}\smash{\begin{tabular}[t]{l}$_A$\end{tabular}}}}%
    \put(0,0){\includegraphics[width=\unitlength,page=28]{proofUnivEnd.pdf}}%
    \put(0.67872225,0.13120001){\color[rgb]{0,0,0}\makebox(0,0)[lt]{\lineheight{1.25}\smash{\begin{tabular}[t]{l}$\beta_A$\end{tabular}}}}%
    \put(0,0){\includegraphics[width=\unitlength,page=29]{proofUnivEnd.pdf}}%
    \put(0.68076589,0.17966046){\color[rgb]{0,0,0}\makebox(0,0)[lt]{\lineheight{1.25}\smash{\begin{tabular}[t]{l}$_A$\end{tabular}}}}%
    \put(0,0){\includegraphics[width=\unitlength,page=30]{proofUnivEnd.pdf}}%
    \put(0.7224516,0.10546912){\color[rgb]{0,0,0}\makebox(0,0)[lt]{\lineheight{1.25}\smash{\begin{tabular}[t]{l}$_A$\end{tabular}}}}%
    \put(0,0){\includegraphics[width=\unitlength,page=31]{proofUnivEnd.pdf}}%
    \put(0.70327205,0.07486498){\color[rgb]{0,0,0}\makebox(0,0)[lt]{\lineheight{1.25}\smash{\begin{tabular}[t]{l}$1_A$\end{tabular}}}}%
    \put(0,0){\includegraphics[width=\unitlength,page=32]{proofUnivEnd.pdf}}%
    \put(0.66177274,0.04341422){\color[rgb]{0,0,0}\makebox(0,0)[lt]{\lineheight{1.25}\smash{\begin{tabular}[t]{l}$_C$\end{tabular}}}}%
    \put(0,0){\includegraphics[width=\unitlength,page=33]{proofUnivEnd.pdf}}%
    \put(0.79706492,0.20627321){\color[rgb]{0,0,0}\makebox(0,0)[lt]{\lineheight{1.25}\smash{\begin{tabular}[t]{l}$m_A$\end{tabular}}}}%
    \put(0,0){\includegraphics[width=\unitlength,page=34]{proofUnivEnd.pdf}}%
    \put(0.80351545,0.25331023){\color[rgb]{0,0,0}\makebox(0,0)[lt]{\lineheight{1.25}\smash{\begin{tabular}[t]{l}$_A$\end{tabular}}}}%
    \put(0,0){\includegraphics[width=\unitlength,page=35]{proofUnivEnd.pdf}}%
    \put(0.84520114,0.17911887){\color[rgb]{0,0,0}\makebox(0,0)[lt]{\lineheight{1.25}\smash{\begin{tabular}[t]{l}$_A$\end{tabular}}}}%
    \put(0,0){\includegraphics[width=\unitlength,page=36]{proofUnivEnd.pdf}}%
    \put(0.83208586,0.14967226){\color[rgb]{0,0,0}\makebox(0,0)[lt]{\lineheight{1.25}\smash{\begin{tabular}[t]{l}$\zeta$\end{tabular}}}}%
    \put(0,0){\includegraphics[width=\unitlength,page=37]{proofUnivEnd.pdf}}%
    \put(0.80147194,0.09032594){\color[rgb]{0,0,0}\makebox(0,0)[lt]{\lineheight{1.25}\smash{\begin{tabular}[t]{l}$h_A$\end{tabular}}}}%
    \put(0,0){\includegraphics[width=\unitlength,page=38]{proofUnivEnd.pdf}}%
    \put(0.84520114,0.12166052){\color[rgb]{0,0,0}\makebox(0,0)[lt]{\lineheight{1.25}\smash{\begin{tabular}[t]{l}$_C$\end{tabular}}}}%
    \put(0,0){\includegraphics[width=\unitlength,page=39]{proofUnivEnd.pdf}}%
    \put(0.84520114,0.06437741){\color[rgb]{0,0,0}\makebox(0,0)[lt]{\lineheight{1.25}\smash{\begin{tabular}[t]{l}$_A$\end{tabular}}}}%
    \put(0,0){\includegraphics[width=\unitlength,page=40]{proofUnivEnd.pdf}}%
    \put(0.82602161,0.03377326){\color[rgb]{0,0,0}\makebox(0,0)[lt]{\lineheight{1.25}\smash{\begin{tabular}[t]{l}$1_A$\end{tabular}}}}%
    \put(0,0){\includegraphics[width=\unitlength,page=41]{proofUnivEnd.pdf}}%
    \put(0.78452231,0.00232245){\color[rgb]{0,0,0}\makebox(0,0)[lt]{\lineheight{1.25}\smash{\begin{tabular}[t]{l}$_C$\end{tabular}}}}%
    \put(0,0){\includegraphics[width=\unitlength,page=42]{proofUnivEnd.pdf}}%
    \put(0.79610141,0.12166054){\color[rgb]{0,0,0}\makebox(0,0)[lt]{\lineheight{1.25}\smash{\begin{tabular}[t]{l}$_A$\end{tabular}}}}%
    \put(0.12702413,0.11024024){\color[rgb]{0,0,0}\makebox(0,0)[lt]{\lineheight{1.25}\smash{\begin{tabular}[t]{l}$=$\end{tabular}}}}%
    \put(0.39491091,0.10943916){\color[rgb]{0,0,0}\makebox(0,0)[lt]{\lineheight{1.25}\smash{\begin{tabular}[t]{l}$=$\end{tabular}}}}%
    \put(0.62333651,0.10908431){\color[rgb]{0,0,0}\makebox(0,0)[lt]{\lineheight{1.25}\smash{\begin{tabular}[t]{l}$=$\end{tabular}}}}%
    \put(0.74700999,0.10913058){\color[rgb]{0,0,0}\makebox(0,0)[lt]{\lineheight{1.25}\smash{\begin{tabular}[t]{l}$=$\end{tabular}}}}%
    \put(0,0){\includegraphics[width=\unitlength,page=43]{proofUnivEnd.pdf}}%
    \put(0.93642087,0.14999037){\color[rgb]{0,0,0}\makebox(0,0)[lt]{\lineheight{1.25}\smash{\begin{tabular}[t]{l}$m_A$\end{tabular}}}}%
    \put(0,0){\includegraphics[width=\unitlength,page=44]{proofUnivEnd.pdf}}%
    \put(0.93797312,0.19622325){\color[rgb]{0,0,0}\makebox(0,0)[lt]{\lineheight{1.25}\smash{\begin{tabular}[t]{l}$_A$\end{tabular}}}}%
    \put(0,0){\includegraphics[width=\unitlength,page=45]{proofUnivEnd.pdf}}%
    \put(0.98431729,0.12183571){\color[rgb]{0,0,0}\makebox(0,0)[lt]{\lineheight{1.25}\smash{\begin{tabular}[t]{l}$_A$\end{tabular}}}}%
    \put(0,0){\includegraphics[width=\unitlength,page=46]{proofUnivEnd.pdf}}%
    \put(0.97120207,0.09076131){\color[rgb]{0,0,0}\makebox(0,0)[lt]{\lineheight{1.25}\smash{\begin{tabular}[t]{l}$\zeta$\end{tabular}}}}%
    \put(0,0){\includegraphics[width=\unitlength,page=47]{proofUnivEnd.pdf}}%
    \put(0.97273837,0.0515975){\color[rgb]{0,0,0}\makebox(0,0)[lt]{\lineheight{1.25}\smash{\begin{tabular}[t]{l}$_C$\end{tabular}}}}%
    \put(0,0){\includegraphics[width=\unitlength,page=48]{proofUnivEnd.pdf}}%
    \put(0.9270342,0.12183578){\color[rgb]{0,0,0}\makebox(0,0)[lt]{\lineheight{1.25}\smash{\begin{tabular}[t]{l}$_A$\end{tabular}}}}%
    \put(0,0){\includegraphics[width=\unitlength,page=49]{proofUnivEnd.pdf}}%
    \put(0.90785474,0.09123163){\color[rgb]{0,0,0}\makebox(0,0)[lt]{\lineheight{1.25}\smash{\begin{tabular}[t]{l}$1_A$\end{tabular}}}}%
    \put(0.86950329,0.11043661){\color[rgb]{0,0,0}\makebox(0,0)[lt]{\lineheight{1.25}\smash{\begin{tabular}[t]{l}$=$\end{tabular}}}}%
  \end{picture}%
\endgroup%

%% file: proofXiAlg.pdf_tex
\begingroup%
  \makeatletter%
  \providecommand\color[2][]{%
    \errmessage{(Inkscape) Color is used for the text in Inkscape, but the package 'color.sty' is not loaded}%
    \renewcommand\color[2][]{}%
  }%
  \providecommand\transparent[1]{%
    \errmessage{(Inkscape) Transparency is used (non-zero) for the text in Inkscape, but the package 'transparent.sty' is not loaded}%
    \renewcommand\transparent[1]{}%
  }%
  \providecommand\rotatebox[2]{#2}%
  \newcommand*\fsize{\dimexpr\f@size pt\relax}%
  \newcommand*\lineheight[1]{\fontsize{\fsize}{#1\fsize}\selectfont}%
  \ifx\svgwidth\undefined%
    \setlength{\unitlength}{467.26843489bp}%
    \ifx\svgscale\undefined%
      \relax%
    \else%
      \setlength{\unitlength}{\unitlength * \real{\svgscale}}%
    \fi%
  \else%
    \setlength{\unitlength}{\svgwidth}%
  \fi%
  \global\let\svgwidth\undefined%
  \global\let\svgscale\undefined%
  \makeatother%
  \begin{picture}(1,0.65691875)%
    \lineheight{1}%
    \setlength\tabcolsep{0pt}%
    \put(0,0){\includegraphics[width=\unitlength,page=1]{proofXiAlg.pdf}}%
    \put(0.06459537,0.55327688){\color[rgb]{0,0,0}\makebox(0,0)[lt]{\lineheight{1.25}\smash{\begin{tabular}[t]{l}$_A$\end{tabular}}}}%
    \put(0,0){\includegraphics[width=\unitlength,page=2]{proofXiAlg.pdf}}%
    \put(0.04070284,0.44414909){\color[rgb]{0,0,0}\makebox(0,0)[lt]{\lineheight{1.25}\smash{\begin{tabular}[t]{l}$m_{\mathcal{A}_{\mathcal{M}}}$\end{tabular}}}}%
    \put(0,0){\includegraphics[width=\unitlength,page=3]{proofXiAlg.pdf}}%
    \put(0.03213584,0.40354587){\color[rgb]{0,0,0}\makebox(0,0)[lt]{\lineheight{1.25}\smash{\begin{tabular}[t]{l}$_{\mathcal{A}_{\mathcal{M}}}$\end{tabular}}}}%
    \put(0,0){\includegraphics[width=\unitlength,page=4]{proofXiAlg.pdf}}%
    \put(0.12265699,0.46984169){\color[rgb]{0,0,0}\makebox(0,0)[lt]{\lineheight{1.25}\smash{\begin{tabular}[t]{l}$=$\end{tabular}}}}%
    \put(0.07159817,0.47838753){\color[rgb]{0,0,0}\makebox(0,0)[lt]{\lineheight{1.25}\smash{\begin{tabular}[t]{l}$_{\mathcal{A}_{\mathcal{M}}}$\end{tabular}}}}%
    \put(0,0){\includegraphics[width=\unitlength,page=5]{proofXiAlg.pdf}}%
    \put(0.06188445,0.50686712){\color[rgb]{0,0,0}\makebox(0,0)[lt]{\lineheight{1.25}\smash{\begin{tabular}[t]{l}$\xi$\end{tabular}}}}%
    \put(0,0){\includegraphics[width=\unitlength,page=6]{proofXiAlg.pdf}}%
    \put(0.08028804,0.40354587){\color[rgb]{0,0,0}\makebox(0,0)[lt]{\lineheight{1.25}\smash{\begin{tabular}[t]{l}$_{\mathcal{A}_{\mathcal{M}}}$\end{tabular}}}}%
    \put(0,0){\includegraphics[width=\unitlength,page=7]{proofXiAlg.pdf}}%
    \put(0.16910866,0.41204762){\color[rgb]{0,0,0}\makebox(0,0)[lt]{\lineheight{1.25}\smash{\begin{tabular}[t]{l}$m_{\mathcal{A}_{\mathcal{M}}}$\end{tabular}}}}%
    \put(0,0){\includegraphics[width=\unitlength,page=8]{proofXiAlg.pdf}}%
    \put(0.16054169,0.37144441){\color[rgb]{0,0,0}\makebox(0,0)[lt]{\lineheight{1.25}\smash{\begin{tabular}[t]{l}$_{\mathcal{A}_{\mathcal{M}}}$\end{tabular}}}}%
    \put(0,0){\includegraphics[width=\unitlength,page=9]{proofXiAlg.pdf}}%
    \put(0.20000401,0.44628607){\color[rgb]{0,0,0}\makebox(0,0)[lt]{\lineheight{1.25}\smash{\begin{tabular}[t]{l}$_{\mathcal{A}_{\mathcal{M}}}$\end{tabular}}}}%
    \put(0,0){\includegraphics[width=\unitlength,page=10]{proofXiAlg.pdf}}%
    \put(0.18169768,0.47630004){\color[rgb]{0,0,0}\makebox(0,0)[lt]{\lineheight{1.25}\smash{\begin{tabular}[t]{l}$\pi_A$\end{tabular}}}}%
    \put(0,0){\includegraphics[width=\unitlength,page=11]{proofXiAlg.pdf}}%
    \put(0.20869389,0.37144441){\color[rgb]{0,0,0}\makebox(0,0)[lt]{\lineheight{1.25}\smash{\begin{tabular}[t]{l}$_{\mathcal{A}_{\mathcal{M}}}$\end{tabular}}}}%
    \put(0,0){\includegraphics[width=\unitlength,page=12]{proofXiAlg.pdf}}%
    \put(0.2038436,0.54087732){\color[rgb]{0,0,0}\makebox(0,0)[lt]{\lineheight{1.25}\smash{\begin{tabular}[t]{l}$\underline{\mathsf{ev}}_{A,A}$\end{tabular}}}}%
    \put(0,0){\includegraphics[width=\unitlength,page=13]{proofXiAlg.pdf}}%
    \put(0.20010863,0.50747032){\color[rgb]{0,0,0}\makebox(0,0)[lt]{\lineheight{1.25}\smash{\begin{tabular}[t]{l}${\scriptscriptstyle \underline{\mathrm{End}}(A)}$\end{tabular}}}}%
    \put(0,0){\includegraphics[width=\unitlength,page=14]{proofXiAlg.pdf}}%
    \put(0.24745272,0.47520715){\color[rgb]{0,0,0}\makebox(0,0)[lt]{\lineheight{1.25}\smash{\begin{tabular}[t]{l}$1_A$\end{tabular}}}}%
    \put(0,0){\includegraphics[width=\unitlength,page=15]{proofXiAlg.pdf}}%
    \put(0.26431159,0.50747032){\color[rgb]{0,0,0}\makebox(0,0)[lt]{\lineheight{1.25}\smash{\begin{tabular}[t]{l}$_A$\end{tabular}}}}%
    \put(0,0){\includegraphics[width=\unitlength,page=16]{proofXiAlg.pdf}}%
    \put(0.22510267,0.58537834){\color[rgb]{0,0,0}\makebox(0,0)[lt]{\lineheight{1.25}\smash{\begin{tabular}[t]{l}$_A$\end{tabular}}}}%
    \put(0,0){\includegraphics[width=\unitlength,page=17]{proofXiAlg.pdf}}%
    \put(0.44460453,0.60508024){\color[rgb]{0,0,0}\makebox(0,0)[lt]{\lineheight{1.25}\smash{\begin{tabular}[t]{l}$\underline{\mathsf{ev}}_{A,A}$\end{tabular}}}}%
    \put(0,0){\includegraphics[width=\unitlength,page=18]{proofXiAlg.pdf}}%
    \put(0.5122898,0.53941007){\color[rgb]{0,0,0}\makebox(0,0)[lt]{\lineheight{1.25}\smash{\begin{tabular}[t]{l}$1_A$\end{tabular}}}}%
    \put(0,0){\includegraphics[width=\unitlength,page=19]{proofXiAlg.pdf}}%
    \put(0.5291486,0.57167324){\color[rgb]{0,0,0}\makebox(0,0)[lt]{\lineheight{1.25}\smash{\begin{tabular}[t]{l}$_A$\end{tabular}}}}%
    \put(0,0){\includegraphics[width=\unitlength,page=20]{proofXiAlg.pdf}}%
    \put(0.46586366,0.64958126){\color[rgb]{0,0,0}\makebox(0,0)[lt]{\lineheight{1.25}\smash{\begin{tabular}[t]{l}$_A$\end{tabular}}}}%
    \put(0,0){\includegraphics[width=\unitlength,page=21]{proofXiAlg.pdf}}%
    \put(0.41679349,0.57167324){\color[rgb]{0,0,0}\makebox(0,0)[lt]{\lineheight{1.25}\smash{\begin{tabular}[t]{l}${\scriptscriptstyle \underline{\mathrm{End}}(A)}$\end{tabular}}}}%
    \put(0,0){\includegraphics[width=\unitlength,page=22]{proofXiAlg.pdf}}%
    \put(0.34761044,0.54039818){\color[rgb]{0,0,0}\makebox(0,0)[lt]{\lineheight{1.25}\smash{\begin{tabular}[t]{l}${\scriptstyle \underline{\Hom}(\mathrm{id}_A,\underline{\mathsf{ev}}_{A,A})}$\end{tabular}}}}%
    \put(0,0){\includegraphics[width=\unitlength,page=23]{proofXiAlg.pdf}}%
    \put(0.37666668,0.50747032){\color[rgb]{0,0,0}\makebox(0,0)[lt]{\lineheight{1.25}\smash{\begin{tabular}[t]{l}${\scriptscriptstyle \underline{\Hom}(A,\underline{\mathrm{End}}(A) \rhd A)}$\end{tabular}}}}%
    \put(0,0){\includegraphics[width=\unitlength,page=24]{proofXiAlg.pdf}}%
    \put(0.3370605,0.4756249){\color[rgb]{0,0,0}\makebox(0,0)[lt]{\lineheight{1.25}\smash{\begin{tabular}[t]{l}${\scriptstyle \underline{\Hom}(\mathrm{id}_A,\mathrm{id}\rhd\underline{\mathsf{ev}}_{A,A})}$\end{tabular}}}}%
    \put(0,0){\includegraphics[width=\unitlength,page=25]{proofXiAlg.pdf}}%
    \put(0.37666668,0.4432674){\color[rgb]{0,0,0}\makebox(0,0)[lt]{\lineheight{1.25}\smash{\begin{tabular}[t]{l}${\scriptscriptstyle \underline{\Hom}(A,\underline{\mathrm{End}}(A)^{\otimes 2} \rhd A)}$\end{tabular}}}}%
    \put(0,0){\includegraphics[width=\unitlength,page=26]{proofXiAlg.pdf}}%
    \put(0.33993928,0.41216335){\color[rgb]{0,0,0}\makebox(0,0)[lt]{\lineheight{1.25}\smash{\begin{tabular}[t]{l}$\underline{\mathsf{coev}}_{\underline{\mathrm{End}}(A)^{\otimes 2},A}$\end{tabular}}}}%
    \put(0,0){\includegraphics[width=\unitlength,page=27]{proofXiAlg.pdf}}%
    \put(0.34220498,0.34789417){\color[rgb]{0,0,0}\makebox(0,0)[lt]{\lineheight{1.25}\smash{\begin{tabular}[t]{l}$\pi_A$\end{tabular}}}}%
    \put(0,0){\includegraphics[width=\unitlength,page=28]{proofXiAlg.pdf}}%
    \put(0.36061597,0.37906448){\color[rgb]{0,0,0}\makebox(0,0)[lt]{\lineheight{1.25}\smash{\begin{tabular}[t]{l}${\scriptscriptstyle \underline{\mathrm{End}}(A)}$\end{tabular}}}}%
    \put(0,0){\includegraphics[width=\unitlength,page=29]{proofXiAlg.pdf}}%
    \put(0.46258545,0.34789417){\color[rgb]{0,0,0}\makebox(0,0)[lt]{\lineheight{1.25}\smash{\begin{tabular}[t]{l}$\pi_A$\end{tabular}}}}%
    \put(0,0){\includegraphics[width=\unitlength,page=30]{proofXiAlg.pdf}}%
    \put(0.48099644,0.37906448){\color[rgb]{0,0,0}\makebox(0,0)[lt]{\lineheight{1.25}\smash{\begin{tabular}[t]{l}${\scriptscriptstyle \underline{\mathrm{End}}(A)}$\end{tabular}}}}%
    \put(0,0){\includegraphics[width=\unitlength,page=31]{proofXiAlg.pdf}}%
    \put(0.34512507,0.30724147){\color[rgb]{0,0,0}\makebox(0,0)[lt]{\lineheight{1.25}\smash{\begin{tabular}[t]{l}$_{\mathcal{A}_{\mathcal{M}}}$\end{tabular}}}}%
    \put(0,0){\includegraphics[width=\unitlength,page=32]{proofXiAlg.pdf}}%
    \put(0.46550558,0.30724147){\color[rgb]{0,0,0}\makebox(0,0)[lt]{\lineheight{1.25}\smash{\begin{tabular}[t]{l}$_{\mathcal{A}_{\mathcal{M}}}$\end{tabular}}}}%
    \put(0,0){\includegraphics[width=\unitlength,page=33]{proofXiAlg.pdf}}%
    \put(0.68536552,0.60508025){\color[rgb]{0,0,0}\makebox(0,0)[lt]{\lineheight{1.25}\smash{\begin{tabular}[t]{l}$\underline{\mathsf{ev}}_{A,A}$\end{tabular}}}}%
    \put(0,0){\includegraphics[width=\unitlength,page=34]{proofXiAlg.pdf}}%
    \put(0.70662462,0.64958127){\color[rgb]{0,0,0}\makebox(0,0)[lt]{\lineheight{1.25}\smash{\begin{tabular}[t]{l}$_A$\end{tabular}}}}%
    \put(0,0){\includegraphics[width=\unitlength,page=35]{proofXiAlg.pdf}}%
    \put(0.63347828,0.44326741){\color[rgb]{0,0,0}\makebox(0,0)[lt]{\lineheight{1.25}\smash{\begin{tabular}[t]{l}${\scriptscriptstyle \underline{\Hom}(A,\underline{\mathrm{End}}(A)^{\otimes 2} \rhd A)}$\end{tabular}}}}%
    \put(0,0){\includegraphics[width=\unitlength,page=36]{proofXiAlg.pdf}}%
    \put(0.60977915,0.41284162){\color[rgb]{0,0,0}\makebox(0,0)[lt]{\lineheight{1.25}\smash{\begin{tabular}[t]{l}$\underline{\mathsf{coev}}_{\underline{\mathrm{End}}(A)^{\otimes 2},A}$\end{tabular}}}}%
    \put(0,0){\includegraphics[width=\unitlength,page=37]{proofXiAlg.pdf}}%
    \put(0.60704211,0.34789417){\color[rgb]{0,0,0}\makebox(0,0)[lt]{\lineheight{1.25}\smash{\begin{tabular}[t]{l}$\pi_A$\end{tabular}}}}%
    \put(0,0){\includegraphics[width=\unitlength,page=38]{proofXiAlg.pdf}}%
    \put(0.62545301,0.37906449){\color[rgb]{0,0,0}\makebox(0,0)[lt]{\lineheight{1.25}\smash{\begin{tabular}[t]{l}${\scriptscriptstyle \underline{\mathrm{End}}(A)}$\end{tabular}}}}%
    \put(0,0){\includegraphics[width=\unitlength,page=39]{proofXiAlg.pdf}}%
    \put(0.7354479,0.34789417){\color[rgb]{0,0,0}\makebox(0,0)[lt]{\lineheight{1.25}\smash{\begin{tabular}[t]{l}$\pi_A$\end{tabular}}}}%
    \put(0,0){\includegraphics[width=\unitlength,page=40]{proofXiAlg.pdf}}%
    \put(0.75385888,0.37906449){\color[rgb]{0,0,0}\makebox(0,0)[lt]{\lineheight{1.25}\smash{\begin{tabular}[t]{l}${\scriptscriptstyle \underline{\mathrm{End}}(A)}$\end{tabular}}}}%
    \put(0,0){\includegraphics[width=\unitlength,page=41]{proofXiAlg.pdf}}%
    \put(0.60996211,0.30724149){\color[rgb]{0,0,0}\makebox(0,0)[lt]{\lineheight{1.25}\smash{\begin{tabular}[t]{l}$_{\mathcal{A}_{\mathcal{M}}}$\end{tabular}}}}%
    \put(0,0){\includegraphics[width=\unitlength,page=42]{proofXiAlg.pdf}}%
    \put(0.73836806,0.30724149){\color[rgb]{0,0,0}\makebox(0,0)[lt]{\lineheight{1.25}\smash{\begin{tabular}[t]{l}$_{\mathcal{A}_{\mathcal{M}}}$\end{tabular}}}}%
    \put(0,0){\includegraphics[width=\unitlength,page=43]{proofXiAlg.pdf}}%
    \put(0.7254924,0.54087732){\color[rgb]{0,0,0}\makebox(0,0)[lt]{\lineheight{1.25}\smash{\begin{tabular}[t]{l}$\underline{\mathsf{ev}}_{A,A}$\end{tabular}}}}%
    \put(0,0){\includegraphics[width=\unitlength,page=44]{proofXiAlg.pdf}}%
    \put(0.76188414,0.57167325){\color[rgb]{0,0,0}\makebox(0,0)[lt]{\lineheight{1.25}\smash{\begin{tabular}[t]{l}$_A$\end{tabular}}}}%
    \put(0,0){\includegraphics[width=\unitlength,page=45]{proofXiAlg.pdf}}%
    \put(0.64480543,0.47940308){\color[rgb]{0,0,0}\makebox(0,0)[lt]{\lineheight{1.25}\smash{\begin{tabular}[t]{l}$\underline{\mathsf{ev}}_{A,\underline{\mathrm{End}}(A)^{\otimes 2} \rhd A}$\end{tabular}}}}%
    \put(0,0){\includegraphics[width=\unitlength,page=46]{proofXiAlg.pdf}}%
    \put(0.79317766,0.41100423){\color[rgb]{0,0,0}\makebox(0,0)[lt]{\lineheight{1.25}\smash{\begin{tabular}[t]{l}$1_A$\end{tabular}}}}%
    \put(0,0){\includegraphics[width=\unitlength,page=47]{proofXiAlg.pdf}}%
    \put(0.81003643,0.4432674){\color[rgb]{0,0,0}\makebox(0,0)[lt]{\lineheight{1.25}\smash{\begin{tabular}[t]{l}$_A$\end{tabular}}}}%
    \put(0,0){\includegraphics[width=\unitlength,page=48]{proofXiAlg.pdf}}%
    \put(0.70570667,0.50747032){\color[rgb]{0,0,0}\makebox(0,0)[lt]{\lineheight{1.25}\smash{\begin{tabular}[t]{l}${\scriptscriptstyle \underline{\mathrm{End}}(A)}$\end{tabular}}}}%
    \put(0,0){\includegraphics[width=\unitlength,page=49]{proofXiAlg.pdf}}%
    \put(0.80201108,0.50747032){\color[rgb]{0,0,0}\makebox(0,0)[lt]{\lineheight{1.25}\smash{\begin{tabular}[t]{l}$_A$\end{tabular}}}}%
    \put(0,0){\includegraphics[width=\unitlength,page=50]{proofXiAlg.pdf}}%
    \put(0.62545301,0.50747032){\color[rgb]{0,0,0}\makebox(0,0)[lt]{\lineheight{1.25}\smash{\begin{tabular}[t]{l}${\scriptscriptstyle \underline{\mathrm{End}}(A)}$\end{tabular}}}}%
    \put(0,0){\includegraphics[width=\unitlength,page=51]{proofXiAlg.pdf}}%
    \put(0.10753917,0.22788806){\color[rgb]{0,0,0}\makebox(0,0)[lt]{\lineheight{1.25}\smash{\begin{tabular}[t]{l}$\underline{\mathsf{ev}}_{A,A}$\end{tabular}}}}%
    \put(0,0){\includegraphics[width=\unitlength,page=52]{proofXiAlg.pdf}}%
    \put(0.12879829,0.27238909){\color[rgb]{0,0,0}\makebox(0,0)[lt]{\lineheight{1.25}\smash{\begin{tabular}[t]{l}$_A$\end{tabular}}}}%
    \put(0,0){\includegraphics[width=\unitlength,page=53]{proofXiAlg.pdf}}%
    \put(0.03724116,0.09910788){\color[rgb]{0,0,0}\makebox(0,0)[lt]{\lineheight{1.25}\smash{\begin{tabular}[t]{l}$\pi_A$\end{tabular}}}}%
    \put(0,0){\includegraphics[width=\unitlength,page=54]{proofXiAlg.pdf}}%
    \put(0.12552012,0.09910788){\color[rgb]{0,0,0}\makebox(0,0)[lt]{\lineheight{1.25}\smash{\begin{tabular}[t]{l}$\pi_A$\end{tabular}}}}%
    \put(0,0){\includegraphics[width=\unitlength,page=55]{proofXiAlg.pdf}}%
    \put(0.14393108,0.13027816){\color[rgb]{0,0,0}\makebox(0,0)[lt]{\lineheight{1.25}\smash{\begin{tabular}[t]{l}${\scriptscriptstyle \underline{\mathrm{End}}(A)}$\end{tabular}}}}%
    \put(0,0){\includegraphics[width=\unitlength,page=56]{proofXiAlg.pdf}}%
    \put(0.04016117,0.05845519){\color[rgb]{0,0,0}\makebox(0,0)[lt]{\lineheight{1.25}\smash{\begin{tabular}[t]{l}$_{\mathcal{A}_{\mathcal{M}}}$\end{tabular}}}}%
    \put(0,0){\includegraphics[width=\unitlength,page=57]{proofXiAlg.pdf}}%
    \put(0.12844028,0.05845519){\color[rgb]{0,0,0}\makebox(0,0)[lt]{\lineheight{1.25}\smash{\begin{tabular}[t]{l}$_{\mathcal{A}_{\mathcal{M}}}$\end{tabular}}}}%
    \put(0,0){\includegraphics[width=\unitlength,page=58]{proofXiAlg.pdf}}%
    \put(0.15569141,0.16368515){\color[rgb]{0,0,0}\makebox(0,0)[lt]{\lineheight{1.25}\smash{\begin{tabular}[t]{l}$\underline{\mathsf{ev}}_{A,A}$\end{tabular}}}}%
    \put(0,0){\includegraphics[width=\unitlength,page=59]{proofXiAlg.pdf}}%
    \put(0.19208315,0.19448105){\color[rgb]{0,0,0}\makebox(0,0)[lt]{\lineheight{1.25}\smash{\begin{tabular}[t]{l}$_A$\end{tabular}}}}%
    \put(0,0){\includegraphics[width=\unitlength,page=60]{proofXiAlg.pdf}}%
    \put(0.20732593,0.09801499){\color[rgb]{0,0,0}\makebox(0,0)[lt]{\lineheight{1.25}\smash{\begin{tabular}[t]{l}$1_A$\end{tabular}}}}%
    \put(0,0){\includegraphics[width=\unitlength,page=61]{proofXiAlg.pdf}}%
    \put(0.22418469,0.13027811){\color[rgb]{0,0,0}\makebox(0,0)[lt]{\lineheight{1.25}\smash{\begin{tabular}[t]{l}$_A$\end{tabular}}}}%
    \put(0,0){\includegraphics[width=\unitlength,page=62]{proofXiAlg.pdf}}%
    \put(0.05565205,0.13027811){\color[rgb]{0,0,0}\makebox(0,0)[lt]{\lineheight{1.25}\smash{\begin{tabular}[t]{l}${\scriptscriptstyle \underline{\mathrm{End}}(A)}$\end{tabular}}}}%
    \put(0,0){\includegraphics[width=\unitlength,page=63]{proofXiAlg.pdf}}%
    \put(0.3723763,0.22788806){\color[rgb]{0,0,0}\makebox(0,0)[lt]{\lineheight{1.25}\smash{\begin{tabular}[t]{l}$\underline{\mathsf{ev}}_{A,A}$\end{tabular}}}}%
    \put(0,0){\includegraphics[width=\unitlength,page=64]{proofXiAlg.pdf}}%
    \put(0.39363537,0.27238909){\color[rgb]{0,0,0}\makebox(0,0)[lt]{\lineheight{1.25}\smash{\begin{tabular}[t]{l}$_A$\end{tabular}}}}%
    \put(0,0){\includegraphics[width=\unitlength,page=65]{proofXiAlg.pdf}}%
    \put(0.29405287,0.04293029){\color[rgb]{0,0,0}\makebox(0,0)[lt]{\lineheight{1.25}\smash{\begin{tabular}[t]{l}$\pi_A$\end{tabular}}}}%
    \put(0,0){\includegraphics[width=\unitlength,page=66]{proofXiAlg.pdf}}%
    \put(0.29697291,0.00227763){\color[rgb]{0,0,0}\makebox(0,0)[lt]{\lineheight{1.25}\smash{\begin{tabular}[t]{l}$_{\mathcal{A}_{\mathcal{M}}}$\end{tabular}}}}%
    \put(0,0){\includegraphics[width=\unitlength,page=67]{proofXiAlg.pdf}}%
    \put(0.31246377,0.07410055){\color[rgb]{0,0,0}\makebox(0,0)[lt]{\lineheight{1.25}\smash{\begin{tabular}[t]{l}${\scriptscriptstyle \underline{\mathrm{End}}(A)}$\end{tabular}}}}%
    \put(0,0){\includegraphics[width=\unitlength,page=68]{proofXiAlg.pdf}}%
    \put(0.42586466,0.16223338){\color[rgb]{0,0,0}\makebox(0,0)[lt]{\lineheight{1.25}\smash{\begin{tabular}[t]{l}$m_A$\end{tabular}}}}%
    \put(0,0){\includegraphics[width=\unitlength,page=69]{proofXiAlg.pdf}}%
    \put(0.44086955,0.19448107){\color[rgb]{0,0,0}\makebox(0,0)[lt]{\lineheight{1.25}\smash{\begin{tabular}[t]{l}$_A$\end{tabular}}}}%
    \put(0,0){\includegraphics[width=\unitlength,page=70]{proofXiAlg.pdf}}%
    \put(0.3758586,0.0418374){\color[rgb]{0,0,0}\makebox(0,0)[lt]{\lineheight{1.25}\smash{\begin{tabular}[t]{l}$1_A$\end{tabular}}}}%
    \put(0,0){\includegraphics[width=\unitlength,page=71]{proofXiAlg.pdf}}%
    \put(0.39216833,0.07163575){\color[rgb]{0,0,0}\makebox(0,0)[lt]{\lineheight{1.25}\smash{\begin{tabular}[t]{l}$_A$\end{tabular}}}}%
    \put(0,0){\includegraphics[width=\unitlength,page=72]{proofXiAlg.pdf}}%
    \put(0.43048394,0.04293032){\color[rgb]{0,0,0}\makebox(0,0)[lt]{\lineheight{1.25}\smash{\begin{tabular}[t]{l}$\pi_A$\end{tabular}}}}%
    \put(0,0){\includegraphics[width=\unitlength,page=73]{proofXiAlg.pdf}}%
    \put(0.44889499,0.0741006){\color[rgb]{0,0,0}\makebox(0,0)[lt]{\lineheight{1.25}\smash{\begin{tabular}[t]{l}${\scriptscriptstyle \underline{\mathrm{End}}(A)}$\end{tabular}}}}%
    \put(0,0){\includegraphics[width=\unitlength,page=74]{proofXiAlg.pdf}}%
    \put(0.43340421,0.00227763){\color[rgb]{0,0,0}\makebox(0,0)[lt]{\lineheight{1.25}\smash{\begin{tabular}[t]{l}$_{\mathcal{A}_{\mathcal{M}}}$\end{tabular}}}}%
    \put(0,0){\includegraphics[width=\unitlength,page=75]{proofXiAlg.pdf}}%
    \put(0.46065533,0.1075076){\color[rgb]{0,0,0}\makebox(0,0)[lt]{\lineheight{1.25}\smash{\begin{tabular}[t]{l}$\underline{\mathsf{ev}}_{A,A}$\end{tabular}}}}%
    \put(0,0){\includegraphics[width=\unitlength,page=76]{proofXiAlg.pdf}}%
    \put(0.48927547,0.13589552){\color[rgb]{0,0,0}\makebox(0,0)[lt]{\lineheight{1.25}\smash{\begin{tabular}[t]{l}$_A$\end{tabular}}}}%
    \put(0,0){\includegraphics[width=\unitlength,page=77]{proofXiAlg.pdf}}%
    \put(0.51228986,0.04183743){\color[rgb]{0,0,0}\makebox(0,0)[lt]{\lineheight{1.25}\smash{\begin{tabular}[t]{l}$1_A$\end{tabular}}}}%
    \put(0,0){\includegraphics[width=\unitlength,page=78]{proofXiAlg.pdf}}%
    \put(0.5291486,0.07410055){\color[rgb]{0,0,0}\makebox(0,0)[lt]{\lineheight{1.25}\smash{\begin{tabular}[t]{l}$_A$\end{tabular}}}}%
    \put(0,0){\includegraphics[width=\unitlength,page=79]{proofXiAlg.pdf}}%
    \put(0.70993559,0.22700824){\color[rgb]{0,0,0}\makebox(0,0)[lt]{\lineheight{1.25}\smash{\begin{tabular}[t]{l}$m_A$\end{tabular}}}}%
    \put(0,0){\includegraphics[width=\unitlength,page=80]{proofXiAlg.pdf}}%
    \put(0.71464997,0.27238909){\color[rgb]{0,0,0}\makebox(0,0)[lt]{\lineheight{1.25}\smash{\begin{tabular}[t]{l}$_A$\end{tabular}}}}%
    \put(0,0){\includegraphics[width=\unitlength,page=81]{proofXiAlg.pdf}}%
    \put(0.7274225,0.09910788){\color[rgb]{0,0,0}\makebox(0,0)[lt]{\lineheight{1.25}\smash{\begin{tabular}[t]{l}$\pi_A$\end{tabular}}}}%
    \put(0,0){\includegraphics[width=\unitlength,page=82]{proofXiAlg.pdf}}%
    \put(0.74583344,0.13027816){\color[rgb]{0,0,0}\makebox(0,0)[lt]{\lineheight{1.25}\smash{\begin{tabular}[t]{l}${\scriptscriptstyle \underline{\mathrm{End}}(A)}$\end{tabular}}}}%
    \put(0,0){\includegraphics[width=\unitlength,page=83]{proofXiAlg.pdf}}%
    \put(0.73034262,0.05845519){\color[rgb]{0,0,0}\makebox(0,0)[lt]{\lineheight{1.25}\smash{\begin{tabular}[t]{l}$_{\mathcal{A}_{\mathcal{M}}}$\end{tabular}}}}%
    \put(0,0){\includegraphics[width=\unitlength,page=84]{proofXiAlg.pdf}}%
    \put(0.75759383,0.16368515){\color[rgb]{0,0,0}\makebox(0,0)[lt]{\lineheight{1.25}\smash{\begin{tabular}[t]{l}$\underline{\mathsf{ev}}_{A,A}$\end{tabular}}}}%
    \put(0,0){\includegraphics[width=\unitlength,page=85]{proofXiAlg.pdf}}%
    \put(0.78621392,0.19207306){\color[rgb]{0,0,0}\makebox(0,0)[lt]{\lineheight{1.25}\smash{\begin{tabular}[t]{l}$_A$\end{tabular}}}}%
    \put(0,0){\includegraphics[width=\unitlength,page=86]{proofXiAlg.pdf}}%
    \put(0.80922835,0.09801499){\color[rgb]{0,0,0}\makebox(0,0)[lt]{\lineheight{1.25}\smash{\begin{tabular}[t]{l}$1_A$\end{tabular}}}}%
    \put(0,0){\includegraphics[width=\unitlength,page=87]{proofXiAlg.pdf}}%
    \put(0.82608705,0.13027811){\color[rgb]{0,0,0}\makebox(0,0)[lt]{\lineheight{1.25}\smash{\begin{tabular}[t]{l}$_A$\end{tabular}}}}%
    \put(0,0){\includegraphics[width=\unitlength,page=88]{proofXiAlg.pdf}}%
    \put(0.5909913,0.09910784){\color[rgb]{0,0,0}\makebox(0,0)[lt]{\lineheight{1.25}\smash{\begin{tabular}[t]{l}$\pi_A$\end{tabular}}}}%
    \put(0,0){\includegraphics[width=\unitlength,page=89]{proofXiAlg.pdf}}%
    \put(0.60940231,0.13027816){\color[rgb]{0,0,0}\makebox(0,0)[lt]{\lineheight{1.25}\smash{\begin{tabular}[t]{l}${\scriptscriptstyle \underline{\mathrm{End}}(A)}$\end{tabular}}}}%
    \put(0,0){\includegraphics[width=\unitlength,page=90]{proofXiAlg.pdf}}%
    \put(0.59391153,0.05845519){\color[rgb]{0,0,0}\makebox(0,0)[lt]{\lineheight{1.25}\smash{\begin{tabular}[t]{l}$_{\mathcal{A}_{\mathcal{M}}}$\end{tabular}}}}%
    \put(0,0){\includegraphics[width=\unitlength,page=91]{proofXiAlg.pdf}}%
    \put(0.62116265,0.16368515){\color[rgb]{0,0,0}\makebox(0,0)[lt]{\lineheight{1.25}\smash{\begin{tabular}[t]{l}$\underline{\mathsf{ev}}_{A,A}$\end{tabular}}}}%
    \put(0,0){\includegraphics[width=\unitlength,page=92]{proofXiAlg.pdf}}%
    \put(0.64978279,0.19207304){\color[rgb]{0,0,0}\makebox(0,0)[lt]{\lineheight{1.25}\smash{\begin{tabular}[t]{l}$_A$\end{tabular}}}}%
    \put(0,0){\includegraphics[width=\unitlength,page=93]{proofXiAlg.pdf}}%
    \put(0.67279715,0.09801496){\color[rgb]{0,0,0}\makebox(0,0)[lt]{\lineheight{1.25}\smash{\begin{tabular}[t]{l}$1_A$\end{tabular}}}}%
    \put(0,0){\includegraphics[width=\unitlength,page=94]{proofXiAlg.pdf}}%
    \put(0.68965588,0.13027811){\color[rgb]{0,0,0}\makebox(0,0)[lt]{\lineheight{1.25}\smash{\begin{tabular}[t]{l}$_A$\end{tabular}}}}%
    \put(0,0){\includegraphics[width=\unitlength,page=95]{proofXiAlg.pdf}}%
    \put(0.93464583,0.18688146){\color[rgb]{0,0,0}\makebox(0,0)[lt]{\lineheight{1.25}\smash{\begin{tabular}[t]{l}$m_A$\end{tabular}}}}%
    \put(0,0){\includegraphics[width=\unitlength,page=96]{proofXiAlg.pdf}}%
    \put(0.93936033,0.23226225){\color[rgb]{0,0,0}\makebox(0,0)[lt]{\lineheight{1.25}\smash{\begin{tabular}[t]{l}$_A$\end{tabular}}}}%
    \put(0,0){\includegraphics[width=\unitlength,page=97]{proofXiAlg.pdf}}%
    \put(0.90454782,0.1216496){\color[rgb]{0,0,0}\makebox(0,0)[lt]{\lineheight{1.25}\smash{\begin{tabular}[t]{l}$\xi$\end{tabular}}}}%
    \put(0,0){\includegraphics[width=\unitlength,page=98]{proofXiAlg.pdf}}%
    \put(0.91488407,0.15388688){\color[rgb]{0,0,0}\makebox(0,0)[lt]{\lineheight{1.25}\smash{\begin{tabular}[t]{l}$_A$\end{tabular}}}}%
    \put(0,0){\includegraphics[width=\unitlength,page=99]{proofXiAlg.pdf}}%
    \put(0.96875078,0.1216496){\color[rgb]{0,0,0}\makebox(0,0)[lt]{\lineheight{1.25}\smash{\begin{tabular}[t]{l}$\xi$\end{tabular}}}}%
    \put(0,0){\includegraphics[width=\unitlength,page=100]{proofXiAlg.pdf}}%
    \put(0.97908689,0.15388688){\color[rgb]{0,0,0}\makebox(0,0)[lt]{\lineheight{1.25}\smash{\begin{tabular}[t]{l}$_A$\end{tabular}}}}%
    \put(0,0){\includegraphics[width=\unitlength,page=101]{proofXiAlg.pdf}}%
    \put(0.89887529,0.08253128){\color[rgb]{0,0,0}\makebox(0,0)[lt]{\lineheight{1.25}\smash{\begin{tabular}[t]{l}$_{\mathcal{A}_{\mathcal{M}}}$\end{tabular}}}}%
    \put(0,0){\includegraphics[width=\unitlength,page=102]{proofXiAlg.pdf}}%
    \put(0.9630782,0.08253128){\color[rgb]{0,0,0}\makebox(0,0)[lt]{\lineheight{1.25}\smash{\begin{tabular}[t]{l}$_{\mathcal{A}_{\mathcal{M}}}$\end{tabular}}}}%
    \put(0.29264616,0.46995458){\color[rgb]{0,0,0}\makebox(0,0)[lt]{\lineheight{1.25}\smash{\begin{tabular}[t]{l}$=$\end{tabular}}}}%
    \put(0.55456609,0.46999469){\color[rgb]{0,0,0}\makebox(0,0)[lt]{\lineheight{1.25}\smash{\begin{tabular}[t]{l}$=$\end{tabular}}}}%
    \put(-0.00080165,0.15962178){\color[rgb]{0,0,0}\makebox(0,0)[lt]{\lineheight{1.25}\smash{\begin{tabular}[t]{l}$=$\end{tabular}}}}%
    \put(0.25442266,0.15948256){\color[rgb]{0,0,0}\makebox(0,0)[lt]{\lineheight{1.25}\smash{\begin{tabular}[t]{l}$=$\end{tabular}}}}%
    \put(0.54579743,0.16009674){\color[rgb]{0,0,0}\makebox(0,0)[lt]{\lineheight{1.25}\smash{\begin{tabular}[t]{l}$=$\end{tabular}}}}%
    \put(0.85118585,0.16045627){\color[rgb]{0,0,0}\makebox(0,0)[lt]{\lineheight{1.25}\smash{\begin{tabular}[t]{l}$=$\end{tabular}}}}%
  \end{picture}%
\endgroup%

%% file: proofHalfBr.pdf_tex
\begingroup%
  \makeatletter%
  \providecommand\color[2][]{%
    \errmessage{(Inkscape) Color is used for the text in Inkscape, but the package 'color.sty' is not loaded}%
    \renewcommand\color[2][]{}%
  }%
  \providecommand\transparent[1]{%
    \errmessage{(Inkscape) Transparency is used (non-zero) for the text in Inkscape, but the package 'transparent.sty' is not loaded}%
    \renewcommand\transparent[1]{}%
  }%
  \providecommand\rotatebox[2]{#2}%
  \newcommand*\fsize{\dimexpr\f@size pt\relax}%
  \newcommand*\lineheight[1]{\fontsize{\fsize}{#1\fsize}\selectfont}%
  \ifx\svgwidth\undefined%
    \setlength{\unitlength}{471.15627169bp}%
    \ifx\svgscale\undefined%
      \relax%
    \else%
      \setlength{\unitlength}{\unitlength * \real{\svgscale}}%
    \fi%
  \else%
    \setlength{\unitlength}{\svgwidth}%
  \fi%
  \global\let\svgwidth\undefined%
  \global\let\svgscale\undefined%
  \makeatother%
  \begin{picture}(1,1.16821347)%
    \lineheight{1}%
    \setlength\tabcolsep{0pt}%
    \put(0,0){\includegraphics[width=\unitlength,page=1]{proofHalfBr.pdf}}%
    \put(0.13842631,1.10914176){\color[rgb]{0,0,0}\makebox(0,0)[lt]{\lineheight{1.25}\smash{\begin{tabular}[t]{l}$\pi_M$\end{tabular}}}}%
    \put(0,0){\includegraphics[width=\unitlength,page=2]{proofHalfBr.pdf}}%
    \put(0.15854477,1.07971817){\color[rgb]{0,0,0}\makebox(0,0)[lt]{\lineheight{1.25}\smash{\begin{tabular}[t]{l}$_{\mathcal{A}_{\mathcal{M}}}$\end{tabular}}}}%
    \put(0,0){\includegraphics[width=\unitlength,page=3]{proofHalfBr.pdf}}%
    \put(0.15095745,1.1535431){\color[rgb]{0,0,0}\makebox(0,0)[lt]{\lineheight{1.25}\smash{\begin{tabular}[t]{l}$_{\underline{\mathrm{End}}(M)}$\end{tabular}}}}%
    \put(0,0){\includegraphics[width=\unitlength,page=4]{proofHalfBr.pdf}}%
    \put(0.10706081,1.03868322){\color[rgb]{0,0,0}\makebox(0,0)[lt]{\lineheight{1.25}\smash{\begin{tabular}[t]{l}$b^{\mathsf{Id}}_X$\end{tabular}}}}%
    \put(0,0){\includegraphics[width=\unitlength,page=5]{proofHalfBr.pdf}}%
    \put(0.08425795,0.9334784){\color[rgb]{0,0,0}\makebox(0,0)[lt]{\lineheight{1.25}\smash{\begin{tabular}[t]{l}$_C$\end{tabular}}}}%
    \put(0,0){\includegraphics[width=\unitlength,page=6]{proofHalfBr.pdf}}%
    \put(0.08585203,1.15146444){\color[rgb]{0,0,0}\makebox(0,0)[lt]{\lineheight{1.25}\smash{\begin{tabular}[t]{l}$_X$\end{tabular}}}}%
    \put(0,0){\includegraphics[width=\unitlength,page=7]{proofHalfBr.pdf}}%
    \put(0.08040664,0.97343476){\color[rgb]{0,0,0}\makebox(0,0)[lt]{\lineheight{1.25}\smash{\begin{tabular}[t]{l}$u$\end{tabular}}}}%
    \put(0,0){\includegraphics[width=\unitlength,page=8]{proofHalfBr.pdf}}%
    \put(0.09428532,1.00415812){\color[rgb]{0,0,0}\makebox(0,0)[lt]{\lineheight{1.25}\smash{\begin{tabular}[t]{l}$_{\mathcal{A}_{\mathcal{M}}}$\end{tabular}}}}%
    \put(0,0){\includegraphics[width=\unitlength,page=9]{proofHalfBr.pdf}}%
    \put(0.1479311,0.9334784){\color[rgb]{0,0,0}\makebox(0,0)[lt]{\lineheight{1.25}\smash{\begin{tabular}[t]{l}$_X$\end{tabular}}}}%
    \put(0,0){\includegraphics[width=\unitlength,page=10]{proofHalfBr.pdf}}%
    \put(0.25157405,1.03715105){\color[rgb]{0,0,0}\makebox(0,0)[lt]{\lineheight{1.25}\smash{\begin{tabular}[t]{l}$\pi_{X M}$\end{tabular}}}}%
    \put(0,0){\includegraphics[width=\unitlength,page=11]{proofHalfBr.pdf}}%
    \put(0.27684742,1.00604118){\color[rgb]{0,0,0}\makebox(0,0)[lt]{\lineheight{1.25}\smash{\begin{tabular}[t]{l}$_{\mathcal{A}_{\mathcal{M}}}$\end{tabular}}}}%
    \put(0,0){\includegraphics[width=\unitlength,page=12]{proofHalfBr.pdf}}%
    \put(0.27559618,1.09966763){\color[rgb]{0,0,0}\makebox(0,0)[lt]{\lineheight{1.25}\smash{\begin{tabular}[t]{l}$J_{X,M,X,M}$\end{tabular}}}}%
    \put(0,0){\includegraphics[width=\unitlength,page=13]{proofHalfBr.pdf}}%
    \put(0.27585006,1.06860101){\color[rgb]{0,0,0}\makebox(0,0)[lt]{\lineheight{1.25}\smash{\begin{tabular}[t]{l}${\scriptscriptstyle \underline{\mathrm{End}}(X M)}$\end{tabular}}}}%
    \put(0,0){\includegraphics[width=\unitlength,page=14]{proofHalfBr.pdf}}%
    \put(0.26731823,0.9334784){\color[rgb]{0,0,0}\makebox(0,0)[lt]{\lineheight{1.25}\smash{\begin{tabular}[t]{l}$_C$\end{tabular}}}}%
    \put(0,0){\includegraphics[width=\unitlength,page=15]{proofHalfBr.pdf}}%
    \put(0.26346691,0.97343472){\color[rgb]{0,0,0}\makebox(0,0)[lt]{\lineheight{1.25}\smash{\begin{tabular}[t]{l}$u$\end{tabular}}}}%
    \put(0,0){\includegraphics[width=\unitlength,page=16]{proofHalfBr.pdf}}%
    \put(0.36282797,0.93347833){\color[rgb]{0,0,0}\makebox(0,0)[lt]{\lineheight{1.25}\smash{\begin{tabular}[t]{l}$_X$\end{tabular}}}}%
    \put(0,0){\includegraphics[width=\unitlength,page=17]{proofHalfBr.pdf}}%
    \put(0.34669094,1.14126872){\color[rgb]{0,0,0}\makebox(0,0)[lt]{\lineheight{1.25}\smash{\begin{tabular}[t]{l}$_{\underline{\mathrm{End}}(M)}$\end{tabular}}}}%
    \put(0,0){\includegraphics[width=\unitlength,page=18]{proofHalfBr.pdf}}%
    \put(0.26914997,1.14293738){\color[rgb]{0,0,0}\makebox(0,0)[lt]{\lineheight{1.25}\smash{\begin{tabular}[t]{l}$_X$\end{tabular}}}}%
    \put(0,0){\includegraphics[width=\unitlength,page=19]{proofHalfBr.pdf}}%
    \put(0.53028878,1.09966766){\color[rgb]{0,0,0}\makebox(0,0)[lt]{\lineheight{1.25}\smash{\begin{tabular}[t]{l}$J_{X,M,X,M}$\end{tabular}}}}%
    \put(0,0){\includegraphics[width=\unitlength,page=20]{proofHalfBr.pdf}}%
    \put(0.53022794,1.06837171){\color[rgb]{0,0,0}\makebox(0,0)[lt]{\lineheight{1.25}\smash{\begin{tabular}[t]{l}${\scriptscriptstyle \underline{\mathrm{End}}(X M)}$\end{tabular}}}}%
    \put(0,0){\includegraphics[width=\unitlength,page=21]{proofHalfBr.pdf}}%
    \put(0.60980681,1.14501827){\color[rgb]{0,0,0}\makebox(0,0)[lt]{\lineheight{1.25}\smash{\begin{tabular}[t]{l}$_{\underline{\mathrm{End}}(M)}$\end{tabular}}}}%
    \put(0,0){\includegraphics[width=\unitlength,page=22]{proofHalfBr.pdf}}%
    \put(0.52384253,1.14293743){\color[rgb]{0,0,0}\makebox(0,0)[lt]{\lineheight{1.25}\smash{\begin{tabular}[t]{l}$_X$\end{tabular}}}}%
    \put(0,0){\includegraphics[width=\unitlength,page=23]{proofHalfBr.pdf}}%
    \put(0.47060391,1.03752244){\color[rgb]{0,0,0}\makebox(0,0)[lt]{\lineheight{1.25}\smash{\begin{tabular}[t]{l}${\scriptstyle \underline{\Hom}(M,\,\beta_{X M})}$\end{tabular}}}}%
    \put(0,0){\includegraphics[width=\unitlength,page=24]{proofHalfBr.pdf}}%
    \put(0.50721258,1.00675676){\color[rgb]{0,0,0}\makebox(0,0)[lt]{\lineheight{1.25}\smash{\begin{tabular}[t]{l}${\scriptscriptstyle \underline{\Hom}(M,CXM)}$\end{tabular}}}}%
    \put(0,0){\includegraphics[width=\unitlength,page=25]{proofHalfBr.pdf}}%
    \put(0.48620638,0.97344536){\color[rgb]{0,0,0}\makebox(0,0)[lt]{\lineheight{1.25}\smash{\begin{tabular}[t]{l}${\scriptstyle \underline{\mathsf{coev}}_{C,\,XM}}$\end{tabular}}}}%
    \put(0,0){\includegraphics[width=\unitlength,page=26]{proofHalfBr.pdf}}%
    \put(0.52201081,0.9334784){\color[rgb]{0,0,0}\makebox(0,0)[lt]{\lineheight{1.25}\smash{\begin{tabular}[t]{l}$_C$\end{tabular}}}}%
    \put(0,0){\includegraphics[width=\unitlength,page=27]{proofHalfBr.pdf}}%
    \put(0.62547964,0.9334784){\color[rgb]{0,0,0}\makebox(0,0)[lt]{\lineheight{1.25}\smash{\begin{tabular}[t]{l}$_X$\end{tabular}}}}%
    \put(0,0){\includegraphics[width=\unitlength,page=28]{proofHalfBr.pdf}}%
    \put(0.79294056,1.11558595){\color[rgb]{0,0,0}\makebox(0,0)[lt]{\lineheight{1.25}\smash{\begin{tabular}[t]{l}$J_{X,M,X,M}$\end{tabular}}}}%
    \put(0,0){\includegraphics[width=\unitlength,page=29]{proofHalfBr.pdf}}%
    \put(0.79360707,1.08350233){\color[rgb]{0,0,0}\makebox(0,0)[lt]{\lineheight{1.25}\smash{\begin{tabular}[t]{l}${\scriptscriptstyle \underline{\mathrm{End}}(X M)}$\end{tabular}}}}%
    \put(0,0){\includegraphics[width=\unitlength,page=30]{proofHalfBr.pdf}}%
    \put(0.88041772,1.1609365){\color[rgb]{0,0,0}\makebox(0,0)[lt]{\lineheight{1.25}\smash{\begin{tabular}[t]{l}$_{\underline{\mathrm{End}}(M)}$\end{tabular}}}}%
    \put(0,0){\includegraphics[width=\unitlength,page=31]{proofHalfBr.pdf}}%
    \put(0.78649429,1.15885573){\color[rgb]{0,0,0}\makebox(0,0)[lt]{\lineheight{1.25}\smash{\begin{tabular}[t]{l}$_X$\end{tabular}}}}%
    \put(0,0){\includegraphics[width=\unitlength,page=32]{proofHalfBr.pdf}}%
    \put(0.73283667,1.05296496){\color[rgb]{0,0,0}\makebox(0,0)[lt]{\lineheight{1.25}\smash{\begin{tabular}[t]{l}${\scriptstyle \underline{\Hom}(M,\,X \beta_M)}$\end{tabular}}}}%
    \put(0,0){\includegraphics[width=\unitlength,page=33]{proofHalfBr.pdf}}%
    \put(0.76911562,0.95775679){\color[rgb]{0,0,0}\makebox(0,0)[lt]{\lineheight{1.25}\smash{\begin{tabular}[t]{l}${\scriptscriptstyle \underline{\Hom}(M,CXM)}$\end{tabular}}}}%
    \put(0,0){\includegraphics[width=\unitlength,page=34]{proofHalfBr.pdf}}%
    \put(0.7488581,0.92569055){\color[rgb]{0,0,0}\makebox(0,0)[lt]{\lineheight{1.25}\smash{\begin{tabular}[t]{l}${\scriptstyle \underline{\mathsf{coev}}_{C,\,XM}}$\end{tabular}}}}%
    \put(0,0){\includegraphics[width=\unitlength,page=35]{proofHalfBr.pdf}}%
    \put(0.78466255,0.88572352){\color[rgb]{0,0,0}\makebox(0,0)[lt]{\lineheight{1.25}\smash{\begin{tabular}[t]{l}$_C$\end{tabular}}}}%
    \put(0,0){\includegraphics[width=\unitlength,page=36]{proofHalfBr.pdf}}%
    \put(0.89609047,0.88572352){\color[rgb]{0,0,0}\makebox(0,0)[lt]{\lineheight{1.25}\smash{\begin{tabular}[t]{l}$_X$\end{tabular}}}}%
    \put(0,0){\includegraphics[width=\unitlength,page=37]{proofHalfBr.pdf}}%
    \put(0.73315617,0.98883429){\color[rgb]{0,0,0}\makebox(0,0)[lt]{\lineheight{1.25}\smash{\begin{tabular}[t]{l}${\scriptstyle \underline{\Hom}(M,\,h_X M)}$\end{tabular}}}}%
    \put(0,0){\includegraphics[width=\unitlength,page=38]{proofHalfBr.pdf}}%
    \put(0.76986429,1.02037554){\color[rgb]{0,0,0}\makebox(0,0)[lt]{\lineheight{1.25}\smash{\begin{tabular}[t]{l}${\scriptscriptstyle \underline{\Hom}(M,XCM)}$\end{tabular}}}}%
    \put(0,0){\includegraphics[width=\unitlength,page=39]{proofHalfBr.pdf}}%
    \put(0.10845433,0.67783321){\color[rgb]{0,0,0}\makebox(0,0)[lt]{\lineheight{1.25}\smash{\begin{tabular}[t]{l}$J_{X,M,X,CM}$\end{tabular}}}}%
    \put(0,0){\includegraphics[width=\unitlength,page=40]{proofHalfBr.pdf}}%
    \put(0.10174857,0.64647468){\color[rgb]{0,0,0}\makebox(0,0)[lt]{\lineheight{1.25}\smash{\begin{tabular}[t]{l}${\scriptscriptstyle \underline{\Hom}(X M,\,XCM)}$\end{tabular}}}}%
    \put(0,0){\includegraphics[width=\unitlength,page=41]{proofHalfBr.pdf}}%
    \put(0.172054,0.78685665){\color[rgb]{0,0,0}\makebox(0,0)[lt]{\lineheight{1.25}\smash{\begin{tabular}[t]{l}$_{\underline{\mathrm{End}}(M)}$\end{tabular}}}}%
    \put(0,0){\includegraphics[width=\unitlength,page=42]{proofHalfBr.pdf}}%
    \put(0.09404883,0.77681692){\color[rgb]{0,0,0}\makebox(0,0)[lt]{\lineheight{1.25}\smash{\begin{tabular}[t]{l}$_X$\end{tabular}}}}%
    \put(0,0){\includegraphics[width=\unitlength,page=43]{proofHalfBr.pdf}}%
    \put(0.14018485,0.74316196){\color[rgb]{0,0,0}\makebox(0,0)[lt]{\lineheight{1.25}\smash{\begin{tabular}[t]{l}${\scriptstyle \underline{\Hom}(M,\,\beta_M)}$\end{tabular}}}}%
    \put(0,0){\includegraphics[width=\unitlength,page=44]{proofHalfBr.pdf}}%
    \put(0.10147936,0.58248597){\color[rgb]{0,0,0}\makebox(0,0)[lt]{\lineheight{1.25}\smash{\begin{tabular}[t]{l}${\scriptscriptstyle \underline{\Hom}(M,CXM)}$\end{tabular}}}}%
    \put(0,0){\includegraphics[width=\unitlength,page=45]{proofHalfBr.pdf}}%
    \put(0.08279576,0.55189496){\color[rgb]{0,0,0}\makebox(0,0)[lt]{\lineheight{1.25}\smash{\begin{tabular}[t]{l}${\scriptstyle \underline{\mathsf{coev}}_{C,\,XM}}$\end{tabular}}}}%
    \put(0,0){\includegraphics[width=\unitlength,page=46]{proofHalfBr.pdf}}%
    \put(0.11609456,0.51164386){\color[rgb]{0,0,0}\makebox(0,0)[lt]{\lineheight{1.25}\smash{\begin{tabular}[t]{l}$_C$\end{tabular}}}}%
    \put(0,0){\includegraphics[width=\unitlength,page=47]{proofHalfBr.pdf}}%
    \put(0.22752247,0.51164386){\color[rgb]{0,0,0}\makebox(0,0)[lt]{\lineheight{1.25}\smash{\begin{tabular}[t]{l}$_X$\end{tabular}}}}%
    \put(0,0){\includegraphics[width=\unitlength,page=48]{proofHalfBr.pdf}}%
    \put(0.0667168,0.61459055){\color[rgb]{0,0,0}\makebox(0,0)[lt]{\lineheight{1.25}\smash{\begin{tabular}[t]{l}${\scriptstyle \underline{\Hom}(M,\,h_X M)}$\end{tabular}}}}%
    \put(0,0){\includegraphics[width=\unitlength,page=49]{proofHalfBr.pdf}}%
    \put(0.18091064,0.70906106){\color[rgb]{0,0,0}\makebox(0,0)[lt]{\lineheight{1.25}\smash{\begin{tabular}[t]{l}${\scriptscriptstyle \underline{\Hom}(M,CM)}$\end{tabular}}}}%
    \put(0,0){\includegraphics[width=\unitlength,page=50]{proofHalfBr.pdf}}%
    \put(0.39035345,0.73334449){\color[rgb]{0,0,0}\makebox(0,0)[lt]{\lineheight{1.25}\smash{\begin{tabular}[t]{l}${\scriptstyle \underline{\Hom}(M,\,\mathrm{ev}_{X\,}CM)}$\end{tabular}}}}%
    \put(0,0){\includegraphics[width=\unitlength,page=51]{proofHalfBr.pdf}}%
    \put(0.3717559,0.63586253){\color[rgb]{0,0,0}\makebox(0,0)[lt]{\lineheight{1.25}\smash{\begin{tabular}[t]{l}${\scriptscriptstyle \underline{\Hom}(M, \,X^* \underline{\Hom}(XM,XCM) X M)}$\end{tabular}}}}%
    \put(0,0){\includegraphics[width=\unitlength,page=52]{proofHalfBr.pdf}}%
    \put(0.37039267,0.60747488){\color[rgb]{0,0,0}\makebox(0,0)[lt]{\lineheight{1.25}\smash{\begin{tabular}[t]{l}${\scriptstyle \underline{\mathsf{coev}}_{X^* \underline{\Hom}(XM,XCM) X, M}}$\end{tabular}}}}%
    \put(0,0){\includegraphics[width=\unitlength,page=53]{proofHalfBr.pdf}}%
    \put(0.47732971,0.70176654){\color[rgb]{0,0,0}\makebox(0,0)[lt]{\lineheight{1.25}\smash{\begin{tabular}[t]{l}${\scriptscriptstyle \underline{\Hom}(M,X^*XCM)}$\end{tabular}}}}%
    \put(0,0){\includegraphics[width=\unitlength,page=54]{proofHalfBr.pdf}}%
    \put(0.38179522,0.67035291){\color[rgb]{0,0,0}\makebox(0,0)[lt]{\lineheight{1.25}\smash{\begin{tabular}[t]{l}${\scriptstyle \underline{\Hom}(M,X^* \underline{\mathsf{ev}}_{XM,\,XCM})}$\end{tabular}}}}%
    \put(0,0){\includegraphics[width=\unitlength,page=55]{proofHalfBr.pdf}}%
    \put(0.31762787,0.83517729){\color[rgb]{0,0,0}\makebox(0,0)[lt]{\lineheight{1.25}\smash{\begin{tabular}[t]{l}$_X$\end{tabular}}}}%
    \put(0,0){\includegraphics[width=\unitlength,page=56]{proofHalfBr.pdf}}%
    \put(0.450624,0.84257064){\color[rgb]{0,0,0}\makebox(0,0)[lt]{\lineheight{1.25}\smash{\begin{tabular}[t]{l}$_{\underline{\mathrm{End}}(M)}$\end{tabular}}}}%
    \put(0,0){\includegraphics[width=\unitlength,page=57]{proofHalfBr.pdf}}%
    \put(0.41792178,0.79852655){\color[rgb]{0,0,0}\makebox(0,0)[lt]{\lineheight{1.25}\smash{\begin{tabular}[t]{l}${\scriptstyle \underline{\Hom}(M,\,\beta_M)}$\end{tabular}}}}%
    \put(0,0){\includegraphics[width=\unitlength,page=58]{proofHalfBr.pdf}}%
    \put(0.47566064,0.76320651){\color[rgb]{0,0,0}\makebox(0,0)[lt]{\lineheight{1.25}\smash{\begin{tabular}[t]{l}${\scriptscriptstyle \underline{\Hom}(M,CM)}$\end{tabular}}}}%
    \put(0,0){\includegraphics[width=\unitlength,page=59]{proofHalfBr.pdf}}%
    \put(0.43513328,0.57532429){\color[rgb]{0,0,0}\makebox(0,0)[lt]{\lineheight{1.25}\smash{\begin{tabular}[t]{l}${\scriptscriptstyle \underline{\Hom}(X M,\,XCM)}$\end{tabular}}}}%
    \put(0,0){\includegraphics[width=\unitlength,page=60]{proofHalfBr.pdf}}%
    \put(0.43597881,0.50977236){\color[rgb]{0,0,0}\makebox(0,0)[lt]{\lineheight{1.25}\smash{\begin{tabular}[t]{l}${\scriptscriptstyle \underline{\Hom}(M,CXM)}$\end{tabular}}}}%
    \put(0,0){\includegraphics[width=\unitlength,page=61]{proofHalfBr.pdf}}%
    \put(0.43524872,0.48021161){\color[rgb]{0,0,0}\makebox(0,0)[lt]{\lineheight{1.25}\smash{\begin{tabular}[t]{l}${\scriptstyle \underline{\mathsf{coev}}_{C,\,XM}}$\end{tabular}}}}%
    \put(0,0){\includegraphics[width=\unitlength,page=62]{proofHalfBr.pdf}}%
    \put(0.46629698,0.44001157){\color[rgb]{0,0,0}\makebox(0,0)[lt]{\lineheight{1.25}\smash{\begin{tabular}[t]{l}$_C$\end{tabular}}}}%
    \put(0,0){\includegraphics[width=\unitlength,page=63]{proofHalfBr.pdf}}%
    \put(0.5777248,0.44001157){\color[rgb]{0,0,0}\makebox(0,0)[lt]{\lineheight{1.25}\smash{\begin{tabular}[t]{l}$_X$\end{tabular}}}}%
    \put(0,0){\includegraphics[width=\unitlength,page=64]{proofHalfBr.pdf}}%
    \put(0.41491124,0.54343629){\color[rgb]{0,0,0}\makebox(0,0)[lt]{\lineheight{1.25}\smash{\begin{tabular}[t]{l}${\scriptstyle \underline{\Hom}(M,\,h_X M)}$\end{tabular}}}}%
    \put(0,0){\includegraphics[width=\unitlength,page=65]{proofHalfBr.pdf}}%
    \put(0.76664817,0.73431738){\color[rgb]{0,0,0}\makebox(0,0)[lt]{\lineheight{1.25}\smash{\begin{tabular}[t]{l}${\scriptstyle \underline{\Hom}(M,\,\mathrm{ev}_{X\,} CM)}$\end{tabular}}}}%
    \put(0,0){\includegraphics[width=\unitlength,page=66]{proofHalfBr.pdf}}%
    \put(0.73787641,0.57218933){\color[rgb]{0,0,0}\makebox(0,0)[lt]{\lineheight{1.25}\smash{\begin{tabular}[t]{l}${\scriptscriptstyle \underline{\Hom}(M, \,X^* \underline{\Hom}(XM,CXM) X M)}$\end{tabular}}}}%
    \put(0,0){\includegraphics[width=\unitlength,page=67]{proofHalfBr.pdf}}%
    \put(0.73923215,0.5442059){\color[rgb]{0,0,0}\makebox(0,0)[lt]{\lineheight{1.25}\smash{\begin{tabular}[t]{l}${\scriptstyle \underline{\mathsf{coev}}_{X^* \underline{\Hom}(XM,CXM) X, M}}$\end{tabular}}}}%
    \put(0,0){\includegraphics[width=\unitlength,page=68]{proofHalfBr.pdf}}%
    \put(0.81957278,0.70176666){\color[rgb]{0,0,0}\makebox(0,0)[lt]{\lineheight{1.25}\smash{\begin{tabular}[t]{l}${\scriptscriptstyle \underline{\Hom}(M,X^*XCM)}$\end{tabular}}}}%
    \put(0,0){\includegraphics[width=\unitlength,page=69]{proofHalfBr.pdf}}%
    \put(0.74870131,0.67085251){\color[rgb]{0,0,0}\makebox(0,0)[lt]{\lineheight{1.25}\smash{\begin{tabular}[t]{l}${\scriptstyle \underline{\Hom}(M,X^* \underline{\mathsf{ev}}_{XM,\,XCM})}$\end{tabular}}}}%
    \put(0,0){\includegraphics[width=\unitlength,page=70]{proofHalfBr.pdf}}%
    \put(0.66783005,0.84313642){\color[rgb]{0,0,0}\makebox(0,0)[lt]{\lineheight{1.25}\smash{\begin{tabular}[t]{l}$_X$\end{tabular}}}}%
    \put(0,0){\includegraphics[width=\unitlength,page=71]{proofHalfBr.pdf}}%
    \put(0.81674456,0.84257057){\color[rgb]{0,0,0}\makebox(0,0)[lt]{\lineheight{1.25}\smash{\begin{tabular}[t]{l}$_{\underline{\mathrm{End}}(M)}$\end{tabular}}}}%
    \put(0,0){\includegraphics[width=\unitlength,page=72]{proofHalfBr.pdf}}%
    \put(0.7845937,0.79841299){\color[rgb]{0,0,0}\makebox(0,0)[lt]{\lineheight{1.25}\smash{\begin{tabular}[t]{l}${\scriptstyle \underline{\Hom}(M,\, \beta_M)}$\end{tabular}}}}%
    \put(0,0){\includegraphics[width=\unitlength,page=73]{proofHalfBr.pdf}}%
    \put(0.84089521,0.7656231){\color[rgb]{0,0,0}\makebox(0,0)[lt]{\lineheight{1.25}\smash{\begin{tabular}[t]{l}${\scriptscriptstyle \underline{\Hom}(M,CM)}$\end{tabular}}}}%
    \put(0,0){\includegraphics[width=\unitlength,page=74]{proofHalfBr.pdf}}%
    \put(0.80177445,0.50968203){\color[rgb]{0,0,0}\makebox(0,0)[lt]{\lineheight{1.25}\smash{\begin{tabular}[t]{l}${\scriptscriptstyle \underline{\Hom}(M,CXM)}$\end{tabular}}}}%
    \put(0,0){\includegraphics[width=\unitlength,page=75]{proofHalfBr.pdf}}%
    \put(0.78069462,0.4799786){\color[rgb]{0,0,0}\makebox(0,0)[lt]{\lineheight{1.25}\smash{\begin{tabular}[t]{l}${\scriptstyle \underline{\mathsf{coev}}_{C,\,XM}}$\end{tabular}}}}%
    \put(0,0){\includegraphics[width=\unitlength,page=76]{proofHalfBr.pdf}}%
    \put(0.81649911,0.44001171){\color[rgb]{0,0,0}\makebox(0,0)[lt]{\lineheight{1.25}\smash{\begin{tabular}[t]{l}$_C$\end{tabular}}}}%
    \put(0,0){\includegraphics[width=\unitlength,page=77]{proofHalfBr.pdf}}%
    \put(0.92792709,0.44001171){\color[rgb]{0,0,0}\makebox(0,0)[lt]{\lineheight{1.25}\smash{\begin{tabular}[t]{l}$_X$\end{tabular}}}}%
    \put(0,0){\includegraphics[width=\unitlength,page=78]{proofHalfBr.pdf}}%
    \put(0.70719606,0.60529544){\color[rgb]{0,0,0}\makebox(0,0)[lt]{\lineheight{1.25}\smash{\begin{tabular}[t]{l}${\scriptstyle \underline{\Hom}(M,X^*\underline{\Hom}(XM,h_X M)  XM)}$\end{tabular}}}}%
    \put(0,0){\includegraphics[width=\unitlength,page=79]{proofHalfBr.pdf}}%
    \put(0.73787635,0.63586232){\color[rgb]{0,0,0}\makebox(0,0)[lt]{\lineheight{1.25}\smash{\begin{tabular}[t]{l}${\scriptscriptstyle \underline{\Hom}(M, \,X^* \underline{\Hom}(XM,XCM) X M)}$\end{tabular}}}}%
    \put(0,0){\includegraphics[width=\unitlength,page=80]{proofHalfBr.pdf}}%
    \put(0.12124658,0.31293127){\color[rgb]{0,0,0}\makebox(0,0)[lt]{\lineheight{1.25}\smash{\begin{tabular}[t]{l}${\scriptstyle \underline{\Hom}(M,\,\mathrm{ev}_{X\,} CM)}$\end{tabular}}}}%
    \put(0,0){\includegraphics[width=\unitlength,page=81]{proofHalfBr.pdf}}%
    \put(0.09318579,0.15035479){\color[rgb]{0,0,0}\makebox(0,0)[lt]{\lineheight{1.25}\smash{\begin{tabular}[t]{l}${\scriptscriptstyle \underline{\Hom}(M, \,X^* \underline{\Hom}(XM,CXM) X M)}$\end{tabular}}}}%
    \put(0,0){\includegraphics[width=\unitlength,page=82]{proofHalfBr.pdf}}%
    \put(0.1295177,0.05845172){\color[rgb]{0,0,0}\makebox(0,0)[lt]{\lineheight{1.25}\smash{\begin{tabular}[t]{l}${\scriptstyle \underline{\mathsf{coev}}_{X^* CX, M}}$\end{tabular}}}}%
    \put(0,0){\includegraphics[width=\unitlength,page=83]{proofHalfBr.pdf}}%
    \put(0.19080052,0.2799322){\color[rgb]{0,0,0}\makebox(0,0)[lt]{\lineheight{1.25}\smash{\begin{tabular}[t]{l}${\scriptscriptstyle \underline{\Hom}(M,X^*XCM)}$\end{tabular}}}}%
    \put(0,0){\includegraphics[width=\unitlength,page=84]{proofHalfBr.pdf}}%
    \put(0.09189908,0.18570679){\color[rgb]{0,0,0}\makebox(0,0)[lt]{\lineheight{1.25}\smash{\begin{tabular}[t]{l}${\scriptstyle \underline{\Hom}(M,X^* \underline{\mathsf{ev}}_{XM,\,CXM})}$\end{tabular}}}}%
    \put(0,0){\includegraphics[width=\unitlength,page=85]{proofHalfBr.pdf}}%
    \put(0.03905781,0.41334272){\color[rgb]{0,0,0}\makebox(0,0)[lt]{\lineheight{1.25}\smash{\begin{tabular}[t]{l}$_X$\end{tabular}}}}%
    \put(0,0){\includegraphics[width=\unitlength,page=86]{proofHalfBr.pdf}}%
    \put(0.172054,0.42073609){\color[rgb]{0,0,0}\makebox(0,0)[lt]{\lineheight{1.25}\smash{\begin{tabular}[t]{l}$_{\underline{\mathrm{End}}(M)}$\end{tabular}}}}%
    \put(0,0){\includegraphics[width=\unitlength,page=87]{proofHalfBr.pdf}}%
    \put(0.14265467,0.3771108){\color[rgb]{0,0,0}\makebox(0,0)[lt]{\lineheight{1.25}\smash{\begin{tabular}[t]{l}${\scriptstyle \underline{\Hom}(M,\,\beta_M)}$\end{tabular}}}}%
    \put(0,0){\includegraphics[width=\unitlength,page=88]{proofHalfBr.pdf}}%
    \put(0.19765335,0.34258966){\color[rgb]{0,0,0}\makebox(0,0)[lt]{\lineheight{1.25}\smash{\begin{tabular}[t]{l}${\scriptscriptstyle \underline{\Hom}(M,CM)}$\end{tabular}}}}%
    \put(0,0){\includegraphics[width=\unitlength,page=89]{proofHalfBr.pdf}}%
    \put(0.17364807,0.0892206){\color[rgb]{0,0,0}\makebox(0,0)[lt]{\lineheight{1.25}\smash{\begin{tabular}[t]{l}${\scriptscriptstyle \underline{\Hom}(M,X^*CXM)}$\end{tabular}}}}%
    \put(0,0){\includegraphics[width=\unitlength,page=90]{proofHalfBr.pdf}}%
    \put(0.07946264,0.12242958){\color[rgb]{0,0,0}\makebox(0,0)[lt]{\lineheight{1.25}\smash{\begin{tabular}[t]{l}${\scriptstyle \underline{\Hom}(M, X^*\underline{\mathsf{coev}}_{C,\,XM}XM)}$\end{tabular}}}}%
    \put(0,0){\includegraphics[width=\unitlength,page=91]{proofHalfBr.pdf}}%
    \put(0.18772689,0.00225883){\color[rgb]{0,0,0}\makebox(0,0)[lt]{\lineheight{1.25}\smash{\begin{tabular}[t]{l}$_C$\end{tabular}}}}%
    \put(0,0){\includegraphics[width=\unitlength,page=92]{proofHalfBr.pdf}}%
    \put(0.27527742,0.00225883){\color[rgb]{0,0,0}\makebox(0,0)[lt]{\lineheight{1.25}\smash{\begin{tabular}[t]{l}$_X$\end{tabular}}}}%
    \put(0,0){\includegraphics[width=\unitlength,page=93]{proofHalfBr.pdf}}%
    \put(0.11806991,0.24957971){\color[rgb]{0,0,0}\makebox(0,0)[lt]{\lineheight{1.25}\smash{\begin{tabular}[t]{l}${\scriptstyle \underline{\Hom}(M,X^*h_X M)}$\end{tabular}}}}%
    \put(0,0){\includegraphics[width=\unitlength,page=94]{proofHalfBr.pdf}}%
    \put(0.15685893,0.21402784){\color[rgb]{0,0,0}\makebox(0,0)[lt]{\lineheight{1.25}\smash{\begin{tabular}[t]{l}${\scriptscriptstyle \underline{\Hom}(M, \,X^* CXM)}$\end{tabular}}}}%
    \put(0,0){\includegraphics[width=\unitlength,page=95]{proofHalfBr.pdf}}%
    \put(0.42572005,0.23241724){\color[rgb]{0,0,0}\makebox(0,0)[lt]{\lineheight{1.25}\smash{\begin{tabular}[t]{l}${\scriptstyle \underline{\Hom}(M,\,\mathrm{ev}_{X\,} CM)}$\end{tabular}}}}%
    \put(0,0){\includegraphics[width=\unitlength,page=96]{proofHalfBr.pdf}}%
    \put(0.43669907,0.10648332){\color[rgb]{0,0,0}\makebox(0,0)[lt]{\lineheight{1.25}\smash{\begin{tabular}[t]{l}${\scriptstyle \underline{\mathsf{coev}}_{X^* CX, M}}$\end{tabular}}}}%
    \put(0,0){\includegraphics[width=\unitlength,page=97]{proofHalfBr.pdf}}%
    \put(0.46154052,0.2003408){\color[rgb]{0,0,0}\makebox(0,0)[lt]{\lineheight{1.25}\smash{\begin{tabular}[t]{l}${\scriptscriptstyle \underline{\Hom}(M,X^*XCM)}$\end{tabular}}}}%
    \put(0,0){\includegraphics[width=\unitlength,page=98]{proofHalfBr.pdf}}%
    \put(0.36538263,0.34171043){\color[rgb]{0,0,0}\makebox(0,0)[lt]{\lineheight{1.25}\smash{\begin{tabular}[t]{l}$_X$\end{tabular}}}}%
    \put(0,0){\includegraphics[width=\unitlength,page=99]{proofHalfBr.pdf}}%
    \put(0.47450136,0.34114465){\color[rgb]{0,0,0}\makebox(0,0)[lt]{\lineheight{1.25}\smash{\begin{tabular}[t]{l}$_{\underline{\mathrm{End}}(M)}$\end{tabular}}}}%
    \put(0,0){\includegraphics[width=\unitlength,page=100]{proofHalfBr.pdf}}%
    \put(0.44266366,0.29719803){\color[rgb]{0,0,0}\makebox(0,0)[lt]{\lineheight{1.25}\smash{\begin{tabular}[t]{l}${\scriptstyle \underline{\Hom}(M,\,\beta_M)}$\end{tabular}}}}%
    \put(0,0){\includegraphics[width=\unitlength,page=101]{proofHalfBr.pdf}}%
    \put(0.45969877,0.26365945){\color[rgb]{0,0,0}\makebox(0,0)[lt]{\lineheight{1.25}\smash{\begin{tabular}[t]{l}${\scriptscriptstyle \underline{\Hom}(M,CM)}$\end{tabular}}}}%
    \put(0,0){\includegraphics[width=\unitlength,page=102]{proofHalfBr.pdf}}%
    \put(0.45981251,0.13552427){\color[rgb]{0,0,0}\makebox(0,0)[lt]{\lineheight{1.25}\smash{\begin{tabular}[t]{l}${\scriptscriptstyle \underline{\Hom}(M,X^*CXM)}$\end{tabular}}}}%
    \put(0,0){\includegraphics[width=\unitlength,page=103]{proofHalfBr.pdf}}%
    \put(0.48221519,0.05001378){\color[rgb]{0,0,0}\makebox(0,0)[lt]{\lineheight{1.25}\smash{\begin{tabular}[t]{l}$_C$\end{tabular}}}}%
    \put(0,0){\includegraphics[width=\unitlength,page=104]{proofHalfBr.pdf}}%
    \put(0.54588831,0.05001378){\color[rgb]{0,0,0}\makebox(0,0)[lt]{\lineheight{1.25}\smash{\begin{tabular}[t]{l}$_X$\end{tabular}}}}%
    \put(0,0){\includegraphics[width=\unitlength,page=105]{proofHalfBr.pdf}}%
    \put(0.42051735,0.1699883){\color[rgb]{0,0,0}\makebox(0,0)[lt]{\lineheight{1.25}\smash{\begin{tabular}[t]{l}${\scriptstyle \underline{\Hom}(M,X^*h_X M)}$\end{tabular}}}}%
    \put(0,0){\includegraphics[width=\unitlength,page=106]{proofHalfBr.pdf}}%
    \put(0.74508356,0.19256632){\color[rgb]{0,0,0}\makebox(0,0)[lt]{\lineheight{1.25}\smash{\begin{tabular}[t]{l}$\underline{\mathsf{coev}}_{C, M}$\end{tabular}}}}%
    \put(0,0){\includegraphics[width=\unitlength,page=107]{proofHalfBr.pdf}}%
    \put(0.63599344,0.30191466){\color[rgb]{0,0,0}\makebox(0,0)[lt]{\lineheight{1.25}\smash{\begin{tabular}[t]{l}$_X$\end{tabular}}}}%
    \put(0,0){\includegraphics[width=\unitlength,page=108]{proofHalfBr.pdf}}%
    \put(0.89864511,0.29395548){\color[rgb]{0,0,0}\makebox(0,0)[lt]{\lineheight{1.25}\smash{\begin{tabular}[t]{l}$_X$\end{tabular}}}}%
    \put(0,0){\includegraphics[width=\unitlength,page=109]{proofHalfBr.pdf}}%
    \put(0.7610305,0.3013491){\color[rgb]{0,0,0}\makebox(0,0)[lt]{\lineheight{1.25}\smash{\begin{tabular}[t]{l}$_{\underline{\mathrm{End}}(M)}$\end{tabular}}}}%
    \put(0,0){\includegraphics[width=\unitlength,page=110]{proofHalfBr.pdf}}%
    \put(0.73421653,0.25689899){\color[rgb]{0,0,0}\makebox(0,0)[lt]{\lineheight{1.25}\smash{\begin{tabular}[t]{l}${\scriptstyle \underline{\Hom}(M,\,\beta_M)}$\end{tabular}}}}%
    \put(0,0){\includegraphics[width=\unitlength,page=111]{proofHalfBr.pdf}}%
    \put(0.76225997,0.22307472){\color[rgb]{0,0,0}\makebox(0,0)[lt]{\lineheight{1.25}\smash{\begin{tabular}[t]{l}${\scriptscriptstyle \underline{\Hom}(M,CM)}$\end{tabular}}}}%
    \put(0,0){\includegraphics[width=\unitlength,page=112]{proofHalfBr.pdf}}%
    \put(0.72894862,0.07389106){\color[rgb]{0,0,0}\makebox(0,0)[lt]{\lineheight{1.25}\smash{\begin{tabular}[t]{l}$_C$\end{tabular}}}}%
    \put(0,0){\includegraphics[width=\unitlength,page=113]{proofHalfBr.pdf}}%
    \put(0.77670355,0.07389106){\color[rgb]{0,0,0}\makebox(0,0)[lt]{\lineheight{1.25}\smash{\begin{tabular}[t]{l}$_X$\end{tabular}}}}%
    \put(0,0){\includegraphics[width=\unitlength,page=114]{proofHalfBr.pdf}}%
    \put(0.74572401,0.12801417){\color[rgb]{0,0,0}\makebox(0,0)[lt]{\lineheight{1.25}\smash{\begin{tabular}[t]{l}$h_X$\end{tabular}}}}%
    \put(0,0){\includegraphics[width=\unitlength,page=115]{proofHalfBr.pdf}}%
    \put(0.78664634,0.16209037){\color[rgb]{0,0,0}\makebox(0,0)[lt]{\lineheight{1.25}\smash{\begin{tabular}[t]{l}$_C$\end{tabular}}}}%
    \put(0,0){\includegraphics[width=\unitlength,page=116]{proofHalfBr.pdf}}%
    \put(0.89609053,0.08185019){\color[rgb]{0,0,0}\makebox(0,0)[lt]{\lineheight{1.25}\smash{\begin{tabular}[t]{l}$_C$\end{tabular}}}}%
    \put(0,0){\includegraphics[width=\unitlength,page=117]{proofHalfBr.pdf}}%
    \put(0.94384546,0.08185019){\color[rgb]{0,0,0}\makebox(0,0)[lt]{\lineheight{1.25}\smash{\begin{tabular}[t]{l}$_X$\end{tabular}}}}%
    \put(0,0){\includegraphics[width=\unitlength,page=118]{proofHalfBr.pdf}}%
    \put(0.91286605,0.12005502){\color[rgb]{0,0,0}\makebox(0,0)[lt]{\lineheight{1.25}\smash{\begin{tabular}[t]{l}$h_X$\end{tabular}}}}%
    \put(0,0){\includegraphics[width=\unitlength,page=119]{proofHalfBr.pdf}}%
    \put(0.94290614,0.1854797){\color[rgb]{0,0,0}\makebox(0,0)[lt]{\lineheight{1.25}\smash{\begin{tabular}[t]{l}$u$\end{tabular}}}}%
    \put(0,0){\includegraphics[width=\unitlength,page=120]{proofHalfBr.pdf}}%
    \put(0.9544406,0.21629911){\color[rgb]{0,0,0}\makebox(0,0)[lt]{\lineheight{1.25}\smash{\begin{tabular}[t]{l}$_{\mathcal{A}_{\mathcal{M}}}$\end{tabular}}}}%
    \put(0,0){\includegraphics[width=\unitlength,page=121]{proofHalfBr.pdf}}%
    \put(0.95378829,0.15413093){\color[rgb]{0,0,0}\makebox(0,0)[lt]{\lineheight{1.25}\smash{\begin{tabular}[t]{l}$_C$\end{tabular}}}}%
    \put(0,0){\includegraphics[width=\unitlength,page=122]{proofHalfBr.pdf}}%
    \put(0.93306107,0.250621){\color[rgb]{0,0,0}\makebox(0,0)[lt]{\lineheight{1.25}\smash{\begin{tabular}[t]{l}$\pi_M$\end{tabular}}}}%
    \put(0,0){\includegraphics[width=\unitlength,page=123]{proofHalfBr.pdf}}%
    \put(0.9281725,0.29339002){\color[rgb]{0,0,0}\makebox(0,0)[lt]{\lineheight{1.25}\smash{\begin{tabular}[t]{l}$_{\underline{\mathrm{End}}(M)}$\end{tabular}}}}%
    \put(0.18810958,1.03734662){\color[rgb]{0,0,0}\makebox(0,0)[lt]{\lineheight{1.25}\smash{\begin{tabular}[t]{l}$=$\end{tabular}}}}%
    \put(0.40065115,1.03712386){\color[rgb]{0,0,0}\makebox(0,0)[lt]{\lineheight{1.25}\smash{\begin{tabular}[t]{l}$=$\end{tabular}}}}%
    \put(0.66277288,1.03722168){\color[rgb]{0,0,0}\makebox(0,0)[lt]{\lineheight{1.25}\smash{\begin{tabular}[t]{l}$=$\end{tabular}}}}%
    \put(0.0049584,0.64250584){\color[rgb]{0,0,0}\makebox(0,0)[lt]{\lineheight{1.25}\smash{\begin{tabular}[t]{l}$=$\end{tabular}}}}%
    \put(0.27346026,0.64063592){\color[rgb]{0,0,0}\makebox(0,0)[lt]{\lineheight{1.25}\smash{\begin{tabular}[t]{l}$=$\end{tabular}}}}%
    \put(0.62380544,0.64257497){\color[rgb]{0,0,0}\makebox(0,0)[lt]{\lineheight{1.25}\smash{\begin{tabular}[t]{l}$=$\end{tabular}}}}%
    \put(-0.0009854,0.19809628){\color[rgb]{0,0,0}\makebox(0,0)[lt]{\lineheight{1.25}\smash{\begin{tabular}[t]{l}$=$\end{tabular}}}}%
    \put(0.32184355,0.19656327){\color[rgb]{0,0,0}\makebox(0,0)[lt]{\lineheight{1.25}\smash{\begin{tabular}[t]{l}$=$\end{tabular}}}}%
    \put(0.60232809,0.19686705){\color[rgb]{0,0,0}\makebox(0,0)[lt]{\lineheight{1.25}\smash{\begin{tabular}[t]{l}$=$\end{tabular}}}}%
    \put(0.86182263,0.19604844){\color[rgb]{0,0,0}\makebox(0,0)[lt]{\lineheight{1.25}\smash{\begin{tabular}[t]{l}$=$\end{tabular}}}}%
  \end{picture}%
\endgroup%

%% file: DY_Mod_Cat_arXiv_v1.bbl
\newpage





\begin{thebibliography}{99}
\bibitem[AM]{AM} N. Andruskiewitsch, J.~M. Mombelli, {\em On finite-dimensional Hopf algebras}, J. Algebra \textbf{314} (2007), no.~1, 383--418.

\bibitem[BBK]{BBK} M. Balodi, A. Banerjee, S. Kour, {\em Gerstenhaber type structures on Davydov--Yetter cohomology with coefficients}, arXiv:2508.02285 (2025).

\bibitem[BD]{BD} M. Batanin, A. Davydov, {\em Cosimplicial monoids and deformation theory of tensor categories}, J. Noncommut. Geom. \textbf{17} (2023), no.~4, 1167--1229.

\bibitem[BM]{BM} N. Bortolussi, M. Mombelli, {\em The character algebra for module categories over Hopf algebras}, Colloq. Math. \textbf{165} (2021), no.~2, 171--197.

\bibitem[BLV]{BLV} A. Brugui\`eres, S. Lack, A. Virelizier, {\em Hopf monads on monoidal categories}, Adv. Math. \textbf{227} (2011), no.~2, 745--800.

\bibitem[BV]{BV} A. Brugui\`eres, A. Virelizier, {\em Hopf monads}, Adv. Math. \textbf{215} (2007), no.~2, 679--733.

\bibitem[CMZ]{CMZ} K.~Coulembier, M.~Stroiński, T.~Zorman, 
{\em Simple algebras and exact module categories}, arXiv:2501.06629 (2025).

\bibitem[CY]{CY} L. Crane, D. N. Yetter, {\em Deformations of (bi)tensor categories}, Cahiers Topologie G\'eom. Diff\'erentielle Cat\'eg. \textbf{39} (1998), no.~3, 163--180.

\bibitem[Dav]{davydov} A. Davydov, {\em Twisting of monoidal structures}, arXiv:q-alg/9703001 (1997).

\bibitem[Dav2]{davCenter} A. Davydov, {\em Centre of an algebra}, Adv. Math. \textbf{225} (2010), no.~1, 319--348.

\bibitem[DSPS]{DSPS} C.L. Douglas, C. Schommer-Pries, N. Snyder, {\em The balanced tensor product of module categories}, Kyoto J. Math. \textbf{59} (2019), no.~1, 167--179.

\bibitem[EGNO]{EGNO} P. Etingof, S. Gelaki, D. Nikshych, V. Ostrik, {\em Tensor categories}, Math. Surveys Monogr. \textbf{205}. American Mathematical Society, Providence, RI, 2015.

\bibitem[EO]{EO} P. Etingof, V. Ostrik, {\em Finite tensor categories}, Mosc. Math. J. \textbf{4} (2004), no.~3, 627–654.

\bibitem[FGS]{FGS} M. Faitg, A.M. Gainutdinov, C. Schweigert, {\em Davydov--Yetter cohomology and relative homological algebra}, Selecta Math. (N.S.) \textbf{30} (2024), no.~2, Paper no.~26, 80 pp.

\bibitem[FGS2]{FGS2} M. Faitg, A. M. Gainutdinov, C. Schweigert, {\em An adjunction theorem for Davydov--Yetter cohomology and infinitesimal braidings}, arXiv:2411.19111, 71 pp. (2024).

\bibitem[FFRS]{fjfrs2}
J.~Fjelstad, J.~Fuchs, I.~Runkel, C.~Schweigert,
{\em Uniqueness of open / closed rational CFT with given algebra of open states,}
Adv. Theor. Math. Phys. \textbf{12} (2008), no.~6, 1283--1375.

\bibitem[FRS]{fuRs4} 
J.\ Fuchs, I.\ Runkel, C.\ Schweigert,
{\em TFT construction of RCFT correlators I: Partition functions}, Nucl.\ Phys.\ B {\bf 646} (2002), no.~3, 353–497.

\bibitem[FSS]{FSS} J. Fuchs, G. Schaumann, C. Schweigert, {\em  Eilenberg-Watts calculus for finite categories and a bimodule Radford $S^4$ theorem}, Trans. Amer. Math. Soc. \textbf{373} (2020), no.~1, 1--40.

\bibitem[FS]{fuSc27} J.~Fuchs, C.~Schweigert, {\em Bulk from boundary in finite CFT by means of pivotal module categories}, Nucl. Phys. B \textbf{967} (2021), Paper no.~115392, 38 pp.

\bibitem[FSV]{fusV}  
J.\ Fuchs, C.\ Schweigert,  A.\ Valentino,   
{\em Bicategories for boundary conditions and for surface defects in 3-d TFT}, Commun. Math. Phys. \textbf{321} (2013), no.~2, 543–575.
           

\bibitem[GHS]{GHS} A.M. Gainutdinov, J. Haferkamp, C. Schweigert, {\em Davydov-Yetter cohomology, comonads and Ocneanu rigidity}, Adv. Math. \textbf{414} (2023), Paper no.~108853, 48 pp.

\bibitem[GL]{GL} A.M. Gainutdinov, R. Laugwitz, {\em Fully exact and fully dualizable module categories}, arXiv:2601.22017.

\bibitem[GM]{GM} A.M. Gainutdinov, M. Mombelli, {\em The relative Nakayama functor}, in preparation.

\bibitem[Ger]{gerst} M. Gerstenhaber, {\em The cohomology structure of an associative ring}, Ann. of Math. 78 (1963), no.~2, 267--288.

\bibitem[GS]{GS} M. Gerstenhaber, S. D. Schack, {\em Algebras, bialgebras, quantum groups and algebraic deformations}, in: {\em Deformation theory and quantum groups with applications to mathematical physics (Amherst, MA, 1990)}, 51--92.
Contemp. Math. \textbf{134}, American Mathematical Society, Providence, RI, 1992.

\bibitem[Joh]{johnstone} P.T. Johnstone, {\em Adjoint lifting theorems for categories of algebras}, Bull. London Math. Soc. \textbf{7} (1975), no.~3, 294--297.

\bibitem[ML]{macLane} S. Mac Lane, {\em Homology}, Reprint of the 1975 edition. Classics Math, Springer-Verlag, Berlin, 1995.

\bibitem[ML2]{MLCat} S. Mac Lane, {\em Categories for the working mathematician}, second edition. Grad. Texts in Math. \textbf{5}, Springer-Verlag, New York, 1998.

\bibitem[Maj]{majidZF} S. Majid, {\em Representations, duals and quantum doubles of monoidal categories}, Proceedings
of the Winter School ``Geometry and Physics'', Rend. Circ. Mat. Palermo (2), Suppl. No. \textbf{26} (1991), 197--206.

\bibitem[Mom]{Mombelli} M. Mombelli, {\em Module categories over pointed Hopf algebras}, Math. Z. \textbf{266} (2010), no.~2, 319--344.


\bibitem[NSS]{NSS} D. Nakamura, T. Shibata, K. Shimizu, {\em Exact module categories over $\mathrm{Rep}(u_q(\mathfrak{sl}_2))$}, arxiv:2503.21265, 49pp (2025).


\bibitem[Sch]{Sch} P. Schauenburg, {\em Hopf-Galois and bi-Galois extensions}, in: {\em Galois theory, Hopf algebras, and semiabelian categories}, 469--515. Fields Inst. Commun. \textbf{43}, American Mathematical Society, Providence, RI, 2004. https://doi.org/10.1090/fic/043

\bibitem[Shi1]{shimizuMonCent} K. Shimizu, {\em The monoidal center and the character algebra}, J. Pure Appl. Algebra \textbf{221} (2017), no.~9, 2338--2371.

\bibitem[Shi2]{shimizuCoend} K. Shimizu, {\em Further results on the structure of (co)ends in finite tensor categories}, Appl. Categ. Structures \textbf{28} (2020), no.~2, 237--286.

\bibitem[Shi3]{shimizuSerre} K. Shimizu, {\em Relative Serre functor for comodule algebras}, J. Algebra \textbf{634} (2023), 237--305.

\bibitem[Shi4]{shimizuFrob} K. Shimizu, {\em Quasi-Frobenius algebras in finite tensor categories}, arXiv:2402.02929.

\bibitem[Wei]{weibel} C.A. Weibel, {\em An Introduction to Homological Algebra}, Cambridge Stud. Adv. Math. \textbf{38}, Cambridge University Press, 1994.

\bibitem[Wil]{willprecht} J.-O. Willprecht, {\em On deformations of module categories over finite tensor categories}, PhD thesis, University of Hamburg, 2023.
\\Available at \url{https://ediss.sub.uni-hamburg.de/handle/ediss/10651}.

\bibitem[Yet]{yetter1} D. N. Yetter, {\em Braided deformations of monoidal categories and Vassiliev invariants}, in: {\em Higher category theory (Evanston, IL, 1997)}, 117--134. Contemp. Math. \textbf{230}, American Mathematical Society, Providence, RI, 1998.
\end{thebibliography}
